\documentclass[a4paper,reqno,10pt]{amsart}
\usepackage[english]{babel}

\usepackage{amsfonts}
\usepackage{amsmath}
\usepackage{amsthm}
\usepackage{amssymb}
\usepackage{indentfirst}
\usepackage{color}
\usepackage{longtable}
\usepackage{eucal}
\usepackage{tcolorbox}
\usepackage[normalem]{ulem}
\usepackage[textsize=tiny]{todonotes}

\definecolor{blue_links}{RGB}{13,0,180} 

\usepackage{hyperref}
\hypersetup{
    colorlinks=true, 
    linktoc=all,     
    linkcolor=blue_links,  
    citecolor=blue_links,
    urlcolor=blue_links,
}

\usepackage[a4paper,top=3.5cm,bottom=1cm,left=2.5cm,right=2.5cm]{geometry}

\numberwithin{equation}{section}

\usepackage{subcaption}

\usepackage{mathrsfs}
\usepackage{paralist}
\usepackage{mathtools}
\usepackage{enumitem}
\usepackage[normalem]{ulem}
\allowdisplaybreaks

\definecolor{lmagenta!30}{rgb}{0.7,0,0.6}

\usepackage[english]{babel}
\usepackage{xcolor}
\usepackage{upgreek}

\theoremstyle{plain}
\begingroup
\theoremstyle{plain}
\newtheorem{theorem}{Theorem}[section]
\newtheorem{corollary}[theorem]{Corollary}
\newtheorem{proposition}[theorem]{Proposition}
\newtheorem{lemma}[theorem]{Lemma}
\theoremstyle{definition}
\newtheorem{definition}[theorem]{Definition}
\newtheorem{example}[theorem]{Example}
\theoremstyle{remark}
\newtheorem{remark}[theorem]{Remark}
\endgroup

\theoremstyle{definition}
\theoremstyle{remark}

\mathchardef\emptyset="001F

\definecolor{dred}{rgb}{1,0.1,0.5}
\definecolor{ddmagenta}{rgb}{0.7,0,0.9}
\definecolor{purple}{rgb}{0.4,0,0.9}
\definecolor{dmagenta}{rgb}{0.8,0,0.6}
\definecolor{ddcyan}{rgb}{0,0.2,1.0}
\definecolor{Orchid}{rgb}{0.3,0.2,0.9}
\definecolor{Turk}{rgb}{0,0.5,0.6}

\DeclareFontFamily{U}{mathb}{\hyphenchar\font45}
\DeclareFontShape{U}{mathb}{m}{n}{
<-6> mathb5 <6-7> mathb6 <7-8> mathb7
<8-9> mathb8 <9-10> mathb9
<10-12> mathb10 <12-> mathb12
}{}
\DeclareSymbolFont{mathb}{U}{mathb}{m}{n}
\DeclareMathSymbol{\llcurly}{\mathrel}{mathb}{"CE}
\DeclareMathSymbol{\ggcurly}{\mathrel}{mathb}{"CF}

\newcommand{\R}{\mathbb{R}}

\DeclareMathOperator{\dive}{div}
\newcommand{\supp}{\spt}

\newcommand{\G}{\mathcal{G}}
\newcommand{\T}{\mathcal{T}}
\newcommand{\di}{\mathrm{d}}

\newcommand{\sP}{\mathcal{P}}
\newcommand{\spt}{\mathrm{spt}}
\newcommand{\ttheta}{\boldsymbol{\theta}}
\newcommand{\rrho}{\boldsymbol{\rho}}

\newcommand{\onu}{\tilde{\boldsymbol{\nu}}}

\newcommand{\III}{{\boldsymbol{\mathrm  I}}}
\newcommand{\sfyy}{{\boldsymbol{\mathsf y}}}
\newcommand{\sfzz}{{\boldsymbol{\mathsf z}}}

\newcommand{\ARBV}{\mathsf{A}\!\mathsf{B}\!\mathsf{V}}
\newcommand{\ABV}{\mathsf{A}\!\mathsf{B}\!\mathsf{V}}

\newcommand{\sfe}{\mathsf{e}}
\newcommand\mres{\mathbin{\vrule height 1.6ex depth 0pt width
0.13ex\vrule height 0.13ex depth 0pt width 1.3ex}}

\newcommand{\QED}{\mbox{}\hfill\rule{5pt}{5pt}\medskip\par}

\newcommand{\N}{\mathbb{N}}

\newcommand{\coloneq }{\hspace{1pt}\raisebox{0.74pt}{\scalebox{0.8}{:}}\hspace{-2.2pt}=}

\newcommand{\norm}[1]{\lVert #1 \rVert}

\newcommand{\dd}{\mathrm{d}}

\newcommand{\loc}{\mathrm{loc}}

\newcommand{\Cc}{\mathrm{C}_{\mathrm{c}}}
\newcommand{\Cb}{\mathrm{C}_{\mathrm{b}}}
\newcommand{\rmC}{\mathrm{C}}
\newcommand{\rmD}{\mathrm{D}}
\newcommand{\rmV}{\mathrm{V}}
\newcommand{\rmv}{\mathrm{v}}
\newcommand{\GG}{\mathcal{G}}
\newcommand{\sfs}{\mathsf{s}}

\newcommand{\RRR}{\color{black}}

\newcommand{\PERME}?
\newcommand{\STE}{\color{black}}

\newcommand{\EEE}{\color{black}}
\newcommand{\eps}{\varepsilon}
\newcommand{\down}{\downarrow}
\newcommand{\BV}{\mathrm{BV}}
\newcommand{\BVloc}{\mathrm{BV}\kern-1pt_{\rm loc}}
\newcommand{\AC}{\mathrm{AC}}
\newcommand{\calMp}{\mathcal{M}_{\rm loc}^+}
\newcommand{\calMb}{\mathcal{M}} 
\newcommand{\calMbp}{\mathcal{M}^+}   

\newcommand{\Puno}{\mathcal{P}_1}
\newcommand{\scrP}{\mathcal{P}}
\newcommand{\foraa}{\text{for a.a.}}
\newcommand{\Var}{\mathrm{Var}}

\newcommand{\vertiii}[1]{{\left\vert\kern-0.25ex\left\vert\kern-0.25ex\left\vert #1 
    \right\vert\kern-0.25ex\right\vert\kern-0.25ex\right\vert}}

\newcommand{\sft}{\mathsf{t}}
\newcommand{\sfr}{\mathsf{r}}

\newcommand{\sfxx}{\boldsymbol{\mathsf{x}}}
\newcommand{\ppi}{\boldsymbol{\pi}}
\newcommand{\bvv}{\boldsymbol{v}}
\newcommand{\bww}{\boldsymbol{w}}

\newcommand{\nnu}{\boldsymbol{\nu}}
\newcommand{\mmu}{\boldsymbol{\mu}}
\newcommand{\llambda}{\boldsymbol{\lambda}}

\newcommand{\zzeta}{\boldsymbol{\zeta}}
\newcommand{\eeta}{\boldsymbol{\eta}}

\newcommand{\ssigma}{\boldsymbol{\sigma}}
\newcommand{\bvarphi}{\boldsymbol{\varphi}}
\newcommand{\weaksto}{\rightharpoonup^*}
\newcommand{\calB}{\mathcal{B}}
\newcommand{\calM}{\mathcal{M}_{\rm loc}}

\newcommand{\pinew}{\ppi}
\newcommand{\Prob}{\mathcal{P}}

\newcommand{\Lip}{\mathrm{Lip}}

\newcommand{\Rep}{\mathfrak{R}}
\newcommand{\btheta}{\boldsymbol{\theta}}

\newcommand{\mderw}[1]{|#1|_{\kern-1pt{W_1}}}

\newcommand{\sigmasel}{\mathfrak{s}}
\newcommand{\selbis}{\mathfrak{l}}
\newcommand{\selteta}{\mathfrak{r}}
\newcommand{\cZ}{\mathcal{Z}}

\newcommand{\llp}[3]{\ell_{#1}^{#2,#3}}

\newcommand{\FCOMM}{\color{black}}
\newcommand{\nc}{\normalcolor}
\newcommand{\NC}{\nc}

\newcommand{\neweta}{\widehat{\eeta}}

\definecolor{pink}{rgb}{0.7,0,0.6}
\definecolor{lmagenta}{rgb}{0.8,0.4,0.6}

\newcommand{\phisc}{\varphi_0}
\newcommand{\phive}{\boldsymbol{\varphi}}

\newcommand{\tetatrue}[1]{\vartheta_{#1}}

\newcommand{\bvc}{\mathsf{u}}
\newcommand{\bvcgr}{\mathfrak{G}}

\newcommand{\Spx}{\mathbf{X}}
\newcommand{\Traj}{\boldsymbol{\Xi}}
\newcommand{\traj}{\boldsymbol{\xi}}
\newcommand{\proxyteta}[1]{\uplambda_{#1}}
\newcommand{\Proxyteta}[1]{\Uplambda^{#1}}
\newcommand{\Proxysigma}[1]{\Upupsilon^{#1}}
\newcommand{\proxym}{\mathsf{m}}
\newcommand{\Vfield}{\boldsymbol{w}}

\newcommand{\DST}{{\mathbb R_+^{d+1}}}  
\newcommand{\DSTT}[1]{{[0,#1]\times \R^d}}  

 \newcommand{\Rdpu}{\R^{d+1}}
\newcommand{\DSTpu}{{\R_+^{d+2}}}
\newcommand{\DSTpuS}{{[0,S]\times \DST}}
\newcommand{\Domain}[1]{\mathrm E(#1)}

\newcommand{\Lipplus}[3]{\Lip^{\uparrow}_{#1}(#2;#3)}
\newcommand{\ALip}[2]{\mathrm{ArcLip}(#1;#2)}

\newcommand{\Lipplusname}[1]{\Lip^{\uparrow}_{#1}}
\newcommand{\Cplus}{\rmC^{\uparrow}(\III;\R^{d+1})}
\newcommand{\sfu}[0]{\mathsf u}

\newcommand{\fre}{\mathfrak e}
\newcommand{\fra}{\mathfrak a}
\newcommand{\frt}[0]{\mathfrak t}

\newcommand{\fry}[0]{\mathfrak y}
\newcommand{\frr}[0]{\mathfrak r}
\newcommand{\frxx}[0]{\boldsymbol{\mathfrak x}}

\newcommand{\teeta}[0]{\eeta_{\mathcal{L}}}
\newcommand{\pushm}{\mathsf{p}}
\newcommand{\etaflat}{\eeta_{\flat}}
\newcommand{\rflat}{\rho_{\flat}}
\newcommand{\BVS}{\mathfrak{V}}
\newcommand{\BVSm}{\mathfrak{v}}

\newcommand{\STEP}[1]{\noindent\underline{\emph{Step #1}}}

\newsavebox\MBox

\newtcolorbox{myboxpurple}{colback=purple!5!white,colframe=purple!75!black}
\newtcolorbox{myboxdred}{colback=red!5!white,colframe=red!75!black}

\newtcolorbox{myboxdmagenta}{colback=dmagenta!5!white,colframe=dmagenta!75!black}

\begin{document}

\title[]{The superposition principle for the continuity equation with singular flux}

\author[Stefano Almi]{Stefano Almi}
\address[Stefano Almi]{Dipartimento di Matematica ed Applicazioni ``R. Caccioppoli'', Universit\`a di Napoli Federico II, Via Cintia,  Monte S.~Angelo, 80126 Napoli, Italy}
\email[Stefano Almi]{stefano.almi@unina.it}
\author[Riccarda Rossi]{Riccarda Rossi}
\address[Riccarda Rossi]{DIMI, Universit\`a degli studi di  Brescia, Via Branze 38, 25133, Brescia, Italy}
\email[Riccarda Rossi]{riccarda.rossi@unibs.it}
\author[Giuseppe Savar\'e]{Giuseppe Savar\'e}
\address[Giuseppe Savar\'e]{Department of Decision Sciences and BIDSA, Universit\`a Bocconi, Via Roentgen 1, 20136, Milano, Italy}
\email[Giuseppe Savar\'e]{giuseppe.savare@unibocconi.it}

\maketitle


\begin{abstract}
Representation results for  absolutely continuous curves $\mu
\colon [0,T]\to \mathcal P_p(\R^d)$, $p>1$, with values in
the Wasserstein space $(\mathcal{P}_p(\R^d),W_p)$
of Borel probability measures in $\R^d$ with finite $p$-moment,
provide a crucial tool to study evolutionary PDEs in a measure-theoretic setting.
They are strictly related to the superposition principle for measure-valued solutions
to the continuity equation.
\par
This paper addresses the extension of these results to the case $p=1$, \emph{and} to curves
$\mu\colon   [0,+\infty)  \to \mathcal{P}_1(\R^d)$
that are only of \emph{bounded variation} in time:  in the corresponding continuity equation, the flux measure  $\nnu \in \mathcal{M}_{\rm loc} ([0, +\infty) \times \R^{d}; \R^{d})$  thus possesses a non-trivial singular part w.r.t.\ $\mu$ in addition to the 
absolutely continuous part featuring the velocity field.
\par
Firstly, we   carefully address the relation between curves in $ 
\BV_{\rm loc}  ([0, + \infty);\mathcal{P}_1(\R^d))$ and solutions  to  the associated continuity equation, among which we select those with \emph{minimal} singular (contribution to the) flux  $\nnu$.  We  show that, with those distinguished solutions it is possible to associate an `auxiliary' continuity equation, in an \emph{augmented} phase space, solely driven by its velocity field. For that continuity equation, a standard version of the superposition principle can be thus obtained. In this way, we derive a first probabilistic representation of the  pair $(\mu, \nnu)$ solutions by projection over the time and space marginals. 
 This representation involves Lipschitz trajectories in the  augmented  phase space, reparametrized in time and solving the characteristic system of ODEs. 
Finally, for  the same  pair $(\mu, \nnu)$  
we also prove a superposition principle in terms of $\BV$ curves on the \emph{actual} time interval,  providing   a fine  description of their behaviour at  jump points. 
\smallskip


\end{abstract} 

\tableofcontents

\section{Introduction}

 Representation results for Lipschitz (or even absolutely continuous) curves $\mu \colon [0,T]\to \mathcal P_p(\R^d)$, $p>1$, with values in
the Wasserstein space $(\mathcal P_p(\R^d),W_p)$ 
of Borel probability measures in $\R^d$ with finite $p$-moment,  metrized by the $L^p$-Kantorovich-Rubinstein-Wasserstein distance
\begin{equation}
    \label{eq:G1}
    W_p(\mu_0,\mu_1):=
    \min\Big\{\int \|x{-}y\|^p\,\dd \gamma (x,y):
        \gamma \in \mathcal P(\R^d\times \R^d),\ 
        \pi^1_\sharp\gamma=\mu_0,\ \pi^2_\sharp\gamma=\mu_1\Big\},
\end{equation}
provide a crucial tool to study evolutionary PDEs and geometric problems in a measure-theoretic setting. 

Such  results are strictly related to the corresponding representation formulae for measure-valued solutions 
to the continuity equation
\begin{equation}
\label{continuity-equation-AC}
\partial_t  \mu +   \dive \nnu = 0
 \qquad \text{ in }   \mathscr D'((0,T)\times \R^d),\quad
 \nnu=\bvv \mu\ll \mu,
\end{equation}
as a superposition of absolutely continuous
curves $\gamma\colon [0,T]\to \R^d$ solving the 
differential equation
\begin{equation}
\label{eq:G0}   
    \dot\gamma(t)=\bvv(t,\gamma(t))\quad\text{a.e.~in }(0,T)
\end{equation}
in an integral sense.
In fact, a curve $(\mu_t)_{t\in [0,T]}$ in $\mathcal P_p(\R^d)$ 
satisfies the $p$-absolute continuity property
\begin{equation}
\label{eq:G00}
    W_p(\mu_s,\mu_t)\le \int_s^t L(r)\,\dd r\quad\text{for every }0\le s<t\le T,\ \text{and some }L\in L^p(0,T),
\end{equation}
if and only  if 
\cite[Sec. 8.1]{AGS08}  the space-time measure $\mu=\mathscr L^1{\otimes} \mu_t=
\int_0^T \delta_t\otimes \mu_t\,\dd t$ 
solves the continuity equation 
\eqref{continuity-equation-AC} for a vector measure
$\nnu\ll \mu$ whose density $\bvv=\frac{\dd \nnu}{\dd \mu}$
satisfies
\begin{equation}
    \label{eq:G2}
    \int_{\R^d}\|\bvv_t(x)\|^p\,\dd \mu_t(x)
    \le L^p(t)\quad\text{for a.a.\ } t \in (0,T).
\end{equation}
The measure
$\nnu$ and its density $\bvv$ 
comply  with the minimality condition
\begin{equation}
    \label{eq:G3}
    \Big(\int_{\R^d}\|\bvv_t(x)\|^p\,\dd \mu_t(x)\Big)^{1/p}
    =\lim_{h\to0}\frac{W_p(\mu_t,\mu_{t+h})}{|h|} \quad\text{for a.a.\ } t \in (0,T),
\end{equation}
and are also uniquely determined by \eqref{eq:G3} if the norm $\|\cdot\|$
is strictly convex.
\nc
On the other hand, if 
$(\mu_t)_{t\in [0,T]}$ is a continuous family of probability measures in $\R^d$ and $(\mu,\nnu)$ is a solution  to  the  \eqref{continuity-equation-AC}
satisfying \eqref{eq:G2} for some $L\in L^p(0,T)$, 
then  by \cite[Sec.~8.2]{AGS08}
there exists a Radon probability measure $\eeta$ in $\rmC([0,T];\R^d)$,
concentrated on the subset 
of absolutely continuous curves satisfying 
\eqref{eq:G0},
such that 
$\mu_t=(\mathsf e_t)_\sharp \eeta$ for every $t\in [0,T]$, where
$\mathsf e_t \colon \gamma\mapsto\gamma(t)$ is the evaluation map. Equivalently,
\begin{equation}
    \label{eq:G5}
    \mu_t(A)=\eeta\Big(\{\gamma\in \rmC([0,T];\R^d):\gamma(t)\in A\}\Big)
    \quad\text{for every }t\in [0,T],\ A\text{ Borel subset of }\R^d.
\end{equation}
\par
  The characterization of $\mathrm{AC}^p$ curves in $\mathcal P_p(\R^d)$ and the lifting result  from \cite{AGS08} have been extended
  to the case in which $\R^d $ is endowed with a non-flat 
Riemannian distance \cite{Lisini09}, or replaced by a separable complete metric space $(X,d)$ \cite{Lisini07}; cf.\ also  \cite{Lisini16} for the extension to spaces endowed with Wasserstein-Orlicz distances. 
\par
A first typical application of the above results concerns curves of measures in $\mathcal P_p(\R^d)$ arising from a suitable approximation method, when
a priori estimates provide the bound \eqref{eq:G00}
(e.g.~in gradient flows, see \cite{AGS08}, or in geometric problems, see \cite{Ambrosio-Gigli-Savare14}).
In order to identify the PDE satisfied by the limit curve,
one can start from the continuity equation 
\eqref{continuity-equation-AC}
and then try to characterize the velocity field
$\bvv$. In this respect, the minimality property 
\eqref{eq:G3} provides a particularly useful information
(see e.g.~\cite{Ambrosio-Lisini-Savare06}).
\par
Another important application stems from problems where one tries to extract finer information from the continuity equation,
using the representation given by the superposition principle. The latter
 in fact establishes a link between the Eulerian representation of the flow of measures solving the continuity equation, and its Lagrangian 
depiction that comes with the associated characteristic system of ODEs. 
This connection is at the core of 
 the Young-measure type technique pioneered in 
 \cite{Ambrosio04} for transport and continuity equations featuring velocitiy fields with low regularity
(see, e.g., \cite{Ambrosio-Crippa08, Ambrosio-Bernard08, Ambrosio-Figalli09, Ambrosio-Trevisan14}).
 The  recent \cite{Stepanov-Trevisan17}  (see also \cite{AmbRenVit25}) 
 has thoroughly investigated
  the correlation 
between
these two facets of the superposition principle, i.e.\ as a bridge between the Eulerian and the Lagrangian standpoints, and 
as a tool to gain insight into the structure of curves of measures absolutely continuous to some Wasserstein distances,
decomposed 
in simpler ones associated with rectifiable curves, cf.\ the cornerstone paper \cite{Smirnov94}. The  approach
 from  \cite{Stepanov-Trevisan17} 
 has in particular led 
to an extension of the probabilistic representations previously obtained in 
\cite{Lisini07, Lisini16}. 
\par
Eventually, 
let us mention that  the superposition principle  has also turned out to be a  key tool for the well-posedness of mean-field particle systems, see, e.g.,
\cite{Ambrosio-Fornasier-Morandotti-Savare21,Morandotti-Solombrino, Almi-Morandotti-Solombrino_JEE}, as well as for the analysis and finite-particle approximation of mean-field optimal control problems \cite{Albi-Almi-Morandotti-Solombrino,Almi-Durastanti-Solombrino1,Almi-Durastanti-Solombrino2,Almi-Morandotti-Solombrino_JDE,Cavagnari-Lisini-Orrieri-Savare22,Fornasier-Lisini-Orrieri-Savare19}, where, compared to previous literature (see, for instance, \cite{bongini2016optimal,Fornasier-Solombrino}), a priori regularity assumptions on the control variable can be dropped. Further applications include the formulation and convergence analysis of gradient methods for dynamic inverse problems \cite{Brediesetal,Bredies-Fanzon}, which build upon an extension of the superposition principle to inhomogeneous continuity equations and on the characterization of extremal points of the Benamou-Brenier or Hellinger-Kantorovich type of energies \cite{Bredies-Carioni-Fanzon-22,Bredies-Carioni-Fanzon-Romero-21}. 
We also mention \cite{BiaBonGus}, where the superposition principle
for absolutely continuous curves of measures
 has been  leveraged to obtain uniqueness results for a transport equation. Finally, we recall that in 
 \cite{BonGus}  two counterexamples  to the validity of the superposition principle
 have been exhibited
  in the case of \emph{signed} measures. 
\medskip

\paragraph{\bf  \emph{The continuity equation with singular flux}}
The main aim of this  paper is to investigate the 
validity and the appropriate formulation of the
\begin{itemize}
\item characterization of solutions to the continuity equation
\item superposition principle
\end{itemize}
 in the 
case $p=1$, for
the space of probability measures with 
finite moment $\mathcal P_1(\R^d)$,
endowed with the metric $W_1.$  Indeed, we will address curves  of measures that have \emph{bounded variation} as functions of time, and, above all, with respect to the 
$W_1$-metric. 
A crucial feature of this setting, that
makes the interpretation of the differential equation
\eqref{eq:G0} much more delicate, 
is the fact that 
the vector flux measure $\nnu$ in 
\eqref{continuity-equation-AC} may have a singular 
part with respect to $\mu$.
Consider, e.g., the simplest example
 (cf.\ \cite[Ex.\ 1.1]{Abedi-Zhenhao-Schultz24})
of the curves $(\mu_t)_{t\in [0,1]}$ 
\begin{equation}
    \label{eq:G6}
    \mu_t=(1-t)\delta_{x_0}+t\delta_{x_1},\quad
    x_0<x_1\in \R,\ t\in [0,1].
\end{equation}
It is immediate to check that 
\[
W_p^p(\mu_s,\mu_t) = |t{-}s| W_p^p (\delta_0,\delta_1) = |t{-}s| \qquad \text{for } s,\, t \in [0,1], \qquad p \in [1,\infty)\,.
\]
Hence, while $\mu \notin \AC ([0,1];\mathcal{P}_p(\R^d) )$ if $p>1$, 
we have $\mu \in \mathrm{Lip} ([0,1];\mathcal{P}_1(\R^d) )$. 
Now, it can be calculated
(see Example
\ref{ex:Dirac-delta} ahead), that 
$\mu$ satisfies the continuity 
equation together with the flux 
measure 
\begin{align}
\label{e:intro-nnu}
    \nnu = (\mathcal{L}^1 \mres  [0,1]) \otimes  (\mathcal{L}^1 \mres [x_{0}, x_{1}])\,.
\end{align} \EEE

\par
 The metric superposition principle for  {\rm BV} curves in 
 $(\mathcal P_1(X),W_1)$ with $X$ a (complete, separable) \emph{metric} space, 
    has been tackled in the 
recent \cite{Abedi-Zhenhao-Schultz24},
by a constructive argument carefully adapting  the line of proof of \cite[Thm.\ 5]{Lisini07} to the $\mathrm{BV}$ setup.  
Therein, 
 the superposition principle 
has indeed been obtained
 by lifting absolutely continuous curves  
  to C\`adl\`ag (i.e., right-continuous, left-limited) {\rm BV}  curves  on $(0,T)$ with  values in $X$ by means of a probability measure $\eeta$. Additionally, in \cite{Abedi-Zhenhao-Schultz24} it has been proved that 
 the total variation of the measure $\mu$ can be reconstructed by averaging  the variation of the  {\rm BV} curves   via the path measure $\eeta$.
\bigskip
\bigskip

\paragraph{\bf \emph{Overview of our results}}
 
 In this paper,  while confining our analysis to the Euclidean setup $X=\R^d$,  we will  adopt a different approach. Indeed, we will primarily focus on the  
  structure of the superposition measure 
 and its link with the flux measure $\nnu$. 
  In this way, we will shed more light   into the properties of the continuity equation in the {\rm BV}  setup.
  More precisely,  
 \begin{enumerate}
 \item  We will carefully address the relation between {\rm BV} curves in $(\mathcal P_1(\R^d),W_1)$ and  the continuity equation
 \begin{equation}
  \label{eq:1}
  \partial_t  \mu +   \dive \nnu = 0
  \qquad \text{ in }   \mathscr D'((0,+\infty)\times \R^d),\quad
  |\nnu|([0,T]\times \R^d)<\infty\quad\text{for every }T>0,  
\end{equation}
where  
the Lebesgue decomposition  $\nnu=\nnu^a+\nnu^\perp$ of the
 flux measure $\nnu$ with respect to $\mu$ 
 may feature a 
 nontrivial \emph{singular} part $
\nnu^\perp $.
\item 
Among solutions 
 to  \eqref{eq:1} we will enucleate a particular class of flux measures,
which we will call \emph{minimal},
and we will show that 
starting from a non-minimal measure,
it is always possible to 
replace 
the singular part $\nnu^\perp$ by a \emph{minimal} one $\bar\nnu^\perp$ such that 
$\bar\nnu=\nnu^a+\bar\nnu^\perp$
satisfies
\begin{equation}
  \label{eq:1bis}
  \partial_t  \mu +   \dive \bar\nnu = 0,\quad
  \bar\nnu^\perp=\lambda  \nnu^\perp
  \quad\text{for a Borel scalar map }
  \lambda\colon  (0,+\infty)\times \R^d\to [0,1].
  \end{equation}
\item 
We will represent 
minimal solutions to \eqref{eq:1bis} 
as marginals of solutions to an auxiliary  continuity equation 
in the \emph{augmented} phase space,
driven by an autonomous bounded vector field.
\item 
By applying the known superposition principle to 
the augmented equation we will obtain a first 
representation of the  solutions to \eqref{eq:1}
by a measure on reparametrized $1$-Lipschitz curves
in the augmented phase space $[0,+\infty)\times \R^d$.
\item Eventually,  
 we will  derive a 
 superposition principle 
 in the original space by 
 a measure on a class of agumented BV curves,
 providing finer information on their jump transitions.
 \end{enumerate}
Let us explain some of the above points in more detail.
\medskip

\paragraph{\bf \emph{Absolutely continuous and {\rm BV} curves in $(\mathcal P_1(\R^d),W_1)$}}
Dealing with $p=1$, a first natural choice
is to include the space $\BV\!_{\rm loc}([0,+\infty);\Puno(\R^d)) $
of BV curves 
with values in $\mathcal P_1(\R^d)$ in the analysis,
i.e.\ the curves 
$\mu\colon [0,+\infty) \to \Puno(\R^d)$
satisfying
$\Var_{W_1}(\mu; [0,T])<\infty$ for every $T>0$, where
\begin{equation}
\label{eq:GG7}
\Var_{W_1}(\mu; [a,b]): = \sup \left \{ \sum_{i=1}^{n} W_1 (\mu_{t_{i-1}}, \mu_{t_i}) \, :  a= t_0< t_1<\ldots<t_n=b \right\}.
\end{equation}
To avoid ambiguities at the jump points of $\mu$ and simplify this
introductory
exposition, we will also assume
that
$\mu$ is right continuous in $[0,+\infty)$.
With every map $\mu\in \BV\!_{\rm loc}([0,+\infty);\Puno(\R^d))$ we will associate the
increasing map $\rmV_\mu(t):=\Var_{W_1}(\mu;[0,t])$ and its
distributional derivative 
\begin{equation}
  \label{eq:3}
  \rmv_\mu=\frac{\dd}{\dd t}\rmV_\mu,\quad\text{a positive locally finite measure
    in $[0,+\infty)$.}
\end{equation}

First of all, 
 in \textsl{\bfseries Theorem \ref{th:1}} 
we will show
that it is possible to associate with  every curve $\mu\in \BV\!_{\rm loc}([0,+\infty);
\Puno( \R^d ))$ a ``minimal'' vector measure $\nnu\in
\calMb_{\rm loc}([0,+\infty)\times\R^d;\R^d)$, with 
local-in-time finite total variation, such that  the pair $(\mu,\nnu)$ fulfills \eqref{eq:1} and 
 the push forward of the variation measure $|\nnu|$ (associated with
the norm $\|\cdot\|$ in $\R^d$),  with respect to the time projection map
$ \sft \colon  [0,+\infty)\times \R^d\ni (t,x)\mapsto t$
satisfies
\begin{equation}
  \label{eq:2}
  \sft_\sharp |\nnu|=\rmv_\mu\quad
  \text{i.e.}\quad
  |\nnu|\big((a,b]\times
  \R^d\big)=\rmV_\mu(b)-\rmV_\mu(a)\quad\text{for every }0\le a<b.
\end{equation}
It is worth noticing that the disintegration $(\mu_t)_{t\ge 0}$ of
a solution $\mu$ of \eqref{eq:1}
w.r.t.~a vector measure $\nnu$ of local-in-time finite total variation admits a 
BV representation satisfying the further condition
\begin{equation}
  \label{eq:4}
  \rmv_\mu\le \sft_\sharp |\nnu|.
\end{equation}
Therefore,  condition \eqref{eq:2} provides a variational characterization
of $\nnu$ similar to \eqref{eq:G3} in the case $p>1$.
\medskip

\paragraph{\bf  \emph{The augmented continuity equation}}
A more difficult task is 
to establish the 
superposition principle for solutions to 
\eqref{eq:1}, 
also dealing with the case when 
the variational condition \eqref{eq:2} is not satisfied.
For this, we have drawn inspiration from the results by \textsc{Smirnov}
 on the representation of \emph{solenoidal} (i.e., null-divergence) \emph{charges} (currents) in terms of simpler ones associated with rectifiable curves, 
 \cite{Smirnov94}. In fact, a pair $(\mu, \nnu)$ solving the continuity equation
\eqref{eq:1}
 can be viewed as a ``solenoidal charge'', too, in the sense that \eqref{eq:1} rewrites as 
\begin{equation}
\label{e:solenoidal}
    \mathrm{Div}_{(t,x)} (\mu,\nnu) =0 \quad \text{in  }\mathscr D'((0,+\infty)\times \R^d),
\end{equation}
with the overall divergence operator 
$ \mathrm{Div}_{(t,x)}(\star,\bullet):= \partial_t (\star) + \mathrm{div}(\bullet)$. This observation 
 is, in fact, also at the core of the approach in \cite[\S 3.1]{Bonicatto:PHD}, where an alternative proof of the superposition principle 
from \cite{Ambrosio04} has been developed, based on Smirnov's decomposition theorem in  \cite{Smirnov94}. 
\par
In the present $\BV$ setup, 
we have developed a self-contained approach, independent of Smirnov's result, 
which exploits the nonnegativity of $\mu$ 
and the irreversible direction of time.
We have also been guided by the reparametrization technique
that has been quite successful in the variational theory of rate-independent evolution
 (cf., e.g., \cite{EfeMie06RILS, MRS09, MielkeRossiSavare12a, MRS13, DalDesSol11}),  
where solution curves 
incorporate further information about jump
transitions (not necessarily segments)
and take values 
in the \emph{augmented} phase space $
\R^{d+1}_+:=[0,+\infty){\times}\R^d$, including time.

Accordingly, our idea 
 (cf.\ \textsl{\bfseries Theorem \ref{thm:augmented}} ahead),  
is to 
represent a minimal solution to 
\eqref{eq:1} as the marginals
$(\mu,\nnu)=\ppi_\sharp (\sigma^0,\ssigma)$ (with $\ppi$ the projection, 
$\ppi(s,t,x):=(t,x)$), 
of a distinguished 
solution to the auxiliary continuity equation 
in the augmented phase space $\R^{d+1}_+\ni (t,x)$, namely
\begin{equation}
\label{augmented-CE}
\partial_s \sigma +   \partial_t \sigma^{0}
+\dive \ssigma =0 \quad \text{in }\mathscr D'((0,+\infty){\times} \R^{d+1}),
\quad
\sigma,\sigma^0\ge 0,
\quad 
\sigma_{s=0}=\delta_0\otimes \mu_0. 
\end{equation}
In \eqref{augmented-CE},  $s$ is an artificial time-like parameter,
the flux pair
 $(\sigma^0,\ssigma)$ is absolutely continuous w.r.t.\ $\sigma$ 
$(\sigma^0,\ssigma)=(\tau,\bvv)\sigma$,
and 
the autonomous velocity field $(\tau,\bvv)$ 
 is related to the original solution pair $(\mu,\nnu)$ via 
 \begin{align}
     \label{e:intro-100}
  \tau = \frac{\dd\mu}{\dd|(\mu,\nnu)|}, \qquad \bvv = \frac{\dd\nnu}{\dd |(\mu,\nnu)|}\,.
  \end{align}
 In particular, 
 the augmented norm of $(\tau,\bvv)$ is $1$ and 
 $\sigma=|(\sigma^0,\ssigma)|$.
A simple modification (cf.\ Proposition \ref{p:superposition}) of the standard superposition principle can then be applied to the \emph{augmented} continuity equation
\eqref{augmented-CE}.  It guarantees  a representation of  any solution $\sigma$ in terms of a probability measure $\eeta$ supported on $1$-Lipschitz curves $ \sfyy $ of the form
\begin{equation}
\label{lip-param-curves}
[0,+\infty) \ni s \mapsto  \sfyy(s)= (\mathsf{t}(s), \sfxx(s)) \in [0,+\infty){\times}\R^d,
\end{equation}
solving the associated characteristic system 
\[
\dot\sft (s)=
\tau(\sft(s),\sfxx(s)),
\quad
\dot \sfxx(s)=\bvv(\sft(s),\sfxx(s)).
\]
We will then derive a probabilistic representation for 
the pair  $(\mu,\nnu)$ in terms of the trajectories $\gamma$.   Specifically,  in \textsl{\bfseries Theorem~\ref{t.1}}  we will  show that
\begin{align}
    \label{e:intro-rep}
(\mu, \nnu) = \fre_{\sharp } (\dot\sfyy  \mathcal{L}^{1} \otimes \eeta),\qquad 
| (\mu, \nnu) | & = \fre_{\sharp } (\|  \dot \sfyy   \| \mathcal{L}^{1} \otimes \eeta)\,,
\end{align}
where $\mathcal{L}^{1}$ denotes the Lebesgue measure in $[0, +\infty)$, and 
$\fre \colon [0, +\infty) \times {\rm Lip} ([0, +\infty);  \R_+^{d+1})$ is 
 the evaluation map defined as $\fre(s,  \sfyy) := \sfyy (s) $. 
\medskip

\paragraph{\bf \emph{Superposition by augmented $\BV$ curves}}

Our final contribution is to recover a superposition result 
for the original continuity equation \eqref{eq:1} by a measure on the space of 
time dependent BV curves.
If $\nnu$ does not satisfy the variational condition
\eqref{eq:2},
one can expect  
the singular part $\nnu^\perp$ of  $\nnu$  to
carry crucial information
about the jump transition of the curves (not necessarily along segments). Seemingly, 
such information cannot be  fully captured by 
the usual descrition of a BV curve, which only characterizes the left and right limit of the curve
at each jump point, but not the \emph{actual} trajectory described along the jump.

To overcome this difficulty, we introduce the 
notion of \emph{augmented} BV curves:
they are maps $ \sfu \colon \cZ\to \R^d$ defined in the augmented
parameter space
$\cZ:=[0,+\infty)\times [0,1]$ 
such that 
\begin{enumerate}
    \item 
    the functions $u_-(t):=\sfu(t,0)$
    (resp.~$u_+(t):=\sfu(t,1)$)
    are left- (resp.~right-) continuous (local) BV maps
    which coincide 
    in the complement of their countable
    jump set 
     $ \mathfrak{J}_{\bvc};$
    \item 
    for every 
     $t\notin  \mathfrak{J}_{\bvc}$  
    the function $[0,1] \ni r\mapsto \sfu(t,r)$
    is constant and coincides with
    $u_-(t)=u_+(t);$
    \item 
    for every  $t\in  \mathfrak{J}_{\bvc}$ 
    the function 
    $[0,1] \ni r\mapsto \sfu(t,r)$
    is a Lipschitz (transition) map
    connecting 
    $u_-(t)$ with $u_+(t)$
    with constant (and strictly positive) velocity,
    thus equal to the length $\ell_\sfu(t)$
    of the transition path.
    \item 
    For every finite time interval $[0,T]$
    we have
    \begin{equation}
        \Var(\sfu;[0,T]\times [0,1])=
        \Var(u_\pm;[0,T])+
        \sum_{t\in  \mathfrak{J}_{\bvc}}
        \Big(\ell_\sfu(t)-\|u_+(t)-u_-(t)\|\Big)
        <\infty\,.
    \end{equation}
\end{enumerate}
We will  denote such space by $\ABV( \cZ \EEE ;\R^d)$  and endow it 
with a Lusin topology, for which the evaluation
maps $\sfe(t,r,\sfu):=\sfu(t,r)$
 are 
Borel.
\par
 In \textsl{\bfseries Theorem \ref{th:BV-representation}}, 
we will 
represent a minimal solution  $( \mu,\nnu) $ 
to \eqref{eq:1}  by means of a probability measure $\widehat\eeta$ on   
$\ABV(  \cZ \EEE  ;\R^d)$ 
concentrated on curves
$\sfu$ 
satisfying
a suitable differential equation formulated in 
a BV sense.
More precisely, we 
can write the Lebesgue decomposition $\nnu^a+\nnu^\perp$ of $\nnu$ as 
$\nnu^a=\bvv^a\mu$,
$\nnu^\perp=\bvv^\perp|\nnu^\perp|$,
and we observe that 
the curves $t\mapsto \sfu(t,r)$ coincide
$\mathcal L^1$-a.e.~on $(0,+\infty)$, 
so that their distributional
time derivative 
$\partial_t \sfu(t,r)$
does not depend on $r$ and can be decomposed 
in the sum of an absolutely continuous part
 $\partial^{\mathrm{L}}_t \sfu \, \mathcal L^1$, 
a Cantor part
$\partial^{\mathrm{C}}_t \sfu$ and a jump part
$\partial^{\mathrm{J}}_t \sfu$ concentrated on $\mathrm{J}_\sfu$:
\begin{equation}
\label{e:intro-decomposition}
    \partial_t \sfu=
    \partial^{\mathrm{L}}_t \sfu\,\mathcal L^1+
    \partial^{\mathrm{C}}_t \sfu+
    \partial^{\mathrm{J}}_t \sfu\,.
\end{equation}
Thus, $\widehat\eeta$-a.e.~curve $\sfu$ satisfy
\begin{subequations}
\begin{align}
    \partial_t^{\mathrm{L}} \sfu(t)& =\bvv^a(\sfu(t,\cdot))\quad
    \text{for $\mathcal L^1$-a.a.~$t\in (0,+\infty),$}  \label{e:intro-Leb}  \\[1mm]
    \partial^{\mathrm{C}}_t \sfu & =
    \bvv^\perp (\sfu(t,\cdot))|\partial^{\mathrm{C}}_t \sfu|,
\end{align}
\end{subequations}
and
\begin{equation}
\label{e:intro-jump}
    \partial_r\sfu(t,r)=
    \bvv^\perp (\sfu(t,r))\ell_\sfu(t)
    \quad\text{for a.a.~$r\in (0,1)$ and every  $t\in \mathfrak{J}_{\bvc}$.}
\end{equation}
We then have
\begin{displaymath}
    \mu_t^+=(\mathsf e_{t,1})_\sharp \widehat\eeta,\quad
    \mu_t^-=(\mathsf e_{t,0})_\sharp \widehat\eeta,
\end{displaymath}
 where the evaluation maps $\sfe_{t, 0}, \sfe_{t, 1} \colon \ARBV(  \cZ \EEE  ;\R^d) \to \R^{d}$ are defined by $\sfe_{t, 0} (\bvc) := \bvc (t, 0)$ and $\sfe_{t, 1} (\bvc) := \bvc(t, 1)$ for every $\bvc \in \ARBV (  \cZ \EEE  ;\R^d)$. Whenever  $\nnu$ is minimal, we can recover
$\nnu$ by superimposing 
integration along $\sfu$.
\medskip

To conclude this introduction, we illustrate our results by  reinterpreting \cite[Ex.\ 1.1]{Abedi-Zhenhao-Schultz24} (cf.~\eqref{eq:G6}--\eqref{e:intro-nnu}) in our setting and  providing a probabilistic representation of the curve $\mu$ in \eqref{eq:G6}, both   in terms of  reparametrized 1-Lipschitz curves and of curves in  $\ARBV ( \cZ ; \R)$. 

\begin{example} 
\label{ex:Dirac-delta}
Let $\mu_{t}$ and $\nnu$ be as in \eqref{eq:G6}--\eqref{e:intro-nnu}, let us extend $\mu_{t}$ to the whole $[0, +\infty)$ by setting $\mu_{t} := \mu_{1}$ for $t \geq 1$, and let us set $\mu := \mathcal{L}^{1} \otimes \mu_{t}$. It is straightforward to check that the pair $(\mu, \nnu)$ solves the continuity equation $\partial_{t} \mu + {\rm div} \nnu = 0$ in distributional sense in $[0,+\infty) \times \R$, with initial condition $\mu_{0} = \delta_{x_{0}}$ (cf.~Definition~\ref{def:solCE}).

We start by illustrating the superposition principle from Theorem \ref{t.1}. For simplicity of notation, let us set $d_{0}:= x_{1} - x_{0}>0$. To define the fields $\tau \colon [0, +\infty) \times \R \to \R$ and $\bvv \colon [0, +\infty) \times \R \to \R$ as in~\eqref{e:intro-100}, notice that by the mutual singularity of $\mu$ and $\nnu$ we have
  \[
  |(\mu,\nnu)| = (\mathcal{L}^1 \mres [0,1]) \otimes (\mu_t {+} \mathcal{L}^1 \mres [x_{0}, x_{1}]) + ( \mathcal{L}^{1} \mres (1, +\infty)) \otimes \mu_{t}\,. 
  \]
Hence,

  \begin{subequations}\label{e:intro-tauv}
  \begin{align}
  \tau(t,x) & = \frac{\dd \mu}{\dd  |(\mu,\nnu)| } (t,x)= \begin{cases}
  1 & \text{if } t \in [0, 1) \text{ and } x\in \{x_0,x_1\}\,,
  \\
  1 & \text{if } t\in [1, +\infty) \text{ and } x = x_{1}\,,
  \\
  0 & \text{otherwise,}
  \end{cases}
  \\[2mm]
  \bvv(t,x) & = \frac{\dd \nnu}{ \dd  |(\mu,\nnu)| } (t,x)= \begin{cases}
  1 & \text{if } t \in [0, 1) \text{ and } x\in (x_0,x_1)\,,
  \\
  0 & \text{otherwise,}
  \end{cases}
  \end{align}
    \end{subequations} 
for   $\mathcal{L}^1$-almost all  $t\in (0, + \infty)$. 

For every $\bar{t} \in [0,1]$ we define the curves $\sfyy^{\bar t} \colon  [0,+\infty) \to \R^2$ by 
\begin{equation}
    \label{e:intro-sfyy}
 \sfyy^{\bar t}(s) : = \begin{cases}
     (s,x_0) & \text{for $s \in [0, \bar{t})$,}\\[2mm]
     \displaystyle \bigg( \bar{t}, \frac{\bar{t} + d_{0} - s}{d_{0} } \, x_{0} + \frac{s - \bar{t}}{ d_{0}} \, x_{1} \bigg) & \text{for $s \in [\bar{t}, \bar{t} + d_{0}]$,}\\[2mm]
      (s - d_{0} ,x_1) & \text{for $s \in (\bar{t} + d_{0}, +\infty)$.}
 \end{cases}  
\end{equation}
Loosely speaking, each curve $ \sfyy^{\bar t} $ moves time till $\bar t$, then it `fills in' the jump from~$x_0$ to~$x_1$, and then moves time again. Furthermore,  notice that $\sfyy^{\bar{t}}$ solves the Cauchy system
 \begin{equation*}
\begin{cases}
  \dot{\sfyy} (s)  = (\tau(\sfyy(s)), \bvv(\sfyy(s))),  \  s \in [0, +\infty), 
\\
\sfyy(0) =(0,x), \  x\in\supp(\mu_{0}) = \{x_0\}\,.
\end{cases} 
\end{equation*} 

  The measure $\eeta$ 
  involved in the representation  \eqref{e:intro-rep}  of $(\mu,\nnu)$ is defined as follows. We consider the mapping
$\Upsilon \colon [0,1]\to   \Lip ( [0, +\infty) ; \R^{2}) $ that with each $\bar t \in [0,1]$ associates the curve $ \sfyy^{\bar t} $. We set
\begin{equation}
\label{eeta-delta}
\eeta : = \Upsilon_{\sharp }(  \mathcal{L}^1 \mres [0,1])\,. 
\end{equation}
Then, $\eeta$ is a probability measure over $ \Lip ( [0, +\infty) ; \R^{2})$ that concentrates over curves that have a non-decreasing first component having image in $[0, +\infty)$, are $1$-Lipschitz continuous and have modulus of velocity constantly equal $1$ (see~\eqref{e:intro-tauv}).
We now check that $\eeta$ provides the probabilistic representation \eqref{e:intro-rep} of the pair $(\mu,\nnu)$. For convenience, for $\sfyy \in \Lip([0, +\infty) ; \R^{2})$ let us use the notation $\sfyy = (\sft, \sfxx)$ , where $\sft$ and $\sfxx$ stand for the time and the space component of $\sfyy$, respectively. For every  $\phisc , \bvarphi \in \rmC_{\rm b}(\R^2) $ there holds
\begin{align}
\label{e:intro-long-equation}
\langle  \fre_{\sharp }( \sfyy'  \mathcal{L}^{1} \otimes\eeta ) , (\phisc, \bvarphi) \rangle 
&= 
\int_{\Lip} \int_{0}^{+\infty} \phisc(  \sfyy(s)) \sft'(s) \, \dd s \, \dd \eeta(\sfyy) + \int_{\Lip} \int_{0}^{+\infty} \bvarphi ( \sfyy(s)) \sfxx'(s) \, \di s \, \di \eeta(\sfyy) 
\\
 & \stackrel{(1)}{=}  \int_0^1 \left( \int_0^{\bar t} \phisc(s,x_0) \, \dd s + \int_{\bar {t}+ d_{0}}^{ + \infty} \phisc (s - d_{0}  , x_1) \, \dd s\right)  \, \dd \bar{t}\nonumber
 \\
 &
 \qquad + \int_{0}^{1} \int_{\bar{t}}^{\bar{t} + d_{0}} \bvarphi \left( \bar{t}, \frac{\bar{t} + d_{0} - s}{d_{0} } \, x_{0} + \frac{s - \bar{t}}{ d_{0}} \, x_{1} \right) \, \di s \, \di \bar{t} \nonumber
\\
&\stackrel{(2)}{=}  \int_0^1 \left(  \int_0^{\bar t} \phisc(s,x_0) \, \dd s + \int_{\bar {t}}^{+\infty} \phisc (s,x_1) \, \dd s \right)  \, \dd \bar{t}
 + \int_{0}^{1} \int_{x_{0}}^{x_{1}} \bvarphi (\bar{t}, x) \, \di x \, \di \bar{t} \nonumber
\\
& 
\stackrel{(3)}{=}  \int_0^1 \int_{s}^1 \phisc(s,x_0) \, \dd \bar{t} \, \dd s  +  \int_0^{+\infty} \int_0^{\min\{s, 1\}} \! \!\!\!\!\!\!\! \phisc (s,x_1) \, \dd \bar t \, \dd s +  \int_{0}^{1} \int_{x_{0}}^{x_{1}} \bvarphi (t, x) \, \di x \, \di t  \nonumber
\\
&
= \int_0^1 (1{-}s)   \phisc(s,x_0) \, \dd s + \int_0^1 s \phisc (s,x_1) \, \dd s +  \int_{1}^{+\infty} \phisc(s, x_{1}) \, \di s \nonumber
\\
& 
\qquad +  \int_{0}^{1} \int_{x_{0}}^{x_{1}} \bvarphi (t, x) \, \di x \, \di t \nonumber
\\
&
=\iint_{[0, +\infty) \times \R} \phisc(t,x) \, \dd \mu_t(x) \, \dd t + \iint_{[0, +\infty) \times \R} \bvarphi(t, x) \, \di \nnu(t, x)\,, \nonumber
\end{align}
 where (1) follows from the fact that $\sft'(s) = 0$ in $(\bar{t}, \bar{t} + d_{0})$ and $\sfxx'(s) = 0$ for $s \notin [\bar{t}, \bar{t} + d_{0}]$, (2) from a change of variable in the last integral, 
and (3) from the Fubini theorem.

 We now recast the above representation in terms of curves in $\ARBV (\cZ ; \R)$. For $\bar{t} \in [0, 1]$ we define $\bvc^{\bar{t}} \colon \cZ \to \R$ as
\begin{align*}
    \bvc^{\bar{t}} (t, r) := 
\begin{cases}
    x_{0} & \text{for $t \in [0, \bar{t})$ and $r \in [0, 1]$,}\\
    rx_{1} + (1-r)x_{0} & \text{for $t = \bar{t}$ and $r \in [0, 1]$,}\\
     x_{1} & \text{for $t \in (\bar{t}, +\infty) $and $r \in [0, 1]$.}
\end{cases}
\end{align*}
We remark that the curves in $\ARBV (\cZ ; \R)$ only track movement in space, differently from the reparametrised curves in~\eqref{e:intro-sfyy}. On the other hand, in agreement with~\eqref{e:intro-sfyy}, each curve $\bvc^{\bar{t}}$ remains constant up to time $\bar{t}$, then performs a 
constant-speed transition between $x_{0}$ and $x_{1}$, and finally sits in~$x_{1}$ after $t=\bar{t}$. Namely, we are squeezing the (parametric) interval $[\bar{t}, \bar{t} + d_{0}]$ in a jump transition for the $\mathrm{BV}$-map $t \mapsto \bvc (t, 0)$. In order to properly represent $\nnu$ we must keep track of how curves transition from~$x_{0}$ to~$x_{1}$. Using the curves $\bvc^{\bar{t}}$, we may rewrite~\eqref{e:intro-long-equation} as
\begin{align}
\label{e:intro-BV-repres}
    \int_{[0, +\infty) \times \R} &  (\varphi_{0} (t, x) , \bvarphi( t, x) ) \, \di (\mu, \nnu) (t, x)  
    \\
    &
    \nonumber = \int_{0}^{1} \int_{0}^{+\infty} \varphi_{0} (t, \bvc^{\bar{t}} (t, 0)) \, \di t \, \di \bar{t} + \int_{0}^{1} \int_{0}^{1} \bvarphi (\bar{t}, rx_{1} + (1-r)x_{0}) \, (x_{1} - x_{0}) \, \di r \, \di \bar{t}
    \\
    &
    = \int_{0}^{1} \int_{0}^{+\infty} \varphi_{0} (t, \bvc^{\bar{t}} (t, 0)) \, \di t \, \di \bar{t} + \int_{0}^{1} \sum_{t \in \mathfrak{J}_{\bvc^{\bar{t}}}}\int_{0}^{1} \bvarphi (t, \bvc^{\bar{t}} (t, r)) \, \partial_{r} \bvc^{\bar{t}} (t, r)  \, \di r \, \di \bar{t} \nonumber\,,
\end{align}
where in the second equality we have used the fact that
 the  jump set $\mathrm{J}_{\bvc^{\bar{t}}} = \{ \bar{t}\}$ 
and that $\partial_{r} \bvc^{\bar{t}} (\bar{t}, r) = (x_{1} - x_{0})$ for $r \in [0, 1]$. To give a compact representation, we define the map $\Xi \colon [0, 1] \to \ARBV (\cZ ; \R)$ by $\Xi(\bar{t}) := \bvc^{\bar{t}}$ and the measure $\widehat{\eeta} := \Xi_{\sharp} ( \mathcal{L}^{1} \mres [0, 1])$. For $\bvc \in \ARBV ( \cZ; \R)$ we denote by $\tetatrue{\bvc}$ the integration over the graph of~$\bvc$, that is (see also~\eqref{e:631})
\begin{align*}
    \langle \tetatrue \bvc , (\varphi_{0}, \bvarphi)\rangle = & \int_{0}^{+\infty} \varphi_{0} (t, \bvc (t, 0)) \, \di t + \int_{0}^{+\infty} \bvarphi (t, \bvc (t, 0)) \,\partial^{L}_{t} \bvc (t, 0)\, \di t
    \\
    &
     + \int_{[0, +\infty)}\bvarphi (t, \bvc (t, 0)) \,\di \partial^{C}_{t} \bvc (t) + \sum_{t \in \mathfrak{J}_{\bvc}} \int_{0}^{1} \bvarphi (t, \bvc(t, r)) \, \partial_{r} \bvc(t, r) \, \di r \,.
\end{align*}
Postponing any discussion concerning measurability to Section~\ref{ss:6.2}, we set
\begin{align*}
  \Theta^{\widehat{\eeta}} := \int_{\ARBV} \tetatrue{\bvc} \, \di \widehat{\eeta} (\bvc)\,.  
\end{align*}
 In view of~\eqref{e:intro-BV-repres} we conclude that 
 \[
 (\mu, \nnu) = \Theta^{\widehat{\eta}}, 
 \]
 which provides the probabilistic representation in $\ARBV(\cZ; \R)$. Notice that~\eqref{e:intro-BV-repres} does not involve any absolutely continuous or Cantor part $\partial_{t}^{L} \bvc^{\bar{t}}$ and $\partial_{t}^{C} \bvc^{\bar{t}}$, as the are both zero for the special curves $\bvc^{\bar{t}}$. In particular, the absence of the absolutely continuous part is in agreement with~\eqref{e:intro-Leb}, since $\mu$ and $\nnu$ are mutually singular. Finally,~\eqref{e:intro-jump} is satisfied along the transition at time $t = \bar{t}$.
\end{example} 

\paragraph{\bf Plan of the paper.} Our analysis is carried out as follows:
\begin{itemize}
\item[-]
In Section \ref{s:2}, after settling some notation and preliminary results from measure theory, we introduce an order relation between Radon measures and delve  into the induced minimality concept, which will play a key role in the selection of distinguished solutions to the continuity equation with singular flux. We also lay the ground for the superposition principle by defining the function spaces that will come into play, and fixing their  properties. 
\item[-] Section \ref{s:3} revolves around the relation between curves $\mu \in \BV_{\mathrm{loc}}([0,+\infty);\mathcal{P}_1(\R^d))$
 and the continuity equation \eqref{eq:1}.
   \item[-] In the main result of Section \ref{s:4} we associate with  \eqref{eq:1} a new continuity equation  in an augmented phase space, driven by a \emph{non-singular} flux measure 
   with bounded velocity field. Their relation is such that suitable marginals of the solutions to the augmented continuity equation provide, indeed, solutions to the continuity equation
   with minimal singular flux. 
   \item[-] Based on this, in Section \ref{s:5} we derive our first, `parametrized' version of the superposition principle. In fact, by leveraging the  probabilistic 
 representation  for the solutions to the augmented continuity equation, we obtain a representation of the solutions to the original continuity equation, in terms of 
 trajectories 
 that are Lipschitz w.r.t.\ an artificial time-like parameter and 
   solve the characteristic system in the extended phase space.
  \item[-] Section \ref{s:6} is devoted to establishing 
  a probabilistic representation for solutions of the continuity equation in terms  of $\BV$ curves  depending on the `true' process time.  For this,  we preliminarily carry out a  thorough 
  analysis of a distinguished class of $\BV$ curves that are  `attached' with their transitions at jump points. We use them as a bridge between the probabilistic representation in terms of reparametrized trajectories, and that involving $\BV$ curves.
\item[-] In Section \ref{s:examples} we discuss our assumptions, and illustrate our results (mostly focusing on the `parametrized version' of the superposition principle) in a series of examples.
\item[-] Finally, in the Appendix we prove some technical results that have been employed at scattered spots in the paper.  
\end{itemize}

\section{Notation and preliminary results}
\label{s:2}
 The following table contains the main notation that we shall use throughout the paper:
 \nopagebreak
\begin{center}
\newcommand{\specialcell}[2][c]{%
  \begin{tabular}[#1]{@{}l@{}}#2\end{tabular}}
\begin{small}
\begin{longtable}{lll}
&  $\norm{\cdot}$  & (generic) norm in $\R^h$
\\
& $B_R$, $\overline{B}_R$ & open/closed ball of center $0$  and radius $R>0$ in   $\R^h$ 
 (w.r.t.\ the norm $\norm{\cdot}$) 
\\
&
$\mathcal{B}(\R^h), \mathcal{B}_{b}(\R^h)$ & Borel  (resp.\ bounded Borel)  subsets of $\R^h$ 
\\
&   $\III$   &     the positive half-line $[0,+\infty)$  
\\
&  $\mathcal{L}^1$   &   Lebesgue measure on $\III$  
\\
&
$\mathcal{P}(\R^h)$ &  Borel probability measures in $\R^h$
\\
&
$\Puno(\R^h)$ &  probability measures in $\R^h$ with finite first moment,  endowed with the 
\\
&
 $W_1$ &  Wasserstein distance 
\\
&
$\calMb(A)$, $\calM(A)$ & finite (resp.\ Radon) Borel measures on $A\in \mathcal{B}(\R^h)$
\\
&
$\calMbp(A)$, $\calMp(A)$ & finite (resp.\ Radon) nonnegative Borel measures on $A$
\\
&
$\calMb(A;\R^m)$, $\calM(A;\R^m)$ & $\R^m$-valued Borel measures with finite total variation,
\\ & &  (resp.\ 
 $\R^m$-valued Radon meas.),
 on $A$
\\
& 
$|\llambda|$ & total variation of $\llambda \in \calM(A;\R^d)$
\\
& 
$\Cc(A), \Cc^k(A) $  & continuous ($\mathrm{C}^k$, $k\geq 1$, resp.) real functions on $A$ with compact support
\\
& 
$ \Cb(A), \Cb^k(A) $  & continuous ($\mathrm{C}^k$, $k\geq 1$, resp.)  and bounded 
real functions on $A$ 
\\
& 
$\|\cdot\|_{p}$ & norm on $L^p(A;\R^m)$ for some $p\geq 1$
\\
&
 $L^{p}_{\theta} (\R^{h}; \R^{k}), L^{p}_{{\rm loc}, \theta}
     (\R^{h}; \R^{k})$ & $L^{p}$-spaces w.r.t.~$\theta \in
                         \mathcal{M}^{+} (\R^{h})$
                           \\
  & $\DST$&the space-time domain $\III\times \R^d$.
\end{longtable}
\end{small}
\end{center}


\subsection{Preliminaries of measure theory}
%

\paragraph{\bf Finite and Radon vector measures.} 
We denote by   $\calMb ( \R^h;\R^m) $  the space of Borel measures
$\mmu \colon  \mathcal{B}(\R^h) \to \R^m$ with finite total variation $\| \mmu \|_{\mathrm{TV}} : = |\mmu|(\R^h)<+\infty$, where for 
every $B \in \mathcal{B}(\R^h)$
\[
 |\mmu|(B): = \sup \left\{ \sum_{i=0}^{+\infty} 
 \|\mmu(B_i)\| \nc 
 \,: \ B_i\in  \mathcal{B}(\R^h), \ B_i \text{ pairwise disjoint,} 
 \ B = \bigcup_{i=0}^{+\infty} B_i \right\},
\]
 and
 $\|{\cdot}\|$ is a norm in $\R^m$. 
 $(\calMb(\R^h;\R^m); \|{\cdot}\|_{\mathrm{TV}})$  is a Banach space.  
We recall 
that a Radon vector measure  in  $\calMb(\R^h;\R^m)$  is a
set function $\llambda \colon  \mathcal{B}_{b}(\R^h) \to
\R^m$ 
such that
for every compact subset $K\Subset \R^h$
 its restriction to  $\mathcal B(K)$ is a (vector) measure 
  with finite total variation. 
 \par
We identify $\llambda\in\calM(\R^h;\R^m)$ with a vector $(\lambda_1,\lambda_2,\ldots, \lambda_m)$ of $m$ measures in $\calM(\R^h)$, so that its integral with a continuous $\R^m$-valued
 function with compact support $\zzeta \in \Cc(\R^h;\R^m)$ is given by 
\begin{equation}
\label{eq:vector-integral}
   \int_{\R^h} \zzeta(x) \dd \llambda(x): = \sum_{i=1}^m \int_{\R^h} \zeta_i(x) \dd \lambda_i(x)\,.
\end{equation}
By the above duality pairing,  $\calM(\R^h;\R^m)$ can be identified with the dual of $\Cc(\R^h;\R^m)$ and is 
thus endowed with the corresponding weak$^*$ topology; for the associated convergence notion we will use the symbol $\weaksto$. 
\par
\nc For every $\llambda \in \calM(\R^h;\R^m)$ 
and every open subset $O\subset \R^h$
 we 
 have that 
 \[
 |\llambda|
(O): = \sup\left\{  \int_{\R^h} \zzeta(x) \dd \llambda(x)\, : \ \zzeta \in \Cc(\R^h;\R^m)\,, \ \spt(\zzeta) \subset O, \
\sup_{x\in O}\|  \zzeta(x) \|_* \leq 1 
\nc
\right \}
 \]
   Clearly, the choice of the norm  $\norm{\cdot}$ 
   on $\R^m$ 
   (and its dual $\|{\cdot}\|_*$)
   affects the definition of 
  the total variation measure 
 $|\cdot|$, which depends  on $\norm{\cdot}$.
 \nc 
 \nc 
 The set function $|\llambda|\colon  \calM(\R^h;\R^m) \to [0,+\infty]$ is a positive 
 Radon measure 
 and 
 every $\llambda \in \calMb(\R^h;\R^m) $ admits the \emph{polar decomposition}
 $\llambda = \bww |\llambda|$ for some 
  Borel map $\bww:\R^h\to\R^m$ 
 with $\|\bww\| \equiv 1$ $|\llambda|$-a.e.\ in $\R^h$. 
 \nc 
 It is trivial to check that 
the integral of \eqref{eq:vector-integral} can also be written as 
\begin{equation}
	\label{eq:vector-integral2}
	\int_{\R^h}\zzeta(x)\,\dd\llambda(x)=
	\int_{\R^h}\zzeta(x)\cdot\bww(x)\,\dd |\llambda|(x)
\end{equation}
and the previous formula can also be used to define a vector integral for a scalar function.

 \paragraph{\bf Weak$^*$ and narrow convergence.}
Every sequence  $(\llambda_k)_{k} \subset \calM(\R^h;\R^m)$ such that 
\[
 \sup_k |\llambda_k|(B_R) <+\infty \qquad \text{for every $R>0$}
 \] 
  admits a subsequence~$(\llambda_{k_j})_{j }$  weakly$^*$-converging to some  $\llambda \in \calM(\R^h;\R^m)$; furthermore, the sequence $(|\llambda_{k_j}|)_{j }$   weakly$^*$ converges to some $\lambda \in \calM^+(\R^h)$ such that 
  $\lambda \geq |\llambda|$.
If  $\sup_k |\llambda_k|(\R^h) <+\infty$, then  up to a subsequence  the measures $(\llambda_k)_k$ weakly$^*$ converge to some $\llambda \in \calMb(\R^h;\R^m)$. 
\par
  We recall that a sequence $(\mu_k)_k\subset \calMb(\R^h)$ \emph{narrowly} converges to $\mu \in \calMb(\R^h)$ if 
  \[
  \lim_{k\to\infty} \int_{\R^h} \varphi(x) \dd \mu_k(x) = \int_{\R^h} \varphi(x) \dd \mu(x)\qquad \text{for all } \varphi \in \Cb(\R^h). 
  \]
  Prokhorov's Theorem \cite[III-59]{Dellacherie-Meyer} asserts  that \NC a subset $M\subset \calMb(\R^h)$ has   compact closure in this topology if and only if it is bounded in the total variation norm $|{\cdot}|$ and  equally tight, namely
  \[
  \forall\, \eps>0 \ \ \exists\, K \Subset \R^h \, : \quad \sup_{\mu \in M} |\mu|(\R^h{\setminus}K) \leq \eps\,.
  \]
On $\Prob(\R^h)$  the narrow topology coincides with the weak$^*$ topology.  

\paragraph{\bf Restriction and push-forward of measures.}
For every $\mmu  \in \calM(\R^h;\R^m) $ and $A \in 
\calB( \R^{h} )$ we denote by $\mmu \mres A\in \calM(\R^h;\R^m) $ the restriction of $\mmu$ to $A$, i.e.\
$ \mmu \mres A(B)  := \mmu(B{\cap}A)$ for every $B \in \calB_{b}(\R^h)$. 
We shall use that, whenever $\mmu_n \weaksto \mmu $ in  $\calM(\R^h;\R^m) $  and $A \subset \R^{h}$ is open, then
$ \mmu_n \mres A \weaksto \mmu \mres  A$ in $\calM (A; \R^{m})$. 
\par
Let $\pushm \colon  \R^h \to \R^k$ be a Borel map. For every $\mmu \in \calMb(\R^h;\R^m)$ we define  the push-forward measure $\pushm_{\sharp}\mmu$
 in $\calMb (\R^k;\R^m)$ via
\[
\pushm_{\sharp}\mmu(B) :=  \mmu (\pushm^{-1}(B)) \qquad \text{for all } B \in\calB(\R^k)\,.
\]
In general, the above definition can  be extended to define a measure in $\calM (\R^k;\R^m)$  from a measure in
$ \calM(\R^h;\R^m)$  if, in addition, the mapping $\pushm\colon  \R^h \to \R^k$  is continuous and \emph{proper}, namely
for every compact subset  $K \in \R^k$
we have that $\pushm^{-1}(K)  $ is a compact subset of $\R^h$. Under this condition, we have that
 (cf., e.g., \cite[Rmk.\ 1.7]{AmFuPa05FBVF}) if 
$\mmu_n \weaksto \mmu$ in $ \calM(\R^h;\R^m)$, then $\pushm_{\sharp}\mmu_n \weaksto \pushm_{\sharp}\mmu$ in 
$ \calM(\R^k;\R^m)$.   We further notice that if $(\mmu_{n})_n, \mmu \in \calMb (\R^{h}; \R^{m})$, $(\mmu_{n})_n$ converges narrow to $\mmu$, and $\pushm\colon \R^{k} \to \R^{h}$ is continuous, then $(\pushm_{\sharp} \mmu_{n})_n$ converges narrow to $\pushm_{\sharp} \mmu$. 
\paragraph{\bf The Wasserstein distance on $ \Puno(\R^h)$.}
 We recall that the  distance $W_1$ on $\Puno(\R^h)$ is defined by 
\begin{equation}
\label{W1}
W_1(\mu_1,\mu_2): = \min\left\{ \int_{\R^h \times  \R^h} \norm{x{-}y} \dd \gamma(x,y)\, : \ \gamma \in \Prob(\R^h {\times}  \R^h), \, \pi^i_{\sharp }\gamma = \mu_i, \ i \in \{1,2\}\right\}.
\end{equation}
Again, notice that the above definition depends on the choice of the norm $\norm{\cdot}$ on $ \R^h $. 
For a given curve  $\mu \colon \III\to \Puno(\R^h)$  we will denote by
 $\Var_{W_1}$ its total variation w.r.t.\  $W_1$, defined on every $[a,b]\subset \III$ by 
\begin{equation}
\label{W1-variation}
\Var_{W_1}(\mu; [a,b]): = \sup \left \{ \sum_{i=1}^{n} W_1 (\mu_{t_{i-1}}, \mu_{t_i}) \, :  a= t_0< t_1<\ldots<t_n=b \right\}\,.
\end{equation}
We will denote by $\BV_{\mathrm{loc}}(\III; \Puno( \R^h ))$ the space of curves  $\mu\colon  \III\to \Puno( \R^h )$   such that $\Var_{W_1}(\mu; [a,b])<+\infty$ for every $[a,b]\subset \III$.
Finally, we  recall (cf.\ \cite[Thm.\ 1.1.2]{AGS08}) that for any  $\mu \in \AC_{\mathrm{loc}} (\III; \Puno( \R^h)) $ the limit
\[
|\mu_t'|_{W_1}: = \lim_{h \to 0} \frac1h W_1(\mu_t,\mu_{t+h}) \qquad \text{exists } \foraa\, t \in (0,+\infty).
\]
 \paragraph{\bf Metrizable spaces.} Following 
\cite[III-16]{Dellacherie-Meyer}
A  topological space 
$(X, \tau)$
is called
\begin{itemize}
\item[-]
 \emph{Polish} if it can be endowed with a metric $d$ 
 inducing the topology $\tau$ such that $(X,d)$ is  a
 complete separable metric space;
 \item[-] \emph{Lusin} if it is 
 the injective and continuous image of
 a Polish space. 
 \end{itemize}

%

%
%

\subsection{Submeasures and minimality}
In the spirit of the definition of \emph{subcurrent}
from \cite[Def.\ 3.1]{Paolini-Stepanov2013}, we introduce the concept of `submeasure' and the induced order relation on $  \calM (O; \R^{d})$, 
where $O$ is some locally compact topological space.
\begin{definition}
\label{d:submeasure}
Let $\ttheta, \zzeta \in \calM (O; \R^{k})$. 
  We say that $\zzeta$ is a \emph{submeasure} 
of~$\ttheta$  and write $\zzeta \prec \ttheta$ if 
\begin{equation}
\label{characterization-submeas}
 \exists\, \lambda \in L^{\infty}_{| \ttheta|} (O; [0, 1]) \quad \text{such that} \quad \zzeta = \lambda \ttheta. 
 \end{equation}
 \end{definition}
 It can be immediately checked that  $\prec$ is an order relation, and that it fulfills
 \begin{equation}
 \label{condition-4-equality}
 \left(  \zzeta \prec \ttheta \text{ and } |\ttheta|(O) \leq
   |\zzeta|(O)
    <+\infty \nc
 \right) \, \Longrightarrow \, \zzeta = \ttheta\,. 
 \end{equation}
The relation $\prec$ can also be characterized by the following result.
\begin{lemma}
\label{r:sub}
Let  $\ttheta, \zzeta \in \calM (O; \R^{k})$
and let $\bww$ be the Lebesgue density of the polar decomposition
of $\ttheta$, i.e.~$\ttheta=\bww|\ttheta|$. The following properties are equivalent:
\begin{enumerate}[label={\emph(\roman*\emph)}, ref=(\roman*)]
\item \label{r:sub:1} $\zzeta \prec \ttheta$,
\item \label{r:sub:2} $|\zzeta|\le |\ttheta|$ and $\zzeta=\bww |\zzeta|$,
\item \label{r:sub:3} (assuming the norm $\|{\cdot}\|$ in $\R^k$ is strictly convex)
there exists $\zzeta_{C} \in \calM (O; \R^{k})$ such that $\ttheta = \zzeta + \zzeta_{C}$ and $| \ttheta| = | \zzeta| + | \zzeta_{C}|$.
\end{enumerate}
\end{lemma}
 \begin{proof}
  The implications $\ref{r:sub:1}\Rightarrow
 \ref{r:sub:2},\ref{r:sub:3}$ are obvious. 

 In order to prove 
 $\ref{r:sub:2}\Rightarrow
 \ref{r:sub:1}$ we observe that $|\zzeta|=\lambda |\ttheta|$
 for some $\lambda\in L^\infty_{|\ttheta|}(O,[0,1])$ since $|\zzeta|\le |\ttheta|$,  so that $\zzeta=\bww \lambda |\ttheta|=\lambda \ttheta.$ \nc 
 
  As for $\ref{r:sub:3}\Rightarrow
 \ref{r:sub:1}$
let $\zzeta_{C} \in \calM (O, \R^{k})$ be such that $\ttheta = \zzeta + \zzeta_{C}$ and $|\ttheta | = | \zzeta| + |\zzeta_{C}|$. 
Then, $\zzeta, \zzeta_{C} \ll |\ttheta|$. 
Denoting by
 $\boldsymbol v \coloneq \frac{\di \zzeta}{\di |\ttheta|} $ 
we may write $\zzeta = \boldsymbol v | \ttheta|$
 and $ \zzeta_{C} = (\boldsymbol w - \boldsymbol v) | \ttheta|$. Now,  the function
    $\boldsymbol v \in L_{ | \ttheta| }^{1}(O; \R^{k})$
 satisfies $  1 = \|\bww\| = \| \bww - \bvv\| + \| \bvv \|$ $|\ttheta|$-a.e.~in~$O$. 
 Since $\|{\cdot}\|$ is strictly convex,
 we deduce $\bvv = \lambda \, \bww$ 
 for some  $\lambda \in L^{\infty}_{| \ttheta|} (O; [0, 1]) $,  i.e.\ \eqref{characterization-submeas} holds. \NC 
\end{proof}
We now consider 
the previous order relation in  the particular case when 
$O$ is 
an open subset of space-time Euclidean space $\R^{d+1} = \R\times \R^d$
(whose elements will be denoted by $(t,x)$), 
and measures have the same $x$-distributional divergence.
The operator $\dive$  is to be understood with respect to the `spatial' variable $x \in \R^d$. This gives
rise to the following definition. 
%
%
\begin{definition}[Minimal vector measures]
\label{d:minimal}
Let $O$ be an open subset  of~$\R\times \R^{d}$   and let $\ttheta
\in \calM (O; \R^{d})$.
We say that $\ttheta$ is \emph{minimal}
if the following property holds:  
\begin{equation}
	\label{eq:minimality}
	\text{whenever $\zzeta \in  \calM (O;
 \R^{d})$ fulfills \quad $\dive \zzeta = \dive \ttheta$ and $\zzeta \prec \ttheta$, \quad then \quad $\zzeta = \ttheta$.}
\end{equation}
\end{definition}
 We illustrate this concept with the following example, where a  \emph{minimal} measure is constructed by juxtaposing the measures carried by finitely many regular and \emph{injective} curves. 
\begin{example}
\label{ex:4minimality}
{\sl
Let
$(\varrho_i)_{i=1}^n$ be  a family of regular injective curves in $\R^d$,
with disjoint image sets.  Let
$t_{\varrho_{i}}$, $i=1,\ldots, n$ be their  tangent vector fields, and 
  $r_{\varrho_i}\colon [0, L_{\varrho_i}] \to \R^d$ 
 their arclength parametrizations.
 Let $(a,b) \subset \R$ be an arbitrary interval and $\lambda \in \calMb((a,b)). $   Then,
  the  measure $\ttheta \in \calM ((a,b){\times}\R^d;\R^d)$ defined by 
\[
\ttheta:  =  
 \lambda 
\otimes \boldsymbol{m}_{\rho} \qquad \text{with } \boldsymbol{m}_\rho: =  \sum_{i=1}^n\mathbf{t}_{\varrho_i}  \mathcal{H}^1  \mres {\varrho_i}  
\]
is
minimal. 
\par
Indeed, let
 $\zzeta \in  \calM ((a,b){\times}\R^d;\R^d)$ fulfill
  $\zzeta \prec \ttheta$ and  \STE $\dive \zzeta = \dive \ttheta$.  Then,  there exists $\ell  \in L^{\infty}_{|\ttheta|} ((a,b) {\times} \R^{d}; [0, 1])$ such that $\zzeta = {\ell} \ttheta$. Moreover, for every $\varphi \in \mathrm{C}^{1}_{\mathrm{c}} ((a,b) {\times} \R^{d})$ we have that
  \begin{align}\label{e:rrho-ex}
  \int_{(a, b)}  \sum_{i=1}^n \int_{0}^{L_{\varrho_i}} \frac{\di}{\di s} \varphi(t, r_{\varrho_i} (s)) \, \di s \, \di  {\lambda} (t) & =
   \int_{(a, b)}  \sum_{i=1}^n \int_{\R^d} \rmD_{x} \varphi(t, x) {\, \cdot\,} \mathbf{t}_{\varrho_i} (x) \, \di  ( \mathcal{H}^{1}\mres \varrho_i)    (x) \, \di {\lambda}(t)
  \\
  &
   = \iint_{(a,b) {\times} \R^{d}} \rmD_{x} \varphi(t, x)\, \di \ttheta(t, x) 
  = \iint_{(a,b) {\times} \R^{d}} \rmD_{x} \varphi(t, x) \, \di \zzeta(t, x)  \nonumber
  \\
  &
  = \int_{(a, b)}   \sum_{i=1}^n \int_{\R^d} {\ell}(t, x)  \rmD_{x} \varphi(t, x) {\, \cdot\,} \mathbf{t}_{\varrho_i} (x)  \, \di  ( \mathcal{H}^{1}\mres \varrho_i)   (x) \, \di  {\lambda}(t) \nonumber
  \\
  &
   =   \int_{(a, b)} \sum_{i=1}^n \int_{0}^{L_{\varrho_i}} {\ell}(t, r_{\varrho_i}(s)) \frac{\di}{\di s} \varphi(t, r_{\varrho_i} (s)) \, \di s \, \di {\lambda} (t)  \,. \nonumber
  \end{align}
  Taking $\varphi (t, x) = \psi_{1}(t) \psi_{2}(x)$ for $\psi_{1} \in \mathrm{C}^{1}_{\mathrm{c}} ((a,b))$ and $\psi_{2} \in \mathrm{C}^{1}_{\mathrm{c}} (\R^{d})$, we deduce from~\eqref{e:rrho-ex} that for  ${\lambda}$-a.a.~$t \in (a,b)$ 
\[
\sum_{i=1}^n  \int_{0}^{L_{\varrho_i}} \frac{\di}{\di s} \psi_{2}( r_{\varrho_i} (s)) \, \di s  = \sum_{i=1}^n \int_{0}^{L_{\varrho_i}} {\ell}(t, r_{\varrho_i} (s))  \frac{\di}{\di s} \psi_{2}( r_{\varrho_i} (s)) \, \di s \,.
  \]
  Since each $\varrho_i$ is 
  regular and injective, and 
   their image sets   are disjoint, we infer that   for every $\psi \in \mathrm{C}^{1} ( [ 0, L_{\rho_{i}} ] )$,
   for  ${\lambda}$-a.a.~$t \in (a,b)$ and every $i =1, \ldots, n$
    it holds 
    \begin{equation}
  \label{e:rrho-3-ex}
 \int_{0}^{L_{\varrho_i}} \Big(1- \ell(t, r_{\varrho_i} (s)) \Big)\frac{\di}{\di s} \psi (s ) \, \di s  =0.
  \end{equation} 
    Choosing a countable set of test functions 
    strongly dense in $\mathrm C^1([0,L])$, with 
    $L:=\max_i L_{\varrho_i}$, then for every $i=1,\ldots, n$ we have that 
 $1-{\ell} (t, r_{\varrho_i} (\cdot))=0$
 a.e.~in~$[0, L_{\varrho_i}]$ for  ${ \lambda}$-a.a.~$t \in (a,b)$.
Thus, ${\ell}\equiv 1
$ $|\ttheta|$-a.e.,   hence $\zzeta = \ttheta$. 
}
\end{example} 

 In the next two statements we discuss the existence of minimal submeasures. We start with the case of bounded Radon measures.

\begin{proposition}
\label{p:submeasure}
Let $O$ be an open subset   of $\R^{d+1}$   and $\ttheta \in \calMb(O; \R^{d})$. Then, the problem
\begin{equation}
\label{e:min-sub}
\min \, \{ | \zzeta | (O) : \zzeta \prec \ttheta \text{ and } \dive \zzeta = \dive \ttheta\}
\end{equation}
admits a solution. Moreover, every solution  to~\eqref{e:min-sub}  is minimal.
\end{proposition}
 We point out for later use that, by  Lemma~\ref{r:sub},  the minimum problem~\eqref{e:min-sub} can be reformulated in terms of densities:
\begin{equation}
\label{e:min-sub-2}
\min\, \bigg\{ \int_{O} \lambda \, \di | \ttheta | : \, \lambda \in
L^{\infty}_{| \ttheta|} (O; [0, 1]) , \, \int_{O} (1 - \lambda )
 \rmD_x \nc
\varphi \, \di \ttheta = 0 \ \text{for every $\varphi \in   \rmC_{\rm c}^1 (O)$}\bigg\}\,. 
\end{equation}
\begin{proof}
A solution to the minimum problem \eqref{e:min-sub-2},  and thus to 
\eqref{e:min-sub},  
 exists, since the constraint is convex and weakly$^{*}$-compact and the functional is weakly$^{*}$-continuous.
\par
Let $\zzeta \in \calMb (O;  \R^{d})$ be a solution of~\eqref{e:min-sub} and let $\widetilde{\zzeta} \in \calMb(O;  \R^{d})$   be such that $\dive \widetilde{\zzeta} = \dive \zzeta$ and $\widetilde{\zzeta} \prec \zzeta$. Then, $\dive \widetilde{\zzeta} = \dive \ttheta$ and $\widetilde{\zzeta} \prec \ttheta$, and $\widetilde{\zzeta}$ is a competitor for~\eqref{e:min-sub},  so that $|\zzeta|(O) \leq |\widetilde{\zzeta}|(O)$. Hence, by  \eqref{condition-4-equality} we conclude that $\widetilde{\zzeta} = \zzeta$. 
\end{proof}

 With our following result we show that 
the existence of  minimal submeasures  extends to the case in which 
 $\ttheta$ is just a \emph{Radon} measure in a cylindrical open set $(a,b)\times \R^d$ (we again emphasize that the 
 divergence operator is only considered w.r.t. the variable $x\in \R^d$). 
\begin{corollary}
\label{c:submeasure}
Let $\ttheta \in \calM ((a,b){\times} \R^d;   \R^{d})$   be such that 
$|\ttheta| ([c,d]{\times}\R^d)<+\infty$ for every $[c,d]\subset (a,b)$. Then
 there exists $\zzeta \in \calM((a,b){\times} \R^d;   \R^{d})$   minimal such that $\zzeta \prec \ttheta$ and $\dive \zzeta = \dive \ttheta$.
\end{corollary}
\begin{proof}
We consider two sequences $(a_j)_j, \, (b_j)_j \subset (a,b)$ with $a_j \searrow a$ and $b_j \nearrow b$ as $j\to \infty$, and set $O_j : =  (a_j,b_j) \times\R^d$. 
 By assumption, for every $j\in \N$ the restriction $\ttheta_{j} \coloneq  \ttheta \mres  O_{j}$   belongs to~$\calMb (O_j;  \R^{d})$; we denote by $\mathrm{div}|_{O_j}$ the divergence operator relative to the open set $O_j$, i.e., restricted to test functions with a compact support in $O_j$, and observe that $\mathrm{div}|_{O_j}(\ttheta_j) = \mathrm{div}|_{O_j}(\ttheta)$. 
We can apply Proposition~\ref{p:submeasure} and find, for every $j\in \N$, a minimal measure 
\[
\zzeta_j \in \calMb (O_j; \R^{d}) \text{ such that }    \zzeta_j \prec \ttheta_j \text{ and } \dive|_{O_j} ( \zzeta_j ) = \dive|_{O_j} (\ttheta_j)\,.
\]
\par
 We now show that it is not restrictive to assume that 
\begin{equation}
\label{constant-zzeta-j}
  \zzeta_j  \mres O_{\ell} = \zzeta_\ell \qquad \text{if } \ell \leq j\,.
\end{equation}
Indeed, since $\zzeta_j \prec \ttheta_j$ and $\zzeta_\ell \prec \ttheta_\ell$, there exist $\lambda_j \in L_{|\ttheta_j|}^\infty (O_j; [0,1])$ and $\lambda_\ell \in L_{|\ttheta_\ell|}^\infty (O_\ell; [0,1])$ such that 
$\zzeta_j = \lambda_j \ttheta_j$ and 
$\zzeta_\ell = \lambda_\ell \ttheta_\ell$. \RRR Recall that
$\lambda_j$ and $\lambda_\ell$ solve the minimum problem
 \eqref{e:min-sub-2} on~$O_{j}$ and on~$O_{\ell}$, respectively. Define now $ \widehat{\lambda}_{j} \colon  O_j \to [0,1]$ via
 \[
 \widehat{\lambda}_j(t,x): = \begin{cases}
  \lambda_\ell (t,x) & \text{if } (t,x) \in O_\ell,
 \\
 \lambda_{j}(t,x) & \text{if } (t,x) \in O_j \setminus O_\ell\,.
 \end{cases}
 \] 
 Then, $\widehat{\lambda}_j  \in L_{|\ttheta_j|}^\infty (O_j; [0,1])$ and the measure $\widehat{\zzeta}_j: = \widehat{\lambda}_j  \ttheta_j$ clearly fulfills $\widehat{\zzeta}_j \prec \ttheta_j$,
 $\mathrm{div}|_{O_j}(\widehat\zzeta_j) =  \mathrm{div}|_{O_j}(\ttheta_j) $, and $ \widehat{\zzeta}_{j}  \mres  O_{\ell} = \zzeta_{\ell}$. By minimality of~$\zzeta_{\ell}$ on~$O_{\ell}$, we have that
  \[
 |\widehat{\zzeta}_j|(O_j)  = |  \widehat{\zzeta}_j|(O_\ell) + | \widehat{\zzeta}_j|(O_j \setminus O_{\ell}) = | \zzeta_{\ell}| (O_{\ell}) + | \zzeta_{j} | (O_{j} \setminus O_{\ell}) \leq | \zzeta_{j}|(O_{j})\,.
 \]
 Hence, $\widehat{\zzeta}_{j}$ solves the minimum problem~\eqref{e:min-sub-2} on~$O_{j}$ and is minimal. This implies that, up to replacing~$\zzeta_{j}$ with~$\widehat{\zzeta}_{j}$, we may assume~\eqref{constant-zzeta-j}.

Let us trivially extend each $\zzeta_j$ to $O = (a,b)\times \R^d$. Since $|\zzeta_j| \leq |\ttheta_j|$ for every $j\in \N$, we find that there exists  $\zzeta\in \calM (O;\R^{d})$ such that, up to a subsequence, $\zzeta_{j} \rightharpoonup^{*} \zzeta$ in $\calM (O;\R^{d})$.  Thus, $\dive \zzeta = \dive \ttheta$.  By the lower semicontinuity of the total variation and by the relation $\zzeta_{j} \prec \ttheta_{j}$ we deduce that $\zzeta \prec \ttheta$. 
\par
Let us now show that $\zzeta$ is minimal. Indeed, let $\boldsymbol{\xi} \in \calM (O; \R^{d})$ be such that $\dive \boldsymbol{\xi} = \dive \zzeta$ and $\boldsymbol{\xi} \prec \zzeta$.  In particular, $\boldsymbol{\xi} \mres O_{j}$ satisfies $\dive|_{O_j} (\boldsymbol{\xi} \mres O_{j}) = \dive |_{O_j} (\zzeta_{j})$  and $\boldsymbol{\xi}\mres O_{j} \zzeta_{j} \prec\zzeta_{j}$. Thus, $\boldsymbol{\xi}\mres O_{j} = \zzeta_{j}$ for every $j\in \N$ and $\boldsymbol{\xi} = \zzeta$. 

 \end{proof}

A crucial step in the proof of  Theorem \ref{thm:augmented} ahead will
consist in relating the weak$^*$ limits of the projections of
(weakly$^*$ converging) sequences  of positive and vector-valued
measures, with the push forward
of their weak$^*$ limits through the projection
$ \pinew(s,t,x):=(t,x)$,
$\R \times \R\times \R^d \to \R \times \R^d$,
which is not proper.
\par
  Now, the last result of this section
addresses this issue in general, for the push forward through a generic continuous map. It provides 
 sufficient conditions under which  the push forward of a  weakly$^*$ converging  sequence of measures is
 a  \emph{submeasure} of the weak$^*$-limits of their push forwards.  
\begin{lemma}
\label{l:sub-nu1}
Let $O,G$ be open subsets of some Euclidean spaces,
let $\pushm \colon O\to G$ be a continuous map,  let $\R^{k}$ be endowed with a strictly convex norm~$\| {\cdot}\|$,  
and let $(\zzeta_n)_n\subset \calM(O;\R^k) $ satisfy
\begin{displaymath}
  \sup_{n \in \mathbb{N}} \, |\zzeta_{n}|(\pushm^{-1}(K)) <+\infty
  \quad\text{for every compact
subset $K\subset G$},
\end{displaymath}
 so that $\llambda_n=\pushm_\sharp\zzeta_n$ is a well defined measure
 in $\calM(G;\R^k)$.
 Let us assume that 
\begin{equation}
\label{convergences-zeta-n1}
\begin{aligned}
\zzeta_n \weaksto \zzeta  \text{ in } \calM(O;\R^k),\qquad
\llambda_n=\pushm_{\sharp } \zzeta_n \weaksto \llambda \text{ in } \calM (G;\R^k),
\end{aligned}
\end{equation}
for some $\zzeta \in  \calM(O;\R^d),$  and $\llambda \in
\calM \FCOMM (G;\R^k)$. 
If
\begin{equation}
\label{cv-variations1}
\pushm_{\sharp } |\zzeta_n| \weaksto |\llambda| \qquad \text{in } \calM^+(G),
\end{equation}
  then
\begin{equation}
\label{they-are-submeasures1}
\pushm_{\sharp } \zzeta \prec \llambda.
\end{equation}
\end{lemma}
\begin{proof}
  Let $\eta^j \in \rmC_{\rm c} (O)$ form an increasing sequence
  such that $0 \leq \eta^j\leq 1$
  for all $j\in \N$
   and $\eta^j(x)\uparrow1$ as $j\to\infty$
  for every $x\in O$.
  For every $n\in \N$ and $j \geq 1$
  we set
  \begin{displaymath}
    \zzeta_n^j:=\eta^j\zzeta_n,\quad
    \widehat\zzeta_n^j:=(1-\eta^j)\zzeta_n,\quad
    \llambda_n:=\pushm_\sharp\zzeta_n,\quad
    \llambda^j_n:=\pushm_\sharp \zzeta^j_n,\quad
    \widehat\llambda^j_n:=\pushm_\sharp \widehat\zzeta^j_n,
  \end{displaymath}
so that 
\begin{equation}
\label{decomposition1}
\zzeta_n  = \zzeta_n^j+\widehat\zzeta_n^j,\quad
\llambda_n=\llambda_{n}^{j} +\widehat\llambda_n^j.
\end{equation}
Since the functions  $\eta^j$ have compact support, if we pass to the limit
as $n\to \infty$ while keeping $j\geq 1$ fixed, we get
\begin{equation}
  \label{eq:5}
  \zzeta_n^j\weaksto \zzeta^j=\eta^j\zzeta,\quad
  \llambda_n^j\weaksto \llambda^j=\pushm_\sharp (\zzeta^j)\quad
  \text{as }n\to\infty,
\end{equation}
and we correspondingly deduce the convergence of
$\widehat\zzeta_n^j$ and $\widehat\llambda_n^j$ to measures
$\widehat\zzeta^j$ and $\widehat\llambda^j$ respectively, satisfying the
decomposition
\begin{equation}
\label{decomposition2}
\zzeta  = \zzeta^j+\widehat\zzeta^j,\quad
\llambda=\llambda^{j} +\widehat\llambda^j.
\end{equation}
(Notice, however, that in general $\widehat\lambda^j$ does not coincide
with $\pushm_\sharp\zzeta^j$).
We can now consider similar decompositions on the level of the total
variations
\begin{displaymath}
  \alpha_n:=|\zzeta_n|,\quad
  \alpha_n^j:=\eta^j|\zzeta_n|,\quad
  \widehat\alpha_n^j:=(1-\eta^j)|\zzeta_n|,\quad
  \beta_n:=\pushm_\sharp \alpha_n,\quad
  \beta^j_n:=\pushm_\sharp \alpha^j_n,\quad
  \widehat\beta^j_n:=\pushm_\sharp \widehat \alpha^j_n,
\end{displaymath}
which satisfy
\begin{equation}
  \label{eq:6}
  \alpha_n=\alpha_n^j+\widehat\alpha_n^j,\quad
  \beta_n=\beta_n^j+\widehat\beta_n^j,\quad
  \beta_n\ge|\llambda_n|,\quad
  \beta_n^j\ge |\llambda_n^j|,\quad
  \widehat\beta_n^j\ge |\widehat \llambda_n^j|,\quad
  \beta_n\weaksto \beta=|\llambda|\quad\text{as }n\to\infty. 
\end{equation}
By a possible extraction of a (not relabeled) subsequence, it is not restrictive to
assume that there exists $\alpha\in \calM^+(O)$ such that
$\alpha_n\weaksto \alpha$ as $n\to\infty$, so that
\begin{equation}
  \label{eq:7}
  \alpha_n^j\weaksto \alpha^j=\eta^j\alpha,\quad
  \widehat \alpha_n^j\weaksto \widehat\alpha^j=(1-\eta^j)\alpha,\quad
  \beta_n^j\weaksto \beta^j=\pushm_\sharp \alpha^j,\quad
  \widehat\beta_n^j\weaksto \widehat\beta^j=\beta-\beta^j.
\end{equation}
By Cantor's diagonal argument, it is possible to extract
an increasing subsequence $m\mapsto n(m)$
and to find limit measures $\lambda^j,\widehat\lambda^j\in \calM^+(G)$ such that
for every $j\in \N$
\begin{equation}
  \label{eq:8}
  |\llambda_{n(m)}^j|\weaksto \lambda^j\ge |\llambda^j|,\quad
  |\widehat\llambda_{n(m)}^j|\weaksto \widehat\lambda^j
  \ge |\widehat\llambda^j|\quad\text{as }m\to\infty.
\end{equation}
Since
\begin{displaymath}
  |\llambda_n|\le |\llambda_n^j|+|\widehat\llambda_n^j|
\end{displaymath}
we deduce
\begin{equation}
  \label{eq:9}
  |\llambda|\le \lambda^j+\widehat\lambda^j.
\end{equation}
On the other hand, the inequalities
\begin{displaymath}
  \lambda_n^j\le \beta_n^j,\quad \widehat\lambda_n^j\le \widehat\beta_n^j
\end{displaymath}
yield
\begin{equation}
  \label{eq:10}
  \lambda^j\le \beta^j,\quad
  \widehat\lambda^j\le \widehat\beta^j,
\end{equation}
and since $\beta^j+\widehat\beta^j=\beta=|\llambda|$ we conclude that
\begin{equation}
  \label{eq:11}
  |\llambda|=\lambda^j+\widehat\lambda^j.
\end{equation}
Similarly, the inequalities $\lambda^j\ge |\llambda^j|,\
\widehat\lambda^j\ge |\widehat \llambda^j|$ and
$|\llambda|\le |\llambda^j|+|\widehat\llambda^j|$ yield
\begin{equation}
  \label{eq:12}
  \lambda^j=|\llambda^j|,\quad
  \widehat\lambda^j=|\widehat\llambda^j|,\quad
  |\llambda|=|\llambda^j|+|\widehat\llambda^j|,\quad
  \llambda=\llambda^j+\widehat\llambda^j.
\end{equation}
We deduce that $\llambda^j\prec\llambda$.
For every $\varphi\in \rmC_c(G;\R^k)$ we easily check that
\begin{align*}
  \int_G \boldsymbol\varphi\cdot \di\llambda^j=
  \int_O \eta^j\boldsymbol\varphi(\pushm(x))\cdot \di\zzeta(x)
  \longrightarrow \int_O \boldsymbol\varphi(\pushm(x))\cdot \di\zzeta(x)=
  \int_G \boldsymbol\varphi\cdot \di (\pushm_\sharp \zzeta)
  \quad\text{as }j\to\infty,
\end{align*}
i.e.~$\llambda^j\weaksto \pushm_\sharp\zzeta$.
We eventually conclude that
$\pushm_\sharp \zzeta\prec \llambda$  by (iii) of Lemma~\ref{r:sub}. 
\end{proof}
\nc
\subsection{Function spaces for the superposition principle} 
\label{ss:funct-spaces}
Recall that 
 $\III$  denotes the interval $[0,+\infty)$;
if $(\mathrm X,d)$ is a complete and separable metric space, we will endow the pathspace 
$\rmC(\III;\mathrm X)$ with the Polish topology of uniform convergence
on the  compact subsets of $\III$ (see 
Lemma \ref{le:C-is-Polish}).
We introduce the spaces
\begin{itemize}
\item[-]
$\Lip_k(\III;\mathrm X)\,,$ 
of 
$k$-Lipschitz 
paths, $k\ge0$
 which is a
(closed, thus Polish) subset of $\rmC(\III;\mathrm X)$; 
\item[-]
$\Lip(\III;\mathrm X)$ 
of Lipschitz paths;
since 
$\Lip(\III;\mathrm X)=\bigcup_{k\in \N}
\Lip_k(\III;\mathrm X),$
$\Lip(\III;\mathrm X)$ is a $F_\sigma$
(namely, a countable union of closed sets), thus a
  Borel
subset of $\rmC(\III;\mathrm X)$.
\end{itemize}
%
We introduce a few more subsets of $\rmC(\III;\R^{d+1})$:
first of all,  the set
\begin{equation}
\label{Cplus}
\rmC^{\uparrow}(\III;  \R^{d+1}) 
 := \Big\{ \sfyy=  (\sft,\sfxx)  \in  \rmC(\III;\R^{d+1}): \ \sft(0) =0, \
 \sft \text{ is non-decreasing},\ \lim_{s\uparrow \infty} \sft(s) = +\infty \Big \}.
\end{equation}
$\Cplus$ is a Polish space, in particular 
a Borel subset of $\rmC(\III;\R^{d+1})$.
In fact,  it can be written as the intersection $A\cap B$ where 
 \begin{align*}
A: ={}& \bigcap_{n \in \mathbb{N}} \Big\{ \sfyy \in \rmC(\III;\R^{d+1}):  
	\ \sup_{s\in \III} \sft(s) > n\Big\},\\
B :={}& \Big\{\sfyy  = (\sft,\sfxx)  \in  \rmC(\III;\R^{d+1}) \, : 
	\ \sft(0) =0, \
 \sft \text{ is non-decreasing}\Big\}\,.
\end{align*}  
Since
the map $\sfyy \mapsto 
\sup_{s\in \III} \sft(s) $ is lower semicontinuous in $\rmC(\III;\R^{d+1})$,  $A$ is a $G_\delta$ set  (namely, the countable intersection of open sets). 
Since $B$ is closed, 
$\Cplus$ is a $G_\delta$ as well, and thus also Polish.
\par
 We further set 
\begin{equation}
\label{ACloc}
 \Lipplus k\III{\R^{d+1}} :=
 \Lip_k(\III;\R^{d+1})\cap \Cplus, \quad 
 \Lipplus {}\III{\R^{d+1}} :=  \bigcup_{k\in \N} \Lipplus k\III{\R^{d+1}}, 
\end{equation}
which are respectively a Polish and a $F_\sigma$ subset of $ \rmC(\III;\R^{d+1})$. 
Finally, we define
\begin{displaymath}
 \ALip \III {\R^{d+1}}  
:= \{ \sfyy \in \Lipplus 1 \III {\R^{d+1}}: \, \| \sfyy'(s)\| = 1 \ \text{for a.e.~$s \in \III$}\}\,.
\end{displaymath}
We notice that 
\begin{displaymath}
 \ALip \III {\R^{d+1}}  = \bigcap_{m \in \mathbb{N}} \bigcap_{n \in \mathbb{N}} \bigg\{ \sfyy \in \Lipplus 1 \III {\R^{d+1}}: \, \int_{0}^{m} \| \sfyy'(s)\| \, \di s 
>m - \frac{1}{n}\bigg\}\,,
\end{displaymath}
so that  $\ALip \III {\R^{d+1}}$    is a 
$G_\delta$, thus Polish and Borel, subset of $ \rmC (\III; \R^{d+1})$. 

\section{$\BV$ curves and the  `relaxed' continuity equation}
\label{s:3}
 The main result of this section, Theorem \ref{th:1} below, will unveil the relation between bounded-variation curves with values in $ \Puno(\R^d)$, and the continuity equation \eqref{continuity-equation}, which,
throughout the paper, 
will be formulated  as in
 Definition
 \ref{def:solCE} below. 
 Recall that we denote by $\DST$ the space-time domain
$\III\times \R^d$, which we can consider as a subset or $\Rdpu$.
\begin{definition}[Distributional and $\Puno$-solutions to the continuity equation]
\label{def:solCE}
We call a pair $(\mu,\nnu) \in \calM^+(\DST) \times  \calM (\DST ;\R^d) $ a (forward, distributional) 
solution  to  the 
 continuity equation    
\begin{equation}
\label{continuity-equation}
	\partial_t \mu+\dive \nnu =0\quad \text{in }\DST,\quad \mu\ge0,\quad
	\text{with initial datum $\mu_0\in \calM^+(\R^d)$},
\end{equation}
if 
\begin{align}
\label{distributional-sense}
&
\iint_{\Rdpu}  \partial_t \varphi (t, x)  \, \di \mu (t, x) + \iint_{\Rdpu } \rmD \varphi (t, x)  \, \di\nnu (t, x)  = -  \int_{\R^d} \varphi(0,x) \, \di \mu_0( x)
\end{align}
for every $\varphi \in   \rmC_{\rm c}^{1}(\Rdpu)$. We say that $(\mu,\nnu)$ is a $\Puno$-solution if 
\begin{equation}
\label{bound-on-nu-rilievo}
\mu_0\in \Puno(\R^d),\qquad | \nnu|(
\DSTT T
)< + \infty \qquad \text{for every $T>0$}.
\end{equation}
\end{definition}
Observe that, since $\mu,\nnu$ are supported in $\DST$,  we could restrict the integrals
in \eqref{distributional-sense} to $\DST$.
We have integrated on  $\Rdpu$,  and thus considered test functions in   $\rmC_{\rm c}^{1}(\Rdpu)$,   to be consistent with the usual 
distributional formulation in $\mathscr D'(\Rdpu)$.


 In this paper we will mainly focus on $\Puno$-solutions;
we will also consider an important subclass characterized by 
a minimality condition.
\begin{definition}[Minimal $\Puno$-solutions]
	\label{def:reduced}
	Let $(\mu,\nnu)$ be a $\Puno$-solution to the continuity equation 
	\eqref{continuity-equation}
	 and let us consider the Lebesgue decomposition of $\nnu$ as 
	\begin{equation}
		\label{eq:Lebesgue-decomp}
		\nnu = \nnu^a+\nnu^\perp, \qquad \nnu^a \ll \mu,\quad
		\nnu^\perp \perp \mu.
	\end{equation}
	We say that $(\mu,\nnu)$ is a \emph{minimal} $\Puno$-solution
	if $\nnu^\perp$ is \emph{minimal} in the sense of 
	Definition \ref{d:minimal}. 
	\par
	Let now  $\bar\nnu\in \calM(\DST;\R^d)$.   We say that $(\mu,\bar\nnu)$ is a \emph{minimal pair induced}
	by $(\mu,\nnu)$ if 
	\begin{equation}
		\label{eq:reduction}
		\bar\nnu=\nnu^a+\bar\nnu^\perp,\quad 
		\bar\nnu^\perp \prec\nnu^\perp,\quad
		\bar\nnu^\perp\text{ is minimal according to Definition \ref{d:minimal}},
	\end{equation}
	so that in particular $(\mu,\bar\nnu)$ is a $\Puno$-solution of 
	\eqref{continuity-equation} as well. 
	\end{definition}
\begin{remark}
	\label{rem:obvious}
	The existence of a minimal pair induced by $(\mu,\nnu)$ 
	is guaranteed by Corollary \ref{c:submeasure}. 
	Notice that the pair $(\mu,\bar\nnu)$ in \eqref{eq:reduction} satisfies
\begin{equation}
	\label{eq:trivial-but-useful}
	|(\mu,\bar \nnu)|=\theta |(\mu,\nnu)|,\quad 
	(\mu,\bar\nnu)=\theta (\mu,\nnu)\prec (\mu,\nnu)\qquad \text{for } 
	\theta:\DST\to [0,1]\text{ Borel,}\quad 
	\theta=1\text{ $\mu$-a.e.}
\end{equation}
We will see that minimality can also be characterized 
directly in terms of $\nnu$.

\end{remark}

We now establish the analogue of \cite[Thm.\ 8.3.1]{AGS08}. 
\begin{theorem}
\label{th:1}
\begin{enumerate}
\item
Let $\mu \in \BVloc(\III ; \mathcal{P}_1(\R^d)) $ and let~$\mu^{\pm}$ be the left- and right-continuous representatives of~$\mu$, respectively. Then, there exists a Borel measure
$\nnu  \in \calM (\DST;\R^d)$ such that for every $T\in [0, +\infty)$
\begin{equation}
\label{UEST}
|\nnu |([0,T) \times \R^d)  =  \mathrm{Var}_{W_1} (\mu^{-}; [0,T]), \qquad  |\nnu |((0,T] \times \R^d)  =  \mathrm{Var}_{W_1} (\mu^{+}; [0,T]),  
\end{equation}
and the pair $(\mu,\nnu)$  
is a  minimal $\Puno$-solution to  the continuity equation \eqref{continuity-equation} in the
sense of  Definition~\ref{def:reduced}.  
\ref{d:minimal}. 
\item Conversely, if $(\mu,\nnu)$ is a $\Puno$-solution
 to  the continuity equation in the sense of  Definition~\ref{def:solCE}
with initial datum $\mu_0\in \Puno(\R^d)$, 
  then
  \begin{enumerate}
  \item $\pi^{0}_{\sharp }\mu = \mathcal{L}^{1}$  (with 
  $\pi^{0}\colon\DST \to \III$ the projection
 $(t,x)\mapsto t$),  
     in particular
    $ \mu ([0, T) \times \R^{d}) = T $
    for every $T \in [0, +\infty)$;
\item  there exists a curve $t\mapsto \mu_t \in \BV_{\mathrm{loc}}(\III;  \Puno(\R^d)) $ such that 
 $\mu =\mathcal{L}^{1} \otimes  \mu_{t}$. The curve
 admits a narrowly left-continuous  representative $\mu^-$ 
(a  right-continuous representative $\mu^+$, respectively),
such that 
 $\mu^-(0):=\mu_0$ , 
 the functions
$t\mapsto \mu_t^{\pm }$ belong to $\mathrm{BV}\kern-1pt_{\rm loc} (\III ; \mathcal{P}_1(\R^d))$, 
and there holds 
\begin{equation}
\label{estimates-w-Var}
    \mathrm{Var}_{W_1}(\mu^-; [a,b]) \leq |\nnu |([a,b) \times \R^d), \quad   \mathrm{Var}_{W_1}(\mu^+; [a,b]) \leq |\nnu |((a,b] \times \R^d)
     \text{ for all } [a,b] \subset \III\,.
\end{equation}
Furthermore, for   every $0  \leq a  <  b <+\infty$
and   
 $\varphi \in  \Cc^1(\R_+^{d+1})$,    there holds 
 \begin{subequations}
\label{new-e.5.1}
\begin{align}
&
\label{e.5.1}
\begin{aligned}
& \int_{\R^{d}} \!\! \varphi ( b, x ) \, \di { \mu }_{b}^{-}(x) - \int_{\R^{d}} \!\!\varphi ( a, x )\,\di { \mu }_{a}^{+} (x) 
= \int_{a}^{b} \!\! \int_{\R^{d}}\!\! \partial_{t} \varphi ( t , x ) \, \di \mu_{t}(x) \, \di t + \iint_{(a,b)\times\R^{d}} \!\!\!\!\!\!\!\!\!\!\!\!\!\! \rmD\varphi(t,x) \, \di \nnu (t,x) \,,
\end{aligned}
\\
&
\label{e.5.1-0}
\begin{aligned}
& \int_{\R^{d}} \!\! \varphi ( b, x ) \, \di { \mu }_{b}^{-}(x) - \int_{\R^{d}} \!\!\varphi ( 0, x )\,\di { \mu }_{0} (x) 
= \int_{0}^{b} \!\! \int_{\R^{d}}\!\! \partial_{t} \varphi ( t , x ) \, \di \mu_{t}(x) \, \di t + \iint_{[0,b)\times\R^{d}} \!\!\!\!\!\!\!\!\!\!\!\!\!\! \rmD\varphi(t,x) \, \di \nnu (t,x) \,.
\end{aligned}
\end{align}
\end{subequations}
\end{enumerate}
 In particular, for every    $\varphi \in \Cc^{1}(\R^{d})$   and $b \in  \III  $ there holds 
\begin{align}
&
 \int_{\R^{d}} \!\! \varphi ( x ) \, \di { \mu }_{b}^{+}(x) - \int_{\R^{d}} \!\!\varphi (x )\,\di { \mu }_{b}^{-} (x) 
=  \iint_{\{b\}\times\R^{d}} \!\!\!\!\!\!\!\!\!\!\! \rmD\varphi(x) \, \di \nnu (t,x)\,. \label{e.5.1-2}
\end{align} 
\end{enumerate}
\end{theorem}
\noindent In fact, 
a partial analogue of part (1) of the statement has been established for $\BV$ curves of currents in
 \cite[Thm.\ 6.1, Prop.\ 6.4]{BonDNiRin} (see also~\cite[Theorems~6.1 and~6.2]{Bonicatto_2024_MJM}).  
We will  develop the \emph{proof} of Theorem 
\ref{th:1}  in the ensuing subsections, starting from the second part of the statement. 

\begin{remark}
\label{rmk:mu-0}
\upshape
The definition $\mu^-(0):=\mu_0$ for the left-continuous representative of the curve $t\mapsto \mu_t$ associated with a solution  to   the continuity equation,
reflects the fact that, if  $|\nnu|(\{0\}{\times}\R^d)>0$   the curve $t\mapsto \mu_t$ has a jump at $t=0$. Hence,
$\mu^+(0) \neq \mu^-(0)$ and it is meaningful to set $\mu^-(0):=\mu_0$.
\end{remark}
\begin{remark}[Continuity equation in ${[}a,b{]}\times \R^d$] 
\label{rmk:0,T}
Let $0\le a<b$,
 $\mu_{a},\mu_b \in \Puno(\R^{d})$, and 
$(\mu,\nnu) \in \calMb^+([a,b]\times \R^d) \times  \calMb ([a,b]\times \R^d;\R^d) $
 (recall that $ \calMb(A;\R^m)$ denotes the space of  $\R^m$-valued Borel measures with finite total variation)
  satisfy
\begin{displaymath}
\iint_{[a,b]\times \R^d} \partial_{t} \varphi (t, x) \, \di \mu(t,
x) + \iint_{[a,b]\times \R^d} \rmD  \varphi(t, x) \, \di \nnu(t,
x) =
\int_{\R^{d}} \varphi(b, x) \, \di \mu_{b}(x)
- \int_{\R^{d}} \varphi(a, x) \, \di \mu_{a}(x)
\end{displaymath}
for every  $\varphi \in \Cc^1([a,b]\times \R^d)$.
It can be immediately checked that  the extensions
$\widetilde\mu,\widetilde\nnu$ defined for every
Borel set $A\subset \DST$ by 
\begin{align*}
  \widetilde{\mu}(A):={}&\mu\Big(A\cap \big([a,b]\times \R^d\big)\Big)+
  (\mathcal L^1{\otimes}
  \mu_a)\Big(A\cap \big([0,a)\times \R^d\big)\Big)
  +
  (\mathcal L^1{\otimes}
  \mu_b)\Big(A\cap \big(b,+\infty)\times \R^d\big)\Big),\\
%
\widetilde{\nnu}(A):={}& \nnu\Big(A\cap \big([a,b]\times \R^d\big)\Big)
\end{align*}
solve~\eqref{distributional-sense} in the sense of Definition~\ref{def:solCE}.
\end{remark} 

\nc
\subsection{Proof of Part (2) of Thm.\ \ref{th:1}}
The proof is carried out in several steps. First of all,
we
 show that, if $\mu$ solves \eqref{continuity-equation}, then its marginal w.r.t.\ the time variable coincides with the $1$-dimensional Lebesgue measure on   $\III$.  
  We also provide a useful chain-rule formula. 
\begin{lemma}[Time marginals and distributional chain rule for
$\Puno$-solutions]
\label{l.1}
Let $\mu_{0}\in \Puno(\R^{d})$ and let $(\mu,\nnu)$ be a $\Puno$-solution of the continuity equation in the sense of  Definition~\ref{def:solCE}.    Then,  $ \mu (\RRR [0, T) \times \R^{d} ) = T $ for all $T>0$, 
$\pi^{0}_{\sharp }\mu = \mathcal{L}^{1}$   and  $\mu = \mathcal{L}^{1} \otimes \mu_t  $ for a family of probability measures 
$(\mu_t)_{t\in  \III}$ in $\mathcal P_1(\R^d)$ with finite first moment.
Furthermore, for every Lipschitz function $\varphi\in \rmC^1(\DST)$ 
the map
$\displaystyle t\mapsto \int_{\R^d} \varphi(t,x)\,\dd\mu_t(x)$ 
(trivially extended to $0$ for $t<0$) has 
distributional derivative
\begin{equation}
\label{eq:distributional-derivative}
	\frac{\dd}{\dd t}
	\int_{\R^d} \varphi(t,\cdot)\,\dd\mu_t=
	\int_{\R^d} \partial_t \varphi(t,\cdot)\,\dd\mu_t+
	\pi^0_\sharp (\rmD\varphi\cdot \nnu)
	+\delta_0\int_{\R^d}\varphi(0,\cdot)\,\dd\mu_0
\quad\text{in }\mathcal D'(\R)\,,
\end{equation}
 (where the scalar product  $\rmD\varphi{\cdot} \nnu$ has to be understood in the sense of  \eqref{eq:vector-integral}),  
and in particular it  has essential bounded variation
	in every bounded interval $(0,T)$, $T>0$.
\end{lemma} 
\begin{proof}
 Let us first observe that selecting $\zeta\in \rmC^1_c(\III)$
and $\varphi_c\in \rmC^1(\DST)$ with support in 
$\III\times B_R(0)$ for some $R>0$, 
\eqref{distributional-sense} yields
\begin{equation}
	\label{eq:distributional-sense-3}
	\iint_{\DST}(\zeta'\varphi_c+\zeta\partial_t\varphi_c)\,\dd\mu+
	\iint_{\DST}\zeta\rmD\varphi_c  \,\dd\nnu  =
	-\int_{\R^d}\zeta(0)\varphi_c(0,x)\,\dd\mu_0(x).
\end{equation}
 In order to evaluate $ \mu ([0, T) \times \R^{d})$,  we 
 consider a regularization of the function $\zeta(t):=t^-=\max(-t,0)$, for instance
\begin{displaymath}
    \zeta_\eps(t): = \begin{cases}
 t+\frac\eps2 & \text{if } t \le -\eps,
 \\
 -\frac{t^2}{2\eps} & \text{if }  -\eps<t\le 0,
 \\
 0 & \text{if } t >0
 \end{cases}
\end{displaymath}
and we set $\zeta_{\eps,T}(t):=\zeta_\eps(t-T)$. 
We also take a function $\theta\in \rmC^\infty_{\rm c}(\R^d)$
fulfilling
\begin{displaymath}
    0\le \theta\le 1,\quad \theta\equiv 1\text{ in }B_1(0),\quad
    \theta\equiv 0\text{ in }\R^d\setminus B_2(0),\quad 
    \|\rmD\theta\|_\infty \le 2,
\end{displaymath}
and for  $\psi\in \rmC^1(\R^d)$ we set
\begin{displaymath}
	\theta_R(x):=\theta(x/R),\quad \psi_R(x):=\psi(x)\theta_R(x).
\end{displaymath}
 %
%
 Choosing $0<\eps<T$ and 
 $\varphi_c=\psi_R$, \eqref{eq:distributional-sense-3} yields
\begin{equation}
\label{test-1-epsR}
\begin{aligned}
 & \iint_{\DSTT {T-\eps}}
 \psi_R  \,\dd \mu
 + \iint_{(T-\eps,T){\times}\R^d} \frac{T-t}{\eps} \psi_R  \,\dd \mu 
& 
=  \bigg(T - \frac{\varepsilon}{2} \bigg) \int_{\R^d}\psi_{R}\, \dd \mu_0 - \iint_{\DST} \zeta_{\eps,T}\rmD \psi_R \,\dd \nnu\,.
\end{aligned}
\end{equation}
We first take $\psi\equiv1$, so that $\psi_R=\theta_R$. 
Since $\zeta_{\eps,T} \to \zeta_T=\zeta(\cdot-T) $ as $\eps \down 0$, uniformly in $\III$, 
 and 
 $\| \zeta_{\varepsilon,T}\|_{\infty} \leq T + \varepsilon/2$, 
 $\|\rmD \theta_R\|_\infty\leq \frac{2}{R}$,
 and $|\nnu|(\DSTT T)<+\infty$, we find that 
 \[
 \lim_{\eps\down 0}  \iint_{\DST } \zeta_{\eps,T}(t)\rmD \theta_R(x) \,\dd \nnu(t,x)= 
   \iint_{\DSTT T}   (T-t)_+ \rmD \theta_R(x) \,\dd \nnu(t,x)\,.
 \]
 Clearly,
 \[
 \lim_{\eps\down 0} \iint_{\DSTT {T-\eps}} \theta_R(x)  \,\dd \mu(t,x)   = \iint_{[0,T){\times}\R^d } \theta_R(x)  \,\dd \mu(t,x)  \,.
 \]
 Finally,
 \[
 \left| \iint_{(T-\eps,T){\times}\R^d} \frac{T-t}{\eps} \theta_R(x)  \,\dd \mu(t,x)   \right| \leq  |\mu|((T-\eps,T){\times}\R^d) \longrightarrow 0 \quad \text{as } \eps \down 0\,.
 \]
 Passing to the limit as $\eps\down0$ in \eqref{test-1-epsR}
 we get 
 \[
 \iint_{[0,T){\times}\R^d}  \theta_R(x)  \,\dd \mu(t,x)  
=  T \int_{\R^d} \theta_R(x) \, \dd \mu_0(x) - \iint_{\DSTT T}  (T-t)_+  \rmD \theta_R(x) \,\dd \nnu(t,x)\,.
\]
We now take the limit as $R\to \infty$ in both sides of the above equality, recalling that  $\theta_R \to 1$  and that $\|  \rmD \theta_R\|_\infty \to 0$. Hence, we conclude that 
 $ \mu ([0, T) \times \R^{d}) = T  \mu_0(\R^d)=T$,
  i.e.~$\pi^0_\sharp \mu([0,T))=T$
 and therefore $\pi^0_\sharp \mu([a,b))=b-a$ for every $0\le a<b$.
 This implies that $\pi^0_\sharp \mu$
 is the Lebesgue measure $\mathcal L^1$.
 \nc 
\par
By the disintegration theorem
(cf., e.g., \cite[Cor.\ A.5]{Valadier90}),
we can disintegrate the measure~$\mu$ with respect to the projection 
$\pi^{0}\colon\DST \to \III$, so that there exists a
Borel family~$\{\mu_{t}\}_{t\in\III}$ of probability measures on~$\R^{d}$ such that 
$\mu=  \mathcal L^1 \otimes \mu_{t}$.  
\par
We now use \eqref{test-1-epsR} by choosing
$\psi(x):=\sqrt{1{+} \| x \|^2}$; since
\begin{displaymath}
	|\rmD\psi_R(x)|\le |\rmD\theta_R(x)\psi(x)|+|\theta_R(x)\rmD\psi(x)|
	\le \frac 2R \sqrt{1+(2R)^2}+1\le 5\quad\text{if }R\ge 2,
\end{displaymath} we obtain the uniform estimate for $R\ge 2$
\begin{align*}
	\int_{\DSTT {T-\eps}}\psi_R\,\dd\mu\le 
	T\int_{\R^d}\psi_R\,\dd\mu_0+
	5T|\nnu|(\DSTT T).
\end{align*}
Passing to the limit as $R\uparrow+\infty$ we deduce that 
\begin{equation}
	\int_0^T \int_{\R^d}\sqrt{1{+}  \| x \|^2}\,\dd\mu_t \,\dd t\le 
	CT,\quad
	C:=\int_{\R^d}\sqrt{1{+}\| x \|^2}\,\dd\mu_0+  5  |\nnu|(\DSTT T),
\end{equation}
so that $\mu_t\in \mathcal P_1(\R^d)$ for $\mathcal L^1$-a.a.~$t>0$.

Eventually, we write \eqref{eq:distributional-sense-3} 
for an arbitrary $\zeta\in \rmC^1_c(\III)$ and $\varphi_c=\varphi\theta_R$ 
(with $\varphi\in \rmC^1(\DST)$ Lipschitz) and  
we get
\begin{equation}
\label{eq:auxiliary}
	\iint_{\DST}(\zeta'\varphi+\zeta\partial_t\varphi)\theta_R\,\dd\mu+
	\iint_{\DST}\zeta\rmD\varphi \theta_R\dd\nnu  =
	-\int_{\R^d}\zeta(0)\varphi(0,\cdot)\theta_R\,\dd\mu_0-
	E_R,
\end{equation}
where
\begin{equation*}
	E_R=\iint_{\DST}\zeta \varphi\rmD \theta_R  \, \dd\nnu\,. 
\end{equation*}
Choosing constants $a,L,T$ such that 
$|\zeta(t)|\le a$, $\mathrm{supp}(\zeta)\subset [0,T]$, $|\varphi(t,x)|\le 
L(1{+} \| x \|)$ whenever $0\le t\le T$, we obtain
\begin{align*}
	|E_R|\le 2aL\frac{(1+2R)}R\,|\nnu|\Big([0,T]\times (B_{2R}(0)\setminus B_R(0))\Big)
     \leq 6 aL |\nnu|\Big([0,T]\times (B_{2R}(0)\setminus B_R(0))\Big)\,.\EEE
\end{align*}
 Now, since $\nnu $ 
is a finite measure on $\DSTT T $,  it is immediate to check that $|\nnu|\Big([0,T]\times (B_{2R}(0)\setminus B_R(0))\Big) \to 0$ as $R\to\infty$,
and thus 
 $\lim_{R\to\infty}|E_R|=0$.
Passing to the limit in \eqref{eq:auxiliary} as $R\uparrow+\infty$
using the fact that $\varphi$ has linear growth and $\partial_t \varphi$ is bounded, 
we get
\begin{equation}
	\label{eq:distributional-sense4}
	\int_{\III} \bigg[\zeta'
	\Big(\int_{\R^d}\varphi_t\,\dd\mu_t\Big)
	+\zeta
	\Big(\int_{\R^d}\partial_t\varphi_t\,\dd\mu_t\Big)\bigg]\,\dd t+
	\iint_{\DST}\zeta\rmD\psi\,\dd \nnu  =
	-\int_{\R^d}\zeta(0)\varphi_0\,\dd\mu_0,
\end{equation}
which in particular yields \eqref{eq:distributional-derivative}.
\end{proof}
We now show the existence of the left- and right-continuous representatives.
 The following result extends \cite[Lemma 8.1.2]{AGS08}
and concludes \textbf{the proof of Part $2.$ of Thm.~\ref{th:1}}.
\begin{lemma}[Left- and right- continuous representatives]
\label{l.2}
Let $\mu_{0}\in \Puno(\R^{d})$ and let $(\mu,\nnu)$ be a $\Puno$-solution  to  the continuity equation in the sense of Definition~\ref{def:solCE}.  
Then,
there exists a narrowly left- (resp.\ narrowly right-)continuous representative 
 $\III\ni t \mapsto {\mu}_{t}^{-} \in  \scrP(\R^d)$   (resp.~$\III\ni t \mapsto {\mu}_{t}^{+} \in \scrP(\R^d)$)  of the curve~$t\mapsto \mu_{t}$,  such that, setting ${\mu}^-_{0}:=\mu_{0}$, 
 for every $0\leq a\leq b< +\infty$
   the following estimates hold
\begin{align}
&
\label{towards-totVar-mu-}
W_1(\mu^-_{a},\mu_{b}^{-}) \leq |\nnu|([a,b)\times\R^{d}),
\\
&
\label{towards-totVar-mu+}
W_1(\mu^+_{a},\mu_{b}^{+}) \leq  |\nnu|((a,b]\times\R^{d}),
\end{align}
as well as  estimates \eqref{estimates-w-Var} and relations \eqref{new-e.5.1} and \eqref{e.5.1-2}.
\end{lemma}
\begin{proof}
 We combine the argument of \cite[Lemma~8.1.2]{AGS08} with the
duality characterization of the Kantorovich-Rubinstein-Wasserstein metric.  

We set $\nu:=\pi^0_\sharp(|\nnu|)$ and $D_\nu:=\{t\in \III:\nu(\{t\})=0\}$,
whose complement is at most countable; 
we also select a Borel set $D_\mu$ such that 
$\mathcal L^1(\III\setminus D_\mu)=0$ and $\int_{\R^d} |x|\,\dd\mu_t(x)<+\infty$ 
for every $t\in D_\mu$.
  Finally, recall the duality representation formula for $W_1$ in terms of  $1$-Lipschitz functions on $\R^d$
\[
W_1(\mu',\mu'')=\sup\Big\{ \int_{\R^d} f \,\dd(\mu'-\mu''):
f \in \mathrm{Lip}_1(\R^d) \Big\},
\]
and let $\widehat{Z}\subset \mathrm{C}_{\mathrm{c}}^1(\R^d)$ 
be a countable set, dense in $ \mathrm{C}_{\mathrm{c}}^1(\R^d)$ w.r.t.\ the $\rmC^1$-norm
$\|f\|_{\rmC^1}: = \max \{ \|f\|_\infty, \|\nabla f\|_\infty\}$. Then, the countable set
$Z:= \widehat{Z}{\cap} \mathrm{Lip}_1(\R^d)$ is dense in $\mathrm{Lip}_1(\R^d)$ and provides 
\EEE
the representation 
\begin{equation}
\label{eq:countable-W-duality}
	W_1(\mu',\mu'')=\sup\Big\{ \int_{\R^d}\psi\,\dd(\mu'-\mu''):
	\psi\in Z\Big\}.
\end{equation} 
%
\nc 
For such $\psi$,  let us still denote by~$\mu_{t}(\psi)$ a good representative~\cite[Theorem~3.28]{AmFuPa05FBVF} of the map $t \mapsto \int_{\R^{d}} \psi \, \di \mu_{t}$, and  let us denote with~$D_{\psi}$ the set of continuity points of
the function
$t\mapsto \mu_{t}(\psi)$. 
We eventually set 
 $D:=D_\nu\cap D_\mu\cap \bigcap_{\psi\in Z} D_{\psi}$. Then,  $\mathcal{L}^1 (\III\setminus D)=0$  and  $t\mapsto \mu_{t}(\psi)$ is continuous at  $t\in D$ for every $\psi\in Z$. 

 For every $r<s\in D$ and every $\psi\in Z$ we have
\begin{equation}\label{e.2}
|\mu_{s} ( \psi ) - \mu_{r}( \psi ) | \leq \iint_{[r,s] \times \R^{d}} | \rmD \psi  | \, \di | \nnu | (t,x) \leq |\nnu| ([r,s] \times \R^{d})=
 \nu((r,s)),
\end{equation}
so that \eqref{eq:countable-W-duality} yields
\begin{equation}
	\label{eq:BV-W}
	W_1(\mu_r,\mu_s)\le \nu((r,s])=\nu((r,s))\quad\text{for every $r,s\in D$, 
	$r<s$}.
\end{equation}
Thus,  the map $ D  \ni r\mapsto \mu_r$ has pointwise bounded variation in every 
bounded subset of $D$. 
Since $(\mathcal P_1(\R^d),W_1)$ is complete,
by a standard density argument we deduce that the limits
\begin{equation}
\label{e.3}
\mu_{t}^{-}:=\lim_{s\in D, s\uparrow t}\mu_s,\quad
\mu_{t}^{+}:=\lim_{s\in D, s\downarrow t}\mu_t
\end{equation}
exist in $(\mathcal P_1(\R^d),W_1)$ for every $t\in \III$ and 
define a left-continuous and a right continuous map respectively satisfying
\eqref{estimates-w-Var}.

Then, relations \eqref{new-e.5.1} and \eqref{e.5.1-2}  immediately   follow from 
\eqref{eq:distributional-derivative} by observing that for every Lipschitz function
$\varphi\in \rmC^1(\DST)$ the map
$t\mapsto \mu^-(\varphi_t)$ (resp.~$t\mapsto \mu^+(\varphi_t)$) is
left- (resp.~right-)continuous and thus provides
the unique left- (resp.~right-) continuous representative of
$t\mapsto \mu_t(\varphi_t)$.
\end{proof}

\par
\subsection{Proof of Part (1) of Thm.\ \ref{th:1}}\

\noindent 
 Let $\mu \in \BVloc(\III ; \mathcal{P}_1(\R^d)) $ and let~$\mu^{\pm}$ be the left- and right-continuous representatives of~$\mu$, respectively.
 We define
 \begin{equation}
 	\label{eq:variation-function}
 	V_\mu(t):=\mathrm{Var}_{W_1}(\mu^-;[0,t])
 	=\mathrm{Var}_{W_1}(\mu^-;[0,t))\,.
 \end{equation} 
 We prove the following  claim:  {\sl 
there exists a Borel measure
$\nnu  \in \calM (\DST;\R^d)$ such that for every $T\in [0, +\infty)$
relations \eqref{UEST} hold,
and the pair $(\mu,\nnu)$  satisfies  the continuity equation \eqref{continuity-equation} in the
sense of  Definition~\ref{def:solCE}. 
 Moreover, writing $\nnu = \nnu^{a} \mu + \nnu^{\perp}$ with $\nnu^{a}\mu \ll \mu$
 and $\nnu^{\perp} \perp \mu$, we have that~$\nnu^{\perp}$  is minimal in the
 sense of Definition 
\ref{d:minimal}.}

Indeed,  let us consider two continuity points $a<b$ for 
$V_\mu$ (and thus for $\mu$) and let us  define the linear functional 
 $\ell_{a,b} \colon \rmC^1_c([a,b]\times \R^d)\to \R$ by 
\begin{equation}
\label{first-functional}
 \ell_{a,b}(\zeta):=   -
\int_a^{b} 
\int_{\R^d} \partial_t \zeta \, \di \mu_t ( x) \,  \dd t
  +
\int_{\R^{d}} \zeta(b, x) \, \di \mu_{b} (x) - \int_{\R^{d}} \zeta(a, x) \, \di \mu_{a} (x)
 \,.
\end{equation}
Observe that, by continuity of ${\mu}$ at $a,b$, and continuity of $\zeta$, we have 
\[
\begin{aligned}
\ell_{a,b}(\zeta)
 & = -  \lim_{h \downarrow 0 } 
 \frac 1h \int_a^{b-h} \int_{\R^d}\Big(\zeta(t+h,x) - \zeta(t,x)\Big) \, \di  \mu_t ( x) \,  \dd t \\&\qquad\qquad
 - \lim_{h \downarrow 0 } 
 \frac 1h \int_a^{a+h}
	\int_{\R^{d}} \zeta(t, x) \, \di \mu_{t} (x)\,\dd t 
	  +
\lim_{h \downarrow 0 } 
 \frac 1h  \int_{b-h}^{b}
	\int_{\R^{d}} \zeta(t, x) \, \di \mu_{t} (x)\,\dd t
 \\ 
 &   =    -  \lim_{h \downarrow 0 }  \frac1h \int_{a+h}^b \int_{\R^d} \zeta(s,x) \, \di (\mu_{s-h} -\mu_s) ( x) \,  \dd s
 =\lim_{h \downarrow 0 }  \frac1h \int_{a+h}^b \int_{\R^d} \zeta(s,x) \, \di (\mu^-_{s} -\mu^-_{s{-}h}) ( x)
  \,  \dd s\,.
 \end{aligned}
\]
The duality formula for the Wasserstein metric yields
\[
\begin{aligned}
\left| \frac1h \int_{\R^d} \zeta_s \,  \di (\mu^-_{s} -\mu^-_{s{-}h})
 \right|  
 \leq \frac1{h}
 W_{1} (\mu^-_{s-h} ,\mu^-_s) 
  \sup_{\R^d}\|\rmD \zeta_s\|\le 
 \frac 1h \Big(V_\mu(s)-V_\mu(s-h)\Big)\sup_{\R^d}\|\rmD \zeta_s\|,
 \end{aligned}
 \]
 so that  
 \begin{displaymath}
 	\left|\frac1h \int_{a+h}^b\int_{\R^d} \zeta_s \,   \di  (\mu^-_{s} -\mu^-_{s{-}h})
	 \,\dd s 
 \right| \le \frac 1h \|\rmD \zeta\|_\infty
 \Big(\int_{b-h}^b V_\mu\,\dd s- \int_a^{a+h}V_\mu\,\dd s\Big)\,.
 \end{displaymath}
 Therefore,  taking the limit as $h\down 0$ we obtain  
 \[
 \left| \ell_{a,b}(\zeta)
\right| 
\leq \| \rmD \zeta \|_\infty \Big(V_\mu(b)-V_\mu(a)\Big)\,.
 \]
\nc
 Therefore, the linear functional $L_{a,b}$
  defined on the space $ \mathbf{V}_{a,b} : = \{ \rmD \zeta \, : \  \zeta \in \mathrm{C}_{\mathrm{c}}^1([a,b]  \times \R^d)\}$
   by 
  $L_{a,b}(\boldsymbol \xi):=\ell_{a,b}(\zeta)$ whenever $\boldsymbol \xi=\rmD \zeta$
 is well defined and 
 it satisfies
 \[
 \| L_{a,b} \|=  \sup_{\boldsymbol \xi\in \mathbf{V}_{a,b} , \| \boldsymbol\xi \|_{\infty} \leq 1} L_{a,b}(\boldsymbol\xi)
\leq V_\mu(b)-V_\mu(a)
\,.
 \]
 By the Hahn-Banach and Riesz representation theorems, 
 we can find a vector measure $\nnu_{a,b}$ on 
 $[a,b] \times \R^d $, which  satisfies
 \begin{gather}
 \label{ad-eqn}
\iint_{[a,b]\times \R^d} \rmD \zeta \, \di {\nnu}_{a,b} = 
 \langle {\nnu}_{a,b} , \rmD \zeta \rangle  =\ell_{a,b}(\zeta), \quad 
\\
\label{ad-est-norm}
 |\nnu_{a,b}|([a,b]\times \R^d)    =    \| L_{a,b} \|  \leq 
 V_\mu(b)-V_\mu(a)\,.
 \end{gather}
 Now, 
from \eqref{ad-eqn}
 we deduce that the pair $(\mu, \nnu_{a,b})$ satisfies the continuity equation on $[a,b]\times \R^d$ in the sense of 
 Remark \ref{rmk:0,T}.  
Estimate \eqref{estimates-w-Var} then yields
\begin{equation}
\label{localizzata-2}
 V_\mu(\beta)-V_\mu(\alpha)
 \le |\nnu_{a,b}|([\alpha,\beta)\times \R^d)\quad\text{for every }a\le \alpha<\beta\le b,
\end{equation}
so that we derive
\begin{equation}
	\label{eq:boundary-estimate}
	|\nnu_{a,b}|(\{a\}\times\R^d)=
	|\nnu_{a,b}|(\{b\}\times\R^d)=0,\quad
	|\nnu_{a,b}|([\alpha,\beta)\times \R^d)=V_\mu(\beta)-V_\mu(\alpha)\,.
\end{equation}
Possibly extending $\mu$ to $(-\infty,0)$ by setting $\mu_t:=\mu_0$
and selecting a diverging sequence $(a_n)_{n\in \N}$ of 
continuity points for $V_\mu$ with $a_0\le 0$,
we can now apply the above results to a sequence of intervals
$[a_n,a_{n+1}]$, $n\in \N$, 
and we   consider a  vector measure $\nnu$ whose
restriction to $[a_n,a_{n+1}]\times \R^d$ coincides with $\nnu_{a_n,a_{n+1}}$.
By \eqref{eq:boundary-estimate} such a gluing process is well defined
and it is easy to check that $(\mu,\nnu)$ satisfies 
the continuity equation and \eqref{UEST}.
   
 \nc Finally, we decompose $\nnu =  \nnu^{a}   + \nnu^{\perp}$ into its absolutely continuous part  $\nnu^{a}$  and singular part~$\nnu^{\perp}$ w.r.t.~$\mu$,  and show that $\nnu^{\perp}$ is minimal. Let $\rrho \in \calM(\DST ; \R^{d})$ fulfill $\dive \rrho = \dive \nnu^{\perp}$ and $\rrho \prec \nnu^{\perp}$.  Setting $\ttheta =  \nnu^{a}   + \rrho$, the pair $(\mu, \ttheta)$ satisfies the continuity equation
\begin{displaymath}
\partial_{t} \mu + \dive \ttheta =0 \qquad \text{in } (0,+\infty) \times \R^d
\end{displaymath}
with initial datum $\mu_{0}$, in the sense of Definition~\ref{def:solCE}. In particular, by~\eqref{UEST} and~\eqref{estimates-w-Var} we have that for $T \in [0, +\infty)$
\begin{equation}
\label{e:min-nu-perp}
\begin{split}
| \nnu^{a}| ([0, T)\times \R^{d}) + | \nnu^{\perp} | ([0, T) \times \R^{d}) & \leq \mathrm{Var}_{W_{1}} (\mu, [0, T)) 
\\
&
\leq | \ttheta| ([0, T) \times \R^{d}) = | \nnu^{a} | ([0, T) \times \R^{d}) + | \rrho| ([0, T) \times \R^{d})\,. 
\end{split}
\end{equation}
 Thus, $| \nnu^{\perp} | ([0, T) \times \R^{d}) \leq | \rrho| ([0, T) \times \R^{d})$. Combining this with the fact that $\rrho \prec \nnu^{\perp}$ and recalling  property  \eqref{condition-4-equality}, we conclude that  $\nnu^{\perp} = \rrho$.  Thus, $\nnu^{\perp}$ is minimal. 
\QED

\section{The augmented continuity equation}
\label{s:4}
\noindent This section revolves around the result at the core of our approach to the superposition principle for the 
 continuity equation 
\eqref{continuity-equation}.
The main idea is to lift a pair $(\mu,\nnu)$ solving \eqref{continuity-equation},
to a solution of an augmented continuity equation in $\DSTpu=\III\times \R^{d+1}$
exhibiting distinguished properties.

Throughout this section we will 
denote by 
\begin{equation}
\label{ppi-projection}
(t,x) \text{ any element in } \R^{d+1} = \R \times \R^d,\quad
\ppi:\R\times \R^{d+1}\to \R^{d+1},\ \ppi(s;t,x):=(t,x).
\end{equation}
and indicate by the symbol $\norm{\cdot}$ a
 norm in $\R^{d+1}$,
 whose restriction on $\{0\}\times \R^d$
induces a norm on $\R^d$ which will be denoted by the same symbol.
As previously observed,  a choice of the norm in
$\R^{d+1}$ affects the $W_1$-distance on $\Puno(\Rdpu)$, cf.\ \eqref{W1}. 

\nc  Theorem \ref{thm:augmented} ahead associates with 
a  solution $(\mu,\nnu)$  to   the continuity
 equation (in the sense of Definition \ref{def:solCE}), a curve of measures  $(\sigma_s)_{s\in \III} \subset  \Puno( \Rdpu )$
  and a vector measure $(\sigma^0,\ssigma)\in \calM(\III\times \Rdpu;\R^{d+1})$ \nc  
that  turn  out to solve 
 the `augmented' continuity equation in $\DSTpu=\III\times  \R^{d+1}$, 
  \begin{equation}
\label{continuity-sigma-lim}
\left\{\begin{aligned}
\partial_s \sigma +\partial_t \sigma^{0} + \dive\ssigma&=0 &&\text{in } 
\DSTpu, 
\\
\sigma,\sigma^0&\ge 0&&\text{in } 
\DSTpu,\\
\sigma_0 &= \delta_0 \otimes \mu_0&&\text{in }\Rdpu.
\end{aligned}\right.
\end{equation} 
In this connection, we mention that, hereafter, 
 with slight abuse of notation we shall denote  by the same symbol \emph{both}  the curve  $\sigma \colon \III  \to   \Puno( \Rdpu )$  \emph{and} the 
 measure $\sigma = \mathcal{L}^1 \otimes \sigma_s \in \calM^{+}(\III \times \Rdpu )$. 
  In equation \eqref{continuity-sigma-lim},  the  `augmented' operator
 $(\partial_t,\dive)$ plays the role that  the `spatial divergence' had for the continuity equation \eqref{continuity-equation}. 
  Thanks to the construction carried out in the proof of Theorem  \ref{thm:augmented} ahead, 
 the 
 vector measure $(\sigma^0,\ssigma)$ will be 
  induced by an autonomous (i.e.~independent of $s$) velocity field given by a pair 
  of bounded Borel maps
 $(\tau,\bvv):\Rdpu\to\Rdpu$  that  satisfy
 \begin{equation}
 	\label{eq:densities}
 	\sigma^0=(\tau\circ\ppi) \sigma,\quad \ssigma=(\bvv\circ\ppi) \sigma,\quad \tau\ge0.
 \end{equation}
 \par
We formalize the above properties in the next Definition
\ref{def:augmented}, after recalling 
 an equivalence relation between positive measures.  
\begin{definition}[Uniformly equivalent measures]
	\label{def:strongly-equivalence}
	We say that two measures $\varrho,\vartheta\in \calM^+(\R^h)$ are 
$k$-uniformly equivalent (for some $k\ge 1)$, and we write $\varrho\sim_k\vartheta$, if 
\begin{equation}
	k^{-1}\varrho\le \vartheta\le k\varrho.	
\end{equation}
We write $\varrho\sim \vartheta$ if there exists $k\ge1$ 
such that $\varrho\sim_k\vartheta$.
\end{definition}
Clearly, two uniformly equivalent measures are mutually absolutely continuous (and thus equivalent, sharing the same collection of null sets). Moreover, their mutual Lebesgue densities are bounded and uniformly bounded away from $0$.

 With this notion at hand, we can introduce `qualified' solutions  $(\sigma,\sigma^0,\ssigma)$  of the augmented continuity equation.
In particular, the second and third  properties  below establish a relation between $\sigma$
 and the pair $(\sigma^0,\ssigma)$.    
\begin{definition}[Solutions of the augmented continuity equation]
\label{def:augmented}
	Let $\mu_0\in \Puno(\R^d)$
	and let 
	$\sigma=\mathcal{L}^1 \otimes \sigma_s \in \calM^{+}(\DSTpu)$,
	$\sigma^{0} \in \calM^{+} ( \DSTpu)$, and  $\ssigma \in \calM (\DSTpu  ; \R^{d})$ be such that 
	$(\sigma,\sigma^0,\ssigma)$ is a $\Puno$-solution
	to the augmented continuity equation 
	\eqref{continuity-sigma-lim} 
	according
	to Definition \ref{continuity-equation}.
	We say that 
	\begin{enumerate}
		\item $(\sigma,\sigma^0,\ssigma)$ \emph{has locally finite $\ppi$-marginals} if 
		\begin{equation}
	    \label{bounded-sigma}
		|(\sigma,\sigma^0,\ssigma)|\Big(\III\times\DSTT T\Big )<+\infty \qquad \text{for every }  T>0,
		\end{equation}	
		so that, in particular, $\ppi_\sharp(\sigma,\sigma^0,\ssigma)$ is a Radon vector measure in $\DST$.	
		\item 
		$(\sigma,\sigma^0,\ssigma)$ \emph{is $k$-adapted}
		if 
		$\sigma\sim_k |(\sigma^0,\ssigma)|$ and
		\emph{normalized} if $\sigma=|(\sigma^0,\ssigma)|$ (i.e.~$k=1$).
		\item 
		$(\sigma,\sigma^0,\ssigma)$ is \emph{$\ppi$-autonomous}
		if there exists a pair of \emph{Borel maps} 
		 $(\tau,\bvv):\DST\to \DST$ such that the autonomous density condition 
		\eqref{eq:densities} holds.
\end{enumerate}
		An \emph{adapted} (resp.~\emph{normalized})
		solution $(\sigma,\sigma^0,\ssigma)$ 
		which has \emph{locally finite $\ppi$-marginals}
		and is \emph{$\ppi$-autonomous}
		will be called \emph{$\ppi$-adapted}
		(resp.~$\ppi$-normalized).
\end{definition}
\begin{lemma}[Elementary properties of augmented solutions]
	\label{rem:properties}
	Let $(\sigma,\sigma^0,\ssigma)$ be a $\Puno$-solution
	 to  the augmented continuity equation 
	\eqref{continuity-sigma-lim} with $\mu_0\in \Puno(\R^d)$
	and locally finite time marginals.
	\begin{enumerate}
		\item If $(\sigma,\sigma^0,\ssigma)$ is $k$-adapted,
		then 
		the disintegration $(\sigma_s)_{s\ge 0}$ 
	of $\sigma$ w.r.t.~the Lebesgue measure in $\III$ 
	admits a representation which belongs to the space $\Lip_k(\III;\Puno(\Rdpu))$ 
	(it is in fact Lipschitz with values in $\mathcal P_p(\Rdpu)$
	if $\mu_0\in \mathcal P_p(\R^d)$, $p\ge 1$).
	\item 
	If $(\sigma,\sigma^0,\ssigma)$ is $\ppi$-autonomous, then
	\begin{gather}
		\label{eq:commutation}
			|(\sigma^0,\ssigma)|=\|(\tau,\bvv)\|\sigma\ll \sigma,\quad
			\ppi_\sharp(\sigma^0,\ssigma)=(\tau,\bvv)\ppi_\sharp\sigma,
			\\
					\label{eq:commutation2}
						\ppi_\sharp |(\sigma^0,\ssigma)|
	=
			\ppi_\sharp\big(			\|(\tau,\bvv)\|\sigma\big)=
			\|(\tau,\bvv)\|\ppi_\sharp\sigma
=	|\ppi_\sharp(\sigma^0,\ssigma)|.
	\end{gather}
		\item 	
		If $(\sigma,\sigma^0,\ssigma)$ is $\ppi$-normalized  then  
	$\|(\tau,\bvv)\|\equiv1$ $\sigma$-a.e.\ and 
			\begin{equation}
	\label{eq:equivalent-densities}		
			\sigma=|(\sigma^0,\ssigma)|,\quad 
			\ppi_\sharp \sigma=|(\ppi_\sharp\sigma^0,\ppi_\sharp\ssigma)|\,.
	\end{equation}
Moreover,  the map $s\mapsto \sigma_s$ belongs to 
	$\Lip_1(\III;\Puno(\R^{d+1}))$.
	\item 
	Conversely, if 
	$\|{\cdot}\|$ is a strictly convex norm of $\R^{d+1}$
	and 
	$(\sigma,\sigma^0,\ssigma)$ 
	satisfies \eqref{eq:equivalent-densities}
	then it is a $\ppi$-normalized solution.
	\end{enumerate}
\end{lemma}
\begin{proof}
	Claim (1) immediately follows from Theorem 
	\ref{th:1} since the Lebesgue density of $(\sigma^0,\ssigma)$ w.r.t.~$\sigma$ is uniformly bounded.
	
	 Claim  (2) is an immediate consequence of 
	the autonomous property. Claim (3) follows from 
	Claim (2) and the 
	normalization condition.

	Concerning the last Claim (4),
	the normalized property is obious since $\sigma=|(\sigma^0,\ssigma)|$.
	The autonomous condition is a consequence of Lemma \ref{le:strict-conv} ahead
	and the strict convexity of the norm.
\end{proof}
Let us now establish a first easy link between  the solutions to the continuity equation   \eqref{continuity-equation}
 and those  to its \emph{augmented} counterpart 
\eqref{continuity-sigma-lim}. 
\begin{lemma}[Marginals of sutonomous solutions]
\label{l:CE-onu}
Let  $(\sigma,\sigma^0,\ssigma)$
be an adapted solution  to  the augmented continuity equation
\eqref{continuity-sigma-lim}
with locally finite $\ppi$-marginals 
according to Definition \ref{def:augmented},
and an initial datum $\mu_0 \in \mathcal{P}_{1}(\R^{d})$.
Then, setting  
\[
\mu :=    \ppi_{\sharp } \sigma^0, \qquad \nnu := \ppi_{\sharp }\ssigma\,, 
\]
 the pair $(\mu,\nnu)$  is a $\mathcal{P}_{1}$-solution to 
 the continuity equation \eqref{continuity-equation}, 
 with initial datum $\mu_0$, in the sense of Definition~\ref{def:solCE}.
\end{lemma} 
\begin{proof}
Recall that $\mu,\nnu$ are well defined Radon measures
thanks to \eqref{bounded-sigma}.
As in Theorem~\ref{th:1}(2), we have that $\sigma = \mathcal{L}^{1} \otimes \sigma_{s}$, where $\sigma_{s} \in \mathcal{P}(\DST )$ for $s \in  \III $ thanks to
the next Theorem
\ref{p:superposition}. 

Let us fix $\zeta \in \rmC_{\rm c}^{1} ([0, 2))$ such that $0 \leq \zeta \leq 1$ and $\zeta \equiv 1$ in $[0, 1]$, and let us define the sequence $\zeta_{n} \in \rmC^{1}_{\rm c} (\III)$ by $\zeta_{n}(s) \coloneq \zeta(s/n)$, so that $\zeta_{n}(s) \to 1$
for every $s\in \III$
 as $n \to \infty$, $0 \leq \zeta_{n} \leq 1$, and $\|\zeta'_{n}\|_{\infty} \leq \frac{ \| \zeta'\|_{\infty}}{n}$. For every $n$ and every 
 $\varphi \in \rmC^{1}_{\rm c}(\R {\times} \R^{d})$  it holds  $\zeta_{n}\varphi \in \rmC^{1}_{\rm c}(\III {\times} \R {\times} \R^{d})$. 
 Since the triple $(\sigma, \sigma^{0}, \ssigma)$ solves the Cauchy problem~\eqref{continuity-sigma-lim} and  $\spt(\sigma^0), \, \spt(\ssigma) \subset \III{\times} [0, +\infty) {\times}\R^{d}$,
  we have that
\begin{equation}
\label{e:Cauchy-n}
\begin{split}
& \int_{\III} \zeta'_{n}(s) \iint_{\DST } \varphi (t, x) \, \di \sigma_{s}(t, x) \, \di s +  \iiint_{ \III \times \DST } \zeta_{n}(s) \partial_{t} \varphi(t, x) \, \di \sigma^{0}(s, t, x)
\\
&
\qquad  +  \iint_{\III \times \DST }  \zeta_{n}(s) \rmD \varphi(t, x) \, \di \ssigma (s, t, x) = - \int_{\R^{d}} \varphi(0, x) \, \di \mu_{0}(x) \,.
\end{split}
\end{equation}
Since $\varphi$ has compact support, there exists $T_{\varphi} <+\infty$ such that   $\spt( \varphi) \subseteq [- T_{\varphi}, T_{\varphi}] \times \R^{d}$.   Hence,
the first integral in~\eqref{e:Cauchy-n} can be estimated by
\begin{align*}
\bigg|  \int_{\III } \zeta'_{n}(s) \iint_{\DST } \varphi (t, x) \, \di \sigma_{s}(t, x) \, \di s \bigg| &  \leq \frac{\| \varphi\|_{\infty} \| \zeta' \|_{\infty}}{n} \,  \sigma ([n, +\infty) {\times} [0, T_{\varphi}] {\times} \R^{d}),
\end{align*}
and therefore it  tends to $0$ as $n \to \infty$. By dominated convergence, we can take the limit in the second and third integrals of~\eqref{e:Cauchy-n},
thus
 obtaining
\begin{displaymath}
\begin{aligned}
  - \int_{\R^{d}} \varphi(0, x) \, \di \mu_{0}(x)= &\iiint_{\III  \times \DST }  \partial_{t} \varphi(t, x) \, \di \sigma^{0}(s, t, x)
+  \iiint_{\III \times \DST }  \rmD \varphi(t, x) \, \di \ssigma (s, t, x) 
\\=& 
\iint_{ \DST } \partial_{t} \varphi(t, x) \, \di \mu( t, x) +  \iint_{\DST }  \rmD \varphi(t, x) \, \di \nnu( t, x) 
 \end{aligned}
\end{displaymath}
\nc for every $\varphi \in \rmC_{\rm c}^{1}(\R{\times} \R^{d})$,  which  concludes the proof.
\end{proof}

\subsection{Superposition results for the augmented continuity equation}
\label{ss:5.1}
Since $\ppi$-adapted solutions of the augmented continuity equations
are driven by a \emph{bounded} Borel velocity field, it is 
easy to state a \emph{superposition principle} in the spirit of \cite[Theorem~8.2.1]{AGS08}.
\par
In order to formulate the probabilistic representation of the curve $\sigma$, we  introduce the \emph{evaluation} and the \emph{augmented evaluation} maps
\begin{equation}
\label{eval-map}
\begin{aligned}
\fre&\colon\III\times  \rmC(\III;\R^{d+1}) \to \R^{d+1}, &
\fre (s,\sfyy)&:= \sfyy(s) = (\sft(s),\sfxx(s)),\\
\fre_s&\colon\rmC(\III;\R^{d+1})\to \R^{d+1},&
\fre_s(\sfyy)&:=\sfyy(s)=(\sft(s),\sfxx(s)),\\
\fra&\colon\III\times  \rmC(\III;\R^{d+1}) \to \III  \times\R^{d+1}, &
\fra (s,\sfyy)&:= (s,\sfyy(s)) = (s,\sft(s),\sfxx(s)).
\end{aligned}
\end{equation}
 Clearly, $\fra$, $\fre$ and $\fre_s$, $s\ge0$, are  continuous maps.
 
 We  also introduce the Borel maps 
 $\fry' \colon \III\times \Lip(\III;\R^{d+1})\to \Rdpu$ defined by
 \begin{equation}
 \label{eq:everywhere-derivative}
 	\fry'(s,\sfyy):=\sfyy'(s),\quad
 	\text{where}
 	\quad
 	(\sfyy')_i(s):=\limsup_{h\to 0}\frac{\sfyy_i(s+h)-\sfyy_i(s)}h,\ 
 	i=0,\cdots,d.
 \end{equation}
 Clearly, $\fry'(s,\sfyy)$ coincides with the usual
 pointwise derivative of $\sfyy(s)$ for $\mathcal L^1$-a.a.~$s\in \III$.
 We will also denote by $\frt'$ (resp~$\frxx'$) the first (corresponding to the index $i=0$)
 component (resp.~the vector of the last $d$ components) of 
 $\fry'$, $\fry'=(\frt',\frxx')$:
\begin{equation}
\label{evaluation-derivatives}
\begin{aligned}
&
\frt' \colon \III\times   \Lip(\III;\R^{d+1})    \to \R_,&\frt'(s,\sfyy): = {}& \sft'(s),
\\
& \frxx' \colon \III \times   \Lip(\III;\R^{d+1})   \to \R^d, &\frxx'(s,\sfyy): 
={}& \sfxx'(s)\,.
\end{aligned}
\end{equation}
Notice that the restriction of $\fry'$ to 
$\Lipplus{k}{\III}{\Rdpu}$ is a bounded Borel vector field whose image
is contained in $\DST$. 



\par
 Now, for every $\sfyy=(\sft,\sfxx)\in \Cplus$ 
 the set 
 $$\{s\ge0:\sft(s)\in [0,T]\}=\sfyy^{-1}\big([0,T]\times \R^d\big)
 \quad\text{is a compact interval}.$$
 For every time interval $[0,T]$ we consider the domain
 \begin{equation}
 	\label{eq:domainT}
 	\Domain T:=\Big\{(s,\sfyy)\in 
	\III\times 	\Cplus:\sft(s)\in [0,T]\Big\}
	=\fre^{-1}(\III\times [0,T]\times \R^d).
 \end{equation}
 If $\eeta$ is a probability measure on $\Lipplus{}\III{\Rdpu}$ 
 we set
 \begin{equation}
 	\label{eq:tensor}
 	 \eeta_{\mathcal{L}}  :=\mathcal L^1\otimes \eeta 
 	\in \calM^+(\III\times \Lipplus{}\III\Rdpu)
 \end{equation}
 and we observe that 
 \begin{equation}
 	\int_{\Domain T}\frt'(s,\sfyy)\,\dd\teeta(s,\sfyy)=T,
 \end{equation}
 so that 
 \begin{equation}
 	\int_{\Domain T}\|\frxx'(s,\sfyy)\|\,\dd\teeta(s,\sfyy)\le 
 	\int_{\Domain T}\|\fry'(s,\sfyy)\|\,\dd\teeta(s,\sfyy)
 	\le T+\int_{\Domain T}\|\frxx'(s,\sfyy)\|\,\dd\teeta(s,\sfyy).
 \end{equation}
 
%
%
%
\par
 We are now in a position to state our result on the probabilistic representation of the solutions to the augmented continuity equation. 
 \begin{theorem}[Superposition principle for solutions  to  the augmented continuity equation]
 \label{p:superposition}
\ \begin{enumerate}
\item 
 Let $k>0$ and $\eeta$ be a probability measure in $\Lipplus k\III{\Rdpu}$ 
 such that 
 \begin{equation}
 	\label{eq:condition}
 	\int \|\fre_0\|\,\dd\eeta<+\infty,\qquad
 	\int_{\Domain{T}}\|\frxx'\|\,\dd\teeta<\infty\quad\text{for every }T>0.
 \end{equation}
 Setting
 \begin{equation}
 	\label{eq:from-eta-to-sigma}
 	\sigma_s:=(\fre_s)_\sharp\eeta,\quad
 	\sigma:=\fra_\sharp\teeta,\quad
 	\sigma_0:=\fra_\sharp\big(\frt'\,\teeta\big),\quad
 	\ssigma:=\fra_\sharp\big(\frxx'\,\teeta\big)
 \end{equation}
 then the curve $s\mapsto \sigma_s$ is $k$-Lipschitz with values in 
 $\Puno(\DST)$, 
 $(\sigma,\sigma^0,\ssigma)$ is a 
 $\Puno$-solution to the augmented continuity equation with locally finite $\ppi$-marginals, and
 \begin{equation}
 	|(\sigma^0,\ssigma)|\le k\sigma.
 \end{equation}
 If moreover there exists a Borel vector field $\bww\colon \DST\to \DST$ 
 such that $\|\bww\|\ge k^{-1}$ and 
 \begin{equation}
 	\fry'=\bww(\fre)\quad\text{$\teeta$-a.e.},
 \end{equation}
then $(\sigma,\sigma^0,\ssigma)$ is a $\ppi$-adapted solution
(with constant $k$; it is $\ppi$-normalized if $k=1$) and
\begin{equation}
\label{eq:crucial}
	(\sigma_0,\ssigma)=\bww\sigma.
\end{equation}
 \item 
Conversely, let $(\sigma,\sigma^0,\ssigma)$ 
be a $\ppi$-adapted solution (with constant $k\ge 1$) 
 to  the augmented continuity equation 
\eqref{continuity-sigma-lim} according to Definition \ref{def:augmented}, and  let it be 
driven by the (autonomous) Borel vector field $\Vfield :=(\tau,\bvv)$. 
Then,
the support of $(\sigma,\sigma^0,\ssigma)$ 
is contained in $\III\times \DST$ and there exists 
$\eeta \in \sP(\Lipplus{k}\III\DST) $ 
that satisfies \eqref{eq:condition}, \eqref{eq:from-eta-to-sigma},
and is concentrated 
 on the  curves   solving the Cauchy problem
 \begin{equation}
\label{Cauchy-gamma}
\begin{cases}
\dot{\sfyy}(s) = \Vfield(\sfyy(s)),  \  s \in (0,\infty), 
\\
\sfyy(0)=(0,x),\quad x\in\supp(\mu_{0}).
\end{cases}
\end{equation}
If moreover $(\sigma,\sigma^0,\ssigma)$ is $\ppi$-normalized, then
$\eeta\in \mathcal P(\Lipplus 1\III\DST)$.
\end{enumerate} \
 \end{theorem}
\begin{proof}
	The first claim is well known (see e.g.~the second part of 
	\cite[Theorem~8.2.1]{AGS08} and can be easily checked 
	by a direct computation. Condition \eqref{eq:condition} ensures
	that $(\sigma,\sigma^0,\ssigma)$ has locally finite $\ppi$-marginals.
	Property \eqref{eq:crucial} easily follows from  Lemma \ref{le:strict-conv}.
	
	In order to prove the second claim, we can still rely 
	on \cite[Theorem~8.2.1]{AGS08} (which corresponds to the case of a finite interval), applied to the restrictions of 
	$(\sigma,\sigma^0,\ssigma)$ to the intervals 
	$[i,i+1]$. We find measures~$\eeta^i$ concentrated on 
	$\Lip_k([i,i+1];\DST)$ and corresponding measures
	$\widetilde\eeta^i=\mathcal L^1\otimes \eeta^i$ 
	satisfying~\eqref{eq:from-eta-to-sigma} 
	in $[i,i+1]$ together with
	\begin{equation}
		\label{eq: estimatei}
		\int_{ \mathrm{E}_{ [i,i+1]}(T)}
		\|\frxx'\|\,\dd\widetilde\eeta^i=
		|\ssigma|([i,i+1]\times [0,T]\times \R^d), 
	\end{equation}
	 where
	\[
	\mathrm{E}_{ [i,i+1]}(T):= \Big\{(s,\sfyy)\in 
	 [i,i+1]\times 	\Cplus:\sft(s)\in [0,T]\Big\}
	=\fre^{-1}( [i,i+1] \times [0,T]\times \R^d)\,.
	\]  
	We can then apply 
	the glueing Lemma \ref{le:glueing}: 
	it is sufficient to use
	\begin{displaymath}
		X:=\mathrm C(\III;\DST),\quad
		X^i:=\rmC([i,i+1];\DST),\quad
		Y^j:=\DST,
	\end{displaymath}
	and choose $
		\mathsf p^i \colon X\to X^i
$ as the operators mapping a
continuous curve defined in $\III$ into its  restriction 
to the interval $[i,i+1]$. 
We eventually set $\mathsf R^i=\mathsf L^{i+1}:=\fre_{i+1}$
and we thus find a measure
$\eeta$ such that $\mathsf p^i_\sharp \eeta=
\eeta^i$ for every $i\in \N$. It is easy to check that $\eeta$ satisfies all the properties stated in Claim (2).
 The second estimate in 
 \eqref{eq:condition} can be derived from
\eqref{eq: estimatei}.
\end{proof}

  We can now state a useful rescaling property, which is
  strongly related to the fact that 
  the velocity vector field is autonomous according to \eqref{eq:densities}.
\begin{lemma}[Rescaling]
	\label{le:rescaling}
	Let $(\sigma,\sigma^0,\ssigma)$ be a $\ppi$-adapted solution  to  the augmented
	continuity equation~\eqref{continuity-sigma-lim} driven by the (autonomous) Borel vector field 
	$\Vfield=(\tau,\bvv) \colon \DST\to \DST$, 
	let 
	\begin{equation}
	\label{labelliamotheta}
	\text{$\theta \colon \DST\to (0,+\infty)$ be a Borel map satisfying
	$c^{-1}\le \theta\le c<+\infty$ in $\DST$ for some constant $c\ge 1$,}
	\end{equation}
	and let 
	\begin{equation}
		\label{eq:rescaled-fields}
		\widehat\tau:=\theta\tau,\quad
		\widehat\bvv:=\theta \bvv,\quad
		\widehat\bww=(\widehat\tau,\widehat\bvv)=\theta\bww.
	\end{equation}
	There exists an autonomous solution $(\widehat \sigma,\widehat \sigma^0,\widehat \ssigma)$ 
	satisfying
	\begin{equation}
		\widehat\sigma^0=\widehat \tau\,\widehat\sigma,\quad
		\widehat\ssigma=\widehat\bvv\, \widehat\sigma,\quad
		\ppi_\sharp (\widehat\sigma^0,\widehat\ssigma)=\ppi_\sharp (\sigma^0,\ssigma).
	\end{equation}
\end{lemma}
\begin{proof}  
     By  Theorem \ref{p:superposition} 
     there exists ${\eeta} \in \mathcal{P}(\Lipplus{k}\III\DST)$ providing the representation formulae
     \eqref{repre-eq-17}
      and supported on solutions of the Cauchy problem \eqref{Cauchy-gamma}, with autonomous velocity field
     ${\Vfield} = ({\tau}, {\bvv})$
     satisfying $k^{-1}\le \|\Vfield\|\le k$.
For every Lipschitz curve $\sfyy\colon \III\to \DST$ we consider the solution
$\ell_\sfyy$  to  the differential equation
\begin{equation}
	\label{eq:rescaling-ODE}
	\dot \ell_\sfyy(r)=(\theta\circ \sfyy)(\ell_\sfyy(r)),\quad
	\ell_\sfyy(0)=0.
\end{equation}
Indeed, $\ell_\sfyy$ can be easily obtained as the inverse of the bi-Lipschitz map
\begin{equation}
	\label{eq:inverse}
	\Theta_\sfyy(s):=\int_0^s \frac1{\theta(\sfyy(r))}\,\dd r,\quad
	c^{-1}\le 	\Theta_\sfyy'\le c.
\end{equation}
 Thus, we may define
the function $\Rep:\Lipplus k\III\Rdpu\to\Lipplus{\widehat k}\III\Rdpu
$, $\widehat k:=ck$, that associates with  every Lipschitz curve $\sfyy$ 
the rescaled curve 
\begin{equation}
\label{def:Rep}
\Rep(\sfyy):= \sfyy\circ \ell_\sfyy.
\end{equation}
Notice that $\widehat \sfyy:=\Rep(\sfyy)$ satisfies the system
\begin{equation}
	\label{eq:rescaled-Cauchy}
	\widehat \sfyy'(r)=\theta(\widehat\sfyy(r))\Vfield(\widehat\sfyy(r))=
	\widehat\Vfield(\widehat\sfyy(r)),\quad
	\widehat\sfyy(0)=\sfyy(0).
\end{equation}
 By Lemma \ref{le:u-measurable} in Appendix \ref{app:meas-R} ahead, $\Rep$ is a Borel map. 
 Let us set 
$\widehat\eeta := \Rep_{\sharp }({\eeta})   \in \mathcal{P}(\Lipplus{}\III\DST)$. 
A further application of Theorem \ref{p:superposition}
yields the thesis.
\end{proof}

\subsection{A representation result by the augmented continuity equation}
\par

We can now apply  the previous results to 
get a first representation for $\Puno$-solutions
to the continuity equation  
\eqref{continuity-equation}.  
\begin{theorem}[Augmented representations of $\Puno$-solutions]
\label{thm:augmented}
Let $(\mu,\nnu)\in \calM^+(\DST ) \times \calM(\DST;\R^{d})$ be a $\Puno$-solution  to 
the continuity equation in the sense of Definition~\ref{def:solCE}, with initial condition~$\mu_{0} \in \mathcal{P}_{1}(\R^{d})$,
	let $(\mu,\bar\nnu)=\theta (\mu,\nnu)$ be 
	a minimal pair induced by $(\mu,\nnu)$ according to 
	Definition \ref{def:reduced} and fulfilling~\eqref{eq:trivial-but-useful},
	let $\varrho\sim |(\mu,\nnu)|$,
	and let $(\tau,\bvv)$ be bounded Borel vector field 
	representing the density of $(\mu,\nnu)$ w.r.t.~$\varrho$, i.e.
 \begin{equation}
 \label{strict-convexity-consequence}
	\mu=\tau\varrho,\qquad 
	\nnu= \bvv \varrho
	\qquad \text{$\varrho$-a.e.\ in $\DST$}.
 \end{equation}
	 Then, there exists a Lipschitz continuous curve $\sigma \in \rm{Lip}( \III  ;  \Puno(\DST))$ 
satisfying the following properties:
   \begin{enumerate}
   	\item the associated measure 
   	$\sigma=\mathcal L^1\otimes \sigma_s
   	\in \calM^+(\III\times \DST)$
   	has marginal
   	  \begin{equation}
\label{OK-marginal1}
\ppi_{\sharp } \sigma=\bar\varrho=\theta\varrho.
\end{equation}
\item 
 The measures
\begin{equation}
	\sigma^0:=(\tau\circ \ppi) \sigma,\qquad
	\ssigma:=(\bvv\circ \ppi) \sigma
\end{equation}
have marginals
   \begin{equation}
\label{OK-marginals}
\ppi_{\sharp } \sigma^0=\mu, \qquad 
\ppi_{\sharp }\ssigma=\bar\nnu.
\end{equation}
\item The triple $(\sigma, \sigma^{0}, \ssigma)$ 
is a $\ppi$-adapted solution  to  the augmented continuity equation
\eqref{continuity-sigma-lim},
 in the sense of  
  Definition \ref{def:augmented}. 
  \end{enumerate}
	In particular, when $\varrho=|(\mu,\nnu)|$
 	then $(\sigma,\sigma^0,\ssigma)$ is also 
	a $\ppi$-normalized solution. 
\end{theorem}
\begin{remark}
	\label{rem:barnu-nu}
	When $(\mu,\nnu)$ is a minimal $\Puno$-solution, 
	then $\nnu^\perp=\bar\nnu^\perp$ is minimal
	and \eqref{OK-marginal1} holds
	for $\nnu=\bar\nnu$.
\end{remark}
\begin{proof}
 Thanks to Lemma \ref{le:rescaling} it is not restrictive to assume that 
$\varrho=|(\mu,\nnu)|$, 
$\|{\cdot}\|$ is the Euclidean norm (so that it is strictly convex),
and $\theta\equiv 1$.    Therefore, in the remainder of the proof we shall use that   $\bar \nnu = 
\nnu$. 
\noindent
We will split the \emph{proof} in the following steps:
\begin{enumerate}
\item Regularization of  the pair $(\mu,\nnu)$ via convolution;
\item
Analysis of the `augmented', regularized system;
\item
 Passage to the limit  in the regularization parameter;
 \item Proof of property \eqref{strict-convexity-consequence}. 
\end{enumerate}
\subsubsection*{\bf Step $1$: regularization}
 It follows from Theorem \ref{th:1} that $\mu$ admits a  left-continuous representative w.r.t.\ narrow convergence.
Therefore,   from now on, without loss of generality we shall suppose that $t\mapsto \mu_t$ is 
 (narrowly)   left-continuous.
We now extend  the  measures $\mu$ and $\nnu$ to the whole
$\R^{d+1} $  by setting 
\[
\mu_t: = \begin{cases}
\mu_0 & \text{ if } t<0,
\\
 \mu_t   & \text{ if } t\ge 0, 
\end{cases} \qquad \text{ and }  \nnu = 0 \text{ on }   (-\infty,0) \times \R^d. 
\]
%
 Let us now consider convolution kernels
 $\kappa^0\in \rmC^\infty_c(\R)$ 
 $\kappa^1\in \rmC^{\infty}(\R^{d})$ 
 satisfying
 \begin{gather}
 	\kappa^0\ge 0,\quad \supp(\kappa^0)\subset [0,1],\quad \int_0^1\kappa^0\,\dd t=1,\\
 	0<\kappa^1\le 1,\quad
 	\int_{\R^d}\kappa^1\,\dd x=1,
 	\quad
 	\int_{\R^d}\|x\|\,\kappa^1(x)\,\dd x=M^1<\infty.
 	\intertext{Let us set}
 	\kappa^0_\eps(t):=\eps^{-1}\kappa(t/\eps),\quad
 	\kappa^1_\eps(x):=\eps^{-d}\kappa^1(x/\eps)\,\dd x,\quad
 	\kappa_\eps(t,x):=\kappa^0_\eps(t)
 	\kappa^1_\eps(x).
 \end{gather}
 For $(t, x) \in \R^{d+1}  $ 
 we define
\begin{equation}
\label{convoluted-measures}
\begin{aligned}
&
 \mu^\eps(t,x) =  (\mu{\star} \kappa_\eps)  (t,x) = 
	\int_{\R^{d+1}} \kappa_\eps(t-\tau, x-y) \dd \mu(\tau,y), 
\\
&  \nnu^\eps(t,x): =  (\nnu {\star} \kappa_\eps)  (t,x) = \int_{\R^{d+1}} \kappa_\eps(t-\tau, x-y) \dd \nnu(\tau,y)\,. 
\end{aligned}
\end{equation}
Since $(\mu,\nnu)$ is a  $\mathcal{P}_{1}$-solution to the continuity equation,
the functions $\mu^\eps  \in \mathrm{C}^\infty(\R^{d+1})$ and $\nnu^\eps \in  \mathrm{C}^\infty(\R^{d+1};\R^d)$
 are smooth solutions  to  the continuity equation
 \begin{equation}
\label{eps-continuity-equation}
 	\partial_t\mu^\eps+\dive\nnu^\eps=0
 	\quad\text{in }\R^{d+1},
 \end{equation}
with
 \begin{equation}
     \label{barmu0eps}
 \begin{aligned}
 \mu^\eps(t,x)  =
\int_{\R^d} \kappa_\eps^1(x-y) \dd \mu_0(y) \doteq  \bar{\mu}_0^\eps(x)\quad
\text{for every }t\le 0.
 \end{aligned}
  \end{equation}
%
It is easy to check that 
\begin{equation}
	\int_{\R^d}\mu_\eps(t,x)\,\dd x=1\quad\text{for every }t\in \R. 
\end{equation}
Moreover
\begin{align*}
	\int_{\R^d}\|x\|\,\bar\mu_0^\eps(x)\,\dd x&=
	\int_{\R^d}\int_{\R^d} \|x\|
	\kappa_\eps^1(x-y) \dd \mu_0(y)\,\dd x
	\le \int_{\R^d}\int_{\R^d}
	\Big(\|x-y\|+\|y\|\Big)
	\kappa_\eps^1(x-y) \dd \mu_0(y)\,\dd x
	\\&\le \eps M^1+\int_{\R^d}\|y\|\,\dd \mu_0(y). 
\end{align*} 
  With slight abuse of notation, we shall denote by $\mu^\eps = (\mu_t^\eps)_\eps$ and $\nnu^\eps = (\nnu_t^\eps)$ (where $\nnu_t^\eps: = \nnu^\eps(t,\cdot)$) also the measures with densities $\mu^\eps$ and $\nnu^\eps$, respectively.
\nc   Due to
 \cite[Thm.\ 2.2]{AmFuPa05FBVF}
 and the previous estimate, 
 \nc we have the following convergences as $\eps \down 0$:
\begin{subequations}
\begin{equation}
\label{converg-mu-nu}
\begin{aligned}
\mu^\eps \weaksto  \mu \quad  \text{ in }  \calM^{+}  (\DST),\quad     \nnu^\eps \weaksto \nnu \quad \text{ in }  \calM  (\DST;\R^d), \quad
 \bar{\mu}_0^\eps \to  \mu_0 \text{ in }  
 \mathcal P_1(\R^d),\\
 \mu^\eps\mres {\DSTT T}\rightharpoonup
  \mu\mres{\DSTT T}\quad\text{narrowly in }
  \mathcal M^+(\DSTT T)\quad\text{for every }T>0. 
\end{aligned}
\end{equation} 
We also have   by 
 \cite[Thm.\ 2.2]{AmFuPa05FBVF}   that 
\begin{equation}
\label{converg-variations}
	|(\mu^\eps,\nnu^\eps)| \weaksto |(\mu,\nnu)|, \quad   |\nnu^\eps| \weaksto |\nnu|    \qquad  \text{ in } \ \calM^+(\DST).
\end{equation}
\end{subequations}
%
In a similar way, we can show that 
\begin{equation}
\label{est-tot-vari-nueps}
|(\mu^\eps,\nnu^\eps)|(\DSTT T)\le 
|(\mu,\nnu)|(\DSTT T),\quad
|\nnu^{\eps}|(\DSTT T) \leq |\nnu|(\DSTT T)\,,
\end{equation}
%
  which implies that
  \begin{equation}
  \label{e:narrow-conv}
  \begin{cases}
  	|\nnu^\eps|\mres \DSTT T
  	\rightharpoonup |\nnu |\mres\DSTT T,
	\\
  	|(\mu^\eps,\nnu^\eps)|\mres \DSTT T
  	\rightharpoonup |(\mu,\nnu )|\mres\DSTT T
	\end{cases}
  	\quad\text{narrowly in $\mathcal M(\DSTT T)$}
   \end{equation}
 for every $T \in [0, +\infty)$ such that 
 $| \nnu| (\{T\} \times \R^{d}) = 0$.

Since $\mu^\eps(t,x)>0$ (see \cite[Lemma 8.1.9]{AGS08}) 
we may introduce the velocity field
\begin{equation}
\label{e:regularvelocity}
 \bww^\eps(t,x)  : = \frac{\nnu_t^\eps(x)}{\mu_t^\eps( x)} \qquad \text{for all } (t, x)\in \R^{d+1}.
\end{equation}
The velocity field $\bww^\eps$ fulfills the local regularity conditions of \cite[Prop.\ 8.1.8]{AGS08}, 
which we may therefore apply to the continuity equation 
\eqref{eps-continuity-equation}. 
We can introduce the characteristic system
\begin{equation}
\label{ODE-X}
\begin{cases}
\dot{X}_t^\eps = \bww_t^\eps(X_t^\eps),
\\
X_0^\eps=x\,,
\end{cases}
\end{equation}
and we denote by $D^\eps$ the subset of $x\in \R^d$ for which the unique
maximal solution is globally defined. 
We know  that  $\R^d\setminus D^\eps $ is 
$\bar\mu^\eps_0$-negligible 
(equivalently $\mathcal L^d(\R^d{\setminus} D^\eps)=0)$)
and \eqref{ODE-X} defines a flow $X^\eps_t \colon D^\eps \to D^\eps$, $t\ge0$, 
inducing the representation formula
\begin{equation}
\label{characteristic}
\mu_t^\eps = (X_t^\eps)_\sharp   \bar{\mu}_0^\eps.
\end{equation}
\nc We finally notice that by~\eqref{e:regularvelocity}, \eqref{characteristic},  and \eqref{est-tot-vari-nueps},   we get for every $T \in [0, +\infty)$ the bound
\begin{equation}
\label{e:boundvelocity}
 \int_{\R^{d}}\int_{0}^{T}  \norm{ \bww^{\eps}  (t, X^{\eps}_t(x))} 
 \,\di t\, \di \bar\mu_{0}^{\eps}(x) \leq |\nnu^{\eps}| ([0,T]{\times}\R^{d}) \leq |\nnu|([0,T] {\times}\R^{d})\,.
\end{equation}
\subsubsection*{\bf Step $2$: Analysis of the augmented system}
We define 
\begin{subequations}
\label{tau-eps}
\begin{align}
\label{tau-veps-1}
\tau^\eps &\colon \R^{d+1}
\to \R, & \tau^\eps(t,x) &\coloneq  \frac1{\norm{(1,\bww^\eps(t,x))}},\\
\label{veps}
\bvv^{\eps}&\colon \R^{d+1}
\to \R^{d}, & \bvv^{\eps} (t, x) &\coloneq \tau^{\eps} (t, x)\bww^{\eps}(t, x) 
 = \frac{\bww^\eps(t,x)}{\norm{(1,\bww^\eps(t,x))}}\,. 
\end{align}
\end{subequations} 
By construction we have that 
\begin{equation}
\label{normalization}
\begin{aligned}
&
 \norm{ (\tau^{\eps} (t, x),\bvv^\eps(t,x)) } \equiv 1 
 \qquad \text{for every $(t, x) \in  \R^{d+1}$}.
\end{aligned}
\end{equation}
For each $\eps>0$ the functions $\tau^{\eps}$ and~$\bvv^{\eps}$ 
are locally Lipschitz and globally bounded.

We now consider the  following `augmented' characteristic system, in the unknowns 
 $T \colon \III   \to \R$ and 
$Y\colon \III   \to \R^d$
\begin{equation}
\label{ODE-augmented}
\begin{cases}
\dot{T}_s= \tau^\eps(T_s,Y_s ),
\\
\dot{Y}_s = \bvv^\eps(T_s,Y_s),
\\
T_0 = t,
\\
Y_0= x\,.
\end{cases}
\end{equation} 	
For every 
$(t,x)\in   \R^{d+1}$, the Cauchy problem 
 possesses a unique solution $s \mapsto (T_s^\eps(t,x),Y_s^\eps(t,x))  $ 
 which is globally defined. Clearly,
 $s\mapsto T^\eps_s$ is an increasing map and in particular
	$T_s^\eps(t,x) \geq 0$ if $(t,x)\in \DST  $. 
   The following result relates the 
	flow map $(T^\eps,Y^\eps)\colon \III\times \R^{d+1}  \to \R^{d+1} $,
defined by
   $T^\eps(s,t,x) := T_s^\eps (t,x)$, $Y^\eps(s,t,x) := Y_s^\eps (t,x)$, 
 with the flow map 
  $X^\eps \colon \III\times D^\eps\to D^\eps$
 of  the ODE system    \eqref{ODE-X}.

\begin{lemma}\label{l:artificial-Teps}
For every $x\in D^\eps$,
 let $(\widehat{T}_{s}^\eps(x))_{s\in \III}$ solve
 the Cauchy problem 
\begin{equation}
\label{artificial-Teps}
\begin{cases}
\widehat{T}_s' = \tau^\eps(\widehat{T}_s,X^\eps(\widehat{T}_s,x)),
\\
 \widehat{T}_0 =0,
\end{cases}
\end{equation}
and thus define a map 
 $\widehat{T}^\eps \colon \III\times D^\eps\to \III$.
Define $\widehat{Y}^\eps\colon \III\times D^\eps \to D^\eps$ via  
 $\widehat{Y}^\eps(s,x): = X^\eps (\widehat{T}^\eps(s,x),x)$. Then,
 \begin{equation}
 \begin{split}
 \label{coincidence}
& \widehat{T}^\eps(s,x) = T^\eps(s,0,x)\,,  \qquad \widehat{Y}^\eps(s,x) = Y^\eps(s,0,x) \qquad \text{for all } (s,x) \in  
 \III\times D^\eps \,,
\end{split}
 \end{equation}
 and $\widehat{T}^{\eps} (\cdot, x)$ is a (strictly increasing and surjective)
 diffeomorphism of  $\III $ 
  for every $x\in D^\eps$.
 In particular, 
\begin{equation}
\label{link-Xeps-Yeps}
Y^\eps(s,0,x) = X^\eps (T^\eps(s,0,x),x) \qquad \text{for all } (s,x) \in  \III  \times \R^d\,.
  \end{equation}
Furthermore, 
if $ S^\eps \colon   \III  \times D^\eps \to \III$ 
is defined by 
  \begin{equation}
  \label{e:S-eps}
  S^{\eps}(t, x) \coloneq \int_{0}^{t} \norm{(1 , \bww^{\eps} (\tau, X^{\eps}(\tau, x)))} \, \di \tau\,,
  \end{equation}
  then 
    \begin{equation}
  \label{e:S-eps-2}
 \int_{\R^d}  S^{\eps}(T, x) \,  \di \bar\mu^{\eps}_{0}(x) \leq |( \mu, \nnu)| ([0, T ] \times \R^{d})
  \end{equation} 
  and 
  \begin{equation}
  \label{inversion}
   S^{\eps}(\cdot,  x)  = (\widehat{T}^{\eps})^{-1}( \cdot,  x)  \qquad 
   \text{for every $x \in D^\eps$}. 
  \end{equation}
\end{lemma}
\begin{proof}
We observe that
 the functions $\widehat{T}^\eps$ and $ \widehat{Y}^\eps$  satisfy
\begin{align}
\nonumber
& 
\begin{aligned}
\partial_s \widehat{T}^\eps(s,x)  & \stackrel{\eqref{artificial-Teps}}{=} \tau^\eps(\widehat{T}^\eps(s,x), X^\eps (\widehat{T}^\eps(s,x), x))
 = \tau^\eps (\widehat{T}^\eps(s,x), \widehat{Y}^\eps(s,x)),
\end{aligned}
\intertext{as well as}
\nonumber
&
\begin{aligned}
\partial_s \widehat{Y}^\eps(s,x)   & = \partial_t X^\eps(\widehat{T}^\eps(s,x),x )\partial_s \widehat{T}^\eps(s,x)
\\
& \stackrel{\eqref{ODE-X}}{=}\bww^\eps(\widehat{T}^\eps(s,x), X^\eps (\widehat{T}^\eps(s,x),x)) \, \tau^\eps (\widehat{T}^\eps(s,x), X^\eps (\widehat{T}^\eps(s,x),x))
\\
& \stackrel{\eqref{veps}}{=} \bvv^\eps (\widehat{T}^\eps(s,x), \widehat{Y}^\eps(s,x))\,.
\end{aligned}
\end{align}
Since we also have that
$\widehat{Y}^\eps(0,x) = X^\eps (T^\eps(0,0,x),x) =X^\eps(0,x) = x$, we conclude that 
 the pair $(\widehat{T}^\eps,\widehat{Y}^\eps)$  solves system \eqref{ODE-augmented} for $t=0$. 
 By uniqueness, \eqref{coincidence} follows. 
\par
 The function $S^{\eps}$ defined in~\eqref{e:S-eps} is finite 
(since the integrand is a continuous function w.r.t.\ $\tau$) and it is clearly  strictly increasing.  Estimate
\eqref{e:S-eps-2} follows immediately by 
\begin{displaymath}
\begin{aligned}
   \int_{\R^d}  S^{\eps}(T, x) \,  \di \bar\mu^{\eps}_{0}(x)  
  & =  \int_{\R^{d}} \int_{0}^{T}  \norm{(1,  \bww^{\eps} (t, X^{\eps}(t, x) )) } \, \di t \, \di \bar\mu^{\eps}_{0}(x) 
  \stackrel{\eqref{e:boundvelocity}}\le 
  |(\mu, \nnu) | \big([0, T] \times \R^{d}\big) <+\infty\,.
%
 \end{aligned}
\end{displaymath}
\FCOMM Finally, we observe that for 
 every $x\in D^\eps$ \nc
\[
\begin{aligned}
\partial_t (S^\eps{\circ}\widehat{T}^\eps)(s,x)  & = \partial_t S^\eps(\widehat{T}^\eps(s,x),x)\, \partial_s\widehat{T}^\eps(s,x)
\\
& 
\stackrel{(1)}{=} \norm{(1 , \bww^{\eps} (\widehat{T}^\eps(s,x), X^{\eps}(\widehat{T}^\eps(s,x), x)))}  \cdot \tau^\eps(\widehat{T}^\eps(s,x), X^{\eps}(\widehat{T}^\eps(s,x), x)) 
\stackrel{(2)}{\equiv} 1\,,
\end{aligned}
\]
where (1) is due to \eqref{e:S-eps} and  \eqref{artificial-Teps}, while (2) is a consequence of \eqref{tau-eps}.
Hence,  \eqref{inversion} follows, whence we conclude that 
 $\widehat{T}^{\eps}(\cdot, x)$ is 
 a  strictly increasing   diffeomorphism of  $\III$  for every $x\in D^\eps$.
\end{proof}
\par
Let us now consider the continuity equation with the vector field $(\tau^\eps,\bvv^\eps)$
and initial datum $\sigma^\eps_0=\delta_0{\otimes}\bar\mu_0^\eps$.
 Since $\sigma^\eps_0$ is supported in $\DST$  then the family of measures
\begin{equation}
 \label{sigma-repr-zero}
\sigma_s^\eps := ({T}_s^\eps,{Y}_s^\eps)_{\sharp } (\delta_0 {\otimes} \bar\mu_0^\eps) \qquad \text{for all } s \in  \III,
\end{equation}
are supported in $\DST$ as well. 
Moreover, 
\eqref{coincidence} shows that 
\begin{equation}
 \label{sigma-repr}
\sigma_s^\eps = (\widehat{T}_s^\eps,\widehat{Y}_s^\eps)_{\sharp }  \bar\mu_0^\eps\qquad \text{for all } s \in  \III.
\end{equation}
\nc 
 It follows from 
\cite[Lemma 8.1.6, Prop.\ 8.1.8]{AGS08} that 
the curve $ \sigma^\eps $ belongs to 
 $ \rm{Lip} (\III;\Puno(\DST)) $ 
 (it is in fact $1$-Lipschitz) and 
fulfills
\begin{equation}
\label{continuity-sigmaeps}
\begin{cases}
\partial_s \sigma^\eps +\partial_t(\tau^\eps \sigma^\eps) +\dive(\bvv^\eps \sigma^\eps) =0 \qquad \text{in }  \III  \times \R^{d+1},
\\
\sigma_0^\eps = \delta_0 \otimes \bar\mu_0^\eps\,.
\end{cases}
\end{equation}
From now on, we will use the short-hand notation
\[
\sigma^{\eps,0}: = \tau^\eps \sigma^\eps, \qquad \ssigma^\eps: = \bvv^\eps \sigma^\eps.
\]
Observe that, in view of \eqref{normalization}, the measures $\sigma^{\eps,0} $ and $\ssigma^\eps$ satisfy
\begin{equation}
\label{normalization-sigma}
\begin{aligned}
  |( \sigma^{\eps,0}, \ssigma^\eps) | = \sigma^{\eps} \qquad \text{in $\calM^{+}(  \III {\times}\DST )$}.
  \end{aligned}
  \end{equation}
 In the following lemma
the relation between~$\mu^{\eps}$,~$\nnu^{\eps}$, and~$\sigma^{\eps}$ is   established in terms of the projection operator
$
\pinew\colon  \R \times \R  \times\R^{d} \to  \R  \times \R^{d}
$, $ \pinew(s,t,x) := (t, x)$
 from \eqref{ppi-projection}. 
\begin{lemma}\label{lemma4.3}
There holds
\begin{equation}
\label{marginals}
\mu^\eps   = \pinew_{\sharp }\sigma^{\eps,0}\,, \qquad   \nnu^\eps   = \pinew_{\sharp }\ssigma^\eps\,, \qquad | (\mu^{\eps}, \nnu^{\eps}) | =  \pinew_{\sharp } \sigma^{\eps}  \qquad \text{in $\III\times \R^{d}$}\,.
\end{equation}
 Moreover for every $S,T>0$
\begin{equation}
\label{eq:tightness}
\sigma^{\eps, 0} ((S, +\infty) \times 
\DSTT T)  \leq \frac{T}{S}\, |(\mu, \nnu)|\big (\DSTT T\big).
\end{equation}
\end{lemma}
\begin{proof}
\par
For every 
 $\phisc \in \Cc(\R^{d+1})$
 we have
%
\[
\begin{aligned}
\int_{ \R^{d+1} } 
 \phisc(t,x) \, \dd \mu^\eps(t,x) &  \stackrel{(1)}{=}\int_{\III }\int_{\R^d} \phisc(t,X^\eps(t,x))\,  \dd \bar\mu_0^\eps(x) \, \dd t
\\
& 
  \stackrel{(2)}{=}\int_{D^\eps} \int_{\III}
   \phisc(\widehat{T}^\eps(s,x),X^\eps(\widehat{T}^\eps(s,x),x)) \,   \partial_s \widehat{T}^\eps(s,x) \, \dd s \, \dd \bar\mu_0^\eps(x)
\\
& 
  \stackrel{(3)}{=}
\int_{D^\eps}  \int_{\III} \phisc(\widehat{T}^\eps(s,x),\widehat{Y}^\eps(s,x)) \,\tau^\eps(\widehat{T}^\eps(s,x),\widehat{Y}^\eps(s,x))\, \dd s   \, \dd \bar\mu_0^\eps(x) 
  \\&=
     \int_{\III}
  \Big(\int_{D^\eps}\phisc(\widehat{T}^\eps(s,x),\widehat{Y}^\eps(s,x)) \,\tau^\eps(\widehat{T}^\eps(s,x),\widehat{Y}^\eps(s,x))  \, \dd \bar\mu_0^\eps(x) \Big)\, \dd s 
  \\
  &
  \stackrel{(4)}{=}
 \int_{\III }\bigg(  \int_{\R^{d+1}} \phisc(t,x)\tau^\eps(t,x) 
 \, \dd \sigma_s^\eps(t,x)\bigg) \, \dd s=
 \int_{\DSTpu}  \phisc(t,x) \, \dd \sigma^{\eps,0}(s,t,x)
 \end{aligned}
\]
 where (1) follows from \eqref{characteristic}, (2)  and (3)  from the change of variables $t= \widehat{T}^\eps(s,x)$ (see  Lemma~\ref{l:artificial-Teps}),  
   (4) from~\eqref{sigma-repr}, and we have repeatedly applied Fubini's Theorem.
 
 \par
  The second of \eqref{marginals} follows from the fact that $\nnu^\eps = \bww^\eps \mu^\eps$ and $\bvv^\eps = \tau^\eps \bww^\eps$, so that for all
  test functions
   $\bvarphi \in \Cc(\R^{d+1};\R^d)$ 
\[
\begin{aligned}
\int_{ \R^{d+1}} & \bvarphi( t, x) \, \dd \nnu^\eps ( t, x )  = 
\int_{\III} \int_{\R^d} \! \bvarphi(t, X^\eps( t, x)) {\,\cdot\,} \bww^\eps(t,X^\eps(t,x)) \, \dd \bar\mu_0^\eps(x) \, \dd t
\\
& 
=\int_{\R^{d}} \int_{\III}
\bvarphi(\widehat{T}^\eps(s,x),X^\eps(\widehat{T}^\eps(s,x),x)) {\,\cdot\,} \bww^\eps(\widehat{T}^\eps(s,x),X^\eps(\widehat{T}^\eps(s,x),x)) \, \partial_s \widehat{T}^\eps(s,x) \, \dd s \,  \dd \bar\mu_0^\eps(x)
\\
& 
  = \int_{{}_{\scriptstyle \R^d}}  \int_{\III} 
   \bvarphi(\widehat{T}^\eps(s,x),\widehat{Y}^\eps(s,x)) {\,\cdot\,}  \tau^\eps(\widehat{T}^\eps(s,x),\widehat{Y}^\eps(s,x))  \bww^\eps(\widehat{T}^\eps(s,x),\widehat{Y}^\eps(s,x)) \, \dd s \,  \dd \bar\mu_0^\eps(x)
  \\
  &  
    =  \int_{\III} \bigg(
   \int_{\R^d}  \bvarphi(\widehat{T}^\eps(s,x),\widehat{Y}^\eps(s,x)) \cdot \bvv^\eps(\widehat{T}^\eps(s,x),\widehat{Y}^\eps(s,x))  \, \dd \bar\mu_0^\eps(x) \bigg)\,\di s
  \\
    & 
=\int_{\III} \bigg(
\int_{\R^{d+1}} \bvarphi(t,x)\cdot \bvv^\eps(t,x)\, \dd \sigma^\eps_s(t,x)\bigg)
\, \dd s
=\int_{\DSTpu}  \bvarphi(t,x) \, \dd \ssigma^\eps(s,t,x)\,.
 \end{aligned}
\]
\par
 Finally, for every open subset $A$ of $\DST $ we have that
\begin{align*}
| ( \mu^{\eps}, \nnu^{\eps})| (A) 
& = \int_A \|(1,\bww^\eps)(t,x) \| \,\dd \mu^\eps(t,x)  
  \stackrel{(1)}{=} \int_{\III{\times} A} 
 \|(1,\bww^\eps)(t,x)\| \,\dd \sigma^{\eps,0}(s,t,x) 
\\
& 
  \stackrel{(2)}{=}  \int_{\III{\times} A}    \|(1,\bww^\eps)(t,x)\|\, 
	\tau^\eps(t,x)  \, \di \sigma^{\eps}(s, t, x) \stackrel{(3)}=
	\int_{\III{\times} A}\di \sigma^{\eps}(s, t, x)
	=  \pinew_{\sharp }\sigma^{\eps} (A)\,,
\end{align*}
\FCOMM where 
(1) follows from the previously proved fact that $\mu^\eps = \pinew_{\sharp } \sigma^{\eps,0}$; for 
(2)  we have used that $\sigma^{\eps,0} = \tau^\eps \sigma^\eps$, while (3) is a consequence of  \eqref{tau-veps-1}. 
\nc This concludes the proof of~\eqref{marginals}.

 In order to check 
the tightness estimate \eqref{eq:tightness},
let us denote by $\iota_{S,T}$ the  
characteristic function of $(S,+\infty)\times [0,T]$ and by $j^\eps_{S,T}(s,x)$ 
the characteristic function of 
$(S,S^\eps(T,x)]$ (which is identically $0$
if $S\ge S^\eps(T,x)$).
We first observe that 
\begin{displaymath}
	\iota_{S,T}(s,\widehat T^\eps(s,x))=
    j^\eps_{S,T}(s,x)
\end{displaymath}
 so that 
\[
\begin{split}
\sigma^{\eps, 0} ((S, +\infty) \times 
\DSTT T) & = 
\int_{\R^{d+2}} 
 \iota_{S, T}  (s,t)\tau^{\eps}(t, x) \, \di \sigma^{\eps}(s, t, x)
\\
&
= 
 \int_{\R^d} \int_{\III }
 \iota_{S, T}  (s,\widehat{T}^{\eps}(s, x))
\tau^{\eps}( \widehat{T}^{\eps}(s, x), \widehat{Y}^{\eps}(s, x)) \, \di s \, \di \bar\mu^{\eps}_{0}(x)
 \\
 &
 = \int_{\R^{d}} \int_{\III}  
 j^\eps_{S,T}(s,x)\partial_s \widehat T^\eps(s,x) \, \di s \, \di \bar\mu^{\eps}_{0}(x) 
 \\&
 = \int_{\R^{d}} \big(T-\widehat T^\eps(S,x)\big)_+ \, \di \bar\mu^{\eps}_{0}(x) 
 \leq T\, \bar\mu^{\eps}_{0} \Big \{ x \in \R^{d} : \,  \widehat{T}^{\eps}(S, x) < T\Big\} 
 \\&
= T \, \bar\mu^{\eps}_{0} \Big \{ x \in D^\eps : \,  S^{\eps}(T, x) > S
\Big \}\,.
\end{split}
\]
Estimate \eqref{eq:tightness} then follows by~\eqref{e:S-eps-2} and by the Chebyschev inequality.
%
\end{proof}

\subsubsection*{\bf Step $3$:   Passage to the limit  as $\eps\down0$}
 As previously observed, \EEE   the curves of measures  $(\sigma^{\eps})_\eps \subset 
\Lip (\III;  \Puno(\R^{d+1})) $ are  $1$-Lipschitz continuous for every $\eps>0$.
 From the estimate
$
W_1(\sigma_s^{\eps}, \sigma_0^\eps) \leq s
$ for all $s\in \III$, 
recalling that $\sigma_0^\eps =\delta_0 \otimes \bar\mu_0^\eps$ with $\bar\mu_0^\eps $ from \eqref{barmu0eps} we deduce that the curves 
$(\sigma^\eps)_\eps$ lie in a bounded set of $ \Puno(\R^{d+1})$, 
narrowly compact in $\mathcal P(\R^{d+1})$. Moreover, the uniform bound on their Lipschitz constant provides an `equicontinuity' estimate that allows us to apply, e.g., the metric version of the Ascoli-Arzel\`a Theorem proved in \cite[Prop.\ 3.3.1]{AmbrosioGigliSavare08}. We thus conclude that there exists \EEE
 a limit curve 
$\sigma \in\Lip (\III;  \Puno(\R^{d+1}))$, 
$1$-Lipschitz and supported in $\DST$, and a 
vanishing subsequence $(\eps_k)_k$ 
such that $\sigma^{\eps_k}_s\rightharpoonup
\sigma_s$ narrowly in
$ \Puno(\R^{d+1})$ for every $s\in \III$.
As usual we will also denote by $\sigma$ 
the Radon measure $\mathcal L^1\otimes \sigma_s$ in $\calM^+(\III\times 
\R^{d+1})$ satisfying

\begin{equation}
\label{basic-wstar-converg}
\sigma^{\eps_k} \mres \DSTpuS\rightharpoonup  \sigma
\mres \DSTpuS \qquad 
\text{narrowly in }\calMb^+(\R^{d+2})
\quad\text{for every }S>0.
\end{equation}
Clearly, the above convergence
also yields
$ 
\sigma^{\eps_k}\weaksto \sigma$
in $\calM^{+}([0, +\infty) \times \DST )$ as $k\uparrow \infty$.
Moreover, we deduce from~\eqref{e:narrow-conv} and from~\eqref{marginals} that 
\begin{equation}
\label{e:ineq-1}
\pinew_{\sharp } \sigma \leq | (\mu, \nnu)|\,.
\end{equation}
 \par
 From \eqref{normalization-sigma} and the projection relations \eqref{marginals} it follows that 
 \ the restrictions
 to $\DSTpuS$ 
  of the families of measures  $(\sigma^{\eps, 0})_\eps$ and~$(\ssigma^{\eps})_\eps$
 are uniformly tight,
 so that 
 there exist $\sigma^{0} \in  \calM^{+} (\DSTpu)$ and $\ssigma \in  \calM  (\DSTpu ; \R^{d})$ such that 
 up to a further (not relabeled) subsequence
 \begin{equation}
 \label{sigma-eps-to-sigma}
 \begin{aligned}
 &   (\sigma^{\eps_k,0},\ssigma^{\eps_k})
 \mres \DSTpuS \rightharpoonup
 (\sigma^0,\ssigma)\mres \DSTpuS   && \text{in  $\calMb(\R^{d+2};\R^{d+1})$}
 \quad \text{for every }S>0.
  \end{aligned}
  \end{equation}

Therefore,    also thanks to the third of \eqref{converg-mu-nu} and  \eqref{basic-wstar-converg}  
we gather that the triple~$(\sigma,\sigma^0,\ssigma)$ satisfies the Cauchy problem   \eqref{continuity-sigma-lim} 
 in the sense of  
  Definition \ref{def:solCE}, 
  with the operator $(\partial_t, \dive)$ playing the role of the `spatial divergence' in \eqref{distributional-sense}. 
  Moreover, 
   it follows from \eqref{normalization}, \eqref{basic-wstar-converg},  and the lower semicontinuity of the total variation functional that 
  \begin{equation}
  \label{citata-dopo}
  |( \sigma^{0}, \ssigma) | \leq \sigma \qquad
  \ \text{in }\R^{d+2},
  \end{equation} 
   so that in particular 
   $\sigma^{0}, \ssigma \ll \sigma$. 
   Hence, since $\sigma$ is concentrated
   in $\III\times \DST$,
   the Radon-Nikod\'ym derivatives  
   can be represented by bounded Borel fields 
\[
\hat \tau \coloneq \frac{\dd \sigma^0}{\dd\sigma} \colon  \III  \times \DST  \to \III \qquad \hat\bvv \coloneq  \frac{\dd \ssigma}{\dd\sigma} \colon  \III  \times  \DST  \to \R^d
\]
satisfying
 \begin{equation}
 \label{Lip-less-1}
\norm{(\hat \tau, \hat\bvv)} \leq 1     \qquad \sigma\text{-a.e.\ in }  \III {\times} \DST .
\end{equation}
 In terms of  $\hat \tau$ and $\hat \bvv$,   we rewrite the continuity equation~\eqref{continuity-sigma-lim} as
\begin{equation}\label{continuity-sigma-lim2}
\begin{cases}
\partial_s \sigma +\partial_t (\hat \tau \sigma) + \dive  (\hat\bvv  \sigma )=0 \qquad \text{in }   \III \times \R^{d+1}
\\
\sigma_0 = \delta_0 \otimes \mu_0\,.
\end{cases}
\end{equation}
%
  It remains to prove the projection properties~\eqref{OK-marginals}. 
 We start by showing that 
 \begin{equation}
 \label{1st-proj-prop}
 \mu = \pinew_{\sharp } \sigma^{0}.
 \end{equation}
  For this,  
  \ the tightness estimate 
  \eqref{eq:tightness}
 implies that 
  $\sigma^{\eps,0}\mres \III\times 
  \DSTT T$ narrowly converge to 
  $\sigma^{0}\mres \III\times \DSTT T$
  for every $T>0$. 
  This shows that $\mu(\DSTT T)=
  \sigma^0(\III\times \DSTT T)$ 
  for every $T>0$.  Since 
  $\pinew_\sharp\sigma^0\le \mu$,
  the above  equality yields \eqref{1st-proj-prop}.
\par
 Let us now show that  
\begin{equation}
\label{2nd-marginal}
\nnu =  \pinew_{\sharp } \ssigma\,.
\end{equation} 
 With this aim, we set 
$\onu \coloneq \pinew_{\sharp } \ssigma $. In order to show that $\onu=\nnu$, we argue in the following way.
 On the one hand,  it follows from  the previously proved Lemma \ref{l:CE-onu}  that 
the pair 
$(\mu, \onu)$ solves 
 the continuity equation
\begin{equation}
\label{CE-bar}
\partial_{t} \mu + \dive \onu = 0 \qquad \text{in } 
 \DST\,, 
\end{equation}
with initial condition $\mu_{0}$, in the sense of Definition~\ref{def:solCE}.
 On the other hand, applying  Lemma
 \ref{l:sub-nu1} with the choices $\pushm = \pinew\colon  \R \times \R  \times\R^{d} \to  \R  \times \R^{d}$, 
$\zzeta_k = (\sigma^{\eps_k,0}, \ssigma^{\eps_k})$, $\zzeta = (\sigma^0,\ssigma)$, $\llambda_k = \pinew_{\sharp} (\sigma^{\eps_k,0}, \ssigma^{\eps_k})
=(\mu^{\eps_k},\nnu^{\eps_k})$, $\llambda = (\mu,\nnu)$, 
we  show that 
\[
(\mu,  \onu)  \stackrel{\eqref{1st-proj-prop}}{=}(\pinew_{\sharp } \sigma^0, \pinew_{\sharp } \ssigma )
\prec (\mu, \nnu).
\]
 Then, by  Lemma~\ref{r:sub}     
 there exists $\lambda \in L^{\infty}_{|(\mu, \nnu)|} (\DST ; [0, 1])$ such that $(\mu, \onu) = \lambda (\mu, \nnu)$. Since the first components  coincide, from that equality  
 we infer that $\onu = \lambda \nnu$ and 
 $\lambda \equiv 1$ $\mu$-a.e., as well.  
We  decompose~$\nnu$ and~$\onu$ into their absolutely continuous and singular part w.r.t.~$\mu$, namely,
\begin{displaymath}
\nnu = \nnu^{a}  + \nnu^{\perp} \qquad \onu= \onu^{a}  + \onu^{\perp} \qquad \text{with $\nnu^{\perp}, \onu^{\perp} \perp \mu$.}
\end{displaymath}
 Since $\onu = \lambda \nnu$, we have 
\begin{displaymath}
   \onu^{a} + \onu^{\perp} = (\lambda \nnu^{a}) + \lambda\nnu^{\perp}\,.
\end{displaymath} 
As $\lambda \equiv 1$ $\mu$-a.e.~in~$\DST $, we have that $\onu^{a} = \nnu^{a}$, $\onu^{\perp} = \lambda \nnu^{\perp}$.
Therefore, $\onu^{\perp} \prec \nnu^{\perp}$. Moreover, by Lemma~\ref{l:CE-onu} it holds $\dive \onu^{\perp} = \dive \nnu^{\perp}$. By hypothesis, $\nnu^{\perp}$ is minimal.
Therefore, we conclude that 
 $\onu^{\perp} = \nnu^{\perp}$  and, since $\onu^{a} \mu= \nnu^{a}\mu$, we ultimately have $\onu = \nnu$. 
 Then, \eqref{2nd-marginal} follows. 
\par
 Lastly,   we may conclude 
\eqref{OK-marginal1} 
 by observing that 
\[
|(\mu, \nnu)| \stackrel{\eqref{e:ineq-1}}{\geq} \pinew_{\sharp }\sigma \stackrel{\eqref{citata-dopo}}{\geq}
 \pinew_{\sharp } | (\sigma^{0}, \ssigma)| \geq |\pinew_{\sharp } (\sigma^{0}, \ssigma)| =  |(\mu, \nnu)|\,. 
 \]
  This finishes the proof.
\par
\end{proof}

\section{The parametrized superposition principle}
\label{s:5}
 The main result of this  section, Theorem \ref{t.1},  provides a probabilistic representation of the solutions of the continuity equation,
which  in fact extends the representation obtained in  \cite[Thm.\ 8.2.1]{AGS08} for absolutely continuous solutions. We will derive it from the 
probabilistic 
 representation \eqref{eq:from-eta-to-sigma},  provided by Theorem \ref{p:superposition}, 
   for the solutions to the `augmented'  system  \eqref{continuity-sigma-lim}. For that, we will need to  
 take into account that $\sigma$ and the measures~$\mu$ and~$\nnu$ are related via~\eqref{OK-marginals}.  Fine properties of our probabilistic representation will be proved in Proposition~\ref{c:2} and in Theorem~\ref{t:injective-curves}.

In the spirit  of  \cite[Thm.\ 8.2.1]{AGS08},
the statement of 
 Theorem \ref{t.1} below consists of two parts:
\begin{enumerate}
\item
   First of all, 
starting from a measure 
 $\eeta \in \sP(\Lipplus{}\III\DST), $ 
 we  construct  the 
measures 
 $ \frt'   \mathcal{L}^{1}  {\otimes}\eeta$ and 
$ \frxx'   \mathcal{L}^{1}{\otimes} \eeta$ (recall the definition  \eqref{evaluation-derivatives} of  $\frt'$ and $\frxx'$), 
 and show that their push-forwards 
 through the evaluation map $\fre$
 from \eqref{eval-map}, cf.\ \eqref{e.15} below, 
  solve the Cauchy problem   \eqref{newCauchy}.   In fact, the Cauchy condition is  expressed in terms of the  
push-forward $(\fre_0)_{\sharp }\eeta$
 (where $\fre_0$ stands for $\fre (0,\cdot)$):   in this respect, let us specify that, while $\fre_0$ is evaluated
   along curves $\sfyy \in  \Cplus$, so that $\fre_0(\sfyy) = (\sft(0),\sfxx(0)) = (0,\sfxx(0))$,
   in \eqref{newCauchy}
    with slight abuse of notation we will  consider 
 $\fre_0$ 
 as 
 valued in $\R^d$.
 \item Conversely, we prove that any solution of the continuity equation admits the probabilistic representation~\eqref{e.15} below in terms of a probability measure $\eeta$ on the space   $\Lipplus 1\III\DST$ from~\eqref{ACloc}.  
  In addition, we show that 
 for $\eeta$-a.a.\ curve $\sfyy$, the velocity~$ \sfyy' (r) $ at a given time $r\in \III$ does not depend explicitly on~$r$ but only on the position~$\sfyy(r)$, see 
 \eqref{Cauchy-gamma-new} 
  ahead. 
 \end{enumerate}
 
 
\begin{theorem}
\label{t.1}
 The following facts hold: 
\begin{enumerate}
\item
Let   $\eeta \in \sP(\Lipplus{k}\III\DST)$ and 
 $\teeta:=\mathcal L^1\otimes \eeta$  
fulfill
 \begin{equation}
 	\label{eq:condition2}
 	\int \|\fre_0\|\,\dd\eeta<+\infty,\qquad
 	\int_{\Domain{T}}\|\frxx'\|\,\dd\teeta<\infty\quad\text{for every }T>0.
 \end{equation} 
Then, the pair $(\mu,\nnu)\in  \calM^+( \R^{d+1}_{+} ) \times \calM( \DST ; \R^{d})$ defined by
\begin{equation}\label{e.15}
 \mu:= \fre_{\sharp }( \frt' \eeta_{\mathcal{L}} )\qquad \nnu:= \fre_{\sharp }(  \frxx' \eeta_{\mathcal{L}} ) 
\end{equation}
 is a $\Puno$-solution  to 
the continuity equation, 
\begin{equation}
\label{newCauchy}
\partial_{t} \mu + \dive \nnu = \mu_{0} \quad \text{in } (0,+\infty)\times \R^d,  \qquad \text{with } \mu_{0} = (\fre_{0})_{\sharp }\eeta \in \Puno(\R^{d})\,,
\end{equation}
 in the sense of Definition  \ref{def:solCE}.
%
Moreover, 
\begin{equation}
\label{one-sided-est}
|(\mu,\nnu)|\leq \fre_{\sharp }(  \| \fry' \| \eeta_{\mathcal{L}})\,.
\end{equation}
\item
Conversely, 
 let  $(\mu,\nnu)\in \calM^+(\DST ) \times \calM( \DST ; \R^{d})$ be a $\Puno$-solution  to 
the continuity equation in the sense of Definition~\ref{def:solCE}, with initial condition~$\mu_{0} \in \mathcal{P}_{1}(\R^{d})$;
	let $(\mu,\bar\nnu)=\theta (\mu,\nnu)$ be 
	a minimal pair induced by $(\mu,\nnu)$ according to 
	Definition \ref{def:reduced} and \eqref{eq:trivial-but-useful},
	let $\varrho\sim_k |(\mu,\nnu)|$  for some $k>0$,  
	and let $(\tau,\bvv)$ be a  bounded Borel vector field 
	representing the density of $(\mu,\nnu)$ w.r.t.~$\varrho$, namely	
 \begin{equation}
 \label{strict-convexity-consequence-bis}
	\mu=\tau\varrho,\qquad 
	\nnu= \bvv \varrho
	\qquad \text{$\varrho$-a.e.\ in $\DST$}.
 \end{equation}
Then, there exists a measure  $\eeta \in \sP(\Lipplus k\III\DST)$ such that 
the representation 
 \begin{equation}\label{repre-eq-17}
  \mu= \fre_{\sharp }\big(\frt' \, \teeta\big)\qquad 
   \bar{ \nnu} = \fre_{\sharp }\big( \frxx'\,  \,\teeta\big), \qquad
  |(\mu,\bar\nnu)|=\fre_{\sharp } \big( \norm{ \fry' }\, \teeta\big),\qquad
  \varrho=\fre_{\sharp } \, \teeta.
 \end{equation}
 holds, and $\eeta$ is supported on curves solving the Cauchy problem
 \begin{equation}
\label{Cauchy-gamma-new}
\begin{cases}
  \sfyy' \RRR (s)  =   (\tau ({\sfyy} (s)), \bvv ({\sfyy} (s))) & 
 \qquad \foraa\, s \in  (0,+\infty)  
\\
\sfyy(0)=(0,x), & \qquad  x\in\supp(\mu_{0})\,.
\end{cases} 
\end{equation}
\end{enumerate}
 Choosing in particular $\varrho=|(\mu,\nnu)|$ 
we get  $\eeta\in \sP(\Lipplus 1\III\DST)$  is concentrated on $\ALip \III\DST$  and 
\begin{equation}\label{e.17}
|(\mu,\bar\nnu)| = \fre_{\sharp } ( \norm{ \fry' }  \teeta)
=\fre_{\sharp } \teeta\,.
\end{equation}
\end{theorem}

\begin{proof}
	Recalling the definitions of the evaluation maps 
	\eqref{eval-map}, we observe that 
	\begin{equation}
		\label{eq:trivial}
		\fre =\ppi\circ \fra.
	\end{equation}
	Claim (1) follows by Claim (1) of Theorem \ref{p:superposition}
	and Lemma \ref{l:CE-onu}.
	
	Claim (2) follows by the augmented representation of 
	$(\mu,\nnu)$ given in Theorem \ref{thm:augmented},
	and by Claim (2) of Theorem \ref{p:superposition}.
\end{proof}

\subsection{Fine properties of the representing measure}
 With the upcoming Proposition \ref{c:2}  we  unveil some refined properties of the probabilistic representation 
provided by Theorem \ref{t.1}
for the solutions of the continuity equation. 
 Namely,  
we give a finer representation formula for $\mu$ and  specify the probabilistic representation 
of the measures $\nnu^{a} \ll \mu$ and $\nnu^\perp$  featuring in the decomposition  $\nnu = \nnu^a + \nnu^\perp$, cf.\  \eqref{e:new-representation} below. 
To this purpose, we need to introduce some further notation. For every $\sfyy
=(\sft,\sfxx) \in 
\Lipplus {}\III \DST$ we define the sets $D_{+}[\sfyy]$,~$D_{0}[\sfyy]$, $ D_{c} [\sfyy] \subseteq  \III $ as
\begin{equation}
\label{def-sets-Dgamma}
\begin{aligned}
& D_{+}[\sfyy] \coloneq \{ s \in   \III  : \, \sft'(s) >0\}\,,\\
& D_{0}[\sfyy] \coloneq \{ s \in   \III : \, \sft' (s) = 0\}\,,\\
& D_{c}[\sfyy] \coloneq \{ s \in   \III : \, \text{$\sft(\cdot)$ is constant in a neighborhood of~$s$}\}\,,
\end{aligned}
\end{equation}
 where $\sft'$ denotes the upper derivative
defined in \eqref{eq:everywhere-derivative}--\eqref{evaluation-derivatives}. 
Its usage is motivated by the fact that 
 $\sft'$ exists at every $s\in \III$ and it is a bounded nonnegative Borel 
map.
 We further set
\begin{align*}
& D_{+} \coloneq \{ (s, \sfyy) \in \III \times \Lipplus {}\III \DST : \, s \in D_{+}[\sfyy]\}\,,\\
& D_{0} \coloneq \{ (s, \sfyy) \in \III \times \Lipplus {}\III \DST : \, s \in D_{0}[\sfyy]\}\,,\\
&  D_{c} \coloneq \{ (s, \sfyy) \in \III \times \Lipplus {}\III \DST : \, s \in D_{c}[\sfyy]\}\,.
\end{align*}
Since  we have $D_{c} [\sfyy] \subseteq D_{0}[\sfyy]$
 for every $\sfyy$, there holds  $D_{c} \subseteq D_{0}$.
\par
We are now in a position to provide the probabilistic representation of $\nnu^a$ and $\nnu^\perp $  in terms of the measure  $\eeta \in 
\mathcal{P}(\Lipplus{}\III\DST)$ 
featuring in \eqref{repre-eq-17} and~\eqref{e.17}. 
	

\begin{proposition}[{Decomposition property of the probabilistic representation}]
\label{c:2}
 Let $(\mu, \nnu)$ be a $\Puno$ so\-lu\-tion  to   the continuity equation~\eqref{continuity-equation} 
in the sense of Definition~\ref{def:solCE}, 
let $(\mu,\bar\nnu)$ be a minimal solution induced by $(\mu,\nnu)$ 
with Lebesgue decomposition $\bar\nnu=\nnu^a+\bar\nnu^\perp$.
Suppose that the representation formulae \eqref{repre-eq-17} and~\eqref{e.17} hold with 
 $\eeta \in \mathcal{P}(\Lipplus{}\III\DST)$
and $\teeta=\mathcal L^1\otimes \eeta$.
Then, we have that
\begin{gather}
\label{e:new-representation}
\mu = \fre_{\sharp } (\frt' \, \teeta \mres D_{+}) \,,\qquad 
\nnu^{a} = \fre_{\sharp } (\frxx'\, \teeta\mres D_{+})\,, \qquad \bar\nnu^{\perp} = \fre_{\sharp } (\frxx'\,  \teeta \mres D_{0})\,,\\
\label{e:new-representation2}
|(\mu,\nnu^{a})| = \fre_{\sharp } (\|\fry'\|\, \teeta\mres D_{+})\,, \qquad 
|\bar\nnu^{\perp}| = \fre_{\sharp } (\|\frxx'\|\,  \teeta \mres D_{0}).
\end{gather}
\end{proposition}
\begin{proof}
Since  $\mu = \fre_{\sharp } ( \frt' \eeta_{\mathcal{L}})$,  the first equality in~\eqref{e:new-representation} holds. Being $\nnu^{a} \ll \mu$ and $D_{0} \cap  D_{+} = \emptyset$,   the second and the third equalities in~\eqref{e:new-representation} follow. 

Concerning \eqref{e:new-representation2}, we have 
	\begin{align*}
		|(\mu,\nnu)|&=
		|(\mu,\nnu^a)|+|\nnu^\perp|
		\le 
		\fre_{\sharp } (\|\fry'\|\, \teeta\mres D_{+})
		+\fre_{\sharp } (\|\fry'\|\,  \teeta \mres D_{0})=
		\fre_{\sharp } (\|\fry'\|\, \teeta)=|(\mu,\nnu)|
	\end{align*}
	so that all inequalities are in fact equalities and \eqref{e:new-representation2} follows.
\end{proof}

 With the next result, 
we use the representation formulae 
\eqref{e:new-representation}
provided by Proposition 
\ref{c:2}, to show that the superposition measure~$ \eeta \in \mathcal{P}( \Lipplus{} \III { \R^{d+1}})$ is supported on injective curves. In the proof, 
we will identify the pair $(\mu, \nnu) $  with an associated mininimal solution $(\mu, \bar\nnu) $. 
\begin{theorem}\label{t:injective-curves}
Let $(\mu, \nnu) $ be a  minimal  $\Puno$-solution  to  the continuity equation~\eqref{continuity-equation} 
and let the measure $\eeta \in \mathcal{P}(\Lipplus{k}\III\DST)$ provide the representation  formulae~\eqref{repre-eq-17}. 
Then, $\eeta$-a.e.~curve 
$\sfyy \in \Lipplus{}\III\DST$ is injective.
\end{theorem}
\begin{proof}
We prove the theorem by contradiction. 
Let us assume that there exist
 $S>0$ and 
 a set
 \begin{equation}
 \label{Lambda-set-added}
 \Lambda \subseteq \Lipplus{}\III\DST \text{ with~$\eeta(\Lambda)>0$ s.t.\
	every $\sfyy \in \Lambda$ is not injective in the interval $[0,S]$}.
  \end{equation}
Since $\eeta$ is a Radon measure, it is not restrictive to assume that 
$\Lambda$ is compact.

 For $\sfyy \in \Lipplus{}\III\DST$  and $s \in [0, S]$   we define  
\begin{align*}
r(s,\sfyy) := \max \, \{ \sigma \in [0, S]: \,   \sfyy(\sigma) = \sfyy(s)  \}\,, \qquad  
 r_{\flat}(s,\sfyy)  &:=  r(s,\sfyy) - s\,, 
\quad 
 \frr (\sfyy):=\max_{s\in [0,S]}  r_{\flat}(s,\sfyy), 
\end{align*}
and we notice that the maps $(s, \sfyy) \mapsto   r_{\flat}(s,\sfyy) $ 
and $\sfyy\mapsto \frr(\sfyy)$ are upper semicontinuous
in $[0,S]\times \Lambda$ and in $\Lambda$ respectively. 
Since $\frr(\sfyy)>0$ for every $\sfyy\in \Lambda$, 
we may   find  
 $h \in \mathbb{N}$  such that the (closed, thus compact) set 
\begin{displaymath}
 \Lambda' := \Big\{ \sfyy \in \Lambda: \, 
  \frr(\sfyy) \geq  \frac{1}{h}  \Big\} 
\end{displaymath}
 satisfies $\eeta(\Lambda') >0$. 
 For every $\sfyy\in \Lambda'$ we now define 
 \begin{displaymath}
 	\Xi(\sfyy):=\Big \{s\in [0,S]:\,   r_{\flat}(s,\sfyy)  \ge 1/h\Big 
 	\},\quad
 	\sigmasel_{\mathrm{l}}(\sfyy):=\min \Xi(\sfyy),
 \end{displaymath}
where $\mathrm{l}$ stands for `lower'.  
 Notice  that $\sigmasel_{\mathrm{l}}$ is well defined since $\Xi(\sfyy)$ is a closed 
 nonempty subset of $[0,S]$; one can easily check that is lower semicontinuous: 
 if $(\sfyy_n)_n\subset  \Lambda'$ converges to $\sfyy$ and 
 $s_n= \sigmasel_{\mathrm{l}}(\sfyy_n)$ is converging (up to the extraction of a subsequence) to $s$, 
 we know that $ r_{\flat}(s,\sfyy) \ge 1/h$ (by the upper semicontinuity of $r_{\flat}$) 
 so that $s\in \Xi(\sfyy)$ and $\sigmasel_{\mathrm{l}}(\sfyy)\le s$.
 
We set $\sigmasel_{\mathrm{u}}(\sfyy):= {r}(\sigmasel_{\mathrm{l}}(\sfyy),\sfyy)$
(with $\mathrm{u}$ for `upper'); 
the function~$\sfyy \mapsto \big( \sigmasel_{\mathrm{l}}(\sfyy), \sigmasel_{\mathrm{u}}(\sfyy)\big) $ is Borel measurable and 
for every $\sfyy \in \Lambda'$ we have 
by construction
\begin{equation}
\label{bigger-than-1/k}
\sfyy ( \sigmasel_{\mathrm{l}}(\sfyy)) = \sfyy ( \sigmasel_{\mathrm{u}}(\sfyy)),
\qquad \sigmasel_{\mathrm{u}}(\sfyy)-\sigmasel_{\mathrm{l}}(\sfyy)=
r_{\flat}(\sigmasel_{\mathrm{l}}(\sfyy),\sfyy)    \geq \frac{1}{h}. 
\end{equation}
Since $\sft$ is non-decreasing we conclude that 
$\sft$ is constant in $(\sigmasel_{\mathrm{l}}(\sfyy),\sigmasel_{\mathrm{u}}(\sfyy))$, hence for the interval $ ( \sigmasel_{\mathrm{l}}(\sfyy), \,\sigmasel_{\mathrm{u}}(\sfyy)  ) $ we have 
\begin{equation}
\label{e:equality}
 ( \sigmasel_{\mathrm{l}}(\sfyy), \,\sigmasel_{\mathrm{u}}(\sfyy)  )  \subseteq   D_{c}[\sfyy] \subseteq D_{0}[\sfyy]\,.
\end{equation}
We introduce the function $\T \colon \Lipplus{}\III\DST \to
\Lipplus{}\III\DST $ defined by $\T(\sfyy) = \sfyy$ for $\sfyy \notin \Lambda'$ and
\begin{displaymath}
\T(\sfyy) (s) = \left\{ \begin{array}{ll}
\sfyy(s) & \text{for $s \leq \sigmasel_{\mathrm{l}}(\sfyy)$}\,,\\
\sfyy(s +   r_{\flat}(\sigmasel_{\mathrm{l}}(\sfyy),\sfyy) ) & \text{for $s > \sigmasel_{\mathrm{l}}(\sfyy)$}
\end{array} \right.
\end{displaymath}
for $\sfyy \in \Lambda'$.
 Now, as previously mentioned, the mapping $\sfyy \mapsto \sigmasel_{\mathrm{l}}(\sfyy)$ is Borel. Since 
$(s, \sfyy) \mapsto   r_{\flat}(s,\sfyy) $  is itself upper semicontinuous, the mapping $\sfyy \mapsto \mathcal{R}(\sfyy)$ with 
$\mathcal{R}(\sfyy)(s): = \sfyy(s +   r_{\flat}(\sigmasel_{\mathrm{l}}(\sfyy),\sfyy) )$ is Borel. Rewriting $\T$ in terms of
$\mathcal{R}$ and of
the characteristic functions of the intervals  $(0,\sigmasel_{\mathrm{l}}(\sfyy)]$ and $(\sigmasel_{\mathrm{l}}(\sfyy),{+}\infty)$, i.e.,
\[
\T(\sfyy) =  \boldsymbol{1}_{(0,\sigmasel_{\mathrm{l}}(\sfyy)]} \sfyy + \boldsymbol{1}_{(\sigmasel_{\mathrm{l}}(\sfyy), +\infty)}\mathcal{R}(\sfyy),
\]
reveals that 
 the map $\T$ is Borel measurable, too. Moreover, we readily check that \EEE   $\T(\sfyy) \in \Lipplus k \III {\R^{d+1}}$ for every $\sfyy \in \Lipplus k \III {\R^{d+1}}$,  so that the push-forward $\etaflat := \T_{\sharp } \eeta$ belongs to $
	\mathcal{P}( \Lipplus{k}\III\DST)$.

\par
 After these preliminary definitions, we are in a position to carry out the contradiction, which will 
be essentially based on the fact that the measure~$\etaflat$ also provides a probabilistic representation of the pair~$(\mu, \nnu)$, cf.\ the upcoming Claims $1$, $2$,  and  $3$. For later use, 
let us set $\nnu_{\flat} := \fre_{\sharp } (\frxx' \mathcal{L}^{1} \otimes \etaflat)$. Let us write $\nnu = \nnu^{a} + \nnu^{\perp}$ and
 $\nnu_{\flat}=\nnu_{\flat}^{a} + \nnu_{\flat}^{\perp}$, where $\nnu^{a}, \nnu_{\flat}^{a} \ll \mu$ and $\nnu^{\perp}, \nnu_{\flat}^{\perp} \perp \mu$. 
\par
\noindent
\textbf{Claim $1$:} \emph{we have} 
\begin{equation}
\label{repre-mu-bareta}
\mu = \fre_{\sharp } (\frt' \mathcal{L}^{1} \otimes \etaflat).
\end{equation}
This follows from the finer representation of~$\mu$ provided by~\eqref{e:new-representation}, and by the
definition of~$\T$. Indeed, using  the place-holder 
$\mu_{\flat}: = \fre_{\sharp } (\frt' \mathcal{L}^{1} \otimes \etaflat)$,
for every $\phisc \in  \Cc (\DST ) $ we have
%
  \begin{equation}
  \label{e:difference-nu}
  \begin{aligned}
  &
  \int_{\DST } \phisc (t, x) \, \di (\mu - \mu_{\flat}) (t, x) 
\\
&  = \int_{\Lambda'} \int_{\III}  \phisc(\sfyy(s))  \sft'(s) \, \di s \, \di \eeta(\sfyy)  - \int_{\Lambda'} \int_{\III}  \phisc(\sfyy(s))  \sft'(s) \, \di s \, \di \eeta_{\flat}(\sfyy) 
\\
& 
 \stackrel{(1)}{=}
  \int_{\Lambda'} \int_{\sigmasel_{\mathrm{l}}(\sfyy)}^{+\infty}  \phisc(\sfyy(s))  \sft'(s) \, \di s \, \di \eeta(\sfyy)  -  \int_{\Lambda'} \int_{\sigmasel_{\mathrm{l}}(\sfyy)}^{+\infty}  \phisc(\sfyy(s{+} r_{\flat} (\sigmasel_{\mathrm{l}}(\sfyy),\sfyy) )
    \sft'(s{+}  r_{\flat} (\sigmasel_{\mathrm{l}}(\sfyy),\sfyy)) \, \di s \, \di \eeta(\sfyy)
 \\
 &
  \stackrel{(2)}{=}
  \int_{\Lambda'} \int_{\sigmasel_{\mathrm{u}}(\sfyy)}^{+\infty}  \phisc(\sfyy(s)) \sft'(s) \, \di s \, \di \eeta(\sfyy)  - 
   \int_{\Lambda'} \int_{ \sigmasel_{\mathrm{u}}(\sfyy) }^{+\infty}  \phisc( \sfyy( \tau)  )
  \sft'(\tau)  \, \di \tau \, \di \eeta(\sfyy)   = 0\,.
   \end{aligned}
  \end{equation}  
 In the above chain of equalities, 
   (1) follows from the fact that $\T$ is the identity in the complement of $\Lambda'$ and  modifies a curve $\sfyy \in \Lambda'$ only on  
  $( \sigmasel_{\mathrm{l}}(\sfyy), +\infty)$, while (2)  ensues
  from the fact that 
   $\sft'\equiv 0$  on $( \sigmasel_{\mathrm{l}}(\sfyy),  \sigmasel_{\mathrm{u}}(\sfyy))$ 
   and from a change of variable in the second integral.
  \\
  \textbf{Claim $2$:} \emph{we have} 
 \begin{equation}
\label{repre-nua-bareta} 
\nnu^{a} = \nnu_{\flat}^{a}\,.
\end{equation}
 With the very same arguments as in \eqref{e:difference-nu} we check that  
\begin{equation}
\label{e:difference}
\begin{aligned}
& 
\int_{\DST } \bvarphi (t, x) \, \di (\nnu - \nnu_{\flat}) (t, x)
= 
\int_{\Lambda'}\Big(\int_{\sigmasel_{\mathrm{l}}(\sfyy)}^{\sigmasel_{\mathrm{u}}(\sfyy)} \bvarphi(\sfyy(s)) {\, \cdot\,} \boldsymbol{\mathsf{x}}'(s) \, \di s \Big)
\, \di \eeta(\sfyy) 
\end{aligned}
\end{equation} 
 for every $\bvarphi \in {\rm C}_{\rm c} (\DST ; \R^{d})$.
%
By~\eqref{e:new-representation} and by~\eqref{e:equality} we therefore have that $(\nnu {-} \nnu_{\flat} )\perp \mu$ and $\nnu^{a} = \nnu_{\flat}^{a}$.
Thus, $\nnu_\flat = \nnu^{a} + \nnu_{\flat}^{\perp}$. 
 \\
 \textbf{Claim $3$:} \emph{we have}
 \begin{equation}
\label{repre-nu-perp-bareta} 
\dive \nnu_\flat^{\perp} = \dive \nnu^{\perp}
\,.
\end{equation}
Denoting $\btheta \coloneq \nnu^{\perp} - \nnu_\flat^{\perp}=
\nnu-\nnu_\flat$, it follows from~\eqref{e:difference} and Theorem~\ref{t.1}(2),
that for every test function $\bvarphi \in {\rm C}_{\rm c} (\DST ; \R^{d})$ there holds
\begin{align}
\label{e:difference-2}
\int_{\DST } \bvarphi (t, x) \, \di \btheta (t, x) & = \int_{\Lambda'} \int_{\sigmasel_{\mathrm{l}}(\sfyy)}^{ \sigmasel_{\mathrm{u}}(\sfyy)} \bvarphi(\sfyy(s)) {\, \cdot\,} \boldsymbol{\mathsf{x}}'(s) \, \di s \, \di \eeta(\sfyy),
\end{align} 
i.e.
\begin{equation}
	\label{e:difference3}
	\btheta=\fre_\sharp \big(\frxx'\,\teeta\mres \Theta\big),\quad
	\Theta:=\Big\{(s,\sfyy):\sfyy\in \Lambda',\ \sigmasel_{\mathrm{l}}(\sfyy)\le 
	s\le\sigmasel_{\mathrm{u}}(\sfyy)\Big\}.
\end{equation}
It is sufficient to select $\bvarphi=\rmD \varphi$  for $\varphi \in {\rm C}^{1}_{\rm c} (\R^{d+1}_{+})$  in \eqref{e:difference-2}
and to observe that for every $\sfyy\in \Lambda'$ 
the inner integral 
\begin{align*}
	\int_{\sigmasel_{\mathrm{l}}(\sfyy)}^{ \sigmasel_{\mathrm{u}}(\sfyy)} \rmD\varphi(\sfyy(s)) {\, \cdot\,} \boldsymbol{\mathsf{x}}'(s) \, \di s 
	=\int_{\sigmasel_{\mathrm{l}}(\sfyy)}^{ \sigmasel_{\mathrm{u}}(\sfyy)} \rmD\varphi(t,\sfxx(s)) {\, \cdot\,} \boldsymbol{\mathsf{x}}'(s) \, \di s 
	=\varphi(t,\sfxx(\sigmasel_{\mathrm{u}}(\sfyy)))-
	\varphi(t,\sfxx(\sigmasel_{\mathrm{l}}(\sfyy)))=0
\end{align*}
since $\sft(s)\equiv t$ is constant in the interval 
$(\sigmasel_{\mathrm{l}}(\sfyy),\sigmasel_{\mathrm{l}}(\sfyy))$. Plugging this formula in \eqref{e:difference-2},
with $\bvarphi=\rmD \varphi$ for an arbitrary $\varphi\in \rmC^1_c(\DST)$, and  we get $\dive \btheta=0$.
%
%
%
%
%
 \\
 \textbf{Claim $4$:} \emph{we have}
 \begin{equation}
\label{repre-nu-perp-bareta2} 
\nnu_\flat^{\perp} \prec \nnu^{\perp}
\,.
\end{equation}
Let us define 
$\eeta_\flat^0:= 
\etaflat{\mres} D_0$,
$\boldsymbol\vartheta=\frxx' \eeta_\flat^0=
(\bvv{\circ}\fre ) \eeta_\flat^0$; let us consider
the Borel function 
\begin{displaymath}
	\lambda(s,\sfyy):=
	\begin{cases}
		0&\text{if }(s,\sfyy)\in \Theta,\\
		1&\text{otherwise.}
	\end{cases}
\end{displaymath}
and the measure $\zzeta=\lambda \boldsymbol\vartheta$. The representation formula 
\eqref{e:new-representation} yields
$\nnu^\perp=
\fre_\sharp \boldsymbol\vartheta$,
whereas 
\eqref{e:difference-2} yields
$\nnu_\flat^\perp =\fre_\sharp \zzeta$.
Since $\zzeta\prec \boldsymbol\vartheta$
and $\boldsymbol\vartheta$ satisfies 
\eqref{eq:fiber2} (w.r.t.~the measure 
$\alpha:=\etaflat^0$ and 
the map $ \pushm:=\fre $), we infer from Lemma \ref{le:strict-conv}  ahead that 
$\nnu_\flat^\perp\prec \nnu^\perp.$
%
%
 \\
 \textbf{Conclusion of the proof:}
 Since $\nnu$ is minimal, we deduce
 that $\nnu_\flat^\perp=\nnu^\perp$; since
 $\|\bvv\|\ge 1/k$ $\etaflat^0$-a.e., we 
have $\lambda \equiv 1$ a.e.~in $\Theta$, i.e.~$\etaflat(\Theta)=0$.
 This implies the for $\etaflat$-a.e.~$\sfyy$
 $\sigmasel_{\mathrm{l}}(\sfyy)=\sigmasel_{\mathrm{u}}(\sfyy)$,
 a contradiction.

\end{proof}
The next result provides another property for any measure~$\eeta
	\in \sP( 
	\Lipplus1\III\DST)$ representing the solutions  to  the continuity equation, 
	 when the pair~$(\mu, \nnu)$ satisfies  condition~\eqref{UEST}
	 (in which case,~$\nnu$ is minimal, cf.\ Theorem \ref{th:1}).   In this situation, we show that 
	the measure $\eeta$ is concentrated on curves~$\sfyy= (\sft, \sfxx) \in  \Lipplus 1 \III {\R^{d+1}} $ 
  given by 
segments on intervals where~$\sft' \equiv 0$.  The proof mimics the contradiction argument
carried out for Theorem \ref{t:injective-curves}.
\begin{theorem}
\label{t:segment-representation}
 Let $(\mu, \nnu) \in \calM^{+} (\DST ) \times \calM(\DST   ; \R^{d} )$ be a  minimal 
 $\Puno$-solution of the continuity equation~\eqref{continuity-equation} with initial datum~$\mu_{0} \in \mathcal{P}_{1}(\R^{d})$, in the sense of 
  Definition~\ref{def:reduced}. 
 Suppose that $(\mu, \nnu)$ comply with \eqref{UEST}.  
 \par
Let $\eeta \in \mathcal{P}(\Lipplus1\III\DST)$ 
satisfy the representation formulae~\eqref{e.15} and~\eqref{e.17}.
Then, 
$\eeta$-almost every  curve $\sfyy \in \Lipplus1\III\DST$
 enjoys the following property:
\begin{equation}
\label{e:segment-x}
\begin{gathered}
\text{if~$s_{1} < s_{2} \in [0, +\infty)$ are such that $\sft' \equiv 0$ in~$(s_{1}, s_{2})$, then}
\\
\sfxx(s) = \sfxx(s_{1}) + (s - s_{1}) \frac{\sfxx(s_{2}) - \sfxx(s_{1}) }{\| \sfxx(s_{2}) - \sfxx(s_{1}) \|} \qquad \text{for all } s \in [s_1,s_2].  
\end{gathered}
\end{equation}
\end{theorem}


\begin{proof}[Proof of Theorem~\ref{t:segment-representation}]
Assume by contradiction that the thesis is false. Then,  there exist
$S>0$ and 
$\Lambda \subseteq   \Lipplus{1}\III\DST$ 
 such that $\eeta (\Lambda) >0$ and for every $\sfyy= (\sft, \sfxx) \in \Lambda$ there exist $s_{1}<s_{2} \in [0, S]$  such that $\sft(s_{1}) = \sft(s_{2})$ and the the restriction of~$\sfxx$ to the interval $ [s_{1}, s_{2}]$ is not of the form~\eqref{e:segment-x}.
 In particular, since  $\sfyy$ is $1$-Lipschitz continuous, 
 we have that 
 $\| \sfxx(s_{1}) - \sfxx(s_{2})\| < s_{2} - s_{1}$.  Arguing similarly as in  the proof of Theorem~\ref{t:injective-curves},  for   $\sfyy \in   \Lipplus{1}\III\DST$ 
  and $s \in [0, S]$  we set 
\begin{equation*}
 \rho(s, \sfyy)  := \sup \, \{ \sigma \in [0, S]: \,   \sft(\sigma) = \sft(s)  \}\,, \qquad   \rflat(s, \sfyy) :=  \rho(s, \sfyy) - s\,. 
\end{equation*}
Then, the maps $(s, \sfyy) \mapsto \rho(s, \sfyy)$ and $(s, \sfyy) \mapsto  \rflat(s, \sfyy) $ are upper semicontinuous. 

Let us show that also the map 
\begin{equation}
\label{e:upp-semicont}
(s, \sfyy) \mapsto  \rflat(s, \sfyy) - \| \sfxx (s) - \sfxx( \rho(s, \sfyy)) \|
\end{equation}
is upper semicontinuous. Let $s_{n} \to s$ and $\sfyy_{n} \to \sfyy$ uniformly on compact  subsets  of~$[0, +\infty)$, and assume that $(s_{n})_n, \, s \in [0, S]$. 
We need to show that 
\begin{equation}
\label{usc-bis}
\limsup_{n\to\infty} \left(  \rflat(s_n, \sfyy_n)  - \| \sfxx_n (s_n) - \sfxx_n( \rho (s_n,\sfyy_n)) \| \right)  \leq  \rflat(s, \sfyy) - \| \sfxx (s) - \sfxx( \rho(s, \sfyy)) \|\,.
\end{equation} 
 Up to a not relabeled subsequence, we may assume that
\begin{align*}
& \overline{s}  = \lim_{n\to \infty} \,  \rho (s_n,\sfyy_n) = \limsup_{n \to \infty} \,  \rho (s_n,\sfyy_n) \leq \rho(s, \sfyy)\,,\\
&\limsup_{n \to \infty}  \|  \sfxx_{n} (s_{n}) - \sfxx_{n} ( \rho (s_n,\sfyy_n)) \|=  \lim_{n\to \infty} \|  \sfxx_{n} (s_{n}) - \sfxx_{n} ( \rho (s_n,\sfyy_n)) \| = \| \sfxx(s) - \sfxx(\overline{s})\| \,.
\end{align*}
Hence, we deduce that
\begin{equation}
\label{e:s-overline}
(\overline{s} - s) - \| \sfxx(s) - \sfxx(\overline{s})\| = \lim_{n\to\infty} \big(  \rflat(s_{n},\sfyy_{n}) -  \|  \sfxx_{n} (s_{n}) - \sfxx_{n} ( \rho (s_n,\sfyy_n)) \| \big)\,.
\end{equation}
If $\overline{s} = \rho(s,\sfyy)$, equality~\eqref{e:s-overline} proves the upper semicontinuity. If $\overline{s} < \rho(s,\sfyy)$, since
  $\sfyy\in    \Lipplus{1}\III\DST$ 
  and $\sft' \equiv 0$ in $(\overline{s}, \rho(s,\sfyy))$, we have that 
\begin{equation}
\label{added-label}
\| \sfxx(\overline{s}) - \sfxx( \rho(s,\sfyy))\| \leq  \rho(s,\sfyy) - \overline{s}.
\end{equation}
 Therefore, we deduce from~\eqref{e:s-overline}  and \eqref{added-label} 
  that
\begin{align*}
\rflat(s,\sfyy) - \| \sfxx(s) - \sfxx( \rho(s,\sfyy)) \|&  \geq (\rho(s,\sfyy) - \overline{s}) + (\overline{s} - s) - \| \sfxx(s) - \sfxx(\overline{s}) \| - \| \sfxx(\overline{s}) - \sfxx( \rho(s,\sfyy))\|
\\
&
\geq (\overline{s} - s) - \| \sfxx(s) - \sfxx(\overline{s}) \| 
\\
&
= \limsup_{n \to \infty} \big( \rflat(s_{n},\sfyy_{n}) - \| \sfxx_{n} (s_{n}) - \sfxx_{n} (\rho (s_{n},\sfyy_{n}) ) \|\big),
\end{align*}
 whence \eqref{usc-bis}.

Since $\eeta(\Lambda)>0$,  there exists $k \in \mathbb{N}$ such that the set 
\begin{displaymath}
\Lambda' := \Big\{ \sfyy \in    \Lipplus{1}\III\DST   \,: \   \max_{s \in [0, S]}  \big( \rflat(s,\sfyy) -  \| \sfxx (s) - \sfxx( \rho(s, \sfyy)) \| \big) \geq  \frac{1}{k}  \Big\}
\end{displaymath}
is Borel measurable and satisfies $\eeta(\Lambda') >0$. We define the multifunction $\Delta \colon \Lambda' \to 2^{[0, S]}$ by $\Delta (\sfyy) := \{ s \in [0, S] : \, \rflat (s,\sfyy) - \| \sfxx(s) - \sfxx(\rho(s, \sfyy) \|  \geq \frac{1}{k}\}$. Since the map~\eqref{e:upp-semicont} is upper semicontinuous, the multifunction~$\Delta$ is upper semicontinuous. Thanks  to~\cite[Theorem~III.6]{Castaing-Valadier77}  we find a Borel measurable selection of~$\Delta$, namely, a Borel function $\selbis \colon \Lambda' \to [0, S]$ such that $\selbis(\sfyy) \in \Delta(\sfyy)$ for every $\sfyy \in \Lambda'$. By construction, we have that
\begin{equation}
\label{3rd-added-label-th6.3}
\rflat(\selbis(\sfyy),\sfyy) - \| \sfxx(\selbis(\sfyy)) - \sfxx( \rho(\selbis(\sfyy),\sfyy))\|  \geq  \frac{1}{k}\,, 
\end{equation}
and the function $\sfyy \mapsto (\selbis(\sfyy) , \rho(\selbis(\sfyy),\sfyy))$ is Borel measurable. Moreover, for every $\sfyy \in \Lambda'$ we have
 $\sft ( \selbis(\sfyy)) = \sft (\rho (\selbis(\sfyy),\sfyy))$ and $\sft'(s) \equiv0 $ for every $ s$ in the interval  $(\selbis(\sfyy), \rho (\selbis(\sfyy),\sfyy))$. Hence, 
\begin{equation}
\label{e:equality-2}
(\selbis(\sfyy), \rho (\selbis(\sfyy),\sfyy)) \subseteq   D_{c}[\sfyy] \subseteq D_{0}[\sfyy]\,.
\end{equation}
For $\sfyy \in \Lambda'$ we further define 
\[
 \selteta (\sfyy) := \selbis(\sfyy) + \| \sfxx(\selbis(\sfyy)) - \sfxx(\rho (\selbis(\sfyy),\sfyy) )\|
 \]
  and $\sfxx_{\sfyy} \colon [\selbis(\sfyy) , \selteta(\sfyy) ] \to \R^{d}$ as the segment
\begin{equation}
\label{another-added-label-6}
\sfxx_{\sfyy} (s) := \sfxx(\selbis(\sfyy)) + (s - \selbis(\sfyy)) \frac{  \sfxx(\rho (\selbis(\sfyy),\sfyy) ) - \sfxx(\selbis(\sfyy))}{\|  \sfxx(\rho (\selbis(\sfyy),\sfyy) ) - \sfxx(\selbis(\sfyy))\|} \qquad \text{for $s \in  [\selbis(\sfyy) , \selteta(\sfyy) ] $}\,.
\end{equation}

We introduce the function   $\G \colon   \Lipplus{1}\III\DST \to   \Lipplus{1}\III\DST $   defined by $\G(\sfyy): = \sfyy$ for $\sfyy \notin \Lambda'$ and, for $\sfyy \in \Lambda'$, we set 
\begin{displaymath}
\G(\sfyy) (s) := \left\{ \begin{array}{lll}
\sfyy(s) & \text{for $s < \selbis(\sfyy)$}\,,\\[1mm]
\sfzz(s) & \text{for $s \in [\selbis(\sfyy), \selteta(\sfyy) ]$}\,,\\[1mm]
\sfyy(s - \selteta(\sfyy) +  \rho (\selbis(\sfyy),\sfyy) ) & \text{for $s >\selteta(\sfyy) $}
\end{array} \right. \qquad \text{with } \sfzz(s) = (\sft(\selbis(\sfyy))) , \sfxx_{\sfyy}(s) )  
\end{displaymath}
for $\sfyy \in \Lambda'$. By construction, the map $\GG$ is Borel measurable, so that the push-forward  $\etaflat := \GG_{\sharp } \eeta$   belongs to
$\mathcal{P}(  \Lipplus{1}\III\DST)$. 
 %
\par
 After these preparations, we are in a position to carry  out  the proof by contradiction, based on the  properties of the pair
\[
 \mu_{\flat} := \fre_{\sharp } (\frt'  \mathcal{L}^{1} \otimes \etaflat), \qquad \nnu_\flat := \fre_{\sharp } (\frxx'  \mathcal{L}^{1} \otimes \etaflat) 
\]
stated in the following Claims $1,\, 2, $
 and $3$. 
 \\
 \textbf{Claim $1$:} 
 \emph{we have that}
 \begin{equation}
 \label{claim1-th6.3}
 \mu = \mu_\flat\,.
 \end{equation}
This follows from the definition of~$\GG$ (in particular, the fact that $\GG$ is the identity in $    \Lipplus{1}\III\DST \setminus \Lambda'$ and modifies curves in~$\Lambda'$ only where~$\sft' \equiv 0$, cf.\ \eqref{e:equality-2}), also taking into account the representation of $\mu$ provided by Proposition 
\ref{c:2}.
 \\
 \textbf{Claim $2$:} 
 \emph{the pair $(\mu, \nnu_\flat)$ solves the continuity equation~\eqref{continuity-equation} with initial datum~$\mu_{0} \in \mathcal{P}_{1}(\R^{d})$, in the sense of Definition~\ref{def:solCE}.}
For this, it is enough to show that $\dive (\nnu - \nnu_\flat) = 0$ in~$\DST $. Since $\GG$ is the identity map in~$  \Lipplus{1}\III\DST  \setminus \Lambda'$ and modifies a curve~$\sfyy \in \Lambda'$ only on~$(\selbis(\sfyy) , \rho(\selbis(\sfyy),\sfyy))$,  for every $\varphi \in {\rm C}_{\rm c} (\DST )$ we have that
\begin{align*}
\int_{\III } \int_{\R^{d}} \rmD \varphi (t, x) \, \di  (\nnu - \nnu_\flat) (t, x) & = \int_{\Lambda'} \int_{\selbis(\sfyy)}^{\rho(\selbis(\sfyy),\sfyy)} \!\!\!\! \rmD \varphi (  \sft  (\selbis(\sfyy)), \sfxx(s)) {\, \cdot\,} \sfxx'(s) \, \di s \, \di \eeta(\sfyy) 
\\
&
\qquad - \int_{\Lambda'} \int_{\selbis(\sfyy)}^{\selteta(\sfyy) }\!\!\! \rmD \varphi(\sft(\selbis(\sfyy)) , \sfxx_{\sfyy} (s)) {\, \cdot\,}\frac{\sfxx(\rho(\selbis(\sfyy),\sfyy)) 
- \sfxx(\selbis(\sfyy))}{\| \sfxx(\rho(\selbis(\sfyy),\sfyy)) - \sfxx(\selbis(\sfyy))\|} \, \di s \, \di \eeta(\sfyy)
\\
&
\stackrel{(\star)}{=}
\int_{\Lambda'} \big( \varphi(\sft(\selbis(\sfyy)) , \sfxx(\rho (\selbis(\sfyy),\sfyy))) - \varphi(\sft(\selbis(\sfyy)) , \sfxx( \selbis(\sfyy))) \big) \, \di \eeta(\sfyy) 
\\
&
\qquad - \int_{\Lambda'} \big( \varphi(\sft(\selbis(\sfyy)) , \sfxx_{\sfyy}(\selteta(\sfyy)))  - \varphi(\sft(\selbis(\sfyy)) , \sfxx_{\sfyy}( \selbis(\sfyy))) \big) \, \di \eeta(\sfyy) = 0
\end{align*}
by definition of~$\sfxx_{\sfyy}$ for $\sfyy \in \Lambda'$, and with $(\star)$ due to the chain rule and to \eqref{e:equality-2}.
 Hence, $\dive ( \nnu - \nnu_\flat) = 0$ as desired.
\par
Since $\eeta(\Lambda')>0$ and $\Lambda' = \bigcup_{T >0} \{ \sfyy \in \Lambda': \, \sft(\selbis(\sfyy)) < T\}$ with
\begin{displaymath}
\{ \sfyy \in \Lambda': \, \sft(\selbis(\sfyy)) < T_{1}\} \subseteq \{ \sfyy \in \Lambda': \, \sft(\selbis(\sfyy)) < T_{2}\} \qquad \text{if $T_{1} \leq T_{2}$}\,,
\end{displaymath} 
we find $T \in [0, +\infty)$ such that 
\begin{equation}
\label{e:positive-eta}
\eeta( \{ \sfyy \in \Lambda': \, \sft(\selbis(\sfyy)) < T\}) >0\,.
\end{equation}
 \textbf{Claim $3$:} 
\emph{we have that} 
\begin{equation}
\label{e:T-positive-eta}
| \nnu| ([0, T) \times \R^{d}) > | \nnu_\flat| ([0, T) \times \R^{d})\,.
\end{equation} 
Indeed, recalling the representation~$ \nnu = \fre_{\sharp } ( \frxx' \eeta_{\mathcal{L}})$ and
the fact 
 that the curves~$\sfyy \in \spt (\eeta)$ solve the Cauchy problem~\eqref{Cauchy-gamma} with velocity field $(\tau, \bvv)$ independent of~$s$, we rewrite the left-hand side of~\eqref{e:T-positive-eta} as
\begin{equation}
\label{e:|nu|}
\begin{aligned}
| \nnu | ([0, T) \times \R^{d}) & = \int_{\Lipplusname{1}} \int_{\{s \in \III \,: \, \sft(s) < T \}} \| \bvv  (\sft(s) , \sfxx(s)) \| \, \di s \,\di \eeta(\sfyy)
\\
&
=  \int_{\Lipplusname{1}\backslash\Lambda'} \int_{\{s \in  \III \, : \, \sft(s) < T\} } \| \bvv  (\sft(s) , \sfxx(s)) \| \, \di s \,\di \eeta(\sfyy)  
\\
&
\qquad + \int_{\Lambda'} \int_{\{s \in [0, \selbis(\sfyy)) : \, \sft(s) < T\}} \| \bvv  (\sft(s) , \sfxx(s)) \| \, \di s \,\di \eeta(\sfyy)
\\
&
\qquad +  \int_{\Lambda'} \int_{\{s \in [ \selbis(\sfyy), \rho  (\selbis(\sfyy),\sfyy) ] : \, \sft(s) < T\}} \| \bvv  (\sft(s) , \sfxx(s)) \| \, \di s \, \di \eeta(\sfyy)
\\
&
\qquad + \int_{\Lambda'} \int_{\{s \in ( \rho(\selbis(\sfyy),\sfyy) , +\infty) : \, \sft(s) < T\}} \| \bvv  (\sft(s) , \sfxx(s)) \| \, \di s \, \di \eeta(\sfyy)\,.
\end{aligned}
\end{equation}
Since the  pair~$(\tau, \bvv)$   fulfills $\| ( \tau, \bvv)(t, x)\|  \equiv 1$ $| (\mu, \nnu)|$-almost everywhere~in~$\DST $, 
$\eeta$-a.e.~curve  $\sfyy \in   \Lipplus{1}\III\DST$    satisfies 
\[
\| \sfyy'(s)\| = \| (\sft'(s), \sfxx'(s)) \| =  \|(\tau(\sfyy(s)), \bvv(\sfyy(s)))\| = 1 \qquad \text{for a.a.~$s \in \III  $.}
\]
 In particular, since $\sft' \equiv 0$ in $(\selbis(\sfyy), \rho(\selbis(\sfyy),\sfyy))$ for $\eeta$-a.a.~$\sfyy \in \Lambda'$, it must be $\| \bvv(\sfyy(s))\| \equiv 1$ for a.a.~$s \in (\selbis(\sfyy), \rho(\selbis(\sfyy),\sfyy))$, for $\eeta$-a.e.~$\sfyy \in \Lambda'$.  Moreover, if $\sft(\selbis(\sfyy)) <T$, then $\sft(s) <T$ for every 
 $s \in [\selbis(\sfyy), \rho(\selbis(\sfyy),\sfyy)]$. Thus, by construction of~$\GG$ and by~\eqref{e:positive-eta} we may continue in~\eqref{e:|nu|} with
\begin{align*}
| \nnu | ([0, T) \times \R^{d}) & \geq \int_{\Lipplusname{1}\backslash\Lambda'} \int_{\{s \in  \III\,  : \, \sft(s) < T\} } \| \bvv  (\GG(\sfyy)(s)) \| \, \di s \,\di \eeta(\sfyy)
\\
&
\qquad + \int_{\Lambda'} \int_{\{s \in [0, \selbis(\sfyy)) : \, \sft(s) < T\}} \| \bvv  (\GG(\sfyy)(s)) \| \, \di s \,\di \eeta(\sfyy)\nonumber
\\
&
\qquad + \int_{\Lambda'} \int_{\selbis(\sfyy)}^{\selteta(\sfyy)} \| \sfxx'_{\sfyy} (s) \| \, \di s \, \di \eeta(\sfyy) + \frac{\eeta( \{ \sfyy \in \Lambda': \, \sft(\selbis(\sfyy)) < T\})}{k} \nonumber
\\
&
\qquad + \int_{\Lambda'} \int_{\{s \in ( \selteta(\sfyy) , +\infty) : \, \sft(s- \selteta(\sfyy) + \rho(\selbis(\sfyy),\sfyy))) < T\}} \| \bvv  (\GG(\sfyy)(s)) \| \, \di s \, \di \eeta(\sfyy)\,. \nonumber
\\
&
\geq | \nnu_\flat| ([0, T) \times \R^{d}) + \frac{\eeta( \{ \sfyy \in \Lambda': \, \sft(\selbis(\sfyy)) \leq T\})}{k} > | \nnu_\flat| ([0, T) \times \R^{d})\,, \nonumber
\end{align*}
 where we have also used that
$ \rho(\selbis(\sfyy),\sfyy) - \selteta(\sfyy) \geq \frac1k$ due to \eqref{3rd-added-label-th6.3}, as well as \eqref{e:positive-eta}.
\\
\textbf{Conclusion of the proof:} 
 From   the assumed  \eqref{UEST}   and~\eqref{e:T-positive-eta} we further deduce that
\begin{equation}
\label{e:last-ineq}
| \nnu_\flat| ([0, T) \times \R^{d}) < {\rm Var}_{W_{1}} (\mu, [0, T])
\end{equation}
for every $T \in  \III $ such that~\eqref{e:positive-eta} holds. Since $(\mu, \nnu_\flat)$ solves the continuity equation~\eqref{continuity-equation}, with initial datum~$\mu_{0} \in \mathcal{P}_{1}(\R^{d})$, in the sense of Definition~\ref{def:solCE}, \eqref{e:last-ineq} contradicts~\eqref{estimates-w-Var} of Theorem~\ref{th:1}.  Hence, the assertion with which we started the proof is false. This concludes the proof of the theorem.
\end{proof}

\section{ Superposition by $\BV $ curves}
\label{s:6}
 The superposition principle obtained in Theorem \ref{t.1}
offers a probabilistic representation of a solution $(\mu,\nnu)$ to the continuity equation
in terms of a measure $\eeta$ concentrated on curves $\sfyy= \sfyy(s) =(\sft(s),\sfxx(s))$  in the augmented space 
$\R^{d+1} = \R \times \R^d$.
%
We may in fact think of them as `parametrized' trajectories, and accordingly regard   $s$ as an `artificial' time variable.
\par
With the main result of this section, Theorem \ref{th:BV-representation}  ahead, 
we provide
an \emph{alternative} probabilistic representation involving
a probability measure $\neweta$   on the space of curves with locally bounded variation,  that are functions of the process time $t$.

 The bridge between the superposition principles of  Theorem \ref{t.1} and Theorem \ref{th:BV-representation}  will be provided by a 
suitable class of curves that  represent   $\BV$ curves  \emph{augmented} by their transition at jumps. The next section revolves around them. 



\subsection{Augmented reparametrized $\BV$ curves}
\label{ss:6.1}
\paragraph{\bf Preliminaries.} 
We denote by $\cZ$ the set $  \III  \times [0, 1]$ endowed with the \emph{dictionary} order relation:
\begin{align*}
(t_{1}, r_{1}), (t_{2}, r_{2}) \in \cZ:\, (t_{1}, r_{1})  \llcurly (t_{2}, r_{2}) \ \Leftrightarrow \ t_{1} < t_{2} \text{ or } (t_{1} = t_{2} \text{ and } r_{1} < r_{2})\,.
\end{align*}
 We can define on  $\cZ$ the 
\emph{order topology}
(cf., e.g., \cite[II.14]{Munkres}):  a basis for such a topology is 
the collection of all the open intervals 
$(z_1,z_2):=\{z\in \cZ:z_1\llcurly z\llcurly z_2\}$
and all the intervals of the form
$[(0,0),z_2):=\{z\in \cZ:z\llcurly z_2\}$.
Notice that, while with this topology $\cZ$ is neither separable nor metrizable, it 
satisfies the first axiom of countability. 
\par  We use the symbol  $\rightharpoondown$ 
 for  the associated notion of convergence. Observe that, 
 for any $t\in  \III $
we have that 
\[
\forall\, (r_n)_n \subset [0,1]\, : \qquad 
(t_n, r_n)   \rightharpoondown 
\begin{cases}
(t,0)  & \text{if } {t_n<t
\text{ and } t_n \to t,}
\\
(t,1)  & \text{if } 
{t_n>t
\text{ and }t_n \to t.}
\end{cases}
\] 
In turn, for any $(t_n, r_n)_n, (t,r) \in \cZ $, 
\[
(t_n, r_n)  \rightharpoondown  (t,r) \ \Longrightarrow \
 \text{ for $n$ big enough, we have }
 \begin{cases}
t_n \equiv t & \text{if } r \in (0,1),
\\
t_n \leq t & \text{if } r =0,
\\
t_n \geq t  & \text{if } r =1. 
\end{cases}
\]
\par
Therefore, let  $\mathsf{v} \colon \cZ \to \R^d$  be a continuous curve. Necessarily, 
\begin{equation}
\label{properties-continuous-curves}
\begin{aligned}
&
\forall\,  t\in [0,+\infty)\,:  &&
\text{the curve $[0,1]\ni r \mapsto \mathsf{v}(t,r)$ is continuous}
\\
&  \forall\,  t\in [0,+\infty)  \ 
\forall (r_n)_n \subset [0,1]\,: &&
 \mathsf{v}(t_n,r_n)  \to 
 \ \begin{cases}
 \mathsf{v}(t,0)  & \text{if } 
  { t_n<t
\text{ and } t_n \to t,}
 \\
  \mathsf{v}(t,1)  & \text{if }
   {t_n>t
\text{ and } t_n \to t.}
 \end{cases}
 \end{aligned}
 \end{equation}
\paragraph{\bf  $\ARBV$ curves.} 
We are now in a position to introduce the class of curves on $\cZ$ we shall  employ  to `bridge' the Lipschitz continuous trajectories $\sfyy$ to their
$\BV$ 
counterpart. 
\begin{definition}
We call \emph{augmented  $\BV$} curve any   $\bvc \colon \cZ \to \R^{d} $ such that 
\begin{enumerate}
\item $\bvc$ is continuous in~$\cZ$;
\item $[0,1] \ni r\mapsto \bvc (t, r)$ is Lipschitz continuous  for all  $t \in [0, +\infty)$;
\item $[0,1] \ni r\mapsto  \| \partial_{r} \bvc (t, r) \|$ is constant  for all  $t \in [0, +\infty)$;
\item for all $T \in (0,+\infty)$ we have that 
\[
\sup_{\mathcal{P} \text{ partition of $ [ 0 , T ] \times [ 0 , 1 ] $}} \sum_{ ( t_{k} , r_{k} ) \in
 \mathcal{P} } \| \bvc (t_{k}, r_{k} ) {-} \bvc ( t_{k-1} , r_{k-1} ) \| <+\infty \,,
  \]
   where the partition $\mathcal{P}$ is constructed using the order relation $\llcurly$ in $\mathcal{Z}$. 
\end{enumerate}
We denote by $\ARBV(\cZ ; \R^{d})$ the set of all such curves.  For $t \in [0, +\infty)$, we further denote by $\ell_{\bvc} (t)$ the length of the curve $r \mapsto \bvc (t, r)$. 
\end{definition}
%
 Loosely speaking,  a curve in  $\ARBV(\cZ ; \R^{d})$  may be interpreted as an \emph{augmented} version of a curve $u \in \mathrm{BV}_{\mathrm{loc}}(\III;\R^d)$,
to which we attach Lipschitz continuous transition curves at the  jump points.
We may indeed associate with any  $u\in  \mathrm{BV}_{\mathrm{loc}}(\III;\R^d)$
(which is, in particular, \emph{regulated}, with left and right limits $u^-(t)$ and $u^+(t)$ at each $t\in  \III $), with \emph{jump set} $ \mathrm{J}_u
$, 
a curve 
$\bvc \in \ARBV(\cZ ; \R^{d})$ by 
setting 
\[
\begin{aligned}
&
\text{if }  t \in \III {\setminus} \mathrm{J}_u && 
\bvc(t,r) \equiv 
u(t)  \text{ for all } r \in [0,1],
 \\
&
\text{if }  t \in  \mathrm{J}_u &&  
\bvc(t,r) := 
\begin{cases} 
  u^-(t) & \text{if } r=0
\\
 \text{a transition curve 
  with constant velocity joining $  u^-(t)$ to $  u^+(t)$} & \text{if } r \in (0,1)
 \\
  u^+(t) & \text{if } r=1
\end{cases}
\,.
\end{aligned}
\]
Because of \eqref{properties-continuous-curves},  the resulting $\bvc$ is indeed continuous on $\cZ$, even when originating from a curve $u$ with  jumps. 
Notice that, at any $t\in \mathrm{J}_u$ the transition curve $r\mapsto \bvc(t,r)$ has constant speed on $[0,1]$,
equal to its length $\ell_{\bvc}(t)$ (observe that 
this property 
 may be obtained through \emph{reparametrization}).
  Moreover, we remark that, for all $r_1, r_2 \in [0,1]$ the mappings $t\mapsto \bvc(t,r_1)$ and  $t\mapsto \bvc(t,r_2)$ coincide a.e.\ in $\III$
 (namely, outside $ \mathrm{J}_u
$). Hence, the mappings $t\mapsto  \bvc(t,r)$ share the same distributional derivative $\partial_t \bvc$, i.e.\ $\partial_t \bvc(t,r) = \partial_t \bvc (t,0)$ for all $r\in [0,1]$. 
\par
 Conversely, we may consider the mapping 
\begin{equation}
\label{BV-skeleton}
\BVS\colon \ARBV(\cZ; \R^{d}) \to \BVloc (\III ; \R^{d}), \quad \bvc \mapsto \BVSm_{\bvc}  \quad \text{ where }  \BVSm_{\bvc}(t): =  \bvc(t, 0)\,.
\end{equation}
 Loosely speaking, $\BVSm_{\bvc}$ is the ``$\BV$ skeleton'' of $\bvc$. 
 We notice that $\BVSm_{\bvc}$ is left-continuous, i.e.,  at each $t\in (0,+\infty)$
   its left limit $\BVSm^{-}_{\bvc} (t)$  coincides with $ \BVSm_{\bvc} (t) =\bvc(t, 0)$, while its right limit  $\BVSm_{\bvc}^{+} (t) $ is $  \bvc(t, 1)$.

\paragraph{\bf Augmented reparametrized $\BV$ curves and Lipschitz trajectories.}
For $\bvc \in \ARBV(\cZ ; \R^{d})$, $t \in [0, +\infty)$, and $r \in [0, 1]$ we define
\begin{align*}
L_{\bvc}^{-} (t) & := \sup_{\mathcal{P} \text{ partition of $ [ 0 , t) \times [ 0 , 1 ] $}} 
\sum_{ ( t_{k} , r_{k} ) \in \mathcal{P} }  \| (t_{k}, \bvc (t_{k}, r_{k} )) - (t_{k-1}, \bvc (t_{k-1}, r_{k-1} )) \|\,,  
\\
L_{\bvc}^{+} (t) & := \sup_{\mathcal{P} \text{ partition of $ [ 0 , t] \times [ 0 , 1 ] $}} 
\sum_{ ( t_{k} , r_{k} ) \in \mathcal{P} }  \| (t_{k}, \bvc (t_{k}, r_{k} )) - (t_{k-1}, \bvc (t_{k-1}, r_{k-1} )) \|\,,  
\\
L_{\bvc} (t, r) & := L^{-}_{\bvc} (t) + r  \ell_{\bvc}(t)   \,,
\end{align*}
 (recall that  $\| \partial_r \bvc(t,r)\| \equiv  \ell_{\bvc}(t) $ for all $r\in [0,1]$).   For later use, we further 
 introduce the set
 \begin{equation}
 \label{jump-Stefano}
  \mathfrak{J}_{\bvc} := \{ t \in \III: \,    \|\partial_r  \bvc(t,\cdot)\|  \equiv  \ell_{\bvc} (t) \neq0\}. 
  \end{equation}
\par
 With each curve $\bvc \in \ARBV(\cZ ; \R^{d})$ we may associate a  trajectory    $\sfyy \in \Lipplus 1 \III{\R^{d+1}}$ as follows. For $s \in  \III $ we define
\begin{align}
\label{e:t(s)}
\sft(s) & := \inf \{ t \in  \III  : \, L^{+}_{\bvc} (t) > s \} \,,
\\
\label{e:r(s)}
\sfr(s) & := \begin{cases}
\displaystyle  \frac{s - L^{-}_{\bvc} ( \sft ( s ) )}{ L^{+}_{\bvc} ( \sft ( s ) ) - L^{-}_{\bvc} ( \sft ( s ) )}  & \text{if $L^{+}_{\bvc} (\sft(s)) \neq L^{-} (\sft(s))$}\,,\\[3mm]
\displaystyle 0 & \text{otherwise.}
\end{cases}
\end{align}
In particular, notice that $L^{-}_{\bvc} (\sft(s)) \leq s \leq L^{+}_{u} (\sft(s))$ for every $s \in  \III $. We define the curve $ \sfyy $ by
\begin{equation}
\label{yy4u}
\sfyy(s):= (\sft(s) , \bvc (\sft(s), \sfr(s))),\qquad s \in  \III .
\end{equation}
 Then, $\sfyy$ is $1$-Lipschitz continuous: for $s_{1} < s_{2}$ it holds
\begin{align*}
\| \sfyy(s_{2}) - \sfyy(s_{2})\| & = \big\| \big(\sft(s_{2}), \bvc (\sft(s_{2}) , \sfr(s_{2})) \big) -  \big( \sft(s_{1}),  \bvc (\sft(s_{1}) , \sfr(s_{1}) ) \big) \big \| 
\\
&
\leq L_{\bvc } (\sft(s_{2}), \sfr(s_{2})) - L_{\bvc} (\sft(s_{1}), \sfr(s_{1})) \leq s_{2} - s_{1}\,.
\end{align*}
 Moreover, notice that   $\|\sfyy'(s)\| \equiv  1$ for a.e.~$s \in \III$.   We denote by $\mathscr{T}\colon \ARBV(\cZ ; \R^{d})\to   \ALip \III{\R^{d+1}} $   the map that associates 
with any $\bvc \in  \ARBV(\cZ ; \R^{d}) $ the curve 
$\sfyy \in  \ALip \III{\R^{d+1}} $   from \eqref{yy4u}. 

We also introduce a map $\mathscr{S}\colon \Lipplus 1 \III{\R^{d+1}} \to   \ARBV(\cZ ; \R^{d})$ as follows. For  $\sfyy = (\sft, \sfxx) \in  \Lipplus 1 \III{\R^{d+1}}$  we set
\begin{align}
\label{e:sgamma}
\sfs_{\sfyy}^{-}(t) & := \sup\, \{ s \in  \III  : \, \sft(s) <t\}  && \text{for every $t \in  \III $}\,,\\
\label{e:sgamma-2}
\sfs_{\sfyy}^{+}(t) & := \inf\, \{ s \in   \III  : \, \sft(s) >t\} && \text{for every $t \in   \III $}\,.
\end{align}
Then, $\sfs^{\pm}_{\sfyy} \colon   \III  \to   \III $ satisfy $\sft(\sfs^{\pm}_{\sfyy}(t)) = t$ for every $t \in   \III $, and 
 \begin{align}
\label{e:ineq-sfs}
\sfs^{-}_{\sfyy}(\sft(s)) \leq s \leq \sfs^{+} (\sft(s)) \qquad \text{for every $s \in  \III $.}
\end{align} 
Moreover, if  $\sft'(s) >0$ at some $s \in   \III $, then $\sfs^{\pm}_{\sfyy}(\sft(s)) = s$. Indeed, if $\sfs^{-}_{\sfyy}(\sft(s)) <s$ or $\sfs^{+} (\sft(s)) >s$, then we would have $\sft(\sigma) = \sft(s)$ for $\sigma \in (\sfs^{-}_{\sfyy} (\sft(s)) , \sfs^{+}_{\sfyy} (\sft(s)))$, which would contradict the assumption $\sft'(s) >0$.
Since $\sfyy$ is $1$-Lipschitz continuous and~$\sft$ is monotone non-decreasing with $\sft(s) \to +\infty$ for $s \to +\infty$, we have that
the functions
$\sfs^{\pm}_{\sfyy}$ are monotone non-decreasing, $\sfs^{\pm}_{\sfyy} (t) \to +\infty$ as $t \to +\infty$. 
We define $\mathscr{S} (\sfyy) \in  \ARBV(\cZ ; \R^{d})$ via
\begin{equation}
	\label{e:function S}
\mathscr{S} (\sfyy) (t, r) := \sfxx \big ( \sfs^{-}_{\sfyy} (t) + r ( \sfs^{+}_{\sfyy} (t) -   \sfs^{-}_{\sfyy} (t) ) \big) \,.
\end{equation}
 We will  prove in Appendix \ref{another appendix}  the following.

\begin{lemma}
\label{l:mapsST}
For every $\sfyy \in  \ALip \III {\R^{d+1}}$   and every $\bvc \in  \ARBV(\cZ ; \R^{d}) $ we have
\begin{align}
\label{inverse-of-another}
 \mathscr{T} (\mathscr{S} (\sfyy))  = \sfyy\,, \qquad   \mathscr{S} (\mathscr{T} (\bvc ))  = \bvc \,.
 \end{align}
 \end{lemma}

Recall that  $ \Lipplus 1 \III{\R^{d+1}}$ is endowed  with the following distance, which metrizes the uniform convergence on compact subintervals of~$  \III $ 
(see Appendix \ref{old:appD}):
 \begin{equation}
 \label{D-recalled-later}
D (\sfyy_{1}, \sfyy_{2}) := \sum_{n =1}^{\infty} 2^{-n} \sup_{s \in [0, n]} \big( \min\{ \| \sfyy_{1}(s) - \sfyy_{2} (s)\| \ , \ 1\} \big) \qquad \text{for $\sfyy_{1}, \sfyy_{2} \in  \Lipplus 1 \III{\R^{d+1}}$.}
\end{equation}
We further define the distance $D_{\ARBV}$ on  $ \ARBV(\cZ ; \R^{d})$ via
\begin{displaymath}
D_{\ARBV}(\bvc_{1}, \bvc_{2}) : = D( \mathscr{T} (\bvc_{1}) , \mathscr{T} (\bvc_{2}))\,.
\end{displaymath}
Then,   the map 
$\mathscr{T}\colon    \ARBV(\cZ ; \R^{d}) \to \Lipplus 1 \III{\R^{d+1}}  $  is trivially continuous. 
Likewise,
the restriction of  
$\mathscr{S}
$ to  $\ALip \III {\R^{d+1}}$ is continuous with values in $ \ARBV(\cZ ; \R^{d})$.   

\subsection{Bridging the  probabilistic representations, from Lipschitz to $\ARBV $ curves}
\label{ss:6.2}
 In order to obtain our superposition principle by $\ARBV$ curves, 
we need to  first  revisit the probabilistic representation  for solutions $(\mu,\nnu)$ of the continuity equation 
guaranteed by Theorem \ref{t.1}.   Recall 
 \eqref{repre-eq-17} \EEE
and \eqref{e.17}, 
 (which in fact   involved the minimal pair $(\mu,   \bar{\nnu} )$
 associated with $(\mu,   \nnu )$), 
namely 
\begin{equation}
\label{e:17-bis}
  \mu= \fre_{\sharp }\big(\frt' \, \teeta\big), \qquad 
   \bar{ \nnu} = \fre_{\sharp }\big( \frxx'\,  \,\teeta\big), \qquad |(\mu,\bar\nnu)| =   \fre_{\sharp } ( \norm{ \fry' }  \teeta)  
=\fre_{\sharp } \teeta\,,
\end{equation}
 (where
$\fre \colon\III\times  \rmC(\III;\R^{d+1}) \to \R^{d+1} $ is the evaluation mapping, and
the Borel  maps $ \frt'$ and $\frxx'$  are defined on 
$\III\times \Lip(\III;\R^{d+1})$ by 
$\frt'(s,\sfyy): = \sft'(s)$,  $\frxx'(s,\sfyy):=  \sfxx'(s)$).  
The above representation brings into play a measure 
 $\eeta \in \sP (\Lipplus 1 \III{\R^{d+1}})$   supported on curves
  $\sfyy  \in   \ALip \III{\R^{d+1}}$  
  solving the Cauchy  problem   \eqref{Cauchy-gamma-new}. \EEE
  \par
   Let us introduce a general construction that provides an equivalent formulation of \eqref{e:17-bis}. 
  Firstly, observe that every 
   $\sfyy \in \Lipplus 1 \III{\R^{d+1}}$  induces the vector measure 
  \begin{equation}
  \label{construction-par}
  \begin{aligned}
  & \omega_{\sfyy}: = \sfyy_{\sharp }((\sft',\sfxx') \mathcal{L}^1)   \in  \mathcal{M}_{\loc}( \R^{d+1}_{+}  ; \R^{d+1})  
  \\
  &
   \text{i.e., }  \quad
   \langle \omega_{\sfyy}, \varphi \rangle  = \int_{\III} \varphi(\sft(s), \sfxx(s)) \cdot (\sft'(s),\sfxx'(s)) \dd s  \text{ for all } \varphi=(\varphi_0,\bvarphi) \in \Cc ( \R^{d+1}_{+}  ;\R^{d+1})\,.
   \end{aligned}
  \end{equation}
Namely, $\omega_\sfyy$ is   the integration measure along the curve $\sfyy$. 
  Now, for a given probability measure  $\llambda \in \sP (\Lipplus 1 \III{\R^{d+1}})$  fulfilling the integrability conditions \eqref{eq:condition2},  
  we may consider the measure 
  \[
 \Upomega^{\llambda}: = \int_{ \Lipplusname 1} \omega_{\sfyy} \, \dd \llambda(\sfyy) \in \mathcal{M}_{\loc} (  \R^{d+1}_{+}; \R^{d+1}) \,.
  \]
  \par
  Let now $\llambda$ be the measure  $\eeta \in \sP (\Lipplus 1 \III{\R^{d+1}})$ supported on  
  $\ALip  \III{\R^{d+1}}$ and 
   providing the representation formulae \eqref{e:17-bis}.  It follows from such representation that  
  for  $\eeta$-almost every 
  $\sfyy \in  \Lipplus 1 \III{\R^{d+1}}$ we have $ | \omega_{\sfyy}| =  \sfyy_{\sharp }(\| \sfyy'\|  \mathcal{L}^1)$ and 
\[
| \Upomega^{\eeta}| = \int_{\Lipplusname 1} |\omega_{\sfyy}| \, \dd \eeta(\sfyy)\,.
\]
 Hence,    \eqref{e:17-bis} reformulates in compact form as
\begin{equation}
\label{concise-reformulation}
(\mu,\nnu) =\Upomega^{\eeta}, \qquad |(\mu,\nnu)|  = |\Upomega^{\eeta}|= \int_{\Lipplusname 1} |\omega_{\sfyy}| \, \dd \eeta(\sfyy).
\end{equation} 
This observation is at the core of the representation provided by Theorem \ref{th:BV-representation} ahead. 
\par
 Indeed,  with any  $\bvc \in \ARBV(\cZ;\R^d)$ we associate  the measure $\vartheta_{\bvc} \in  \mathcal{M}_{\loc} (\R^{d+1}_{+}; \R^{d+1})$  defined by 
\begin{equation}
\label{tetau-orig-def}
\vartheta_{\bvc} := \omega_{ \mathscr{T} (\bvc)} \quad \text{i.e., } \ 
\int_{\R^{d+1}} \varphi ( t, x) \, \di \vartheta_{\bvc} (t, x) = \int_{\III} \varphi (\sfyy (s)) \cdot \sfyy'(s) \, \di s 
\text{ with $\sfyy= \mathscr{T} (\bvc)$, } 
\end{equation}
for 
$\varphi= (\varphi_{0}, \boldsymbol{\varphi}) \in \Cc (\R^{d+1}_{+};\R^{d+1})$. By construction, we have that
 for every $\varphi \in {\rm C}_{\rm c} (   \R^{d+1}_{+} ; \R^{d+1})$
 the map $\bvc \mapsto \int_{\R^{d+1}} \varphi \, \di \vartheta_{\bvc}$ is continuous with respect to the convergence in  $\ARBV (\cZ; \R^{d})$. 
  Hence, for $A \subseteq   \R^{d+1}_{+} $ open we infer that $\bvc \mapsto \vartheta_{\bvc} (A)$ is Borel. The following result   allows us to express the measure~$\Upomega^{\eeta}$ 
in terms of a probability measure $\neweta$ on  curves in $\ARBV  (\cZ; \R^{d}).$ 
\begin{lemma}
\label{l:switch}
 For every $\llambda \in \mathcal{P}(\Lipplus 1 \III { \R^{d+1}})$  concentrated in 
 $\ALip \III {\R^{d+1}}$,   let $\widehat\llambda: = \mathscr{S}_{\sharp }\llambda$. Then, 
\begin{enumerate}
\item
$\widehat\llambda$ is a Borel probability measure over $\ARBV (\cZ; \R^{d})$;
\item
 if $\widehat\llambda$ satisfies the integrability condition \eqref{eq:condition2}, then 
\begin{equation}
\label{e:Psi-Omega}
\Uptheta^{\widehat{\llambda}}:= \int_{\ARBV} \vartheta_{\bvc}\, \di \widehat{\llambda}(\bvc) \text{ is a measure in } \mathcal{M}_{\loc} (  \R^{d+1}_{+}; \R^{d+1})\,.
\end{equation}
\item Let $\llambda = \eeta $  fulfill \eqref{e:17-bis}.  Then,
\begin{equation}
\label{final-claim-lemma}
\Uptheta^{\neweta} =\Upomega^{\eeta} = (\mu,\nnu) \quad \text{and} \quad |\Uptheta^{\neweta}| =  \int_{\ARBV}| \vartheta_{\bvc}|\, \di \neweta(\bvc)=|  (\mu,\nnu) |\,.
 \end{equation}
 \end{enumerate}
 \end{lemma}
\begin{proof}
 Since the map $\mathscr{S}$ is continuous  on  $\ALip \III {\R^{d+1}}$,    we have that 
  $\widehat\llambda= \mathscr{S}_{\sharp} \llambda$  is a Borel probability measure over $\ARBV (\cZ; \R^{d})$.

Let $\llambda$ additionally satisfy ~\eqref{eq:condition2}. 
To show that $\Uptheta^{\widehat\llambda} \in \mathcal{M}^{+}_{\loc} ( \R^{d+1}_{+};; \R^{d+1})$, it is enough to notice that for every $T \in [0, +\infty)$ 
\begin{displaymath}
\Uptheta^{\widehat\llambda} ([0, T] \times \R^{d}) = \int_{\Domain{T}} \| \sfyy' (s)\| \, \di s \, \di \widehat\llambda(\sfyy) <+\infty
\end{displaymath}
(where $\Domain T$ is from \eqref{eq:domainT}). 

 When  $\llambda = \eeta$ fulfills  \eqref{e:17-bis},  \eqref{final-claim-lemma} obviously follows from \eqref{concise-reformulation}.  
\end{proof}

In the upcoming Proposition \ref{p:identify-theta-alpha}, we give an alternative formula for $\vartheta_{\bvc}$. 
It brings into play
\begin{itemize}
\item[-] the mapping $\BVS \colon \ARBV(\cZ;\R^d) \to \BVloc(\III;\R^d) $  from \eqref{BV-skeleton} that associates  with each $\bvc \in  \ARBV(\cZ;\R^d) $  its $\BV$ 
 `skeleton' $\BVSm_{\bvc}(t):= \bvc(t,0)$,
\item[-] as well as the jump transitions $[0,1] \ni r  \mapsto \bvc(t,r) $ at the discontinuity points $t$ of $\BVSm_{\bvc}$. 
\end{itemize}
To prepare the statement of   Proposition \ref{p:identify-theta-alpha},  let  us recall that 
%
%
 the distributional derivative $(\BVSm_{\bvc})'_{\mathrm{d}} $ is a Radon vector measure on $\III$ that  can be decomposed into the sum of three mutually singular measures
\[
(\BVSm_{\bvc})'_{\mathrm{d}} = (\BVSm_{\bvc})'_{\mathcal{L}^1} +  (\BVSm_{\bvc})'_{\mathrm{C}} +  (\BVSm_{\bvc})'_{\mathrm{J}}. 
\]
Now,
\begin{itemize}
\item[-]
 $(\BVSm_{\bvc})'_{\mathcal{L}^1} $  is  the absolutely continuous part with respect to $\mathcal{L}^1$, is  given by $(\BVSm_{\bvc})'_{\mathcal{L}^1}  = \BVSm_{\bvc}'  \mathcal{L}^1$
    (with $\BVSm_{\bvc}'$ the a.e.-defined pointwise derivative of $\BVSm_{\bvc}$).
 In turn,  since the a.e.\ defined derivative
  of $t\mapsto \partial_t \bvc (t,r)$ does not, in fact, depend on $r$, we have that $ \BVSm_{\bvc}' = \partial_t \bvc (\cdot,0) =  \partial_t \bvc (\cdot,r)$ $ \mathcal{L}^1$-a.e.\ in 
  $\III$ for every $r\in [0,1]$. Therefore, hereafter we will use the more evocative notation
  \begin{equation}
  \label{partialL}
    \partial^{\mathrm{L}}_t \bvc \quad \text{in place of} \quad  \BVSm_{\bvc}' \,.
    \end{equation}
    \item[-] $ (\BVSm_{\bvc})'_{\mathrm{C}}$ is the so-called Cantor part of $\BVSm_{\bvc}$, still satisfying 
$(\BVSm_{\bvc})'_{\mathrm{C}}(\{ t\}) =0$ for all $t\in [0,+\infty)$. In accordance with \eqref{partialL}, we will write 
  \begin{equation}
  \label{partialC}
    \partial^{\mathrm{C}}_t \bvc \quad \text{in place of} \quad   (\BVSm_{\bvc})'_{\mathrm{C}}\,.
    \end{equation}
 \item[-] $ (\BVSm_{\bvc})'_{\mathrm{J}}$ is a discrete measure concentrated on the
 (at most countable)
   jump set  of $\BVSm_{\bvc}$, i.e.\
\[
\mathrm{J}_{\BVSm_{\bvc}}: = \{ t \in \III\, : \ \BVSm_{\bvc}^{+} (t) \neq    
 \BVSm_{\bvc}^{-} (t)    \} = \{ t \in \III\, : \ \bvc(t, 0) \neq \bvc (t, 1)\}  \subseteq \mathfrak{J}_{\bvc}.
\]
\end{itemize} 
Observe that
$  \partial^{\mathrm{L}}_t \bvc   \mathcal{L}^1 + \partial^{\mathrm{C}}_t \bvc $  does not charge  $\mathfrak{J}_{\bvc}$ 
(it is indeed known 
 as the diffuse part of $(\BVSm_{\bvc})'_{\mathrm{d}}$). 
 Clearly, it  is concentrated on  
\begin{equation}
\label{complement-of-jump}
C_{\BVSm_{\bvc}}:= \III \setminus \mathrm{J}_{\BVSm_{\bvc}}. 
\end{equation}
 We further denote by 
\begin{equation}
\label{complement-of-transition}
C_{\bvc} := \III \setminus \mathfrak{J}_{\bvc}
\end{equation}
and notice that $C_{\bvc} \subseteq C_{\BVSm_{\bvc}}$ and $| \partial^{\mathrm{L}}_t \bvc   \mathcal{L}^1 + \partial^{\mathrm{C}}_t \bvc| (C_{\BVSm_{\bvc}} \setminus C_{\bvc}) =0$.  While $  \partial^{\mathrm{L}}_t \bvc   \mathcal{L}^1 $ and $ \partial^{\mathrm{C}}_t \bvc $  will be encompassed in the alternative representation of $\vartheta_{\bvc}$, 
the jump contribution will feature, in place of  $(\BVSm_{\bvc})'_{\mathrm{J}} $,  the measures 
\begin{equation}
\label{construction-u-2}
\begin{aligned}
&
 \boldsymbol{j}_{t, \bvc}  \in  \mathcal{M}_{\loc} (  \R^{d+1}_{+}; \R^{d+1}),  \ t \in \mathfrak{J}_{\BVSm_{\bvc}},  \text{ given by } 
 \\
 &
\left\langle \boldsymbol{j}_{t, \bvc} , \boldsymbol{\varphi} \right\rangle := \int_{0}^{1} \boldsymbol{\varphi} (t, \bvc (t,r) ) {\cdot} \partial_{r} \bvc (t, r) \, \di r 
\text{ for all $\boldsymbol{\varphi} \in \Cc(\DST ; \R^{d})$.}
\end{aligned}
\end{equation} 
\par 
Finally, for fixed $\bvc \in  \ARBV (\cZ; \R^{d}) $, we 
  introduce the map (where the  gothic letter $G$ stands for 'graph')  
\[
\bvcgr_{\bvc} \colon \III \to \R^{d+1}_{+}, \quad \bvcgr_{\bvc}( t): =  (t, \BVSm_{\bvc} (t)) = (t, \bvc (t, 0))\,,
\]
and observe that  
it is Borel measurable, 
since
may rewrite it as the composition $\gamma_{2} \circ \gamma_{1}$ of the following  functions:
  \begin{align*}
  \gamma_{1} & \colon \III  \to  \III \times \III \qquad  \gamma_{1} (t) := ( t,  \sfs_{\mathscr{T} (\bvc)}^{-} (t) )  \,,\\
  \gamma_{2} & \colon \III \times \III \to  \R^{d+1}_{+} \qquad \gamma_{2} (t, \sigma ) = (t, \pi   ( \mathscr{T}(\bvc) ( \sigma)) ) \,,
  \end{align*}
  where $\pi \colon \R^{d+1}_{+} \to \R^{d}$ is the projection $\pi (t, x) := x$. Then, $\gamma_{2}$ is continuous, while $\gamma_{1}$ is Borel measurable, as the second component is lower-semicontinuous.
In fact, $\bvcgr_{\bvc}$ will  come into play in the representation formula \eqref{teta-equals-alpha}   for $\tetatrue{\bvc} $,
where  $\tetatrue{\bvc} $ results from the sum of three contributions involving the 
\emph{absolutely continuous}, the 
\emph{Cantor} and the 
\emph{jump} parts of the distributional derivative $(\BVSm_{\bvc})'_{\mathrm{d}}$;  the corresponding mappings 
$\mathscr{A}$, $\mathscr{C}$, and $\mathscr{J}$, cf.\ \eqref{bvc2measures}, 
will be shown to  be Borel   measurable.
\par
 Let us emphasize that, the `jump contribution' to $\tetatrue{\bvc} $ features 
 the jump transitions $r
\mapsto \bvc(t,r)$  at the  transition times $t \in \mathfrak{J}_{\bvc}$, i.e., those times $t$ such that $\ell_{\bvc} (t) = \| \partial_{r} \bvc (t, r) \| \neq 0$,  by means of the measures $ \boldsymbol{j}_{t, \bvc} $ of  \eqref{construction-u-2}.  We notice that we cannot, in principle, only consider the jump points $t \in \mathrm{J}_{\BVSm_{\bvc}}$ in the jump contribution $ \boldsymbol{j}_{t, \bvc}$ of~$\vartheta_{\bvc}$, as it may happen that $\bvc (t, 0) = \bvc (t, 1)$ but a transition occurs. 
Nevertheless, 
 in the representation result of Theorem~\ref{th:BV-representation}  there will come into play curves $\bvc$ such that 
$\mathfrak{J}_{\bvc}$ may be replaced by $\mathrm{J}_{\BVSm_{\bvc}}$. 
  \begin{proposition}
  \label{p:identify-theta-alpha}
  For any  $\bvc \in  \ARBV (\cZ; \R^{d})$, the measure  $\tetatrue{\bvc} $  from \eqref{tetau-orig-def} fulfills \EEE
  \begin{align}
&
  \label{teta-equals-alpha}
 \tetatrue{\bvc}  =  (\bvcgr_{\bvc})_{\sharp}  \left(   (1, \partial^{\mathrm{L}}_t \bvc) \mathcal{L}^1  + (0,  \partial^{\mathrm{C}}_t \bvc)   \right) 
+  \sum_{t \in \mathfrak{J}_{\bvc}} \delta_t {\otimes}  \boldsymbol{j}_{t, \bvc}\,.
\end{align}
The above measures can be expressed in terms of $\sfyy =  \mathscr{T} (\bvc)$  as 
  \begin{subequations}
  \label{e:631}
 \begin{align}
   \label{e:616}
 \left\langle  (\bvcgr_{\bvc})_{\sharp} \big((1, \partial^{\mathrm{L}}_t \bvc) \mathcal{L}^1  + (0,  \partial^{\mathrm{C}}_t \bvc)  \big) ,  \varphi \right\rangle &  =
 \int_{ \mathscr{C}_{\sfyy} } \varphi_{0} (\sfyy(s) ) \, \sft'(s) \, \di s+
  \int_{\mathscr{C}_{\sfyy}} \bvarphi( \sfyy(s) ) \cdot \sfxx'(s)\, \di s \,,
 \\
 \label{e:615}
  \left\langle \boldsymbol{j}_{t, \bvc} , \boldsymbol{\varphi} \right\rangle
  & = 
   \int_{\sfs_{\sfyy}^{-}(t)}^{\sfs_{\sfyy}^{+}(t)}  \bvarphi (\sfyy(s) ) \cdot  \sfxx'(s) \, \di s
 \end{align}
 \end{subequations}
 for all $\varphi = (\varphi_0,\bvarphi) 
 \in \Cc(\DST ; \R^{d+1})$, 
 where  $\mathscr{C}_{\sfyy}: =  \sft^{-1} (C_{\bvc})$ and $C_{\bvc}$ is from \eqref{complement-of-transition}. 
 \par
 Finally, the mappings 
 \begin{equation}
 \label{bvc2measures}
 \begin{cases}
 \mathscr{A}\colon \ARBV (\cZ; \R^{d}) \to   \mathcal{M} ( \III \times \R^{d}; \R^{d+1}),  & \bvc \mapsto  (\bvcgr_{\bvc})_{\sharp} (1, \partial^{\mathrm{L}}_t \bvc) \mathcal{L}^1 
 \\
 \mathscr{C}\colon \ARBV (\cZ; \R^{d}) \to   \mathcal{M} ( \III \times \R^{d}; \R^{d+1}),  & \bvc \mapsto   (\bvcgr_{\bvc})_{\sharp}  (0,  \partial^{\mathrm{C}}_t \bvc) 
 \\
  \mathscr{J}\colon \ARBV (\cZ; \R^{d}) \to   \mathcal{M} ( \III \times \R^{d}; \R^{d+1}),  & \bvc \mapsto   \sum_{t \in \mathfrak{J}_{\bvc}}  \delta_t {\otimes}  \boldsymbol{j}_{t, \bvc}
 \end{cases} \qquad \text{are Borel}.
 \end{equation}
  \end{proposition}
  We postpone the \emph{proof} to Appendix \ref{app:meas}.

\subsection{The ${\rm BV}$ probabilistic representation} 
We are now in a position to state the `BV version' to Theorem \ref{t.1}.
In the same way as  the latter result,   Theorem \ref{th:BV-representation}  also provides information on the 
curves $\bvc \in \ARBV (\cZ ; \R^{d})$
 on which the  representation measure $\neweta$ is concentrated. 
 Recall that in  Theorem \ref{t.1}   the superposition principle involved trajectories solving the characteristic system. Now,
 relations \eqref{proxy-Cauchy} ahead,
 which involve $ \partial^{\mathrm{L}}_t \bvc$ \eqref{partialL},  $ \partial^{\mathrm{C}}_t \bvc$ \eqref{partialC}, and $\partial_r \bvc(t,\cdot)$ 
at each $t\in    \mathfrak{J}_{\bvc} = \mathrm{J}_{\BVSm_{\bvc}}$, 
 may be interpreted as a counterpart to the Cauchy problem \eqref{Cauchy-gamma}.
\par
  Without loss of generality,  in what follows we will identify a given solution to the continuity equation with an induced minimal pair, and thus suppose minimality straight away.  
\begin{theorem}
\label{th:BV-representation}
 Let $(\mu,\nnu)\in \calM^+(\DST ) \times \calM( \DST ; \R^{d})$ be a  minimal  $\Puno$-solution  to 
the continuity equation in the sense of  Definition~\ref{def:reduced}, with initial condition~$\mu_{0} \in \mathcal{P}_{1}(\R^{d})$. 
 Let $(\tau, \bvv)$ be bounded Borel vector field representing the density of $(\mu, \nnu)$ w.r.t.~$|(\mu, \nnu)|$.
\par
 Then, there exists a Borel probability measure~$\neweta$ on~$\ARBV (\cZ ; \R^{d})$ that  provides the  probabilistic representation
\begin{equation}
\label{alternative-prob-repre}
   (\mu,\nnu) =\Uptheta^{\neweta}, \qquad |(\mu,\nnu)|  =|\Uptheta^{\neweta}|\,, 
\end{equation}
with $\Uptheta^{\neweta}$ defined as in~\eqref{e:Psi-Omega}.  The measure $\neweta$ is concentrated on curves $\bvc \in \ARBV (\cZ ; \R^{d})$ fulfilling $\bvc(0) \in    \spt(\mu_0)$   and
\begin{subequations}
\label{proxy-Cauchy}
\begin{align}
&
\label{proxy-Cauchy-1}
  \partial^{\mathrm{L}}_t \bvc(t)  = \frac{\bvv(t,\bvc(t, r )))}{\tau(t,\bvc(t, r )))} = \frac{\dd \nnu^a}{\dd\mu}(t,\bvc(t, r )) && \text{for $\mathcal{L}^1$-a.a.\ } t \in  \III  \text{ and $\mathcal{L}^{1}$-a.a.\ } r \in [0, 1],
\\
&
\label{proxy-Cauchy-2}
  \partial^{\mathrm{C}}_t \bvc
= \bvv(t,\bvc(t, r ))    |  \partial^{\mathrm{C}}_t \bvc  |  = \frac{\dd \nnu^\perp}{\dd|\nnu^\perp| }(t,\bvc(t, r))      |  \partial^{\mathrm{C}}_t \bvc  |
  &&   \text{for    $ |\partial^{\mathrm{C}}_t \bvc|  $-a.a.\ } t \in  \III  \text{ and $\mathcal{L}^{1}$-a.a.\ }r \in [0, 1],
\\
& 
\label{proxy-Cauchy-3}
 \partial_{r} \bvc (t, r)  = \frac{\dd \nnu^\perp}{\dd|\nnu^\perp| }\left(t,  \bvc (t, r) \right) \| \partial_{r} \bvc(t, r)  \| && \text{for a.a.\ } r \in [0,1] \text{ and all } t \in  \mathrm{J}_{\BVSm_{\bvc}}\,,
\end{align}
with $\nnu = \nnu^a+\nnu^\perp$ the Lebesgue decomposition of $\nnu$ into 
$\nnu^a \ll \mu$ and $\nnu^\perp \perp \mu$. 
\end{subequations}
\end{theorem}
\par

As a consequence of Theorem~\ref{th:BV-representation}  (and thus, ultimately, of Proposition \ref{p:identify-theta-alpha}), \EEE  and of the Borel measurability in~\eqref{bvc2measures} we can rewrite the representation  formulae \eqref{alternative-prob-repre} and \eqref{proxy-Cauchy} as follows. 

\begin{corollary}\label{c:unnamed}
Under the assumptions of Theorem~\ref{th:BV-representation}, the representation formulae \eqref{alternative-prob-repre} and \eqref{proxy-Cauchy} rephrase as 
\begin{subequations}
\label{true-time-representation}
\begin{align}
 \label{e:representation-mu}
& \iint_{\DST } \varphi_0(t, x) \, \di \mu (t, x)  =  \int_{\ARBV}  \int_{\III}  \varphi_{0} (t,  \BVSm_\bvc (t)  ) \, \di t \, \di \neweta(\bvc) \quad\text{for all } \varphi_{0} \in \Cc (\DST ),
\intertext{while  for $\nnu^a$ and $\nnu^\perp$  we have}
&
\label{e:representation-nu}
\begin{aligned}
&
\iint_{\DST } \boldsymbol{\varphi}(t, x) \, \di \nnu^a (t, x)   = \int_{\ARBV}
\int_{\III} \boldsymbol{\varphi}(t, \BVSm_{\bvc}(t))  \frac{\bvv(t,  \BVSm_{\bvc}(t))}{\tau(t,  \BVSm_{\bvc}(t) ))} \, \dd t  \, 
 \di \neweta (\bvc)\,, 
\\
&
\begin{aligned}
\iint_{\DST } \boldsymbol{\varphi}(t, x) \, \di \nnu^\perp (t, x ) 
 & = 
\int_{\ARBV} \int_{\III} \boldsymbol{\varphi}(t, \BVSm_{\bvc}(t)  )  \bvv(t,  \BVSm_{\bvc}(t)  ))  \, \dd | \partial_{t} ^{C} \bvc|(t)  \, 
 \di \neweta (\bvc) 
\\
&  \quad +
\int_{\ARBV}   \sum_{t \in \mathrm{J}_{\BVSm_{\bvc}}}  \int_{0}^{1}    \boldsymbol{\varphi} (t, \bvc (t,r) ) {\cdot} \partial_{r} \bvc (t, r) 
 \, \di r \, \di \neweta(\bvc)\,
\end{aligned}
\end{aligned}
\end{align}
for every $\boldsymbol{\varphi} \in {\rm C}_{\rm c} (\DST ; \R^{d})$\,. 
\end{subequations} 

\par
 Finally,   the left and right representatives $\mu_t^-=\mu_t$ and $\mu_t^+$ of $\mu$  fulfill
\begin{equation}
\label{pm-limits}
\begin{cases}
\displaystyle
\int_{\R^d} \psi(x) \dd \mu_t(x) =  \int_{\R^d} \psi(x) \dd \mu_t^-(x) =
 \int_{\ARBV} \psi( \BVSm_{\bvc} (t))  \, \di \neweta(\bvc)
 \smallskip
 \\
 \displaystyle
  \int_{\R^d} \psi(x) \dd  \mu_t^+(x)  
 =  \int_{\ARBV} \psi(\BVSm_{\bvc}^{+} (t))  \, \di \neweta(\bvc) 
\end{cases}
\qquad \text{for all } \psi \in \Cc(\R^d)\,.
\end{equation} 
\end{corollary}

 In the proof of Theorem \ref{th:BV-representation} we will also resort to 
 some 
 measure-theoretic results in Appendix  \ref{s:app-B}.

\begin{proof}[Proof of Theorem~\ref{th:BV-representation}]
 We divide the proof in three steps, proving~\eqref{alternative-prob-repre}--\eqref{pm-limits} separately. 

\STEP{1: proof of  \eqref{alternative-prob-repre}.}
Since
$\nnu$ fulfills the minimality condition 
 by Theorem \ref{t.1} there exists $\eeta \in \mathcal{P}(\Lipplus 1 \III { \R^{d+1}} ) $   concentrated on $\ALip \III {\R^{d+1}}$ and  such that the representation formulae \eqref{concise-reformulation}
 hold. 
  In  view of 
  Lemma \ref{l:switch}, we conclude that the Borel measure $\neweta:= \mathscr{S}_{\sharp} \eeta$ fulfills \eqref{alternative-prob-repre}. 
  
 \STEP{2: proof of  \eqref{proxy-Cauchy}.}
From \eqref{alternative-prob-repre}  we gather  in particular that  $|\Uptheta^{\neweta}| =  \int_{\ARBV}| \vartheta_{\bvc}|\, \di \neweta(\bvc)$. 
 Therefore, 
 we are in a position to apply Proposition \ref{prop:measure-theor}, thus concluding that 
\[
\tetatrue{\bvc} = \boldsymbol{\mathsf{f}} |\tetatrue{\bvc}| \qquad \text{for } \neweta\text{-a.a. } \bvc \in   \ARBV(\cZ;\R^d) 
 \qquad \text{with }  \boldsymbol{\mathsf{f}} = \frac{\dd \Uptheta^{\neweta}}{\dd|\Uptheta^{\neweta}|} = \frac{\dd(\mu,\nnu)}{\dd |(\mu,\nnu)|} = (\tau,\bvv)\,.
\]
Combining this with \eqref{teta-equals-alpha},  we thus obtain
 \begin{equation}
 \label{KEY-BV-PROOF}
 \mathscr{A}(\bvc)+  \mathscr{C}(\bvc)+  \mathscr{J}(\bvc) = \tetatrue{\bvc}  =  (\tau,\bvv) |\tetatrue{\bvc}|  \quad 
 \text{with } 
\begin{cases}
 \mathscr{A}(\bvc) =  (\bvcgr_{\bvc})_{\sharp}      (1, \partial^{\mathrm{L}}_t \bvc) \mathcal{L}^1,  
 \\
 \mathscr{C}(\bvc) =   (\bvcgr_{\bvc})_{\sharp}      (0,  \partial^{\mathrm{C}}_t \bvc),  
 \\
 \mathscr{J}(\bvc) =     \sum_{t \in \mathrm{J}_{\BVSm_{\bvc}}}  \delta_t {\otimes}  \boldsymbol{j}_{t, \bvc} . 
\end{cases}
 \end{equation}
 Notice that in the last equality in~\eqref{KEY-BV-PROOF} we have used that $\mathrm{J}_{\BVSm_{\bvc}} = \mathfrak{J}_{\bvc}$, which is a consequence of Theorem~\ref{t:injective-curves}. Namely, a transition at time $t$ may occur if and only if $\bvc (t, 0) \neq \bvc (t, 1)$, since $\eeta$ is supported on injective curves in $\ALip \III {\R^{d+1}}$.   We shall obtain   \eqref{proxy-Cauchy} by restricting \eqref{KEY-BV-PROOF} to the support of each of the 
  three mutually singular contributions $ \mathscr{A}(\bvc)$, $ \mathscr{C}(\bvc)$, $ \mathscr{J}(\bvc)$. 
\par
Indeed, 
restricting \eqref{KEY-BV-PROOF} to $\spt(\mathscr{A}(\bvc))$ we infer
\begin{align*}
 (1, \partial^{\mathrm{L}}_t \bvc(t) ) 
 & = (\tau(t, \BVSm_{\bvc} (t)), \bvv(t, \BVSm_{\bvc} (t)) ) | (1, 
\partial^{\mathrm{L}}_t \bvc(t)  ) | 
\\
&
= (\tau(t, \bvc (t, r) ), \bvv(t, \bvc (t, r)) ) | (1,  \partial^{\mathrm{L}}_t \bvc(t)  ) | \qquad \text{for }  \mathcal{L}^1\text{-a.a.\ } t \in \III \text{ and every $r \in [0, 1]$}\,.
\end{align*}
Hence,  for $\mathcal{L}^1$-a.a.~$t \in \III $  and every $r \in [0, 1]$ we have
\[
\tau(t, \bvc(t, r)) =  \frac{1}{| (1, \partial^{\mathrm{L}}_t \bvc(t)  ) |}, \qquad \bvv(t, \bvc(t, r))= \frac{ \partial^{\mathrm{L}}_t \bvc(t)  }{ | (1, \partial^{\mathrm{L}}_t \bvc(t)   ) |}.
\]
 Ultimately, we deduce \eqref{proxy-Cauchy-1}.
 \par
 Analogously, restricting \eqref{KEY-BV-PROOF} to $\spt(\mathscr{C}(\bvc))$ we obtain
 \[
(0,  \partial^{\mathrm{C}}_t \bvc ) = (\tau(\cdot, \BVSm_{\bvc} (\cdot)), \bvv(t, \BVSm_{\bvc}(\cdot)) ) | (0, \partial^{\mathrm{C}}_t \bvc  ) | \qquad  | \partial^{\mathrm{C}}_t \bvc  |\text{-a.e.~in } \III \,.
\]
Therefore,  we obtain 
\[
\begin{aligned}
\tau(t, \bvc(t, r)) \equiv 0, \qquad
 \bvv(t, \bvc(t, r))  = \frac{\dd  \partial^{\mathrm{C}}_t \bvc }{\dd |   \partial^{\mathrm{C}}_t \bvc |}(t)  \qquad   \text{for } |   \partial^{\mathrm{C}}_t \bvc  |\text{-a.a.\ } t \in \III \text{ and every $r \in [0, 1]$}.
 \end{aligned}
\]
Observing that  $\spt (\partial^{\mathrm{C}}_t \bvc )$ coincides with the image set 
   $ \sft( D_0[\sfyy] {\setminus} D_c[\sfyy] )$
   %
and recalling the representation formula \eqref{e:new-representation} for $\nnu^\perp$, we deduce that 
$ \bvv(t,\bvc(t, r))  = \tfrac{\dd \nnu^\perp}{ \dd |\nnu^\perp|} (t)$ for $|\partial^{\mathrm{C}}_t \bvc |$-a.a.\ $t\in \III$ and every $r \in [0, 1]$, and  \eqref{proxy-Cauchy-2} ensues. 
\par
Finally,  restricting \eqref{KEY-BV-PROOF} to $\spt(\mathscr{J}(\bvc))$
we deduce that for $t \in \mathrm{J}_{\BVSm_{\bvc}}$ and $r \in [0, 1]$
\begin{align}
\left(0,  \partial_{r} \bvc (t, r) \right) &  = 
(\tau (t, \bvc(t, r) )  ,\bvv (t, \bvc(t, r) ) )  \|  \partial_{r} \bvc (t, r) \|  =  (\tau (t, \bvc(t, r) )  ,\bvv (t, \bvc(t, r) ) )   \ell_{\bvc}(t)\,,
\end{align}
 with  $\ell_{\bvc}(t)$ the length of the curve connecting $\bvc(t,0) $ to $\bvc(t,1)$. 
Hence, at every $t \in   \mathrm{J}_{\BVSm_{\bvc}} $ and $\mathcal{L}^{1}$-a.a.~$r \in [0, 1]$ there holds 
\[
\tau (t,  \bvc (t, r))  \equiv 0, \qquad 
\bvv ( t ,  \bvc (t, r) )   =  \frac{ \partial_{r} \bvc(t, r)}{\| \partial_{r} \bvc(t, r)\|} =  \frac{\dd \nnu^\perp}{\dd |\nnu^\perp|} ( t, \bvc(t, r)) \,,
\]
whence \eqref{proxy-Cauchy-3}.  
\end{proof}

  We conclude with the proof of Corollary~\ref{c:unnamed}. 

\begin{proof}[Proof of Corollary~\ref{c:unnamed}]
 The representation formulae in \eqref{true-time-representation} are a consequence of Theorem~\ref{th:BV-representation}.  Let us plug in \eqref{e:representation-mu} the test function $\varphi_\eps (r ,x) = \eta_\eps(r) \psi(x)$, with $\psi \in \Cc(\R^d)$ and $\eta_\eps \in \Cc(\III)$, $0<\eps\ll 1$,  such that  
\[
\begin{cases}
\spt(\eta_\eps) \subset (t-2\eps-\eps^2, t-\eps+\eps^2),
\\
\eta_\eps(r) \equiv \frac1\eps = \max_{[t-2\eps-\eps^2, t-\eps+\eps^2]}\eta_\eps  \ \text{ for all } r \in [t-2\eps,t-\eps],
\end{cases}
\quad \text{for any fixed } t\in (0,+\infty)\,.
\]
On the one hand, we have 
\[
\begin{aligned}
 \iint_{\DST }  & \varphi_\eps(r, x) \, \di \mu_{r} ( x) \, \di r  
\\
&
=  \iint_{(t-2\eps-\eps^2, t-2 \eps) {\times} \R^{d}}  \eta_\eps(r) \psi( x) \, \di \mu_{r} ( x) \, \di r
+ \frac1\eps  \iint_{( t-2 \eps, t-\eps){\times}\R^d}  \psi(x)  
 \, \di \mu_{r} ( x) \, \di r  
 \\
 &
  \qquad +  \iint_{[t-\eps, t- \eps+\eps^2) {\times} \R^{d}} \eta_\eps(r) \psi( x) \, \di \mu_{r}( x) \, \di r   \doteq I_{1,\eps}+I_{2,\eps}+I_{3,\eps}\,.
\end{aligned}
\]
We observe that 
\begin{align}
&
\nonumber
\left|  I_{1,\eps} \right| \leq \frac1\eps \| \psi\|_{\infty} \eps^2 = \eps \| \psi\|_{\infty} \longrightarrow 0 \text{ as } \eps \down 0,
\intertext{and with analogous calculations we have $I_{3,\eps} \to 0$, while}
&
I_{2,\eps} = \frac1\eps \int_{t-2\eps}^{t-\eps} \int_{\R^d}\psi(x) \, \dd \mu_r(x)\,  \dd r \longrightarrow   \int_{\R^d}\psi(x) \, \di \mu_t^- (x)=  \int_{\R^d}\psi(x) \, \di \mu_t (x)
\nonumber
\end{align}
%
%
where we have used the assumed left-continuity of $t\mapsto \mu_t$.
On the other hand,
\[
\begin{aligned}
 \int_{\ARBV}  \int_0^{+\infty} &   \eta_{\eps} (r) \psi( \BVSm_{\bvc}(r) ) \, \di r \, \di \neweta(\bvc) 
 \\
 & =  \int_{\ARBV} \int_{t-2\eps-\eps^2}^{t-2\eps} \eta_\eps(r) \psi(\BVSm_{\bvc}(r)) \, \dd r 
 \, \di \neweta(\bvc)  + \frac1\eps \int_{\ARBV} \int_{t-2\eps}^{t-\eps} \psi(\BVSm_{\bvc}(r)) \, \dd r  \, \di \neweta(\bvc)  
  \RRR 
  \\
  &
  \qquad + \int_{\ARBV} \int_{t-\eps}^{t-\eps+\eps^2} \eta_{\eps}(r)  \psi(\BVSm_{\bvc}(r)) \, \dd r  \, \di \neweta(\bvc)  
   \\
   & \doteq I_{4,\eps}+I_{5,\eps}+I_{6,\eps}\,.
 \end{aligned}
\]
Arguing in the same way as in the above lines we conclude that 
$I_{4,\eps} \to 0$ and $I_{6,\eps} \to 0$ as $\eps\down 0$ 
while (recall that $\BVSm_{\bvc}$ is assumed to be left-continuous)
\[
\lim_{\eps \down 0} \frac1\eps \int_{\ARBV} \int_{t-2\eps}^{t-\eps} \psi(\BVSm_{\bvc}(r)) \, \dd r  \, \di \neweta(\bvc)  
= \int_{\ARBV} \psi( \BVSm_{\bvc} (t))  \, \di \neweta(\bvc)\,.
\]
We thus have the first of \eqref{pm-limits}. 
The very same argument yields  the second of \eqref{pm-limits}.
\par
This concludes the proof.
 \end{proof}

\section{Examples}
\label{s:examples}
In this final section we illustrate our results, and discuss our assumptions, in the context of some examples.
 In what follows we will often work with time-dependent measures
$ (\llp t{p_0}{p_1})_{t\in \III} $ on $\R^h$
 given by the linear combination of two Dirac masses at  $p_0,\, p_1 \in \R^h,$ i.e.
\begin{equation}
\label{general-notation}
\llp t{p_0}{p_1}: =  \max\{1{-}t, 0 \} \delta_{p_0} + \min\{t, 1\} \delta_{p_1}= 
 \left\{ \begin{array}{ll}
t \delta_{p_1}+(1{-}t) \delta_{p_0} &  t\in [0,1]\,,\\
\delta_{p_{1}} & t \in (1, +\infty)\,.
\end{array}
\right.
\end{equation}

 The first example generalizes Example~\ref{ex:Dirac-delta}, as we allow the measure $\nnu$ to transition between two points $\mathbf{x}_0, \mathbf{x}_1 \in \R^{2}$ following a regular and injective curve, not necessarily a segment. \RRR

\begin{example}
\label{ex:6.5}
{\sl
Let $\varrho$ be a regular and injective curve connecting $\mathbf{x}_0$ to $\mathbf{x}_1$, 
$ \alpha_{\varrho} \EEE \colon [0,L_\varrho] \to \R^2$ be its arclength parametrization and 
$\mathbf{t}_\varrho$ its tangent vector.
Consider the measures 
\[
\begin{aligned}
&
\mu=  \mathcal{L}^1  \otimes \mu_t  \qquad \text{with } 
 \mu_t = \llp t{\mathbf{x}_0}{\mathbf{x}_1}, 
\\
&
\nnu =   ( \mathcal{L}^1\mres[0,1])   \otimes \mathbf{t}_\varrho  \mathcal{H}^1 \mres \varrho \,.
\end{aligned}
  \]
  Here, $\nnu^\perp = \nnu$ and
   the minimality condition is satisfied,
   cf.\ Example \ref{ex:4minimality}. 

   In order to illustrate the probabilistic representation of the pair $(\mu,\nnu)$, we consider the fields
   defined  for $\mathcal{L}^1$-a.e.~$t\in \III $ 
 by
   \[
  \tau(t,x) = \frac{\dd \mu}{\dd  |(\mu,\nnu)| } (t,x)= \begin{cases}
  1 & \text{if } x \in \{ \mathbf{x}_0, \mathbf{x}_1\},
  \\
  0 & \text{otherwise,}
  \end{cases}
  \quad
  \bvv(t,x) = \frac{\dd \nnu}{ \dd  |(\mu,\nnu)| } (t,x)= \begin{cases}
  \mathbf{t}_\varrho(x) & \text{if } x\in  \varrho{\setminus}  \{ \mathbf{x}_0, \mathbf{x}_1\} ,
  \\
  0 & \text{otherwise.}
  \end{cases}
  \quad  
    \]
    The  family of curves $(\sfyy_{\bar t})_{\bar t \in [0,1]}$  defined by 
      \begin{equation}
  \label{expression-bargamma-new}
  \sfyy_{\bar t}(s): = \begin{cases}
  (s, \mathbf{x}_0) & \text{if } 0 \leq  s <\bar t\,,
  \\
  (\bar{t},  \alpha_{\varrho} \EEE (s{-}\bar t)) & \text{if } \bar t\leq s \leq \bar t+L_\varrho\,,
  \\
  (s-L_{\varrho}, \mathbf{x}_1) & \text{if } \bar t+L_\varrho <s <  +\infty 
  \end{cases}
  \end{equation}
  provide the solutions to the Cauchy problem \eqref{Cauchy-no-exist}.
  Let $\Upsilon \colon [0,1]\to   \Lipplus 1 \III {  \R^{3} \EEE} $  associate with each $\bar t \in [0,1]$
  the corresponding curve $\sfyy_{\bar t}$.  Following the lines of Example~\ref{ex:Dirac-delta} (cf.~\eqref{e:intro-long-equation}), one can check \RRR that the measure 
  $\eeta = \Upsilon_{\sharp }(\mathcal{L}^1\mres[0,1]) $ fulfills
  $
  \mu =  \fre_{\sharp }(\frt' \eeta_{\mathcal{L}}) 
  $ and $\nnu =  \fre_{\sharp }(\frxx' \eeta_{\mathcal{L} })$.

  As we have done in Example~\ref{ex:Dirac-delta}, we may define $\Xi \colon [0, 1] \to \ARBV (\cZ ;  \R^{2})$ as
  \begin{align*}
      \Xi ( \bar{t}) & := \bvc_{\bar{t}}\,,\\
      \bvc_{\bar {t}} (t, r) & := 
      \begin{cases}
      \mathbf{x}_0 & \text{for $t \in [0, \bar{t})$ and $r \in [0, 1]$},\\[1mm]
       \alpha_{\varrho} \left( L_{\varrho} r \right)  & \text{for $t = \bar{t}$ and $r \in [0, 1]$,}\\[1mm]
       \mathbf{x}_{1} & \text{for $t \in (\bar{t}, +\infty)$ and $r \in [0, 1]$.}
      \end{cases}
  \end{align*}
  We set $\widehat{\eeta} := \Xi_{\sharp} (\mathcal{L}^{1} \mres [0, 1]))$ and notice that $\widehat{\eeta} = \mathscr{{S}}_{\sharp} \eeta$ and $\widehat{\eeta}$ is a Borel probability measure over~$\ARBV (\cZ; \R^{2})$. We define $\Uptheta^{\widehat{\eeta}}$ as in~\eqref{e:Psi-Omega}. Then, for every $\varphi_{0} \in {\rm C}_{\rm b} (\R^{2})$ and every $\bvarphi \in {\rm C}_{\rm b} (\R^{2}; \R^{2})$ we have that
  \begin{align*}
   \int_{\R^{3}_{+}} & (\varphi_{0} (t, x) , \bvarphi(t, x)) \, \di  \Uptheta^{\widehat{\eeta}} (t, x)   = \int_{\ARBV} \langle \tetatrue{\bvc} , (\varphi_{0}, \bvarphi)\rangle\, \di \widehat{\eeta} (\bvc) = \int_{0}^{1} \langle \tetatrue{\bvc_{\bar{t}}} , (\varphi_{0}, \bvarphi)\rangle\, \di \bar{t} 
    \\
    &
    = \int_{0}^{1} \int_{\III} \varphi_{0} (t, \bvc_{\bar{t}} (t, 0)) \, \di t \, \di \bar{t} + \int_{0}^{1} \int_{\III} \bvarphi (t, \bvc_{\bar{t}} (t, 0)) \cdot \partial^{L}_{t} \bvc_{\bar{t}} (t) \, \di t \, \di \bar{t} 
    \\
    &
    \qquad + \int_{0}^{1} \int_{\III}\bvarphi (t, \bvc_{\bar{t}} (t, 0)) \, \di \partial^{C}_{t} \bvc_{\bar{t}} (t) \, \di \bar{t}  + \int_{0}^{1} \sum_{t \in \mathfrak{J}_{\bvc^{\bar{t}}}} \int_{0}^{1} \bvarphi (t, \bvc_{\bar{t}} (t, r))  \cdot \partial_{r} \bvc_{\bar{t}} (t, r) \, \di r \, \di \bar{t}
    \\
    &
    \stackrel{(1)}{=}\int_{0}^{1} \int_{0}^{\bar{t}} \! \varphi_{0} (t, \mathbf{x}_{0}) \, \di t \, \di \bar{t} + \int_{0}^{1} \int_{\bar{t}}^{+\infty} \! \!\!\!\!\!\! \varphi_{0} (t, \mathbf{x}_{1}) \, \di t \, \di \bar{t} 
    \\
    &
    \qquad + \int_{0}^{1}  \int_{0}^{1} \!\bvarphi (\bar{t} , \alpha_{\varrho} (L_{\varrho}r) )  \cdot L_{\varrho} \mathbf{t}_\varrho (\alpha_{\varrho} (L_{\varrho} r) )  \, \di r \, \di \bar{t}
    \\
    &
    \stackrel{(2)}{=} \int_{0}^{1} \int_{t}^{1} \varphi_{0} (t, \mathbf{x}_{0}) \, \di \bar{t} \, \di t + \int_{0}^{1} \int_{0}^{t} \varphi_{0} (t, \mathbf{x}_{1}) \, \di \bar{t} \, \di t + \int_{1}^{+\infty} \varphi_{0} (t, \mathbf{x}_{1}) \, \di t 
    \\
    &
    \qquad + \int_{0}^{1}  \int_{0}^{L_{\varrho}} \bvarphi (\bar{t} , \alpha_{\varrho} (s) )  \cdot  \mathbf{t}_\varrho (\alpha_{\varrho}(s))   \, \di s \, \di \bar{t}
    \\
    &
    =  \int_{0}^{1} (1 - t) \varphi_{0} (t, \mathbf{x}_{0}) \,  \di t + \int_{0}^{1} t \varphi_{0} (t, \mathbf{x}_{1})  \, \di t + \int_{1}^{+\infty} \varphi_{0} (t, \mathbf{x}_{1}) \, \di t 
    \\
    &
    \qquad + \int_{0}^{1}  \int_{\R^{2}} \bvarphi (t , x )  \cdot  \mathbf{t}_\varrho (x)   \, \di (\mathcal{H}^{1} \mres \varrho) (x) \, \di t
    = \int_{\R^{3}_{+}}  (\varphi_{0} (t, x) , \bvarphi (t, x) ) \, \di (\mu, \nnu) (t, x) \,,
  \end{align*}
  where (1) follows from the fact that $\partial^{L}_{t} \bvc^{\bar{t}} = 0$ and $\partial^{C}_{t} \bvc^{\bar{t}} = 0$ and (2) from the
  Fubini Theorem and a change of variable. Finally, we notice that each curve $\bvc^{\bar{t}}$ satisfies~\eqref{proxy-Cauchy}.
}
\end{example}

In our  next example the measures $\mu$ and $\nnu$ are again mutually singular and the minimality  of $\nnu^{\perp}$  does not hold. In this case, the pair
$(\mu,\nnu)$  lacks a probabilistic representation. 
\begin{example}
\label{ex:no-exist}
{\sl
Let $\mathbf{x}_0=(0,0)$, $\Lambda$ be the unitary circle centered at $\mathbf{x}_0$ with tangent vector $\mathbf{t}_\Lambda$,  and
\[
\begin{cases}
\mu :=   \mathcal{L}^1  \otimes \mu_t \text{ with }  \mu_t = \delta_{\mathbf{x}_0} 
\\
\nnu := \delta_{t_0} \otimes \mathbf{t}_\Lambda  \mathcal{H}^1\mres \Lambda  
\end{cases}
\]
with any $t_0 \in  \III$. The pair $(\mu,\nnu)$ solves the continuity equation with initial datum $\mu_0=\delta_{\mathbf{x}_0}$, since for any  $\varphi \in \Cc^1( \III{\times}\R^2)$ there holds
\begin{equation}
\label{CP-ex:no-exist}
\begin{aligned}
\iint_{\III{\times} \R^2} \partial_{t} \varphi (t, x) \, \di \mu(t, x) &  =\int_{\III} \partial_t \varphi(t,\mathbf{x}_0) \dd t = 
- \varphi(0,\mathbf{x}_0) 
\\
&
= -  \iint_{\III{\times} \R^2}  \nabla \varphi  (t, x)  \cdot \mathbf{t}_\Lambda   (t, x)  \dd   (\mathcal{H}^{1}\mres \Lambda)  (x) \dd t - \varphi(0,\mathbf{x}_0) \,.
\end{aligned}
\end{equation}
Since the measures $\mu$ and $\nnu$ are mutually singular (and, in particular,
$\supp(\mu) \cap \supp (\nnu)=\emptyset$),
we have that
 $|(\mu,\nnu)|=\mu+ \RRR | \nnu | = \mathcal{L}^{1} \otimes \delta_{\mathbf{x}_{0}} + \delta_{t_{0}} \otimes  (\mathcal{H}^{1}\mres \Lambda) $ and 
  \begin{equation}
  \label{tau-v-6.2}
  \begin{aligned}
  &
  \tau(t,x) = \frac{\dd \mu}{\dd  |(\mu,\nnu)| } (t,x)= \begin{cases}
  1 & \text{if } x = \mathbf{x}_0, \, 
  \\
  0 & \text{otherwise}
  \end{cases} \quad \text{for }\mathcal{L}^1\text{-a.a. } t\in  \III,
  \\
  &
  \bvv(t,x) = \frac{\dd \nnu}{ \dd  |(\mu,\nnu)| } (t,x)= \begin{cases}
  \mathbf{t}_\Lambda(x) & \text{if } x\in \Lambda, \, t=t_0,
  \\
  0 & \text{otherwise}\,.
  \end{cases}
  \end{aligned}
  \end{equation}
  Now, any probabilistic representation of the pair $(\mu,\nnu)$ would involve a measure $\eeta$ supported on the solutions of the Cauchy problem
  \begin{equation}
\label{Cauchy-no-exist}
\begin{cases}
  \dot{\sfyy} (s)   = (\tau(\sfyy(s)), \bvv(\sfyy(s) )) ,  \  s \in  \III\,, 
\\
 \sfyy(0)  =(0,x), \  x\in\supp(\mu_{0}) = \{\mathbf{x}_0\}\,.
\end{cases} 
\end{equation}
 However, it is immediate to check that 
the solution to \eqref{Cauchy-no-exist}  is given by $ \sfyy(s)  = (s,\mathbf{x}_0)$ for all $s\in  \III $ so that, in particular, 
$\bvv(\sfyy(s)) \equiv 0$ for all $s\in  \III $  and $\sfyy$ never intersects $\spt (\nnu)$. Hence, no probability measure supported on the solution trajectories 
of \eqref{Cauchy-no-exist} could represent the measure $\nnu$  in the sense of Theorem~\ref{t.1}. 
}
\end{example}
A  modification of Example \ref{ex:no-exist} provides a situation in which the measures $(\mu,\nnu)$ are still  mutually singular and  the minimality of $\nnu^{\perp}$  
 does not hold, but it is still possible to provide a probabilistic representation.  Hence,  minimality  is not a necessary condition for the validity of the representation from Theorem \ref{t.1}. 
 \begin{example}
 \label{ex:3-variant}
 {\sl
 Let $\mathbf{x}_1=(1,0)$ and 
\[
\mu =  \mathcal{L}^1 \otimes \delta_{\mathbf{x}_1}, \qquad 
\nnu = \delta_{t_0} \otimes \mathbf{t}_\Lambda  \mathcal{H}^1 \mres \Lambda  \,.
\]
  with the   same notation as in Example \ref{ex:no-exist}. 
 The very same calculations as  in \eqref{CP-ex:no-exist} show that the pair~$(\mu,\nnu) $ solves the continuity equation with initial datum $\mu_0 =  \delta_{\mathbf{x}_1}$. 
 In this case, 
   \begin{equation}
   \label{ex:6.3}
  \tau(t,x) = \frac{\dd \mu}{\dd  |(\mu,\nnu)| } (t,x)= \begin{cases}
  1 & \text{if } x = \mathbf{x}_1,
  \\
  0 & \text{otherwise}
  \end{cases}
   \quad \text{for }\mathcal{L}^1\text{-a.e. } t\in  \III \,, 
  \end{equation}
  and $\bvv$ is as in \eqref{tau-v-6.2}. 
  Let us now examine the Cauchy problem \eqref{Cauchy-no-exist}. Its solution is provided by the curve \RRR $\sfyy^{t_0}\colon  \III  \to  \R^3$  defined by
  \begin{equation}
  \label{expression-bargamma}
   \sfyy^{t_0} (s) : = \begin{cases}
  (s, \mathbf{x}_1) & \text{if } 0 \leq  s <t_0 \,,
  \\
  (t_0, r(s-t_0)) & \text{if } t_0\leq s \leq t_0+2\pi\,,
  \\
  (s-2\pi, \mathbf{x}_1) & \text{if }  s > t_0+2\pi\,,
  \end{cases}
  \end{equation}
    where
  $r\colon [0,2\pi]\to \R^2$ is the arclength parametrization of $\Lambda$, 
   $r(\tau) := (\cos(\tau),\sin(\tau))$. 
   The measure 
   \[
   \eeta: = \delta_{\sfyy^{t_0}} \in \Prob(  \Lipplus 1 \III {\R^{3}})
   \]
  gives the probabilistic representation of  the pair $(\mu,\nnu)$. Indeed,   for every $\phisc \in \rmC(\R^3) $ we have 
  (with the notation $ \sfyy^{t_0}  = (\bar \sft,\bar\sfxx)$)
  \begin{align}
  &
  \begin{aligned}
 \langle   \fre_{\sharp }(\frt'  \eeta_{\mathcal{L}}), \phisc \rangle 
 & = 
 \int_{\III} \phisc(\sfyy^{t_0}(s)) \bar\sft'(s) \, \dd s 
 \\
 & 
 \stackrel{\eqref{expression-bargamma}}{=}
 \int_0^{t_0}  \phisc(s,\mathbf{x}_1) \dd s +  \int_{t_0{+}2\pi}^{+\infty}  \phisc(s{-}2\pi,\mathbf{x}_1) \dd s 
 \\
 & =  \int_0^{t_0}  \phisc(s,\mathbf{x}_1) \dd s +  \int_{t_0}^{+\infty}  \phisc(s,\mathbf{x}_1) \dd s = \langle \mu,\phisc\rangle\,,
 \end{aligned}
\nonumber
\intertext{and  for all $\phive \in \Cc(\R^3;\R^2)$ there holds}
&
\nonumber
\begin{aligned}
 \langle   \fre_{\sharp }(\frxx' \eeta_{\mathcal{L}}),  \phive  \rangle 
 & = 
 \int_{\III} \phive( \sfyy^{t_0}(s) ) \cdot \bar\sfxx'(s)  \, \dd s 
 \\
 & 
 \stackrel{\eqref{expression-bargamma}}{=}
 \int^{t_0{+}2\pi}_{t_0}  \phive(t_0, r(s-t_0)) \cdot  r'(s{-}t_0)\dd s  = \int_0^{2\pi} \phive(t_0,r(s)) \cdot  r'(s) \dd s = 
 \langle \nnu,\phive\rangle\,.
 \end{aligned}
  \end{align}
  \par
  Another instance of a pair $(\mu,\nnu)$ for which  the minimality of $\nnu^{\bot}$  
 does not hold, but  a probabilistic representation still exists, is offered by 
 a variant of the above example in which $\nnu$ is diffuse in time, i.e.\
\[
\mu =   \mathcal{L}^1 \otimes \delta_{\mathbf{x}_1}, \qquad 
\nnu =   \mathcal{L}^1  \otimes \mathbf{t}_\Lambda  \mathcal{H}^1 \mres \Lambda  \,.
\]
In this case, $\tau$ is still given by \eqref{ex:6.3} while
\[
  \bvv(t,x) = \begin{cases}
  \mathbf{t}_\Lambda(x) & \text{if } x\in \Lambda\setminus\{\mathbf{x}_1\}
  \\
  0 & \text{otherwise}
  \end{cases}  \quad \text{for }\mathcal{L}^1\text{-a.a. } t\in  \III \,, 
\]
so that $\tau^2 + |\bvv|^2 \equiv 1$ $|(\mu,\nnu)|$-a.e.\ in $ \III  \times \R^2$. The solutions to the Cauchy problem 
\eqref{Cauchy-no-exist} are  provided by the family of curves
 $ \sfyy^{\bar t} \colon \III   \to\R^3$, $\bar t\in [0,1]$
 defined by \eqref{expression-bargamma} (with $\bar t $ in place of $t_0$). 
  Let us consider the measure  
$\eeta : = \Upsilon_{\sharp }( \mathcal{L}^1 \mres  [0,1])$, where  in this case 
the mapping
$\Upsilon \colon [0,1]\to   \Lipplus 1 \III {\R^{2}}  $  associates 
 with each $ \bar t \in [0,1]$ the curve  $\sfyy^{\bar t}$.  A straightforward adaptation of the above calculations shows that $\eeta$ represents $(\mu,\nnu)$ in the sense of \eqref{e.15}. 
}
 \end{example}
 The following example shows that the representation provided by  Theorem.\ \ref{t.1}  is not
in general  stable for weak$^*$ convergence. 
\begin{example}
\label{ex:4}
{\sl
Recall the notation  $\mathbf{x}_0 =(0,0)$ and  $\mathbf{x}_1 =(1,0)$ from Examples \ref{ex:no-exist} and \ref{ex:3-variant}. 
For every $n\geq 1$ consider the probability measures  on $(0,1)\times\R^2$ 
\[
\mu^n=  \mathcal{L}^1  \otimes \left( \left( 1{-}\tfrac 1n\right) \delta_{\mathbf{x}_0} {+} \tfrac1n \delta_{\mathbf{x}_1}\right)
\]
and $\nnu = \delta_{t_0} \otimes \mathbf{t}_\Lambda  \mathcal{H}^1 \mres  \Lambda  $. In this case as well, the minimality condition is not satisfied, but for each $n\geq 1$ the pair $(\mu^n,\nnu)$ admits a probabilistic representation. Indeed, the associated fields $(\tau_n,\bvv_n)$ are 
\[
\tau_n (t,x) = \begin{cases}
 1  &\text{if } x= \mathbf{x}_0, \mathbf{x}_{1}\,, 
\\
0 & \text{otherwise}
\end{cases}
\quad \mathcal{L}^1\text{-a.a.\ in }  \III \,, 
\]
while $\bvv_n \equiv \bvv$ with $\bvv$  as in \eqref{tau-v-6.2}. The Cauchy problem \eqref{Cauchy-no-exist}  featuring the fields  
$(\tau_n,\bvv) $ is solved by the curves
\[
 \sfyy_{0} (s)  = (s ,\mathbf{x}_0) \qquad \text{for all } s \in \R
\]
and
\RRR \[
\sfyy_{1} ^n(s) = \begin{cases}
\left( s , \mathbf{x}_1\right) & \text{if } s<t_0\,,
\\
(t_0, r (s{-}t_0) ) &\text{if } s \in [t_0,t_0+2\pi n]\,,
\\
\left( s -2\pi n , \mathbf{x}_1\right) & \text{if } s\in [t_0+2\pi n, +\infty)\,.
\end{cases}
\]
We then consider the probability measure
 \[
 \eeta^n =\left( 1{-}\frac1n \right)  \delta_{\sfyy_0} \EEE +\frac1n   \delta_{\sfyy_1^n}\,. \EEE
 \]
  For every  $\phisc \in \mathrm{C}_{\mathrm{c}} (  \III  {\times} \R^{2})$  we have that
 \begin{align*}
 \left\langle    \fre_{\sharp }(\frt' \eeta^{n}_{\mathcal{L}}), \phisc \right\rangle & = \frac{1}{n} \int_{0}^{t_{0}} \phisc(s, \mathbf{x}_{1}) \, \di s + \frac{1}{n} \int_{t_{0} + 2\pi n}^{+\infty} \phisc(s - 2\pi n, \mathbf{x}_{1}) \, \di s + \bigg( 1 - \frac{1}{n} \bigg) \int_{0}^{+\infty} \phisc (s, \mathbf{x}_{0}) \, \di s 
 \\
 &
 = \frac{1}{n} \int_{0}^{+\infty} \phisc(s, \mathbf{x}_{1}) \, \di s+ \bigg( 1 - \frac{1}{n} \bigg) \int_{0}^{+\infty} \phisc(s, \mathbf{x}_{0}) \, \di s = \left\langle \mu^{n}, \phisc \right\rangle.
 \end{align*}
 In a similar way, for $\bvarphi \in \mathrm{C}_{\mathrm{c}} (  \III{\times} \R^{2} ; \R^{2})$ we have that
 \begin{align*}
 \left\langle   \fre_{\sharp }(\frxx' \eeta^{n}_{\mathcal{L}}), \bvarphi \right\rangle & = \frac{1}{n} \int_{t_{0}}^{t_{0} + 2\pi n} \bvarphi(t_{0}, r(s - t_{0}) ) \cdot \mathbf{t}_{\Lambda} (r(s - t_{0})) \, \di s 
 \\
 &
 = \int_{0}^{2\pi} \bvarphi(t_{0}, r(s)) \cdot \mathbf{t}_{\Lambda} (r(s)) \, \di s = \left\langle \nnu , \bvarphi\right\rangle\,.
 \end{align*} 
 
 \par
 Nonetheless, it turns out that, as $n\to \infty$,  
 $\mu^n\weaksto \mu^\infty = \mathcal{L}^1  \otimes  \delta_{\mathbf{x}_0} $ and 
 $\eeta^n \weaksto \eeta^\infty  = \delta_{\sfyy_{0}}$, which represents $\mu^\infty$ but no longer provides a representation for $\nnu^\infty = \nnu$.
}
\end{example}

 In our last example we consider a solution pair $(\mu,\nnu)$ such that $\nnu \ll \mu$. Therefore, 
in  this  absolutely continuous case  \cite[Theorem~8.2.1]{AGS08}  applies. 
We show that representation from our  Theorem~\ref{t.1} follows from that provided by   
\cite[Theorem~8.2.1]{AGS08} via a reparametrization. Hence, Theorem~\ref{t.1} is consistent with 
the classical result.  
\begin{example}
\label{ex:6.6}
\RRR
{\sl
Let us consider the scalar measures
\begin{displaymath}
\mu = \mathcal{L}^{1} \otimes \bigg(
 \frac12 \llp t{0}{1} + 
 \frac12 \mathcal{L}^{1}  \mres (0,1) \EEE \bigg) \,, \qquad \nnu = \mathcal{L}^{1} \otimes  \frac12( \mathcal{L}^{1}  \mres  (0,1) )\,.
\end{displaymath}
Then, $\nnu \ll \mu$ and $\nnu = \bww \mu$ with
\begin{displaymath}
\bww(t, x) := \left\{ \begin{array}{ll}
1 & \text{for $t \in  \III $ and $x \in (0,1)$},\\
0 & \text{elsewhere}.
\end{array}
\right.
\end{displaymath}
The representation of Theorem~\ref{t.1}(2) follows, for instance, from~\cite[Theorem~8.2.1]{AGS08} by an arc-length reparametrization, arguing as in~\eqref{tau-eps}. In particular, we may write $\mu =  \fre_{\sharp } (\frt' \eeta_{\mathcal{L}} )$ and $\nnu = \fre_{\sharp }  \eeta_{\mathcal{L} }$, where the measure $\eeta \in \mathcal{P} (  \Lipplus 1 \III {\R^{2}} )$ is supported on the set of curves $ \sfyy \in  \Lipplus 1 \III {\R^{2}} $ solving the Cauchy problem
\begin{displaymath}
\left\{ \begin{array}{ll}
 \dot{\sfyy}(s)  = (\tau, \bvv) (\sfyy(s)) \,,\\
\sfyy (0) = (0, x_{0}) ,\  x_{0} \in [0, 1]\,,
\end{array}\right.
\end{displaymath}
where
\begin{align}\label{e:tt}
\tau(t, x) & = \left\{
\begin{array}{llll}
1 & \text{for $t \in [0,1)$ and $ x \in \{0, 1\}$},\\
\frac{1}{\sqrt{2}} & \text{for $t \in [0,+\infty)$ and $x \in (0,1)$},\\
1 & \text{for $t \in [1, +\infty)$ and $x =1$},\\
0 & \text{elsewhere},
\end{array}
\right.\\
 \bvv(t, x)&  = \left\{
\begin{array}{ll}
\frac{1}{\sqrt{2}} & \text{for $t \in [0, +\infty)$ and $x \in (0,1)$}\,,\\
0 & \text{elsewhere}.
\end{array}
\right. \label{e:vv}
\end{align}
\par
On the other hand, we may write $\mu = \frac12 \mu_{1} + \frac12 \mu_{2}$ with
\begin{displaymath}
\mu_{1} :=  \mathcal{L}^{1} \otimes  \llp t 01  \,,
\qquad \mu_{2} :=  \mathcal{L}^{1} \otimes (  \mathcal{L}^{1} \mres  (0,1))  \,.
\end{displaymath}
The pairs $(\mu_{1}, 2 \nnu)$ and $(\mu_{2}, 0)$ solve the continuity equation in the sense of Definition~\ref{def:solCE} and both admit a representation in the form~\eqref{e.15} satisfying the conditions of Theorem~\ref{t.1}(2). Precisely, we take $\eeta_{1} \in \mathcal{P}_{1}(\Lipplus 1 \III {\R^{2}})$ as in Example~\ref{ex:Dirac-delta} (with the obvious modifications) and write $\mu_{1} = \fre_{\sharp } (\frt' \mathcal{L}^{1} \otimes \eeta_{1})$ and $\nnu = \fre_{\sharp } (\frxx' \mathcal{L}^{1} \otimes \eeta_{1})$. As for $\mu_{2}$, we consider the measure $\eeta_{2} \in   \mathcal{P}_{1}(\Lipplus 1 \III {\R^{2}} )$ of the form $\eeta_{2} = \Upsilon_{\sharp }(\mathcal{L}^1 \mres [ 0,1]  )$ where $\Upsilon \colon (0,1) \to   \mathcal{P}_{1}(\Lipplus 1 \III {\R^{2}})$ is defined as  $\Upsilon (x) := \sfyy_{x}$  with $\sfyy_{x} (s) := (s, x)$ for every $s \in \III$ and every $x \in (0,1)$. Then, it is easy to see that $\mu_{2} = \fre_{\sharp } (\frt' \mathcal{L}^{1} \otimes \eeta_{2})$.

As a consequence, we obtain the alternative representation
\begin{displaymath}
\mu = \fre_{\sharp } \bigg( \frt' \mathcal{L}^{1} {\otimes} \bigg( \frac12 \eeta_{1} + \frac12 \eeta_{2} \bigg) \bigg), \qquad \nnu = \fre_{\sharp } \bigg( \frxx' \mathcal{L}^{1} {\otimes} \bigg( \frac12 \eeta_{1} + \frac12 \eeta_{2} \bigg) \bigg)\,.
\end{displaymath}
However, we notice that this second representation does not fulfill the conditions of Theorem~\ref{t.1}(2). Indeed, the curves contained in~$\spt(\eeta_{1} ) \cup \spt(\eeta_{2})$ do not solve $\sfyy'(s) = (\tau(\sfyy(s)) , \bvv(\sfyy(s)))$ with $\tau, \bvv$ as in~\eqref{e:tt}--\eqref{e:vv}.
 This shows that the superposition of two representations from Theorem~\ref{t.1} does not, in general, yield a representation in the sense  Theorem~\ref{t.1}. 
}
\end{example}

\appendix

\section{Push forward of vector measures}
Let $\|\cdot\|$ be a strictly convex norm on $\R^h$ and let $\|\cdot\|_*$ denote its dual norm. 
The corresponding duality (multivalued) map $J \colon \R^h\twoheadrightarrow \R^h$ is defined by
\begin{displaymath}
	J_1(\boldsymbol x):=\Big\{\boldsymbol y\in \R^h:
	\|\boldsymbol y\|_*\le 1,\ 
	\boldsymbol{y}\cdot\boldsymbol{x}=\|\boldsymbol{x}\|
	\Big\}.
\end{displaymath}
Notice that if $\|\cdot\|$ is strictly convex then for every $\boldsymbol{x}_1,\boldsymbol{x}_2\in \R^h$
\begin{equation}
	\label{eq:injectivity}
	\|\boldsymbol x_1\|=\|\boldsymbol{x}_2\|,\quad\boldsymbol{y}\in J(\boldsymbol{x}_1)\cap J(\boldsymbol{x}_2)
	\quad\Longrightarrow\quad
	\boldsymbol{x}_1= \boldsymbol{x}_2.
\end{equation}
In fact, setting $\bar{\boldsymbol{x}}:=\frac 12\boldsymbol{x}_1+\frac{1}{2}\boldsymbol{x}_2$, we get
\begin{displaymath}
	\|\bar{\boldsymbol x}\|\ge \boldsymbol{y}\cdot \bar{\boldsymbol{x}}=\frac 12 
	\boldsymbol{y}\cdot {\boldsymbol{x}}_1+
	\frac 12	\boldsymbol{y}\cdot {\boldsymbol{x}}_2=
	\frac 12 \|\boldsymbol{x}_1\|+\frac{1}{2}\|\boldsymbol{x}_2\|
\end{displaymath}
which implies $\boldsymbol{x}_1=\boldsymbol{x}_2$ by the strict convexity of the norm.   In the following statement, $X$ and  $Y$  are two locally compact
topological spaces. 

\begin{lemma}
\label{le:strict-conv}
	Let $\boldsymbol{\vartheta}
	\in \calM(X;\R^h)$, let 
	$\pushm \colon X\to Y$ be a $|\boldsymbol{\vartheta}|$-proper map 
	(i.e.~$|\boldsymbol{\vartheta}|(\pushm^{-1}(K))<+\infty$ for every compact subset $K\subset Y$).
	\begin{enumerate}
		\item If there exists $\alpha\in \calMp(X)$ 
	 	and a Borel map 
	 	$\boldsymbol f \colon Y\to \R^h$ such that 
		 \begin{equation}
			\label{eq:fiber2}
			\boldsymbol{\vartheta}=
			(\boldsymbol{f}\circ \pushm)\, \alpha, 
		\end{equation}
		then 
	\begin{equation}
		\label{eq:fiber}
		|\pushm_\sharp \boldsymbol\vartheta|=
		\pushm_\sharp |\boldsymbol{\vartheta}|.
	\end{equation}
	\item 
	If the norm 
		$\|{\cdot}\|$ on $\R^h$ is strictly convex and \eqref{eq:fiber} holds, then
		\eqref{eq:fiber2} holds 
		with respect to $\alpha:=|\boldsymbol{\vartheta}|$ 
		and $\boldsymbol f$ the density of the polar decomposition of $\pushm_\sharp\boldsymbol{\vartheta}$,
		i.e.
		\begin{equation}
			\label{eq:fiber3}
			\pushm_\sharp \boldsymbol{\vartheta}=
			\boldsymbol{f}|\pushm_\sharp \boldsymbol{\vartheta}|,\quad
			\boldsymbol{\vartheta}=
			(\boldsymbol{f}\circ \pushm)\, |\boldsymbol\vartheta|.
		\end{equation}
	\item If \eqref{eq:fiber2} holds and $\zzeta\prec \boldsymbol{\vartheta}$ then
	\begin{equation}
		\label{eq:push-forward}
		\pushm_\sharp \zzeta\prec \pushm_\sharp\boldsymbol{\vartheta}.
	\end{equation}
	\end{enumerate}
\end{lemma}
\begin{proof}
	First of all, we observe that \eqref{eq:fiber2} implies
	a similar identity for $|\boldsymbol\vartheta|$ up to rescaling $\boldsymbol f$ by a suitable positive Borel function, therefore it is not restrictive to assume that 
	$\alpha=|\boldsymbol\vartheta|$ and therefore $\|\boldsymbol{f}\|=1$
	$\pushm_\sharp |\boldsymbol\vartheta|$-a.e.\ in $Y$. 
	
	 In order to prove \textbf{Claim 1}, notice that, by~\eqref{eq:fiber2}, we have that $\pushm_{\sharp} \boldsymbol{\vartheta} = \pushm_{\sharp} (\boldsymbol{f} {\circ} \pushm) | \boldsymbol{\vartheta}| = \boldsymbol{f} \pushm_{\sharp} |\boldsymbol{\vartheta}|$. Since $\| \boldsymbol{f}\| = 1$ $\pushm_\sharp |\boldsymbol\vartheta|$-a.e.\ in $Y$, we immediately deduce that $|\pushm_{\sharp} \boldsymbol{\vartheta}| \leq \pushm_{\sharp}|\boldsymbol{\vartheta}|$. On the other hand, 
	let us select $\bvarphi\in  L^\infty_{\pushm_\sharp |\boldsymbol{\vartheta}|}(Y ;\R^h)  $ so that
	$\bvarphi(y)\in J(\boldsymbol{f}(y))$ for $\pushm_\sharp|\boldsymbol{\vartheta}|$-a.a.\ $y\in Y$; 
	for every $K$ compact in $Y$ 
	we have
		\begin{align}
		\notag
		 |\pushm_{\sharp} \boldsymbol{\vartheta}|(K)&\geq 
		 \int_K  \bvarphi \cdot \dd  \pushm_{\sharp} \boldsymbol{\vartheta} 
		=
		\int_K \bvarphi\cdot \boldsymbol{f}\,\dd \pushm_{\sharp} |\boldsymbol{\vartheta}|
		= \pushm_{\sharp} |\boldsymbol{\vartheta}| (K) \,.
	\end{align}
	This implies that $\pushm_{\sharp} |\boldsymbol{\vartheta}| = |\pushm_{\sharp} \boldsymbol{\vartheta}|$. 
	
	\par
	In order to prove 
	\textbf{Claim 2} (which is also well known, see e.g. \cite[Lemma 2.4]{Ambrosio-Lisini-Savare06}
	for a similar statement)
	we observe that for $\bvarphi$ as above it holds
	\begin{align}
	\label{eq:useful}
	 |\pushm_{\sharp} \boldsymbol{\vartheta}|(K)& = \int_{K} \| \boldsymbol{f} \| \dd |\pushm_{\sharp} \boldsymbol{\vartheta}| = \int_{K} \bvarphi \cdot \boldsymbol{f} \dd |\pushm_{\sharp} \boldsymbol{\vartheta}| = \int_{K} \bvarphi \cdot \dd \pushm_{\sharp} \boldsymbol{\vartheta} 
	\\
	&
	\nonumber = \int_{\pushm^{-1} (K)} \bvarphi {\circ} \pushm \cdot \dd \boldsymbol{\vartheta} = \int_{\pushm^{-1} (K)} \bvarphi {\circ} \pushm \cdot \boldsymbol{g} \dd |\boldsymbol{\vartheta}| \leq  \int_{\pushm^{-1} (K)}\|  \bvarphi {\circ} \pushm\|_{*} \, \| \boldsymbol{g}\| \dd |\boldsymbol{\vartheta}| \leq \pushm_{\sharp} |\boldsymbol{\vartheta}| (K)\,,
	\end{align}
	where $\boldsymbol{\vartheta} = \boldsymbol{g} |\boldsymbol{\vartheta}|$. Then, \eqref{eq:fiber} and \eqref{eq:useful} yield that 
	$(\boldsymbol\varphi{\circ}\pushm)\cdot \boldsymbol{g}=\|\boldsymbol{g}\|$ holds
	$|\boldsymbol{\vartheta}|$-a.e.~on $\pushm^{-1}(K)$ so that $\boldsymbol\varphi(\pushm(x))\in J(\boldsymbol g(x))$ 
	for $|\boldsymbol{\vartheta}|$-a.a.~$x\in \pushm^{-1}(K)$. 
	On the other hand, by construction $\boldsymbol\varphi(\pushm(x))\in J(\boldsymbol f(\pushm(x)))$ 
	so that \eqref{eq:injectivity} yields $\boldsymbol f(\pushm(x))=\boldsymbol{g}(x)$ for 
	$|\boldsymbol{\vartheta}|$-a.e.~$x\in \pushm^{-1}(K)$. Exhausting $Y$ with a countable sequence of compact sets,
	we conclude.
	
	Let us eventually consider \textbf{Claim 3}: we can write $\zzeta=\lambda\boldsymbol{\vartheta}$
	for a Borel map $\lambda$ with values in $[0,1]$. 
	We can select the Euclidean norm and we thus have
	\begin{displaymath}
		|\zzeta|=\lambda \,|\boldsymbol{\vartheta}|,\quad
		\zzeta=\lambda \boldsymbol{\vartheta}=\lambda \boldsymbol{f}{\circ}\pushm\,|\boldsymbol{\vartheta}|=
		\boldsymbol{f}{\circ}\pushm\,|\zzeta|
	\end{displaymath}
	and therefore, by Claim 1, 
	\begin{displaymath}
		\pushm_\sharp |\zzeta|=
		|\pushm_\sharp \zzeta|.
	\end{displaymath}
	Similarly $\pushm_\sharp|\boldsymbol\vartheta-\zzeta|=
	\pushm_\sharp|(1-\lambda)\boldsymbol\vartheta|=|\pushm_\sharp 
	((1-\lambda)\boldsymbol\vartheta)|=
	|\pushm_\sharp 
	(\boldsymbol\vartheta-\zzeta)|$.
	We deduce that 
	\begin{displaymath}
		|\pushm_\sharp \zzeta|+|\pushm_\sharp 
	(\boldsymbol\vartheta-\zzeta)|=
	\pushm_\sharp \big(|\zzeta|+|\boldsymbol\vartheta-\zzeta|\big)=
	\pushm_\sharp |\boldsymbol\vartheta|=
	|	\pushm_\sharp \boldsymbol\vartheta|
	\end{displaymath}
	so that $\pushm_\sharp\zzeta\prec \pushm_\sharp\boldsymbol{\vartheta}$.
\end{proof}

\section{Topological properties of functions spaces}
\label{old:appD}
\noindent Recall that $\rmC(\III;\R^h)$ denotes the space
of $\R^h$-valued continuous paths endowed with the topology of uniform convergence on
compact sets of $\III$.
 \begin{lemma}
 \label{le:C-is-Polish}
The 
metric 
\begin{equation}
\label{metric-4-uniform-compact}
D(\sfyy_{1},\sfyy_{2} ): = \sum_{n=0}^\infty 2^{-n} (\| \sfyy_{1} - \sfyy_{2} \|_{\infty, n}\wedge 1), \qquad \text{with } \| \sfyy_{1} - \sfyy_{2} \|_{\infty, n} = \max_{s\in [0,n]} \norm{ \sfyy_{1} (s) - \sfyy_{2} (s)},
\end{equation}
makes the
topological space  $\rmC(\III;\R^{h})$ 
complete, separable, and induces on  $\rmC(\III;\R^{h})$ 
  the  topology of uniform convergence on compact sets.
In particular $\rmC(\III;\R^h)$ is Polish.
\end{lemma}
\begin{proof}
It is easy to check that $(\rmC(  \III  ;\R^{d+1}),D)$  is complete. It is also separable: indeed, for every $n\geq 1$ the space $\rmC([0,n];\R^{d+1})$  has a countable and dense subset $(\sfyy_{i})_{i\in I_n}$; we then extend each $\gamma_i$ to a function $\tilde{\sfyy}_i \in \rmC(  \III  ;\R^{d+1})$ by setting
$\tilde{\sfyy}_i(x) \equiv \sfyy_{i} (n)$  for $x\in (n,+\infty)$. 
Then,  the set 
$\bigcup_{n \geq 1} (\sfyy_{i})_{i\in I_n}$ is countable and dense in  $(\rmC(  \III  ;\R^{d+1}),D)$.
\end{proof}

\section{Glueing properties}
We establish a useful generalization of the glueing Lemma
\cite[Lemma 5.3.2, 5.3.4]{AGS08}.

\begin{lemma}
	\label{le:glueing}
	Let $N\in \N\cup\{\infty\}$, 
	$\mathcal I(N):=\{1,2,\cdots,N\}$ if  $N\in \N$ and
	$\mathcal I(\infty):=\N$ ($N=\infty$),
	let $X,X^i,Y^j$,  be Polish spaces and let 
	$\mathsf p^i \colon X\to X_i$, 
	$\mathsf R^{i} \colon X^{i}\to Y^i$,
	$\mathsf L^{j+1}\colon X^{j+1}\to Y^j$,  be Borel maps 
	for $i\in \mathcal I(N)$ and $j\in \mathcal I(N{-}1) .$
	We set $\boldsymbol X:=\Pi_{i\in I} X^i$,
	$\boldsymbol X_0:=\big\{\boldsymbol x=(x_i)_{i\in I}\in \boldsymbol X: 
	\mathsf R^i(x_i)=\mathsf L^{i+1}(x_{i+1}), \ i\in \mathcal I(N{-}1)
	\big\}$,
	and we suppose that the  image of the map 
	$\boldsymbol{\mathsf p}:X\to \boldsymbol X$, 
	$\boldsymbol{\mathsf p}(x):=(\mathsf p^i(x))_{i\in \N}$ 
	contains $\boldsymbol X_0.$
	
	If $\mu^i\in \sP(X^i)$, $i\in \mathcal I(N)$, satisfy the compatibility conditions
	\begin{equation}
		\label{eq:compatibility1}
		\mathsf R^i_\sharp\mu^i=\mathsf L^{i+1}_\sharp \mu^{i+1},
		\quad 
		i\in \mathcal I(N{-}1),
	\end{equation}
	then there exists $\mu\in \sP(X)$ such that 
	$\mathsf p^i_\sharp\mu=\mu^i$
	for every $i\in \mathcal I(N).$	
\end{lemma}
\begin{proof}
We consider the case $N=\infty$, $\mathcal I(N)=\mathcal I(N{-}1)=\N$; 
the argument in the finite case is even simpler.

	Let $d_i$ be a metric inducing the topology of $Y^i$ taking values in $[0,1]$.
	Let us set $\nu^i:= \mathsf R^i_\sharp \mu^i=\mathsf L^{i+1}_\sharp \mu^{i+1} \in \sP(Y^i)$, 
	$\hat\mu^{i\to}:=(\mathsf i_{X^i},\mathsf R^i)_\sharp \mu^i\in 
	\Gamma(\mu^i,\nu)\subset \sP(X^i\times Y)$,
	$\hat\mu^{i\leftarrow}:=(\mathsf i_{X^{i+1}},\mathsf L^{i+1})_\sharp \mu^{i+1}\in 
	\Gamma(\mu^{i+1},\nu)\subset \sP(X^{i+1}\times Y)$.
	Notice that 
	\begin{equation}
		\label{eq:metric-id}
		\int_{X^i\times Y}d_i(\mathsf R^i(x_i),y)\,\dd \hat 	\mu^{i\to}(x_i,y)=0,\quad
		\int_{X^{i+1}\times Y}d_i(\mathsf L^{i+1}(x_{i+1}),y)\,\dd \hat 	\mu^{i\leftarrow}(x_{i+1},y)=0,\quad
		i\in \N.
	\end{equation}
	By the standard glueing Lemma (see e.g. \cite[Lemma 5.3.2]{AGS08})
	there exist $\beta^i\in \sP(X^i\times X^{i+1}\times Y^i)$ 
	such that 
	\begin{displaymath}
		\pi^{i\to}_\sharp \beta^i=\hat\mu^{i\to},\quad 
		\pi^{i\leftarrow}_\sharp \beta^i=\hat\mu^{i\leftarrow},\quad 
		\text{where}\quad 
		\pi^{i\to}(x_i,x_{i+1},y):=(x_i,y),\quad 
		\pi^{i\leftarrow} (x_i,x_{i+1},y):=(x_{i+1},y).
	\end{displaymath}
	In particular, using \eqref{eq:metric-id} we deduce 
	\begin{align*}
		\int d_i(\mathsf R^i(x_i),\mathsf L^{i+1}(x_{i+1}))\,\dd\beta^i
		&\le 
			\int d_i(\mathsf R^i(x_i),y)\,\dd\beta^i
+
		\int d_i(y,\mathsf L^{i+1}(x_{i+1}))\,\dd\beta^i
				\\&=
				\int d_i(\mathsf R^i(x_i),y)\,\dd\hat \mu^{i\to}
					+
				\int d_i(\mathsf L^{i+1}(x_{i+1}),y)\,\dd\hat\mu^{i\leftarrow}
				=0,
	\end{align*}
	so that $\alpha^{i}:=\pi^{i}_\sharp \gamma$
	(where $\pi^{i}(x_i,x_{i+1},y)=(x_i,x_{i+1})$)
	is concentrated on $\{(x_i,x_{i+1})\in X^i\times X^{i+1}:
	\mathsf R^i(x_i)=\mathsf L^{i+1}(x_{i+1})\}.$
	
	We can then use the glueing Lemma \cite[5.3.4]{AGS08} to find
	a probability measure $\boldsymbol \alpha\in \sP(\boldsymbol X)$
	such that $\boldsymbol \pi^{i}_\sharp \boldsymbol\alpha=\alpha^i$
	for every $i\in \N$, where $\boldsymbol \pi^i(\boldsymbol x)=(x_i,x_{i+1}).$
	Clearly, $\boldsymbol\alpha$ is concentrated on $\boldsymbol X_{0}$;
	since the image of $\boldsymbol{\mathsf p}$ contains $\boldsymbol X_{0} $,
	we can find $\mu\in \sP(X)$ such that 
	$\boldsymbol{\mathsf p}_\sharp \mu=\boldsymbol \alpha$,
	so that $\mathsf p^i_\sharp \mu=\mu^i.$
\end{proof}
%

\section{A measurability result}
\label{app:meas-R}
  In this section we prove the measurability of the mapping 
$\Rep$ from \eqref{def:Rep}, 
 coming into play in the proof of Lemma \ref{le:rescaling}. 
In the proof we shall resort to the following functional version of the \emph{monotone class theorem}, which we 
record here 
(in a shortened and simplified version, adapted to our usage) for the reader's convenience.
We refer to
 \cite[Thm.\ 2.12.9]{Bogachev07} for the general statement.
\begin{theorem}[Functional monotone class theorem]
Let $\boldsymbol{H}$  be a class of real functions on a set $O \subset \R^k$
containing $f\equiv 1$. Suppose that $\boldsymbol{H}$ is closed with respect to the formation of uniform and monotone limits  and that $f \equiv 1 \in \boldsymbol{H}$.
Let  $\boldsymbol{H}_0\subset  \boldsymbol{H}$ be a subclass closed with respect to multiplication (i.e., $fg \in \boldsymbol{H}_{0}$ for every $f, g \in \boldsymbol{H}_{0}$).
\par
Then, $\boldsymbol{H}$ contains all bounded functions measurable with respect to the $
\sigma$-algebra generated by $\boldsymbol{H}_0$. 
\end{theorem}
We will then prove the following result.
\begin{lemma}
\label{le:u-measurable}
The mapping  $\Rep \colon \Lipplus k\III\Rdpu\to\Lipplus{\widehat k}\III\Rdpu
$ from 
\eqref{def:Rep}
 is Borel. 
\end{lemma}
\begin{proof}
 Let $\theta$ be a Borel measurable function $\theta\colon \R^{d+1} \to (0, +\infty)$ such that $ c^{-1} \leq \theta \leq c$ for some $c \in [1, +\infty)$ (cf.\ \eqref{labelliamotheta}). For every function $\zeta \colon \R^{d+1} \to (0, +\infty)$ such that $\zeta \geq c_{\zeta}>0$ we consider the mapping 
\[
\boldsymbol{\mathsf{F}}_\zeta \colon  \Lipplus k\III\Rdpu\to L^{1}_{\loc} (\III), \qquad \sfyy \mapsto \frac1{\zeta(\sfyy)}\,.
\]
 We consider $ L^{1}_{\loc} (\III)$ endowed with the (Fr\'echet, hence metrizable) topology that induces  the $L^1$ convergence on the compact subsets of $\III$, 
 whereas we recall that 
 $\Lipplus k\III\Rdpu$ is with the (metrizable) topology of the convergence on compact subsets. 
 We claim that for every  $\zeta$ as above we have
 \begin{equation}
 \label{FBorel}
 \boldsymbol{\mathsf{F}}_\zeta  \text{ is Borel.}
 \end{equation}
   To show this, we introduce the class $\boldsymbol{H}$ as 
\[
\zeta \in 
\boldsymbol{H} \quad
\Longleftrightarrow \quad 
 \begin{cases}
  \zeta\colon \Rdpu \to (0,+\infty)  \text{ is Borel}, 
  \\
     \exists\, c_{\zeta}>0 \ \ \zeta \geq c_{\zeta}  \text{ in } \Rdpu, 
    \\
\boldsymbol{\mathsf{F}}_\zeta \colon  \Lipplus k\III\Rdpu\to  L^{1}_{\loc} (\III) \text{ is Borel}. 
\end{cases}
\]
Now, we clearly have that $\zeta \equiv 1$ belongs to $\boldsymbol{H}$ and that  $\boldsymbol{H}$ is closed w.r.t.\ monotone limits of uniformly bounded sequences. 
 Moreover,  it is immediate to check that $\boldsymbol{H}$ contains the  set
 \begin{displaymath}
 \boldsymbol{H}_0 := \{ \zeta \in {\rm C}(\R^{d+1}): \, \exists \, c>0 \ \ \zeta \geq c \text{ in } \Rdpu\}\,.
 \end{displaymath}
  Furthermore, $\boldsymbol{H}_{0}$ is closed with respect to multiplication. Hence, by the monotone class theorem the family $\boldsymbol{H}$ contains all positive Borel functions bounded away from $0$. In particular, $\theta \in \boldsymbol{H}$ and \eqref{FBorel} follows. 
 \par
We now consider the mapping 
\[
\boldsymbol{\mathsf{A}}_\theta \colon \Lipplus k\III\Rdpu\to  {\rm C}(\III), 
 \qquad \boldsymbol{\mathsf{A}}_\theta(\sfyy)(t):= \Theta_{\sfyy}(t) = \int_0^t \frac 1{\theta(\sfyy(r))} \dd r =  \int_0^t \boldsymbol{\mathsf{F}}_\theta(\sfyy(r)) \dd r\,.
\]
 In particular, we notice that for every $\sfyy \in \Lipplus k\III\Rdpu$ we have
\begin{displaymath}
\boldsymbol{\mathsf{A}}_\theta (\sfyy) \in \mathrm{biLip}_{c,c^{-1}}(\III): = \bigg\{ g: \III \to  [0,+\infty) \, : \ g(0)=0, \  g \text{ is bi-Lipschitz, with  } \ \  \frac1{c} \leq g' \leq c \text{ in } \III \bigg\},
\end{displaymath}
where $c \geq 1$ the constant from  \eqref{labelliamotheta}. Recalling that  ${\rm C}(\III)$ is endowed with the (metrizable) topology that induced the uniform convergence on compact sets, we have that the map $\boldsymbol{\mathsf{A}}_\theta$ is the composition  of the Borel mapping $\boldsymbol{\mathsf{F}}_\theta$ with the  function 
\[
A\colon  L^{1}_{\loc} (\III)  \to {\rm C} (\III), \qquad f \mapsto A(f) \text{ with } A(f)(t):=  \int_0^t f(r) \dd r\,.
\]
Since $A$ is continuous, we  have that 
 $\boldsymbol{\mathsf{A}}_\theta$ is Borel.
\par
 Finally,  we show that the mapping 
 \[
 \boldsymbol{\mathsf{L}}_\theta \colon \Lipplus k\III\Rdpu\to  {\rm C} (\III) \,,  \qquad \sfyy \mapsto \ell_{\sfyy}= \Theta_{\sfyy}^{-1}\,,
 \]
 is Borel measurable. We notice that $ \boldsymbol{\mathsf{L}}_\theta (\sfyy)$ is well defined, as $\boldsymbol{\mathsf{A}}_\theta (\sfyy)$ is invertible for every $\sfyy \in  \Lipplus k\III\Rdpu$.  Moreover, $ \boldsymbol{\mathsf{L}}_\theta (\sfyy)$ is the composition of $\boldsymbol{\mathsf{A}}_\theta$ with the 
 inversion operator $\boldsymbol{\mathsf{I}}\colon  \mathrm{biLip}_{c,c^{-1}}(\III) \to  \mathrm{biLip}_{c,c^{-1}}(\III)$. Now,   $\boldsymbol{\mathsf{I}}$ is  continuous: 
 indeed, let  $(g_n)_n, g \in  \mathrm{biLip}_{c,c^{-1}}(\III)  $ with $g_n \to g $ uniformly con compact subsets of~$\III$.  Consider $g^{-1}, (g_n^{-1})_n \subset \mathrm{biLip}_{c,c^{-1}}(\III)$. Taking into account that 
 $(g_n^{-1})_n$ is bounded in $ L^\infty_{\mathrm{loc}}(\III)$ with 
 $ \frac1{c} \leq (g_n^{-1})' \leq c$, 
 in order to check that  $g_n^{-1} \to g^{-1}$ on compact subsets of~$\III$ 
 it is sufficient to show that $g_n^{-1} \to g^{-1}$   pointwise in $\III$.  Hence, let  $r \in \III$ and  
  $s_n := g_n^{-1}(r) $, i.e.\
  $g_n(s_{n}) = r$.  Since $(s_n)_n$ is bounded, it admits a subsequence
 $(s_{n_k})_k$
  converging to some  $s^*$. Recalling that $(g_{n})_n$
converges uniformly to $g$ on the compact subsets of $\III$, we gather that   $r= g_{n_k}(s_{n_k}) \to g(s^*)$. Hence, $s^*= g^{-1}(r)$.  As  the limit 
does not depend on the extracted subsequence, then we have that the \emph{whole} sequence $(s_n=g_n^{-1}(r))_n$ converges to 
$g^{-1}(r)$. 
Therefore, $\boldsymbol{\mathsf{L}}_\theta = \boldsymbol{\mathsf{I}}{\circ} \boldsymbol{\mathsf{A}}_\theta$ is the composition of a Borel and of a continuous mapping: a fortiori, it is Borel.
\par
Ultimately, the map 
 $
 \Rep \colon \Lipplus k\III\Rdpu\to\Lipplus{\widehat k}\III\Rdpu$ defined by 
$ \sfyy \mapsto \sfyy\circ\ell_{\sfyy} = \sfyy\circ\boldsymbol{\mathsf{L}}_\theta(\sfyy)$  is  Borel. 
\end{proof}

\section{Proof of Lemma~\ref{l:mapsST}}
\label{another appendix}

We divide the proof in 2 steps, proving the two equalities in~\eqref{inverse-of-another} separately.

\STEP {1: $\mathscr{S} (\mathscr{T}(\bvc)) = \bvc$ for $\bvc \in \ARBV(\cZ ; \R^{d})$.}  Let us denote by $\sfyy = (\sft, \sfxx) = \mathscr{T} (\bvc)$, constructed according to \eqref{e:t(s)}--\eqref{yy4u}. Then, we have that for $t \in \III$
 \begin{align*}
 \sfs_{\sfyy}^{-}(t) & := \sup\, \{ s \in  \III  : \, \sft(s) <t\} =  \sup\, \{ s \in  \III  : \, \inf\{ \tau \in \III:\, L^{+}_{\bvc} (\tau) >s\} <t\}\,.
 \end{align*}
 For every $s < L^{-}_{\bvc} (t)$ we have that $\inf\{ \tau \in \III:\, L^{+}_{\bvc} (\tau) >s\} < t$, since $L^{+}_{\bvc} (\tau) < L^{-}_{\bvc} (t)$ for $\tau < t$ and $L^{+}_{\bvc} (\tau) \to L^{-}_{\bvc} (t)$ as $\tau \nearrow t$. Hence, $ L^{-}_{\bvc} (t) \leq \sfs_{\sfyy}^{-}(t) $. On the other hand, for $s > L^{-}_{\bvc} (t)$ we get that 
 \begin{align*}
 \inf\{ \tau \in \III: \,  L^{+}_{\bvc} (\tau) >s\} \geq  \inf\{ \tau \in \III: \,  L^{+}_{\bvc} (\tau) > L^{-}_{\bvc} (t) \} \geq t\,.
 \end{align*}
 Hence, $L^{-}_{\bvc} (t) = \sfs_{\sfyy}^{-}(t) $. With a similar argument we infer $L^{+}_{\bvc} (t) = \sfs_{\sfyy}^{+}(t) $. 
 
 Recalling~\eqref{e:function S}, we write for $(t, r) \in \cZ$
 \begin{align*}
 \mathscr{S} (\sfyy) (t, r)&  = \mathscr{S} (\mathscr{T} (\bvc) ) ( t , r ) = \bvc \Big ( \sft \big( \sfs^{-}_{\sfyy} (t) + r (\sfs^{+}_{\sfyy} (t) - \sfs_{\sfyy}^{-} (t)) \big) , \sfr \big( \sfs^{-}_{\sfyy} (t) + r (\sfs^{+}_{\sfyy} (t) - \sfs_{\sfyy}^{-} (t) ) \big) \Big)  
 \\
 &
 = \bvc \Big( t, \sfr \big( L^{-}_{\bvc} (t) + r (L^{+}_{\bvc} (t) - L_{\bvc}^{-} (t) ) \big) \Big) = \bvc (t, r)\,,
 \end{align*}
 where we have used that $\sft(s) = t$ for every $s \in [\sfs^{-}_{\sfyy} (t) , \sfs^{+}_{\sfyy} (t)]$ and the definition of $\sfr$ in~\eqref{e:r(s)}.
 
\STEP{2: $\mathscr{T} (\mathscr{S} (\sfyy) ) = \sfyy$ for $\sfyy = (\sft, \sfxx)  \in  \ALip  \III {\R^{d+1}}$.}  For $(t, r) \in \cZ$ we notice that
\begin{align}
\label{e:alternative-S}
\mathscr{S} (\sfyy) (t, r) & = \sfxx( \sfs_{\sfyy}^{-} (t) + r (\sfs^{+}_{\sfyy} (t) - \sfs^{-}_{\sfyy} (t))) \,.
\end{align}
To shorten the notation, we set $\bvc = \mathscr{S}(\sfyy)$ and $(\sft_{\bvc}, \sfxx_{\bvc}) = \mathscr{T} (\mathscr{S}(\sfyy))$. Since $\| \sfyy'\| = 1$ a.e.~in~$\III$, it holds
\begin{align}
\label{e:LLL}
L^{\pm}_{\bvc} (t) & = \int_{0}^{\sfs^{\pm}_{\sfyy} (t)} \| \sfyy ' (s) \| \, \di s = \sfs^{\pm}_{\sfyy} (t) \qquad \text{for $t \in \III$,}\\
L_{\bvc} (t, r) & = \sfs_{\sfyy}^{-} (t) + r ( \sfs^{+}_{\sfyy} (t) - \sfs^{-}_{\sfyy} (t)) \qquad \text{for $(t, r) \in \cZ$.} \label{e:LLL2}
\end{align}
By definition of $\mathscr{T}$ and by the characterization of $L^{\pm}_{\bvc}$ above we have that for $s \in \III$
\begin{align*}
\sft_{\bvc} (s) = \inf\, \{ t \in \III: \, s^{+}_{\sfyy} (t) >s\}\,.
\end{align*}
In particular, it is immediate to see that $\sft ( s) \geq \sft_{\bvc} (s)$. By contradiction, if $\sft(s) > \sft_{\bvc} (s)$, then it must be $\sfs_{\sfyy}^{+} (t) >s$ for every $t \in (\sft_{\bvc} (s) , \sft(s))$, which implies $\sfs_{\sfyy}^{-} (\sft(s)) >s$, whence \eqref{e:ineq-sfs}. Thus, $\sft = \sft_{\bvc}$ in~$\III$.

We now consider the second component $\sfxx_{\bvc}$ of $\mathscr{T} (\mathscr{S} (\sfyy))$. We recall that, in view of~\eqref{e:r(s)} and~\eqref{e:LLL}, it holds for $s \in \III$
\begin{align}
\label{e:alternative-r}
\sfr(s) & =   \frac{s - \sfs^{-}_{\sfyy} (\sft(s))}{ \sfs^{+}_{\sfyy} (\sft(s)) - \sfs^{-}_{\sfyy} (\sft(s))} \qquad \text{if $\sfs^{+}_{\sfyy} (\sft(s)) \neq \sfs^{-}_{\sfyy} (\sft(s))$,}\\[3mm]
\sfxx_{\bvc} (s) &  =  \bvc ( \sft_{\bvc} (s) , \sfr(s)) = \bvc (\sft(s) , \sfr(s)) 
\\
  & =   \sfxx (\sfs^{-}_{\sfyy} (\sft(s)) ) + (\sfs^{+}_{\sfyy} (\sft(s)) - \sfs^{-}_{\sfyy} (\sft(s))) \int_{0}^{\sfr(s)}  \sfxx'(\sfs^{-}_{\sfyy} (\sft(s) ) + \ell (\sfs^{+}_{\sfyy} (\sft(s) ) - \sfs^{-}_{\sfyy} (\sft(s) ) ) ) \, \di \ell\,.\nonumber
\end{align}
Hence, if $\sfs^{+}_{\sfyy} (\sft(s)) = \sfs^{-}_{\sfyy} (\sft(s))$, from~\eqref{e:ineq-sfs} we immediately conclude that $\sfxx_{\bvc} (s) = \sfxx (s)$. If $\sfs^{+}_{\sfyy} (\sft(s)) > \sfs^{-}_{\sfyy} (\sft(s))$,~\eqref{e:alternative-S} and~\eqref{e:alternative-r} yield $\sfxx_{\bvc} (s) = \sfxx(s)$. Hence, $\mathscr{T} (\mathscr{S} (\sfyy) ) = \sfyy$.



 \section{Proof of Proposition \ref{p:identify-theta-alpha}}
 \label{app:meas}
 \noindent
$\vartriangleright\, $ \eqref{teta-equals-alpha} \& \eqref{e:631}. 
  We will prove the representation formula \eqref{teta-equals-alpha} for 
$\tetatrue{\bvc}$
  by showing \eqref{e:631}. To do so, we need to relate the curve $\BVSm_{\bvc}$ to the trajectory $\mathcal{T}(\bvc) = \sfyy =  (\sft, \sfxx)$. 
  In fact,
  \[
  \BVSm_{\bvc}(t) = \bvc(t,0) = \sfxx( \sfs_{\sfyy}^{-}(t)) \qquad  \text{for all } t \in \III\,.
  \]
  Moreover, $ \mathfrak{J}_{\bvc}= \mathrm{J}_{\sfs_{\sfyy}^{\pm}}$, and 
   $\mathscr{C}_{\sfyy} := \sft^{-1} (C_{\bvc}) = \sft^{-1}(\III {\setminus} \mathfrak{J}_{\bvc}) $ 
   is the set where $\sft(\cdot)$ is injective. In particular, $\III \setminus \mathscr{C}_{\sfyy}$ is union of the intervals $[\sfs_{\sfyy}^{-} (t) , \sfs_{\sfyy}^{+} (t)]$. 
   \par
   Now, 
for every $\boldsymbol{\varphi} \in \mathrm{C}_{\mathrm{c}} (\R^{d+1}_{+}; \R^{d})$ we have
  \begin{align}
  \label{e:upphi-jumps}
\langle  \sum_{t \in \mathfrak{J}_{\bvc}}  \delta_t {\otimes}  \boldsymbol{j}_{t, \bvc}, \bvarphi \rangle = 
  \sum_{t \in \mathfrak{J}_{\bvc}}  \int_{0}^{1} \boldsymbol{\varphi} (t, \bvc (t,r) ) \cdot \partial_{r} \bvc (t, r) \, \di r & =   \sum_{t \in \mathfrak{J}_{\bvc}}  \int_{\sfs_{\sfyy}^{-} (t)}^{\sfs_{\sfyy}^{+} (t)} \boldsymbol{\varphi} (t, \bvc (t, r(s) ) ) \cdot \sfxx'(s) \, \di s 
  \\
  &
  =  \sum_{t \in \mathfrak{J}_{\bvc}} \int_{\sfs_{\sfyy}^{-} (t)}^{\sfs_{\sfyy}^{+} (t)} \boldsymbol{\varphi} (\sft (s) , \bvc (\sft (s) , r(s) ) ) \cdot \sfxx'(s) \, \di s \nonumber
  \\
  &
  = \sum_{t \in \mathfrak{J}_{\bvc}}  \int_{\sfs_{\sfyy}^{-} (t)}^{\sfs_{\sfyy}^{+} (t)} \boldsymbol{\varphi} (\sfyy (s)  ) \cdot \sfxx'(s) \, \di s \,,\nonumber
  \end{align}
  where, in the last equality, we have used the fact that $\sft(s) \equiv t$ for every $s \in  [\sfs_{\sfyy}^{-} (t), \sfs_{\sfyy}^{+} (t) ]$. Hence, \eqref{e:615} follows. 
  \par

In order to show   \eqref{e:616}, we start by recalling  that   for every $\zeta \in \mathrm{C}_{\mathrm{c}} (\III)$ and $\boldsymbol{\zeta} \in \mathrm{C}_{\mathrm{c}} (\III;  \R^{d})$ we further have (see, e.g.,~\cite[Proposition~6.11]{MielkeRossiSavare12a})
  \begin{align}
  \label{e:upphi1}
 & 
 \int_{ \III } \zeta ( t ) \, \di t   = \int_{ \mathscr{C}_{\sfyy} } \zeta ( \sft( s ) ) \, \sft'(s) \, \di s \,,
 \\
  &
   \label{e:upphi2} \int_{ C_{ \bvc}}  \boldsymbol{\zeta} (t) \, \di ( ( \BVSm_{\bvc} )'_{\mathcal{L}^{1}} + ( \BVSm_{\bvc} )'_{\mathrm{C}}) (t)  =  \int_{ C_{ \BVSm_{\bvc} }}  \boldsymbol{\zeta} (t) \, \di ( ( \BVSm_{\bvc} )'_{\mathcal{L}^{1}} + ( \BVSm_{\bvc} )'_{\mathrm{C}}) (t)  = \int_{\mathscr{C}_{\sfyy}}    \boldsymbol{\zeta} (\sft(s) ) \cdot \sfxx'(s) \, \di s 
  \end{align}
   (in \eqref{e:upphi2} and in what follows in this proof, we use the more compact notation $ \BVSm_{\bvc}'$,  $( \BVSm_{\bvc} )'_{\mathcal{L}^{1}}$, and $ ( \BVSm_{\bvc} )'_{\mathrm{C}}$, in place of 
  \eqref{partialL} and \eqref{partialC}). 
For $\varphi_{0} \in \mathrm{C}_{\mathrm{c}} (\R^{d+1}_{+}) $ and $\boldsymbol{\varphi} \in \mathrm{C}_{\mathrm{c}} (\R^{d+1}_{+}; \R^{d})$, we test~\eqref{e:upphi1} and~\eqref{e:upphi2} with $\zeta_{\eps} := \varphi_{0} (\cdot, \BVSm_{\bvc} (\cdot)) * \rho_{\eps}$ and $\boldsymbol{\zeta}_{\eps} := \boldsymbol{\varphi} (\cdot, \BVSm_{\bvc} (\cdot)) * \rho_{\eps}$, for a mollifier $\rho_{\eps}$ supported in $[0, \eps]$. Since $\zeta_{\eps}(t) \to   \varphi_{0} (t, \BVSm_{\bvc} (t))$ and $\boldsymbol{\zeta}_{\eps}(t) \to  \boldsymbol{\varphi} (t, \BVSm_{\bvc} (t))$ for every $t \in  \III $, we infer that
    \begin{align}
  \label{e:upphi3}
 & 
 \int_{ \III } \varphi_{ 0 } ( t, \BVSm_{\bvc} (t) ) \, \di t   = \int_{ \mathscr{C}_{\sfyy} } \varphi_{0} ( \sft( s ), \BVSm_{\bvc} (\sft(s)) ) \, \sft'(s) \, \di s \,,
 \\
  &
   \label{e:upphi4} \int_{ C_{ \BVSm_{\bvc } }}  \boldsymbol{\varphi} (t, \BVSm_{\bvc} (t) ) \, \di ( ( \BVSm_{\bvc} )'_{\mathcal{L}^{1}} + ( \BVSm_{\bvc} )'_{C}) (t)  = \int_{\mathscr{C}_{\sfyy}}    \boldsymbol{\varphi} (\sft(s) , \BVSm_{\bvc} (\sft(s)) ) \cdot \sfxx'(s) \, \di s \,.
  \end{align}
  Since $\BVSm_{\bvc} (\sft(s) ) = \bvc(\sft(s), 0)$ and, for $s \in \mathscr{C}_{\sfyy}$, $\bvc ( \sft(s), 0) = \bvc(\sft(s), r)$ for every $r \in [0, 1]$, we rewrite \eqref{e:upphi3}--\eqref{e:upphi4} as
    \begin{align}
  \label{e:upphi5}
 & 
 \int_{ \III } \varphi_{ 0 } ( t, \BVSm_{\bvc} (t) ) \, \di t  =  \int_{ \mathscr{C}_{\sfyy} } \varphi_{0} ( \sfyy(s)  ) \, \sft'(s) \, \di s \,,
 \\
  &
   \label{e:upphi6} \int_{ C_{ \BVSm_{\bvc } }}  \boldsymbol{\varphi} (t, \BVSm_{\bvc} (t) ) \, \di ( ( \BVSm_{\bvc} )'_{\mathcal{L}^{1}} + ( \BVSm_{\bvc} )'_{C}) (t)  = \int_{\mathscr{C}_{\sfyy}}    \boldsymbol{\varphi} ( \sfyy (s) ) \cdot \sfxx'(s) \, \di s \,,
  \end{align}
  whence \eqref{e:616}. 
  \par
  Combining~\eqref{e:upphi5}--\eqref{e:upphi6} with~\eqref{e:upphi-jumps} we conclude  
    that for all $\varphi=(\varphi_0,\bvarphi) \in \Cc ( \R^{d+1}_{+}  ;\R^{d+1})$
  \[
  \langle \omega_{\sfyy}, \varphi \rangle  =\big \langle (\bvcgr_{\bvc})_{\sharp } \left( (1,  \BVSm_{\bvc}' ) \mathcal{L}^1  + (0,  (\BVSm_{\bvc})'_{\mathrm{C}})  \right) 
+ \sum_{t \in \mathfrak{J}_{\bvc}}  \delta_t {\otimes}  \boldsymbol{j}_{t, \bvc}, \varphi  \big \rangle\,,
  \]
  and \eqref{teta-equals-alpha} follows. 
  \medskip
  
  \noindent
  $\vartriangleright  \eqref{bvc2measures}: $ We will show the measurability of $\mathscr{A}$, $\mathscr{C}$, and $\mathscr{J}$ 
  by proving that  the 
  following mappings
  \[
  \begin{cases}
\ARBV(\cZ;\R^d) \ni  \bvc \mapsto  \langle \mathscr{A}(\bvc), \varphi \rangle
  \\
\ARBV(\cZ;\R^d) \ni  \bvc  \mapsto \langle \mathscr{J}(\bvc), \bvarphi \rangle
  \\
\ARBV(\cZ;\R^d) \ni   \bvc  \mapsto \langle \mathscr{C}(\bvc), \bvarphi \rangle
  \end{cases}
  \quad\text{ are Borel for every test function } 
  \varphi=(\varphi_0,\bvarphi) \in \Cc ( \R^{d+1}_{+} ;\R^{d+1})\,.
  \]
  In turn, this will be shown via 
   the representation formulae \eqref{e:631}. 
   We start by observing that the mappings
   \begin{equation}
   \label{I0-I}
   \begin{cases}
   \Lipplus 1 \III{\R^{d+1}}  \ni   \sfyy \mapsto \mathrm{I}_0(\sfyy) : =  \int_{\III} \varphi_0(\sfyy(s)) \sft'(s) \, \dd s
  \\
   \Lipplus 1 \III{\R^{d+1}} \ni   \sfyy \mapsto  \mathrm{I}(\sfyy) : = \int_{\III} \bvarphi(\sfyy(s)) {\,\cdot\, } \sfxx'(s) \, \dd s
  \end{cases} \qquad \text{are } Borel.
   \end{equation}
Indeed, they are continuous with respect to the topology of uniform convergence on compact sets of $
  \III$ induced by the metric $D$ from \eqref{metric-4-uniform-compact}: to check this, it  suffices to take $(
\sfyy_n)_n$, $\sfyy 
\in   \Lipplus 1 \III{\R^{d+1}}  $ with $D(\sfyy_n,\sfyy) \to 0$ as $n\to\infty$, and and observe that, since  $\|( \sft_n' , \sfxx_n')\|  \leq 1 $  a.e.\ in $\III$, we may suppose (up to 
a not relabeled subsequence), that $\sft_n'\weaksto \sft'$ and $\sfxx_n'\weaksto \sfxx$ in $L^\infty(\III) $. Then, the convergences
 \begin{align*}
\lim_{n\to \infty} \int_{\III}\varphi_0(\sfyy_n(s)) \sft_n'(s) \, \dd s & = \int_{\III} \varphi_0(\sfyy(s)) \sft'(s) \, \dd s\,,
\\
 \lim_{n\to \infty}\int_{\III} \bvarphi(\sfyy_n(s)) {\, \cdot \, } \sfxx_n'(s) \, \dd s &  =  \int_{\III} \bvarphi(\sfyy(s)) {\, \cdot \,} \sfxx'(s) \, \dd s
 \end{align*}
 follow by dominated convergence.  Now, recalling \eqref{e:upphi3} and \eqref{e:upphi5}, we conclude that
 the mapping
 \begin{equation}
 \label{Borel1}
\ARBV(\cZ;\R^d) \ni \bvc  \mapsto \mathrm{I}_0( \mathscr{T}(\bvc))  =  \int_{ \III } \varphi_{ 0 } ( t, \BVSm_{\bvc} (t) ) \, \di t     \qquad \text{is Borel},
\end{equation}
  as it  is  given  by the composition of two Borel mappings.
  \par
  In turn, we observe that  for all $\sfyy \in  \Lipplus 1 \III{\R^{d+1}} $ we have 
  \[
    \mathrm{I}(\sfyy) =   \mathrm{I}_{\mathrm{cont}}(\sfyy) +  \mathrm{I}_{\mathrm{sing}}(\sfyy)  \qquad \text{with } 
    \begin{cases}
      \mathrm{I}_{\mathrm{cont}}(\sfyy) =\int_{\III \cap \{ \sft' > 0\}}   \bvarphi (\sfyy (s)) \cdot \sfxx'(s) \, \di s,
      \\
     \mathrm{I}_{\mathrm{sing}}(\sfyy)= \int_{\III \cap \{ \sft' = 0\}}  \bvarphi (\sfyy (s)) \cdot \sfxx'(s) \, \di s.
    \end{cases}\,.
  \]
Indeed, the pedices $\mathrm{cont}$ and $\mathrm{sing}$ refer to the fact that $\mathrm{I}_{\mathrm{cont}}$
and  $ \mathrm{I}_{\mathrm{sing}}$
 represent the contributions to $\tetatrue{\bvc}$ involving the measures $ (\BVSm_{\bvc})'_{\mathcal{L}^1}$
 and $  (\BVSm_{\bvc})'_{\mathrm{C}} +  (\BVSm_{\bvc})'_{\mathrm{J}}$, respectively, as 
\begin{equation}
\label{continuousVSsing}
\begin{cases}
\mathrm{I}_{\mathrm{cont}}(\sfyy) = \int_{\III} \bvarphi (t, \BVSm_{\bvc} (t)) \cdot  \BVSm_{\bvc}' (t) \, \di t,
\\
\mathrm{I}_{\mathrm{sing}}(\sfyy) =  \int_{\III}  \boldsymbol{\varphi} (t, \BVSm_{\bvc} (t) ) \, \di  ( \BVSm_{\bvc} )'_{\mathrm{C}} (t) +   \langle  \sum_{t \in \mathfrak{J}_{\bvc}} \delta_t {\otimes}  \boldsymbol{j}_{t, \bvc},\bvarphi \rangle.
\end{cases}
\end{equation}
Now, we claim that  the mapping
\begin{equation}
\label{measu-Icont}
     \ALip \III{\R^{d+1}}     \ni   \sfyy \mapsto     \mathrm{I}_{\mathrm{sing}}(\sfyy)  = \int_{\III}  \boldsymbol{1}_{\{ s\in \III\, : \ \sft'(s)=0\}}(r)     \bvarphi (\sfyy (r)) \cdot \sfxx'(r) \, \di r
   \qquad \text{is Borel.}
\end{equation}
 Since we can write $ \mathrm{I}_{\mathrm{sing}}(\sfyy) = \mathrm{I}^{+}_{\mathrm{sing}}(\sfyy) - \mathrm{I}^{-}_{\mathrm{sing}}(\sfyy)$ with 
\begin{displaymath}
\mathrm{I}^{\pm}_{\mathrm{sing}}(\sfyy) :=  \int_{\III}  \boldsymbol{1}_{\{ s\in \III\, : \ \sft'(s)=0\}}(r) (  \bvarphi (\sfyy (r)) \cdot \sfxx'(r) ) _{\pm}\, \di r  \quad \text{with }  \begin{cases}
   r_+ = \max\{ r, 0\},
   \\
   r_-=  \max \{ -r, 0\}\,,
   \end{cases}
\end{displaymath}
it is enough to prove that $\sfyy \mapsto \mathrm{I}^{\pm}_{\mathrm{sing}}(\sfyy)$ are Borel measurable. We proceed with the proof for $\mathrm{I}^{+}_{\mathrm{sing}}$. The very same argument applies to $\mathrm{I}^{-}_{\mathrm{sing}}$. 

 For $n, k \in \mathbb{N}\setminus \{0\}$ we consider the map
\begin{displaymath}
  \ALip \III{\R^{d+1}}    \ni   \sfyy \mapsto     J_{n, k}(\sfyy)  = \int_{\III} \boldsymbol{1}_{S_{n, k} (\sft) }(r)    ( \bvarphi (\sfyy (r)) \cdot \sfxx'(r))_{+} \, \di r\,,
\end{displaymath}
where $\boldsymbol{1}_{S_{n, k}}$   is the characteristic function of the  set $S_{n, k} (\sft) := \{ s\in \III\, : \ k ( \sft(s + \frac{1}{k}) - \sft(s)) \leq \frac{1}{n}\}$. Then, $ J_{ n, k}$ is Borel measurable, as it is upper semicontinuous with respect to the uniform convergence on compact subsets of~$\III$. Notice that here we are also using  that,
 whenever   $\sfyy_{m}, \sfyy \in \ALip  \III {\R^{d+1}}$   are such that $D(\sfyy_{m}, \sfyy) \to 0$ as $m \to \infty$, then we also have 
 $\sfyy'_{m} \to \sfyy'$ in $L^{p}_{\loc} (\III; \R^{d+1})$ for every $1
 \leq p <+\infty$. 

The Borel measurability of $J_{n, k}$ implies that also the maps
\begin{align*}
  \ALip \III{\R^{d+1}}   \ni   \sfyy \mapsto     \liminf_{k\to \infty}J_{n, k}(\sfyy) \,, \qquad   \ALip  \III{\R^{d+1}}   \ni   \sfyy \mapsto     \limsup_{k\to \infty}J_{n, k}(\sfyy)
  \end{align*}
  are Borel measurable for every $n \in \mathbb{N}$. By Fatou lemma, we further notice that for every $\sfyy \in   \ALip \III{\R^{d+1}}$  it holds
  \begin{align}
  \label{e:limk-J}
\mathrm{I}^{+}_{\mathrm{sing}}(\sfyy) &  \leq \int_{\III} \liminf_{k\to \infty}  \boldsymbol{1}_{S_{n, k} (\sft) }(r)    ( \bvarphi (\sfyy (r)) \cdot \sfxx'(r))_{+} \, \di r \leq \liminf_{k\to \infty}J_{n, k}(\sfyy) 
\\
&
\leq  \limsup_{k\to \infty}J_{n, k}(\sfyy) \leq \int_{\III} \limsup_{k\to \infty} \boldsymbol{1}_{S_{n, k} (\sft) }(r)    ( \bvarphi (\sfyy (r)) \cdot \sfxx'(r))_{+} \, \di r \nonumber
\\
&
\leq \int_{\III} \boldsymbol{1}_{\{ s \in \III : \, \sft'(s) \leq \frac{1}{n}\} }(r)    ( \bvarphi (\sfyy (r)) \cdot \sfxx'(r))_{+} \, \di r \,.\nonumber
  \end{align}
Taking  the limit as $n \to \infty$ in the chain of inequalities~\eqref{e:limk-J}, we infer that
 \begin{displaymath}
 \mathrm{I}^{+}_{\mathrm{sing}}(\sfyy) = \lim_{n\to \infty} \liminf_{k\to\infty} J_{n, k} (\sfyy)\,.
 \end{displaymath}
 This implies that $\mathrm{I}^{+}_{\mathrm{sing}}$ is a Borel map. This concludes the proof of~\eqref{measu-Icont}.

Combining \eqref{I0-I} and~\eqref{measu-Icont} we deduce that also $\mathrm{I}_{\mathrm{cont}}$ is Borel measurable. Thus, in view of \eqref{continuousVSsing} we have that 
 \begin{equation}
 \label{Borel2}
\ARBV(\cZ;\R^d) \ni \bvc  \mapsto     \mathrm{I}_{\mathrm{cont}}( \mathscr{T}(\bvc)) =  \int_{\III} \bvarphi (t, \BVSm_{\bvc} (t)) \cdot  \BVSm_{\bvc}' (t) \, \di t    \qquad \text{is Borel}.
\end{equation}
From \eqref{Borel1} and \eqref{Borel2} we then have that the mapping 
\[
\ARBV(\cZ;\R^d) \ni \bvc  \mapsto   \langle\mathscr{A}(\bvc),\varphi\rangle  \text{ is Borel for every } \varphi \in  \Cc ( \R^{d+1}_{+} ;\R^{d+1})\,.
\]
\par
Now,  Lemma \ref{l:Borel-jump} 
ahead ensures that the mapping  $\ARBV(\cZ;\R^d) \ni  \bvc  \mapsto \langle \mathscr{J}(\bvc), \bvarphi \rangle, 
$ is Borel for all test functions $\bvarphi$. Ultimately, 
\[
\ARBV(\cZ;\R^d) \ni \bvc  \mapsto  \mathrm{I}_{\mathrm{cont}}( \mathscr{T}(\bvc)) -  \langle \mathscr{J}(\bvc), \bvarphi \rangle  =\langle \mathscr{C}(\bvc), \bvarphi \rangle 
\qquad \text{is Borel for all } \bvarphi \in  \Cc ( \R^{d+1}_{+}  ;\R^{d})\,.
\]
We have thus proven  \eqref{bvc2measures}.
  \QED

  The last result of this section addresses the measurability of the mapping $\mathscr{J}$. 
   \begin{lemma}
  \label{l:Borel-jump}
  For all $\bvarphi \in  \Cc (  \R^{d+1}_{+} ;\R^{d})$ we consider the mapping 
  \begin{equation}
   \label{Borel3}
   \mathrm{I}_{\mathrm{jump}} \colon    \ALip  \III{\R^{d+1}} \to \R   \qquad 
    \sfyy \mapsto \sum_{t \in \mathrm{L}(\sfyy)} \int_{\sfs_{\sfyy}^{-} (t)}^{\sfs_{\sfyy}^{+} (t)} \boldsymbol{\varphi} (\sfyy (s)  ) \cdot \sfxx'(s) \, \di s
  \end{equation}
  with the short-hand notation  $L _{\sfyy}:=\{ s \in \III  : \, \sfs^{+}_{\sfyy} (s) - \sfs^{-}_{\sfyy} (s) >0\}$. 
  Then, 
  \[
  \forall\, \bvarphi \in   \Cc (\III{\times}\R^{d};\R^{d}) \qquad \bvc \mapsto \mathrm{I}_{\mathrm{jump}} (\mathscr{T}(\bvc)) =  \langle \mathscr{J}(\bvc), \bvarphi \rangle  \quad \text{is Borel}.
  \]
  \end{lemma}
  \begin{proof}
  First of all, observe that 
  \[
   \mathrm{I}_{\mathrm{jump}} (\sfyy) =    \mathrm{I}_{\mathrm{jump}}^+ (\sfyy)  -  \mathrm{I}_{\mathrm{jump}}^- (\sfyy) 
  \]
  with 
  \[
   \mathrm{I}_{\mathrm{jump}}^{\pm} (\sfyy)  :=   \sum_{t \in \mathrm{L}_\sfyy}  \int_{\sfs_{\sfyy}^{-} (t)}^{\sfs_{\sfyy}^{+} (t)} 
   \left( \boldsymbol{\varphi} (\sfyy (s)  ) \cdot \sfxx'(s) \right)_{\pm} \, \di s \quad \text{with }  \begin{cases}
   r_+ = \max\{ r, 0\},
   \\
   r_-=  \max \{ -r, 0\}\,.
   \end{cases} 
  \]
 Hence, we can show the measurability property for the functions $ \mathrm{I}_{\mathrm{jump}}^{\pm}$. We provide the full proof for~$ \mathrm{I}_{\mathrm{jump}}^{+}$. A similar argument applies to $ \mathrm{I}_{\mathrm{jump}}^{-}$. 
\par
Let us now introduce the continuous function
\[
A\colon \III \to \R^+, \qquad A(r) : = \int_0^r   (\boldsymbol{\varphi} (\sfyy (s)  ) \cdot \sfxx'(s))_{+}   \, \di s,
\]
so that  
\begin{equation}
\label{decomposition-via-integral-function}
\int_{\sfs_{\sfyy}^{-} (t)}^{\sfs_{\sfyy}^{+} (t)} 
   ( \boldsymbol{\varphi} (\sfyy (s)  ) \cdot \sfxx'(s) ) \,   \di s = A(\sfs_{\sfyy}^{+} (t))- A(\sfs_{\sfyy}^{-} (t))
   \end{equation}
We will thus prove that the function
   \begin{equation}
   \label{BorelAinfty}
  \mathcal{A} \colon   \ALip  \III{\R^{d+1}}  \to \R^+
 \qquad 
   \sfyy \mapsto  \sum_{t \in \mathrm{L}(\sfyy)} [ A(\sfs_{\sfyy}^{+} (t)) {-} A(\sfs_{\sfyy}^{-} (t)) ]\qquad \text{is Borel.}
   \end{equation}
   We split the argument for \eqref{BorelAinfty} in the following steps.
   \par
  \textbf{Claim $1$:} 
    \begin{equation}
    \label{Borel-claim1}
    \begin{cases}
  \III\times   \ALip \III{\R^{d+1}}   \ni (t,\sfyy) \mapsto  A(\sfs_{\sfyy}^{+} (t)) & \text{is upper semicontinuous.}
 \\
  \III\times    \ALip  \III{\R^{d+1}}    \ni (t,\sfyy) \mapsto  A(\sfs_{\sfyy}^{-} (t))  &  \text{is lower semicontinuous.}
  \end{cases}
  \end{equation}
As for the first property,  it suffices to observe that  $(t,\sfyy)\mapsto \sfs_{\sfyy}^{+} (t)$ is upper semicontinuous. Hence,  since every $\sfyy \in  \ALip \III{\R^{d+1}}$   satisfies  $\| \sfyy'(s)\| =1$ for a.e.~$s \in \III$, we have that $A\circ\sfs_{\sfyy}^{+}$ is upper semicontinuous. Indeed, the constraint $\| \sfyy'(s)\|=1$ implies that 
\begin{displaymath}
\sfyy \mapsto \int_{0}^{r} ( \boldsymbol{\varphi} (\sfyy (s)  ) \cdot \sfxx'(s) )_{+} \, \di s 
\end{displaymath}
is continuous in  $\ALip \III{\R^{d+1}}$   for every $r \in \III$. This yields the desired upper semicontinuity.
Analogously, the second statement follows from the lower semicontinuity of $(t,\sfyy)\mapsto \sfs_{\sfyy}^{-} (t)$. 
   \par
  \textbf{Claim $2$:} {\sl for every $T>0$, $S>0$, and $\zeta>0$,  the mapping }
    \begin{equation}
    \label{Borel-claim2}
    \begin{gathered}
\mathcal{A}_{\zeta}^{T,S} \colon      \ALip  \III{\R^{d+1}} \to \R^+,   \qquad \sfyy \mapsto
\begin{cases}
 \sum_{t \in \mathrm{L}_{\zeta}^{T,S}(\sfyy)} 
[ A(\sfs_{\sfyy}^{+} (t)) {-} A(\sfs_{\sfyy}^{-} (t)) ] & \text{if }  \sfs_{\sfyy}^{-} (T) \leq S,
\\
0 &   \text{if }   \sfs_{\sfyy}^{-} (T) > S,
\end{cases} 
\\
\text{with } \mathrm{L}_{\zeta}^{T}(\sfyy) := \{ t \in \mathrm{L}(\sfyy)\, : 
\ t \in  [0,T],    \ | \sfs_{\sfyy}^{+} (t) {-} \sfs_{\sfyy}^{-} (t)| \geq \zeta \}
\end{gathered}
\end{equation}
 is upper semicontinuous.   Let us consider $(\sfyy_j)_j, \, \sfyy \in   \ALip \III{\R^{d+1}}$ such that $D(\sfyy_j ,  \sfyy ) \to 0$ and show that
 \begin{displaymath}
 \limsup_{j\to \infty} \,  \mathcal{A}_{\zeta}^{T,S}(\sfyy_j)\leq  \mathcal{A}_{\zeta}^{T,S}(\sfyy)\,.
 \end{displaymath}
Up to a subsequence, we may assume that the limsup is a limit. If $\sfs_{\sfyy_{j}}^{-} (T) > S$ definitely for $j$ large enough, there is nothing to prove. Let us therefore assume that for every $j \in \mathbb{N}$ we have $\sfs_{\sfyy_{j}}^{-} (T) \leq S$. By lower-semicontinuity,  observe that  $\sfs_{\sfyy}^{-} (T) \leq S. $ 
Moreover, 
 we observe that, in correspondence with the sequence $(\sfyy_j)_j$ there exists $N\in \N$ such that 
for every $j \in \N$ there exist at most $N$ times $t_1^j<\ldots < t_N^j \in [0,T]$  such that 
$ | \sfs_{\sfyy_j}^{+} (t_i^j) {-} \sfs_{\sfyy_j}^{-} (t_i^j)| \geq \zeta$ for all $i=1,\ldots, N$. 
In fact,  by definition of $\sfs^{\pm}_{\sfyy_{j}}$, for every $j \in \mathbb{N}$ it holds that
\begin{displaymath}
S \geq \sfs_{\sfyy_{j}}^{-} (T) \geq \sum_{t \in \mathrm{L}_{\zeta}^{T}(\sfyy)\cap  [0, T) } | \sfs^{+}_{\sfyy} (t) - \sfs^{-}_{\sfyy} (t)| \geq \left( \# \big[ \mathrm{L}_{\zeta}^{T}(\sfyy)\cap  [0, T) \big]\right) \zeta \,.\
\end{displaymath} 
This implies that $\# L^{T}_{\zeta} (\sfyy_{j}) \leq \frac{S}{\zeta} + 1$ for every $j \in \mathbb{N}$. 
Then,
\[
\mathcal{A}_{\zeta}^{T,S}(\sfyy_j)  = \sum_{i=1}^N [ A(\sfs_{\sfyy_j}^{+} (t_i^j)) {-} A(\sfs_{\sfyy_j}^{-} (t_i^j)) ]\,.
\]
Now, up to a non-relabeled subsequence we have that there exist $(t_i)_{i=1}^N \subset  [0,T] $ such that 
$t_i^j \to t_i$ as $j\to\infty$.  In particular, it holds $ | \sfs_{\sfyy}^{+} (t_i) {-} \sfs_{\sfyy}^{-} (t_i)| \geq \zeta $ for every $i =1, \ldots, N$, so that $(t_i)_{i=1}^N \subset  \mathrm{L}_{\zeta}^{T}(\sfyy)$. We notice that some of the $t_{i}$'s may coincide. With a slight abuse of notation, we denote by $t_{k}$, for $k =1, \ldots, M\leq N$ the distinct limit points of $t_{i}^{j}$. By \eqref{Borel-claim1} we have that, whenever $t^{j}_{i} \to t_{k}$ as $j \to \infty$, then 
\begin{equation}
\label{e:Borel-claim2100}
 A(\sfs_{\sfyy}^{+} (t_k)) {-} A(\sfs_{\sfyy}^{-} (t_k)) \geq \limsup_{j \to \infty}  A(\sfs_{\sfyy_{j}}^{+} (t_i^{j})) {-} A(\sfs_{\sfyy_{j}}^{-} (t_i^{j}))\,. 
\end{equation}
If we have that $t_{i}^{j}, \ldots, t^{j}_{i+\ell} \to t_{k}$ for some $\ell>0$ and some $k = 1, \ldots, M$, then
\begin{align}
\label{e:Borel-claim2101}
 A(\sfs_{\sfyy}^{+} (t_k)) {-} A(\sfs_{\sfyy}^{-} (t_k)) & \geq \limsup_{j \to \infty}  A(\sfs_{\sfyy_{j}}^{+} (t_{i + \ell}^{j})) {-} A(\sfs_{\sfyy_{j}}^{-} (t_{i}^{j})) 
 \\
 &
 \geq \limsup_{j\to \infty} \sum_{n=0}^{\ell} A(\sfs_{\sfyy_{j}}^{+} (t_{i + n}^{j})) {-} A(\sfs_{\sfyy_{j}}^{-} (t_{i + n}^{j})) \nonumber \,. 
\end{align}
Combining~\eqref{e:Borel-claim2100}--\eqref{e:Borel-claim2101} we conclude that
\[
 \mathcal{A}_{\zeta}^{T,S}(\sfyy)  \geq \sum_{k=1}^M  [ A(\sfs_{\sfyy}^{+} (t_k)) {-} A(\sfs_{\sfyy}^{-} (t_k)) ]
 \geq  \limsup_{j\to\infty}  \mathcal{A}_{\zeta}^{T,S}(\sfyy_j) \,.  
\] 
     \par
  \textbf{Conclusion:} Clearly, we have that 
  \[
  \forall\, \sfyy \in     \ALip \III{\R^{d+1}}    \, : \qquad   \mathcal{A}(\sfyy) = \lim_{T \uparrow \infty,\,  S \uparrow \infty, \,
  \zeta \downarrow 0  }\mathcal{A}_{\zeta}^{T,S}(\sfyy) \,.
  \]
  Then,  $\mathcal{A}$ is the pointwise limit of  Borel mappings. Thus, \eqref{BorelAinfty} follows.
 This  finishes the proof. 
  \end{proof}

\section{Auxiliary measure-theoretic tools}
\label{s:app-B}
Let $\Spx$  be a Polish metric space and 
$\calMb(\Spx;\R^h)$ the space of $\R^h$-valued Borel measures on $\Spx$ with finite total variation, endowed with the weak$^*$ topology.
Let $\Traj$ also be a Polish space, and let  $(\proxyteta{\traj})_{\traj \in \Traj} \subset \calMb(\Spx;\R^h)$  be a Borel family. 
With any given $\proxym \in \Prob(\Traj)$ with 
\[
\int_{\Traj} |\proxyteta{\traj}|(\Spx) \, \dd \proxym(\traj)<+\infty
\]
we may associate the measures
\[
\Proxyteta{\proxym} : =  \int_{\Traj} \proxyteta{\traj} \, \dd \proxym(\traj)  \in \calMb(\Spx;\R^h) \quad \text{and} \quad  \Proxysigma{\proxym} : =  \int_{\Traj} |\proxyteta{\traj}| \, \dd \proxym(\traj)   \in \calM^+ (\Spx)\,.
\]
Clearly,
we have that $|\Proxyteta{\proxym} |  \leq \Proxysigma{\proxym}$.  The following result
provides a useful property of the `generating' 
measures
$(\proxyteta{\traj})_{\traj \in \Traj} \subset \calMb(\Spx;\R^d)$ in the case when the measures 
$|\Proxyteta{\proxym} |$ and $ \Proxysigma{\proxym}$ coincide.
\begin{proposition}
\label{prop:measure-theor}
Let 
$\boldsymbol{\mathsf{f}}: X \to \R^h$ be a Borel density of $\Proxyteta{\proxym} $ w.r.t.\ $\Proxysigma{\proxym} $, 
with $|\boldsymbol{\mathsf{f}}(x) | \leq 1$ for all $x\in \Spx$.
Suppose that $|\Proxyteta{\proxym} |=  \Proxysigma{\proxym}$.
Then
\begin{equation}
\label{thesis-b1}
\proxyteta{\traj} = \boldsymbol{\mathsf{f}} |\proxyteta{\traj}| \qquad \text{for } \proxym \text{-a.e. }  \traj \in \Traj\,.
\end{equation}
\end{proposition}
\begin{proof}
Observe that, for a given measure $\uplambda \in  \calMb(\Spx;\R^h)$, a Borel function
$\boldsymbol{h}\colon X \to \R^h$ is the density of $\uplambda$ w.r.t.\ $|\uplambda|$ if and only if 
\begin{equation}
\label{general-obs}
|\boldsymbol{h}|\leq 1 \text{ $|\uplambda|$-a.e.\ in $\Spx$}, \quad \text{and} \qquad \int_{\Spx} \boldsymbol{h}(x) \dd \uplambda(x) = 
|\uplambda|(\Spx)\,. 
\end{equation}
\par
Now,
 from $ |\Proxyteta{\proxym}  | =  \Proxysigma{\proxym}$ we have that 
 $|\boldsymbol{\mathsf{f}}(x)| \equiv 1$ for  $\Proxysigma{\proxym} $-a.e.\ $x \in \Spx$.
 Therefore, we have the following chain of identities
 \[
 \begin{aligned}
 \int_{\Traj} |\proxyteta{\traj}| (\Spx) \,\dd  \proxym(\traj) = \int_{\Spx} |\boldsymbol{\mathsf{f}}(x)|^2 \, \dd \Proxysigma{\proxym}(x)
 \stackrel{(1)}= \int_{\Spx} \boldsymbol{\mathsf{f}}(x)  \dd 
 \Proxyteta{\proxym}(x)
 \stackrel{(2)}=  \int_{\Traj}
 \left(  \int_{\Spx}   \boldsymbol{\mathsf{f}}(x)  \,  \dd \proxyteta{\traj}(x) \right) \, \dd \proxym(\traj)
 \end{aligned}
 \]
 where (1) follows from the fact that $ \Proxyteta{\proxym} = \mathsf{f}   \Proxysigma{\proxym}$
 and (2) from the Fubini theorem. 
 Then, we immediately conclude that 
 \[
  |\proxyteta{\traj}| (\Spx) = \int_{\Spx}   \boldsymbol{\mathsf{f}}(x)  \,  \dd \proxyteta{\traj}(x) \qquad    \text{for } \proxym \text{-a.e. }  \traj \in \Traj\,.
 \]
Thus, on account of  \eqref{general-obs} we obtain \eqref{thesis-b1}.

\end{proof}
\medskip

\paragraph{\bf Acknowledgements}
The work of SA was partially funded by the Austrian Science Fund through the project 10.55776/P35359, by the University of Naples Federico II through FRA
Project ``ReSinApas", by the MUR - PRIN 2022 project ``Variational Analysis of Complex Systems in Materials Science, Physics and Biology'', No.
2022HKBF5C, funded by European Union NextGenerationEU, and by  Gruppo Nazionale per l'Analisi Matematica, la Probabilit\`a e le loro
Applicazioni,   through Project 2025: ``DISCOVERIES - Difetti e Interfacce in
Sistemi Continui: un'Ottica Variazionale in Elasticit\`a con Risultati Innovativi ed Efficaci Sviluppi'' and Project 2026: ``Sistemi multi-agente e replicatore: derivazione particellare e ottimizzazione''. SA further acknowledges the kind hospitality of the Technical University of Munich and of the University of Vienna, where part of this work has been carried out.
\par
RR has been partially supported by GNAMPA-INdAM, through Project 2025: ``Variational analysis of some models for materials with defects''. 
\par
GS has been partially supported by the MIUR-PRIN 202244A7YL project
Gradient Flows and Non-Smooth Geometric Structures with Applications
to Optimization and Machine Learning,  by the INDAM project
E53C23001740001, by the Institute for Advanced Study of the
Technical University of Munich, funded by the German Excellence
Initiative, and by funding from the European Research Council (ERC) under
the European Union’s Horizon Europe research and innovation programme
(grant agreement No~101200514, project acronym OPTiMiSE).
Views and opinions expressed are however those of the
authors only and do not necessarily reflect
those of the European Union or the European Research
Council Executive Agency. Neither the European Union nor the granting authority can be held
responsible for them.

\bibliographystyle{siam}
\bibliography{ricky_lit.bib}

\begin{thebibliography}{10}

\bibitem{Abedi-Zhenhao-Schultz24}
{\sc E.~Abedi, Z.~Li, and T.~Schultz}, {\em Absolutely continuous and
  {BV}-curves in 1-{W}asserstein spaces}, Calc. Var. Partial Differential
  Equations, 63 (2024), pp.~Paper No. 16, 34.

\bibitem{Albi-Almi-Morandotti-Solombrino}
{\sc G.~Albi, S.~Almi, M.~Morandotti, and F.~Solombrino}, {\em Mean-field
  selective optimal control via transient leadership}, Applied Math. {\&}
  Optim., 85 (2022), p.~22.

\bibitem{Almi-Durastanti-Solombrino1}
{\sc S.~Almi, R.~Durastanti, and F.~Solombrino}, {\em A {P}ontryagin maximum
  principle for agent-based models with convex state space}, ESAIM: Control
  Optim.~Calc.~Var., 31 (2025), p.~article n.~37.

\bibitem{Almi-Durastanti-Solombrino2}
\leavevmode\vrule height 2pt depth -1.6pt width 23pt, {\em Mean field first
  order optimality condition under low regularity of controls}, accepted on
  SIAM J. Control. Optim., 64 (2026).

\bibitem{Almi-Morandotti-Solombrino_JEE}
{\sc S.~Almi, M.~Morandotti, and F.~Solombrino}, {\em A multi-step {L}agrangian
  scheme for spatially inhomogeneous evolutionary games}, J. Evol. Equ., 21
  (2021), pp.~2691--2733.

\bibitem{Almi-Morandotti-Solombrino_JDE}
{\sc S.~Almi, M.~Morandotti, and F.~Solombrino}, {\em Optimal control problems
  in transport dynamics with additive noise}, J. Differential Equations, 373
  (2023), pp.~1--47.

\bibitem{Ambrosio04}
{\sc L.~Ambrosio}, {\em Transport equation and {C}auchy problem for {$BV$}
  vector fields}, Invent. Math., 158 (2004), pp.~227--260.

\bibitem{Ambrosio-Bernard08}
{\sc L.~Ambrosio and P.~Bernard}, {\em Uniqueness of signed measures solving
  the continuity equation for {O}sgood vector fields}, Atti Accad. Naz. Lincei
  Rend. Lincei Mat. Appl., 19 (2008), pp.~237--245.

\bibitem{Ambrosio-Crippa08}
{\sc L.~Ambrosio and G.~Crippa}, {\em Existence, uniqueness, stability and
  differentiability properties of the flow associated to weakly differentiable
  vector fields}, in Transport equations and multi-{D} hyperbolic conservation
  laws, vol.~5 of Lect. Notes Unione Mat. Ital., Springer, Berlin, 2008,
  pp.~3--57.

\bibitem{Ambrosio-Figalli09}
{\sc L.~Ambrosio and A.~Figalli}, {\em On flows associated to {S}obolev vector
  fields in {W}iener spaces: an approach \`a la {D}i{P}erna-{L}ions}, J. Funct.
  Anal., 256 (2009), pp.~179--214.

\bibitem{Ambrosio-Fornasier-Morandotti-Savare21}
{\sc L.~Ambrosio, M.~Fornasier, M.~Morandotti, and G.~Savar\'{e}}, {\em
  Spatially inhomogeneous evolutionary games}, Comm. Pure Appl. Math., 74
  (2021), pp.~1353--1402.

\bibitem{AmFuPa05FBVF}
{\sc L.~Ambrosio, N.~Fusco, and D.~Pallara}, {\em Functions of Bounded
  Variation and Free Discontinuity Problems}, Oxford University Press, 2005.

\bibitem{AGS08}
{\sc L.~Ambrosio, N.~Gigli, and G.~Savar{\'e}}, {\em Gradient flows in metric
  spaces and in the space of probability measures}, Lectures in Mathematics ETH
  Z\"urich, Birkh\"auser Verlag, Basel, second~ed., 2008.

\bibitem{AmbrosioGigliSavare08}
{\sc L.~Ambrosio, N.~Gigli, and G.~Savar\'e}, {\em Gradient Flows in Metric
  Spaces and in the Space of Probability Measures}, Lectures in Mathematics ETH
  Z{\"u}rich, Birkh{\"a}user, 2008.

\bibitem{Ambrosio-Gigli-Savare14}
{\sc L.~Ambrosio, N.~Gigli, and G.~Savar\'{e}}, {\em Calculus and heat flow in
  metric measure spaces and applications to spaces with {R}icci bounds from
  below}, Invent. Math., 195 (2014), pp.~289--391.

\bibitem{Ambrosio-Lisini-Savare06}
{\sc L.~Ambrosio, S.~Lisini, and G.~Savar\'{e}}, {\em Stability of flows
  associated to gradient vector fields and convergence of iterated transport
  maps}, Manuscripta Math., 121 (2006), pp.~1--50.

\bibitem{AmbRenVit25}
{\sc L.~Ambrosio, F.~Renzi, and F.~Vitillaro}, {\em The superposition principle
  for local 1-dimensional currents}, 2025.
\newblock Preprint arXiv 2503.18157.

\bibitem{Ambrosio-Trevisan14}
{\sc L.~Ambrosio and D.~Trevisan}, {\em Well-posedness of {L}agrangian flows
  and continuity equations in metric measure spaces}, Anal. PDE, 7 (2014),
  pp.~1179--1234.

\bibitem{BiaBonGus}
{\sc S.~Bianchini, P.~Bonicatto, and N.~A. Gusev}, {\em Renormalization for
  autonomous nearly incompressible {BV} vector fields in two dimensions}, SIAM
  J. Math. Anal., 48 (2016), pp.~1--33.

\bibitem{Bogachev07}
{\sc V.~I. Bogachev}, {\em Measure theory. {V}ol. {I}, {II}}, Springer-Verlag,
  Berlin, 2007.

\bibitem{bongini2016optimal}
{\sc M.~Bongini and G.~Buttazzo}, {\em Optimal control problems in transport
  dynamics}, Math. Models Methods Appl. Sci., 27 (2017), pp.~427--451.

\bibitem{Bonicatto:PHD}
{\sc P.~Bonicatto}, {\em Untangling of trajectories for non-smooth vector
  fields and Bressan's compactness conjecture}, PhD thesis, SISSA Trieste,
  2017.

\bibitem{Bonicatto_2024_MJM}
{\sc P.~Bonicatto}, {\em On the transport of currents}, Milan J. Math., 92
  (2024), pp.~371--395.

\bibitem{BonDNiRin}
{\sc P.~Bonicatto, G.~Del~Nin, and F.~Rindler}, {\em Transport of currents and
  geometric {R}ademacher-type theorems}, Trans. Amer. Math. Soc., 378 (2025),
  pp.~4011--4075.

\bibitem{BonGus}
{\sc P.~Bonicatto and N.~A. Gusev}, {\em Non-uniqueness of signed
  measure-valued solutions to the continuity equation in presence of a unique
  flow}, Atti Accad. Naz. Lincei Rend. Lincei Mat. Appl., 30 (2019),
  pp.~511--531.

\bibitem{Bredies-Carioni-Fanzon-22}
{\sc K.~Bredies, M.~Carioni, and S.~Fanzon}, {\em A superposition principle for
  the inhomogeneous continuity equation with {H}ellinger-{K}antorovich-regular
  coefficients}, Comm. Partial Differential Equations, 47 (2022),
  pp.~2023--2069.

\bibitem{Bredies-Carioni-Fanzon-Romero-21}
{\sc K.~Bredies, M.~Carioni, S.~Fanzon, and F.~Romero}, {\em On the extremal
  points of the ball of the {B}enamou-{B}renier energy}, Bull. Lond. Math.
  Soc., 53 (2021), pp.~1436--1452.

\bibitem{Brediesetal}
\leavevmode\vrule height 2pt depth -1.6pt width 23pt, {\em A generalized
  conditional gradient method for dynamic inverse problems with optimal
  transport regularization}, Found. Comput. Math., 23 (2023), pp.~833--898.

\bibitem{Bredies-Fanzon}
{\sc K.~Bredies and S.~Fanzon}, {\em An optimal transport approach for solving
  dynamic inverse problems in spaces of measures}, ESAIM Math. Model. Numer.
  Anal., 54 (2020), pp.~2351--2382.

\bibitem{Castaing-Valadier77}
{\sc C.~Castaing and M.~Valadier}, {\em Convex analysis and measurable
  multifunctions}, Lectures Notes in Mathematics, Vol. 580, Springer-Verlag,
  Berlin-New York, 1977.

\bibitem{Cavagnari-Lisini-Orrieri-Savare22}
{\sc G.~Cavagnari, S.~Lisini, C.~Orrieri, and G.~Savar\'{e}}, {\em Lagrangian,
  {E}ulerian and {K}antorovich formulations of multi-agent optimal control
  problems: equivalence and gamma-convergence}, J. Differential Equations, 322
  (2022), pp.~268--364.

\bibitem{DalDesSol11}
{\sc G.~{Dal Maso}, A.~{DeSimone}, and F.~{Solombrino}}, {\em Quasistatic
  evolution for cam-clay plasticity: a weak formulation via viscoplastic
  regularization and time rescaling}, Calc. Var. Partial Differential
  Equations, 40 (2011), pp.~125--181.

\bibitem{Dellacherie-Meyer}
{\sc C.~Dellacherie and P.-A. Meyer}, {\em Probabilities and potential},
  vol.~29 of North-Holland Mathematics Studies, North-Holland Publishing Co.,
  Amsterdam-New York; North-Holland Publishing Co., Amsterdam-New York, 1978.

\bibitem{EfeMie06RILS}
{\sc M.~Efendiev and A.~Mielke}, {\em On the rate--independent limit of systems
  with dry friction and small viscosity}, J. Convex Analysis, 13 (2006),
  pp.~151--167.

\bibitem{Fornasier-Lisini-Orrieri-Savare19}
{\sc M.~Fornasier, S.~Lisini, C.~Orrieri, and G.~Savar\'{e}}, {\em Mean-field
  optimal control as gamma-limit of finite agent controls}, European J. Appl.
  Math., 30 (2019), pp.~1153--1186.

\bibitem{Fornasier-Solombrino}
{\sc M.~Fornasier and F.~Solombrino}, {\em Mean-field optimal control}, ESAIM
  Control Optim. Calc. Var., 20 (2014), pp.~1123--1152.

\bibitem{Lisini07}
{\sc S.~Lisini}, {\em Characterization of absolutely continuous curves in
  {W}asserstein spaces}, Calc. Var. Partial Differential Equations, 28 (2007),
  pp.~85--120.

\bibitem{Lisini09}
\leavevmode\vrule height 2pt depth -1.6pt width 23pt, {\em Nonlinear diffusion
  equations with variable coefficients as gradient flows in {W}asserstein
  spaces}, ESAIM Control Optim. Calc. Var., 15 (2009), pp.~712--740.

\bibitem{Lisini16}
\leavevmode\vrule height 2pt depth -1.6pt width 23pt, {\em Absolutely
  continuous curves in extended {W}asserstein-{O}rlicz spaces}, ESAIM Control
  Optim. Calc. Var., 22 (2016), pp.~670--687.

\bibitem{MRS09}
{\sc A.~Mielke, R.~Rossi, and G.~Savar{\'e}}, {\em Modeling solutions with
  jumps for rate-independent systems on metric spaces}, Discrete Contin. Dyn.
  Syst., 25 (2009), pp.~585--615.

\bibitem{MielkeRossiSavare12a}
\leavevmode\vrule height 2pt depth -1.6pt width 23pt, {\em {BV} solutions and
  viscosity approximations of rate-independent systems}, ESAIM: Control,
  Optimisation and Calculus of Variations, 18 (2012), pp.~36--80.

\bibitem{MRS13}
\leavevmode\vrule height 2pt depth -1.6pt width 23pt, {\em Balanced viscosity
  ({BV}) solutions to infinite-dimensional rate-independent systems}, J. Eur.
  Math. Soc. (JEMS), 18 (2016), pp.~2107--2165.

\bibitem{Morandotti-Solombrino}
{\sc M.~Morandotti and F.~Solombrino}, {\em Mean-field {A}nalysis of
  {M}ultipopulation {D}ynamics with {L}abel {S}witching}, SIAM J. Math. Anal.,
  52 (2020), pp.~1427--1462.

\bibitem{Munkres}
{\sc J.~R. Munkres}, {\em Topology}, Prentice Hall, Inc., Upper Saddle River,
  NJ, 2000.
\newblock Second edition of [MR0464128].

\bibitem{Paolini-Stepanov2013}
{\sc E.~Paolini and E.~Stepanov}, {\em Decomposition of acyclic normal currents
  in a metric space}, J. Funct. Anal., 263 (2012), pp.~3358--3390.

\bibitem{Smirnov94}
{\sc S.~K. Smirnov}, {\em Decomposition of solenoidal vector charges into
  elementary solenoids, and the structure of normal one-dimensional flows},
  Algebra i Analiz, 5 (1993), pp.~206--238.

\bibitem{Stepanov-Trevisan17}
{\sc E.~Stepanov and D.~Trevisan}, {\em Three superposition principles:
  currents, continuity equations and curves of measures}, J. Funct. Anal., 272
  (2017), pp.~1044--1103.

\bibitem{Valadier90}
{\sc M.~Valadier}, {\em Young measures}, in Methods of nonconvex analysis
  (Varenna, 1989), A.~Cellina, ed., Springer, Berlin, 1990, pp.~152--188.

\end{thebibliography}

\end{document}